\documentclass[11pt,letterpaper]{amsart}
\usepackage{amsmath,amssymb,mathrsfs,bm,amsthm,amsfonts}
\usepackage{multicol}
\usepackage[inline]{enumitem} 
\usepackage[usenames,dvipsnames]{xcolor}
\usepackage{hyperref}
\usepackage{algorithm}
\usepackage{graphicx}
\usepackage[noend]{algpseudocode}
\usepackage{stmaryrd}
\usepackage{comment}

\hypersetup{%
	colorlinks=true, linkcolor=black,
	citecolor=ForestGreen
}
\usepackage[paper=letterpaper,margin=1in]{geometry}
\newtheorem{theorem}{Theorem}[section]
\newtheorem{lemma}[theorem]{Lemma}
\newtheorem{proposition}[theorem]{Proposition}

\theoremstyle{definition}
\newtheorem{definition}[theorem]{Definition}

\newtheorem{remark}[theorem]{Remark}
\newtheorem{assumption}[theorem]{Assumption}
\numberwithin{equation}{section}
\allowdisplaybreaks
\usepackage{acronym}

\newcommand{\fn}{p_N}
\newcommand{\qn}{q_N}
\newcommand{\gn}{r_N}
\newcommand{\F}{F}
\newcommand{\Bc}{B_0}

\newcommand{\MA}{D_1}
\newcommand{\MB}{D_2}
\newcommand{\MC}{D_3}

\def\lM{ {\tt M}}

\renewcommand{\P}{\mathbb{P}}	
\newcommand{\Ex}{\mathbb{E}}	
\renewcommand{\d}{d}	
\newcommand{\ind}{\mathbf{1}}	

\newcommand{\pv}{\phi_V^{\theta,\infty}}
\newcommand{\fs}{F_V^{\theta,s}}
\newcommand{\fni}{F_V^{\theta, \infty}}

\newcommand{\norm}[1]{\Vert#1\Vert}

\newcommand{\R}{\mathbb{R}} 
\newcommand{\Z}{\mathbb{Z}} 

\newcommand{\e}{\varepsilon}

\newcommand{\ds}{\displaystyle}

\newcommand{\calD}{\mathcal{D}}

\renewcommand{\hat}{\widehat}

\newcommand{\til}{\widetilde}

\usepackage{graphicx}

\title[Large Deviations for Discrete $\beta$-ensembles]{Large Deviations for Discrete $\beta$-ensembles}

\author[S.\ Das]{Sayan Das}
\address{S.\ Das,
	Department of Mathematics, Columbia University,
	\newline\hphantom{\quad \ \ S. Das}
	2990 Broadway, New York, NY 10027 USA
}
\email{sayan.das@columbia.edu}
\author[E.\ Dimitrov]{Evgeni Dimitrov}
\address{E.\ Dimitrov,
	Department of Mathematics, Columbia University,
	\newline\hphantom{\quad \ \ S. Das}
	2990 Broadway, New York, NY 10027 USA
}
\email{esd2138@columbia.edu}

\begin{document}
\begin{abstract} We consider discrete $\beta$-ensembles as introduced by Borodin, Gorin and Guionnet in (Publications math{\' e}matiques de l'IH{\' E}S 125, 1-78, 2017). Under general assumptions, we establish a large deviation principle for their rightmost particle. We apply our general results to two classes of measures that are related to Jack symmetric functions.
\end{abstract}

\maketitle

\tableofcontents

%
\section{Introduction and  main results}\label{Section1}

%
\subsection{Discrete $\beta$-ensembles}\label{Section1.1} A {\em continuous log-gas} (or {\em continuous $\beta$-ensemble}) is a probability distribution on $N$-tuples of reals $x_1>x_{2}>\cdots>x_N$ with density proportional to 
\begin{align}\label{clg}
\prod_{1\le i<j\le N}(x_{i}-x_j)^{\beta}\prod_{i=1}^{N}\exp(-NV(x_i)),
\end{align}	
where $V(x)$ is a continuous function, called the {\em potential}. When $V(x)=x^2$ and $\beta=1,2,$ and $4$, continuous log-gases describe the joint density of eigenvalues of random matrices from the Gaussian Orthogonal, Unitary and Symplectic ensembles \cite{agz}. Such models for general $\beta > 0$ and potential $V(x)$ are now fairly well studied and understood. As the literature on log-gases and their connections to random matrices is vast we will not attempt to give a comprehensive review of it here. The reader is referred to the monographs \cite{agz,DB99,DB09,forrester,mehta,pastur} for a review of the classical results on the subject. For more recent results pertaining to the bulk universality of the measures in (\ref{clg}) we refer to \cite{bourgade2012bulk,bourgade2014universality}; for edge universality see \cite{bourgade2014edge}. The introductions in the last few papers provide a good summary of many of the approaches used to study the measures in (\ref{clg}), starting from the pioneering works of Dyson, Gaudin and Mehta.

In this article, we consider a discrete analogue of \eqref{clg}, called {\em a discrete $\beta$-ensemble}, which was introduced in \cite{bgg}. To define the model we begin with some necessary definitions and notation. Let $\theta > 0, M \in [0, \infty]$ and $N \in \mathbb{N}$. We set
\begin{equation}\label{GenState}
\begin{split}
&\mathbb{Y}_N = \{  (\lambda_1, \dots, \lambda_N)\in \mathbb{Z}^N:  0  \leq \lambda_N \leq \cdots \leq \lambda_1  \}, \mathbb{Y}^M_N = \{  (\lambda_1, \dots, \lambda_N)\in \mathbb{Y}_N : 0  \leq \lambda_1  \leq  M  \}, \\
&\mathbb{W}^{\theta,M}_{N} = \{ (\ell_1, \dots, \ell_N):  \ell_i = \lambda_i + (N - i)\cdot\theta, \mbox{ with } (\lambda_1, \dots, \lambda_N) \in \mathbb{Y}^M_N\}.
\end{split}
\end{equation}
We interpret $\ell_i$'s as locations of particles. If $\theta = 1$ then all particles live on the integer lattice, while for general $\theta$ the particle of index $i$ lives on the shifted lattice $\mathbb{Z} +(N - i) \cdot \theta$. 
	
We define a probability measure $\mathbb{P}_N^{\theta,M}$ on $\mathbb{W}^{\theta,M}_{N}$ through
\begin{equation}\label{PDef}
\mathbb{P}^{\theta,M}_N(\ell_1, \dots, \ell_N) := \frac1{Z_N} \prod_{1 \leq i < j \leq N} Q_{\theta}(\ell_i-\ell_j)  \prod_{i = 1}^N w(\ell_i; N), \qquad Q_{\theta}(x):=\frac{\Gamma(x + 1)\Gamma(x + \theta)}{\Gamma(x)\Gamma(x +1-\theta)}.
\end{equation}
Here $Z_N$ is a normalization constant (called the {\em partition function}) and $w(x,N)$ is a weight function, which is assumed to be positive and continuous for $x \in [0 , M  + (N-1)\theta]$ if $M < \infty$ and if $M = \infty$ we assume that $w(x,N)$ is positive and continuous on $[0, \infty)$ and that for some $T \geq 0$ we have
\begin{equation}\label{WDecay}
\log w(x;N) \leq - [\theta \cdot N   + 1] \cdot \log(1 + (x/N)^2)  \mbox{ for $x \geq T$}.
\end{equation}
Condition \eqref{WDecay} ensures that $Z_N < \infty$ so that \eqref{PDef} defines an honest probability measure when $M = \infty$ (see Lemma \ref{S7WellDef}). When $M$ and $\theta$ are clear from the context we will simply write $\P_N(\cdot)$ to denote the measure \eqref{PDef}.

The measures in (\ref{PDef}) are called {\em discrete $\beta$-ensembles}, a term coined by \cite{bgg} where these measures were introduced as discrete analogues of (\ref{clg}) and extensively studied. To get a sense of why one might consider (\ref{PDef}) as a discrete version of (\ref{clg}) note that $Q_{\theta}(\ell_i-\ell_j) \sim (\ell_i-\ell_j)^{2\theta}$ as $\ell_i - \ell_j \rightarrow \infty$ (see Lemma \ref{InterApprox}), which agrees with \eqref{clg} for $\beta=2\theta$. 

It is worth mentioning that there are other discrete analogues of (\ref{clg}); for example, one can consider the following measure on $\mathbb{W}^{1,\infty}_{N} $ as in (\ref{GenState})
\begin{equation}\label{Coulomb}
\begin{split}
& \mathbb{P}(\ell_1, \dots, \ell_N)  \propto \prod_{1\leq i<j \leq N} |\ell_i-\ell_j|^\beta \prod_{i=1}^{N}w(\ell_i; N).
\end{split}
\end{equation}
When $\beta = 2\theta = 2$ we observe that $Q_{\theta}(x) = x^2$ so that the measures in (\ref{PDef}) and (\ref{Coulomb}) are the same. For general $\beta = 2\theta$ the measures in (\ref{PDef}) and (\ref{Coulomb}) are {\em different}, since for the former each particle $\ell_i$ belongs to a different shifted lattice  $\mathbb{Z} +(N - i) \cdot \theta$ for $i = 1, \dots, N$. While both (\ref{PDef}) and (\ref{Coulomb}) are reasonable discretizations of (\ref{clg}), there is a much higher interest in the former coming from connections to discrete Selberg integrals and integrable probability; specifically, due to connections to uniform random tilings, $(z,w)$-measures and Jack measures --- see \cite[Section 1]{bgg} for more details. We also mention that (\ref{Coulomb}) seemingly lacks the integrability that is present in (\ref{PDef}). In particular, while both global and edge fluctuations have been successfully obtained for (\ref{PDef}) in \cite{bgg} and \cite{huang}, respectively, neither has been established for (\ref{Coulomb}), except when $\theta = 1$.

The main question we investigate in the present paper is about establishing a large deviation principle (LDP) for the location of the rightmost particle $\ell_1$ of the measures in (\ref{PDef}) for general parameters $\theta > 0$ and under general assumptions on the weight functions $w(\cdot ; N)$ -- see Theorem \ref{TMain} in the next section. For the measures in (\ref{Coulomb}) this question was investigated previously in \cite{fe} and \cite{jo} by essentially adapting many of the continuous log-gas techniques from \cite{jocont}. An important structure that is present in (\ref{Coulomb}) and not in (\ref{PDef}) is that the measure is symmetric with respect to $(\ell_1, \dots, \ell_N) $. This symmetry is reflected {\em both} in the form of pairwise interaction of particles, captured by the double product, {\em and } in the state space of the model (all particles live on $\mathbb{Z}$). The absence of this structure in (\ref{PDef}) makes it impossible to directly extend the arguments in \cite{jo} (that appear to strongly depend on symmetry) to the general $\theta > 0$ case of (\ref{PDef}), specifically for establishing an LDP for the upper tail of $\ell_1$. Faced with this difficulty, we had to find a significantly more involved set of arguments (relying on detailed estimates) that would allow us to obtain the correct rate and rate function for the upper tail LDP. We will give more details about the nature of this difficulty and how we overcome it in Section \ref{Section1.2.4} . We also mention that in both \cite{fe} and \cite{jo} there is an error in the rate function for the upper tail LDP -- we discuss this at length in Section \ref{Section7.4}.

In a different direction, \cite{huang} studied the edge distribution of the measures in (\ref{PDef}) and showed that under fairly technical conditions, including the analyticity of the weight function $w(x; N)$, the edge process formed by $\ell_1, \ell_2, \cdots$  (under appropriate scaling) asymptotically matches the edge process for the continuous log-gases in (\ref{clg}). In particular, the rightmost particle converges to the $\beta$-Tracy-Widom distribution. The way these results were established in \cite{huang} is through an intricate use of what are called {\em discrete loop equations} or {\em Nekrasov's equations} \cite{bgg, nek1,nek2,nek3}. Compared to \cite{huang}, the present paper does not utilize loop equations and relies on more direct combinatorial constructions and estimates. We also mention that the arguments in \cite{huang} are specifically catered to answer the question of fluctuations of $\ell_1$ and are not suitable for addressing the large deviations we investigate. Thus, while studying a similar family of models, \cite{huang} has little in common with our paper in terms of both results and methodology.

We now turn to explaining our results in more detail. 
%
\subsection{Main Results}\label{Section1.2} We present here our main results concerning the law of large numbers of the empirical measures and the large deviation of the rightmost particle of discrete $\beta$-ensembles. For simplicity of the exposition we only consider the case when $M = \infty$ in (\ref{PDef}). 

%
\subsubsection{Assumptions and basics from potential theory}\label{Section1.2.1} We summarize the assumptions we make on $w(x;N)$ in the following definition.
\begin{definition}[Assumptions] \label{Assumptions}
Throughout we fix $\theta > 0$. 
\begin{itemize} 
\item We assume that $w(x;N)$ has the form $w(x;N) = \exp\left( - N V_N(x/N)\right)$ for a function $V_N$ that is continuous in $[0, \infty)$. 
\item We assume that $V_N(t) \rightarrow V(t)$ uniformly on compact subsets of $[0, \infty)$, where $V(t)$ is a continuous function on $[0, \infty)$. More specifically, we assume that there is a sequence $r_N > 0$ with $\lim_{N \rightarrow \infty} N^{3/4} \cdot r_N = 0$ and an increasing function $\F_1: (0,\infty) \rightarrow (0, \infty)$ such that the following holds for all $a > 0$ and $N \in \mathbb{N}$
\begin{align}\label{conv-rate}
\sup_{x\in [0,a]}|V_N(x)-V(x)|\le \F_1(a) \cdot r_N. 
\end{align}
\item We assume that $V$ is differentiable on $(0,\infty)$ and that there exists $\Bc > 0$ and an increasing function $\F_2: (0,\infty) \rightarrow (0, \infty)$  such that for all $a > 0$ and $x \in (0, a]$ we have 
\begin{equation}\label{VPotU}
|V'(x)|\le \Bc \left( \F_2(a) + |\log (x)| \right)
\end{equation}
\item We assume that there is a constant $\xi > 0$ such that for all $t \geq 0$ and $N \in \mathbb{N}$
\begin{equation}\label{VPot}
V_N(t) \geq (1 + \xi) \cdot \theta \cdot \log(1 + t^2).
\end{equation}
\end{itemize}
The measures $\mathbb{P}_N$ will then be as in \eqref{PDef} for this choice of $\theta$, $M = \infty$, $w(x;N)$ and $N \geq \theta^{-1} \xi^{-1}$. 
\end{definition}
\begin{remark}\label{RemDecay}
Observe that \eqref{VPot} implies \eqref{WDecay} as long as $N \geq \theta^{-1} \xi^{-1}$ so that $\P_N^{\theta, \infty}$ is well-defined (see Lemma \ref{S7WellDef}). Note that $V$ also satisfies \eqref{VPot} in view of the uniform convergence of $V_N$ to $V$. 	
\end{remark}

To state our theorems we also need some results and notions from potential theory, for which we will use \cite{ds97, saff-tot}. Suppose that $V: [0, \infty) \rightarrow \mathbb{R}$ is a continuous function satisfying (\ref{VPot}). For such a function $V$ we define
\begin{equation}\label{kv}
k^{\theta}_V(x,y) = - \theta \cdot \log|x-y|+ \frac{1}{2} V(x) + \frac{1}{2}V(y).
\end{equation}
For a probability measure $\mu$ on $(\mathbb{R}_+, \mathcal{B}(\mathbb{R}_+))$, where $\mathbb{R}_+ = [0,\infty)$ and $\mathcal{B}(\mathbb{R}_+)$ is the Borel $\sigma$-algebra on $\mathbb{R}_+$, we define the weighted energy integral
\begin{equation}\label{IV} 
\begin{split}
I^{\theta}_V(\mu) & = \int_ 0^\infty\int_0^\infty {\bf 1}\{ x \neq y \} \cdot k^{\theta}_V(x,y) \mu(dx) \mu(dy)  \\ 
& = -\theta \int_ 0^\infty\int_0^\infty {\bf 1}\{ x \neq y \} \log |x-y| \mu(dx) \mu(dy)+\int_0^{\infty}V(x) \mu(dx).
\end{split}
\end{equation}
Since $V$ satisfies (\ref{VPot}), the integral on the first line of (\ref{IV}) is always well-defined and possibly equal to $\infty$. The second representation in (\ref{IV}) is valid whenever both integrals exist and are finite, cf. \cite[pp. 26]{saff-tot}.

For $s \in [\theta, \infty]$ we let $\mathcal{A}^{\theta}_s$ denote the set of all $\phi \in L^1([0,s])$ such that
\begin{align}\label{ainf}
\mbox{ $0 \leq \phi(x) \leq \theta^{-1}$ for Lebesgue a.e. $x \in [0,s]$, and $\int_0^s \phi(x)dx = 1$.}
\end{align}
In words, the family $\mathcal{A}^{\theta}_s$ contains all probability density functions $\phi$, whose support is contained in $[0,s]$ and which are bounded from above by $\theta^{-1}$. The assumption $s \geq \theta$ ensures that the family $\mathcal{A}^{\theta}_s$ is non-empty as it contains $\theta^{-1} \cdot {\bf 1}_{[0, \theta]}$. We will typically not distinguish between the probability measure with density $\phi \in \mathcal{A}^{\theta}_s$ and the function $\phi$ itself, using the same letter for both. In particular, we will write
\begin{equation}\label{IV2}
\begin{split}
I^{\theta}_V (\phi)& = \int_ 0^\infty\int_0^\infty {\bf 1}\{ x \neq y \} \cdot k^{\theta}_V(x,y) \phi(x) \phi(y) dx dy  \\ 
& = -\theta\iint\limits_{x\neq y}\log |x-y|\phi(x)\phi(y)dx dy+\int_0^{\infty}V(x) \phi(x) dx,
\end{split}
\end{equation}
where we note that for $\phi \in \mathcal{A}^{\theta}_s$ and $s < \infty$ all the integrals above are well-defined and finite.

As a special case of \cite[Theorem 2.1]{ds97} we have the following result.
\begin{lemma}\label{iv} Fix $\theta > 0$, $s \in [\theta, \infty]$. If $s < \infty$ suppose that $V$ is continuous on $[0, s]$; if $s = \infty$ suppose that $V$ is continuous on $[0,\infty)$ and satisfies (\ref{VPot}). There is a unique $\phi_V^{\theta,s} \in \mathcal{A}_s^{\theta}$ such that 
\begin{align}\label{fs}
\inf_{\phi \in \mathcal{A}_s^{\theta}}  I^{\theta}_V(\phi) = I^{\theta}_V(\phi_V^{\theta,s}) =: \fs \in \mathbb{R}.
\end{align}
 Moreover, $\phi_V^{\theta,s}$ has compact support. 
\end{lemma}
We will refer to the measure $\phi_V^{\theta,s}$, as the {\em equilibrium measure associated with $V$} (the term {\em extremal measure} is also used in the literature) and write $\operatorname{supp} (\phi_V^{\theta,s})$ for the support of this measure. Lemma \ref{iv} implies that $\operatorname{supp} (\phi_V^{\theta,s})$ is compact and so $b_V^{\theta,s}=\sup \left(\operatorname{supp} (\phi_V^{\theta,s}) \right)$ is finite -- this is the right end point of the support. Moreover, since $\phi_V^{\theta,s} \leq \theta^{-1}$, we have that $b_V^{\theta, s} \geq \theta$. 

%
\subsubsection{Law of large numbers}\label{Section1.2.2} With the notation from the previous section in place we can formulate our first result, Theorem \ref{emp}, which concerns the law of large numbers for the empirical measures associated to \eqref{PDef}. It is proved in Section \ref{Section4.1}.
\begin{theorem}\label{emp} Let $\P_N$ be as in Definition \ref{Assumptions}. If $\vec{\ell} = (\ell_1, \dots, \ell_N)$ is distributed according to $\mathbb{P}_N$ we define the (random) empirical measures $\mu_N(\vec{\ell})$ through
\begin{align}\label{emp-meas}
\mu_N(\vec{\ell}):=\frac1N\sum_{i=1}^N\delta\left(\frac{\ell_i}{N}\right).
\end{align}
The sequence of measures $\mu_N$ converges weakly in probability to $\phi_V^{\theta, \infty}$ from Lemma \ref{iv} in the sense that for any bounded real continuous function $f$ on $[0,\infty)$ the sequence of random variables
\begin{align*}
 \int_{0}^{\infty} f(x)\mu_N(dx) - \int_{0}^{\infty} f(x)\pv(x)dx
\end{align*}
converges to $0$ in probability.
\end{theorem}
	
 We prove an analogue of Theorem \ref{emp} for the measures $\mathbb{P}^{\theta, M}_N$ when $M$ is finite and scales linearly with $N$ as Proposition \ref{S4emp} in the main text. Theorem \ref{emp} is deduced from  Proposition \ref{S4emp} by showing that under the assumptions in Definition \ref{Assumptions}, the measures $\mu_N$ are supported in a large (but finite) window with exponentially high probability so that $M= \infty$ can effectively be replaced with $M = A N$ for some large enough $A$. We call this concentration of the supports of $\mu_N$ an {\em exponential tightness for the empirical measures} $\mu_N$ -- this is Proposition \ref{exptight} in the main text.

\begin{remark} The finite version of Theorem \ref{emp}, namely Proposition \ref{S4emp}, appears as Theorem 5.3 in \cite{bgg} under a different set of assumptions, including $|V_N(x)-V(x)|=O(\log N/N)$. We discuss the proof strategy for this finite case and compare with \cite{bgg} below in this section. As explained above, one of the keys to handling the $M=\infty$ case is exponential tightness for the rightmost particle. \cite[Theorem 10.1]{bgg} provides exponential tightness under the additional assumption that $V_N \equiv V$ for all $N$ and $V$ is eventually increasing. These additional assumptions simplify their argument considerably and we are not sure if its possible to apply their ideas in our setting. Rather, we take a different route in proving Proposition \ref{exptight} that involves proving several a priori estimates (Lemma \ref{mid-region} and Lemma \ref{large-region}) on the number of particles beyond a certain window. The proof of those lemmas rely on a novel construction of a certain transportation map (similar to what is described in Section \ref{Section1.2.4}, see Figure \ref{S1_2}), several technical estimates, as well as intrinsic properties of the $I_V^{\theta}$ functional proved in Section \ref{Section7.3}.
 \end{remark}

The way we prove Proposition \ref{S4emp} (which is the finite $M$ case of Theorem \ref{emp}, where $M \approx \lM N$ for some fixed $\lM > 0$) is by establishing a {\em global large deviation estimate} for the empirical measures $\mu_N$ -- this is Proposition \ref{gldp}. In simple words, we effectively show that 
$$\mathbb{P}^{\theta, M}_N ( \vec{\ell})  \approx C_N \cdot \exp \left( - N^2 I^{\theta}_V(\mu_N) \right),$$
where $C_N$ is some $N$-dependent constant and $I^{\theta}_V$ is as in (\ref{IV}). We have
$$C_N = \exp \left( N^2 F_V^{\theta, \lM + \theta}  + O(N^2 p_N + N^2 q_N + N \log N) \right),$$
provided $|V_N(x) - V(x)| = O(p_N)$ and $|MN^{-1}- \lM | = O(q_N)$, see \eqref{up-bound} and \eqref{pf-bound}. The latter forces the $\mu_N$ to concentrate around the unique minimizer of $I^{\theta}_V$ over $\mathcal{A}^{\theta}_{\lM + \theta}$ at rate $N^2$ so that effectively 
\begin{equation}\label{S1EConc}
\mathbb{P}^{\theta, M}_N ( \vec{\ell})  \approx  \exp \left( \operatorname{Error}_N +  N^2 \left[ F_V^{\theta, \lM + \theta} - I^{\theta}_V(\mu_N) \right] \right),
\end{equation}
where $\operatorname{Error}_N$ is small compared to $N^2$ and $F_V^{\theta, \lM+ \theta}$ is as in (\ref{fs}). The latter equation and the uniqueness of the minimizer $\phi_V^{\theta, \lM + \theta}$ in Lemma \ref{iv} are enough to prove that $\mu_N$ converge weakly in probability to $\phi_V^{\theta, \lM + \theta}$. We mention that analogous arguments to the one we presented above can be found for continuous log-gases in \cite{BenGui, jocont} and for the measures (\ref{Coulomb}) in \cite{fe,jo}.

The fact that the $\mu_N$ concentrate around the unique minimizer of $I^{\theta}_V$ over $\mathcal{A}^{\theta}_{\lM + \theta}$ (and not say the space of all probability measures on $\mathbb{R}_+$ as in the context of continuous log-gases, cf. \cite{jocont}) has to do with the discrete nature of the support of $\mathbb{P}^{\theta, M}_N$. In particular, the fact that $\ell_i - \ell_j \geq (j - i) \theta$ ensures that any weak subsequential limit of $\mu_N$ has a density bounded by $\theta^{-1}$.

Our proof of Proposition \ref{gldp} is inspired by \cite[Corollaries 2.17 and 5.7]{bgg}; however, it is much more quantified than these results in that we obtain uniform bounds on $\operatorname{Error}_N$ both in terms of $N$ and a set of parameters that are related to the measures $\mathbb{P}^{\theta, M}_N$ -- see Remark \ref{gldpRem}. The detailed estimates on $\operatorname{Error}_N$ in Proposition \ref{gldp} have been made not just for the purposes of the present paper, but with an outlook to future applications.

In the proof of Theorem \ref{emp}, the condition $\lim_{N \rightarrow \infty} N^{3/4} \cdot r_N = 0$ in Definition \ref{Assumptions} is required for the application of Proposition \ref{exptight}. This condition can probably be relaxed to $\lim_{N \rightarrow \infty} N^{1/2 + \delta} \cdot r_N = 0$ for some $\delta > 0$, but we stick to $\delta = 1/4$ to make the proofs easier to follow and also because for the applications we have in mind we have $r_N = O(\log N /N)$. Proposition \ref{S4emp} (which is the finite $M$ case of Theorem \ref{emp}) is proved under much milder assumptions on what would be analogues of $\gn$ in the finite $M$ case, introduced in Section \ref{Section2.1}. \\

%
\subsubsection{Large deviation principle}\label{Section1.2.3} 
In order to state our large deviation principle, we require some additional notation that is presented in the following lemma. Its proof can be found in Section \ref{Section7} (see Lemma \ref{S7GJ}). 
\begin{lemma}\label{S1GJ} Suppose that $V(x) : [0, \infty) \rightarrow \mathbb{R}$ is continuous and satisfies (\ref{VPot}) for some $\theta, \xi > 0$. Let $\phi^{\theta, \infty}_V$ be as in Lemma \ref{iv} and let $b_V^{\theta, \infty}$ be the rightmost point of its support. Then the function 
\begin{equation}\label{S1DefG}
G_V^{\theta, \infty}(x) = -2\theta  \int_{0}^{\infty} \log |x - t|\phi_V^{\theta, \infty}(t)dt + V(x)
\end{equation}
is well-defined and continuous on $[0, \infty)$. In addition, the function 
\begin{equation}\label{S1DefJ}
J^{\theta, \infty}_V(x) =\begin{cases} 0 &\mbox{ if $x < b^{\theta, \infty}_V$, } \\  \inf_{y \geq x} G^{\theta, \infty}_V(y)  - G^{\theta, \infty}_V(b_V^{\theta, \infty})   &\mbox{ if $x \geq b_V^{\theta, \infty}$}. \end{cases}
\end{equation}
is continuous on $[0,\infty)$.
\end{lemma}
With the above notation, we state our large deviation result for the rightmost particle $\ell_1$. 

\begin{theorem}\label{TMain} Let $\P_N$ be as in Definition \ref{Assumptions} and let $b_V$ be the rightmost point of the support of $\phi_V^{\theta, \infty}$ as in Lemma \ref{iv}.
\begin{enumerate}[label=(\alph*), leftmargin=25pt]
\item For any $t\in [\theta,\infty)$, we have
\begin{equation}\label{LeftTail}
\lim_{N \rightarrow \infty} \frac{1}{N^2} \log \P_N(\ell_1\le tN) = \fni -  F_V^{\theta, t}, 
\end{equation}
where $F^{\theta, s}_V$ is as in (\ref{fs}). Moreover, $F_V^{\theta, t} >\fni $ for $t \in [\theta, b_V)$ and $ F_V^{\theta, t} = \fni \mbox{ for } t \geq b_V$. 
\item With the same notation as in Lemma \ref{S1GJ} assume that $J^{\theta, \infty}_V(x)  > 0$ for $x > b_V$. Then for any $t\in [\theta,\infty)$, we have
\begin{equation}\label{RightTail}
\lim_{N \rightarrow \infty} \frac{1}{N} \log \P_N(\ell_1\ge tN) = - J^{\theta, \infty}_V(t).
\end{equation}
\end{enumerate}
\end{theorem}
\begin{remark} Theorem \ref{TMain} provides both upper and lower LDP rate function for the rightmost particle. The proof strategy for the lower tail is similar to that for the discrete case \cite{fe,jo}, which in turn is adapted from \cite{jocont}. However, the proof idea for the upper tail is markedly different. We explain this difference in Section \ref{Section1.2.4} in detail.
\end{remark}
\begin{remark} It follows from the form of our state space that $\ell_1 \geq (N-1) \theta$ with probability $1$ for all $N$ and so the probabilities on the left side of (\ref{LeftTail}) and (\ref{RightTail}) become $0$ and $1$ respectively for all large $N$ when $t < \theta$. This is why we restricted Theorem \ref{TMain} to the case $t \geq \theta$, which is the non-trivial case of the large deviations principle. 

The different rates, namely $N$ and $N^2$, for the upper and lower tail LDPs in Theorem \ref{TMain} are the same as for the continuous log-gases in (\ref{clg}) and the measures in (\ref{Coulomb}), see \cite{fe, jo, maju}. To see why these rates are different one can consider the following heuristic. By Theorem \ref{emp} we know that the sequence of empirical measures $\mu_N(\vec{\ell})=\frac1N\sum_{i=1}^N\delta\left(\frac{\ell_i}{N}\right)$ converge to the equilibrium measure $\phi_V^{\theta, \infty}$ and so one expects that the rescaled particle locations $\ell_i/N$, should be close to the quantiles of the equilibrium measure. In particular, one expects that $\ell_1/N$ is close to $b_V$ -- the right endpoint of the support of $\phi_V^{\theta, \infty}$. In order for the event $\{ \ell_1 \geq t N\}$ to occur for $t > b_V$ it is enough for one particle (namely $\ell_1$) to be away from its typical location and instead be far to the right of $b_V N$. On the other hand, in order for the event $\{ \ell_1 \leq t N\}$ to occur for $t < b_V$ one needs order $N$ of the particles to be away from their typical locations and instead be far to the left of $b_V N$. This produces the order $N$ difference between the two rates for the upper and lower tail LDPs.
\end{remark}

\begin{remark} For the measures in (\ref{Coulomb}), the analogue of Theorem \ref{TMain} was established in \cite{fe} and \cite{jo}, although both papers made the same mistake when computing the rate function for the lower tail -- see Section \ref{Section7.4}. 

 In the case of continuous log-gases as in (\ref{clg}), the formula (\ref{LeftTail}) for certain $V$ is a consequence of the result in \cite{BenGui}, see also \cite{HP}. For the spectrum of the GOE (which we recall is the continuous log-gas (\ref{clg}) with $\beta = 1$ and $V(x) = x^2$), the formula (\ref{LeftTail}) was established in \cite{BDG}. To our knowledge, the general form of Theorem \ref{TMain} for continuous log-gases has not been written down anywhere, although as explained in \cite[Remark 2.3]{jo}, the result should be attainable by directly modifying the arguments in that paper.
\end{remark}

 We prove an analogue of Theorem \ref{TMain}(a) and Theorem \ref{TMain}(b) for the measures $\mathbb{P}^{\theta, M}_N$ when $M$ is finite and scales linearly with $N$ as Proposition \ref{FinLLDP} and Proposition \ref{FinULDP}, respectively,  in the main text. Theorem \ref{TMain} is deduced from these two propositions by using the exponential tightness for the empirical measures $\mu_N$, Proposition \ref{exptight}, which effectively allows us to replace $M= \infty$ with $M = A N$ for some large enough $A$.

The proof of Proposition \ref{FinLLDP} (which is the finite $M$ analogue of Theorem \ref{TMain}(a)) is presented in Section \ref{Section4.2} and relies on a careful analysis of the partition function $Z_N$, see Lemma \ref{LemmaS42}, which in turn is based on our work in Section \ref{Section2}. The proof of Proposition \ref{FinULDP} (which is the finite $M$ analogue of Theorem \ref{TMain}(b))  forms the crux of our argument and is split into two parts, given in Sections \ref{Section5.2} and \ref{Section5.3}. The two parts of the proof are devoted to showing that if $M \approx \lM N$ for a fixed $\lM > 0$ we have the following two inequalities for any $t \in (\theta, \lM + \theta)$
\begin{equation}\label{LINFSUP}
\limsup_{N \rightarrow \infty} \frac{1}{N} \log \P^{\theta, M} _N(\ell_1\geq tN) \leq - J^{\theta, \lM + \theta}_V(t) \mbox{ and }\liminf_{N \rightarrow \infty} \frac{1}{N} \log \P^{\theta, M}_N (\ell_1\geq tN) \geq - J^{\theta, \lM + \theta}_V(t),
\end{equation}
where $J^{\theta, \lM + \theta}_V$ has an analogous definition as $J^{\theta, \infty}_V$ in Lemma \ref{S1GJ}, with $\phi_V^{\theta, \infty}$ replaced with $\phi_V^{\theta, \lM + \theta}$ as in Lemma \ref{iv}. The proof of the inequalities in (\ref{LINFSUP}) rely on careful combinatorial constructions, combined with many of the results of Section \ref{Section2} (most notably the global large deviation estimate in Proposition \ref{gldp}).

%
\subsubsection{Ideas behind the proof}\label{Section1.2.4} 
In this section explain the basic ideas behind the proof of the first inequality in (\ref{LINFSUP}). We will try to illustrate how these ideas differ from what would be a natural approach for the measures in (\ref{Coulomb}), which were considered in \cite{fe} and \cite{jo}. We will also explain the origin of the assumptions in Definition \ref{Assumptions}, which are stronger than the assumptions in \cite{fe} and \cite{jo}, but essentially necessary to carry out our arguments. Since $J^{\theta, \lM + \theta}_V(t) = 0$ for $t \in (\theta, b_V^{\theta, \lM + \theta}]$ we see that the first inequality in (\ref{LINFSUP}) trivially holds and so we focus on the case $t > b_V^{\theta, \lM + \theta}$. 

To illustrate the the difficulty in working with (\ref{PDef}), let us first explain how one would (heuristically) establish the first inequality in (\ref{LINFSUP}) for the measures in (\ref{Coulomb}). We mention that the argument we present here is not the one in \cite{fe} and \cite{jo} (although it has a similar spirit), since the latter papers exclusively work with symmetrized versions of (\ref{Coulomb}). Instead, we present an argument that allows for a more direct comparison with what is done in the present paper. 

Suppose that we have a sequence of measures on $\mathbb{W}_{N}^{1, M}$ as in (\ref{Coulomb}), where $M = \lM N$ with $\lM > 0$ fixed and $w(\cdot; N) = \exp\left( - N V_N(x/N)\right)$ with $V_N(x)$ uniformly converging to $V(x)$ on $[0, \lM + 1]$ as $N \rightarrow \infty$. As shown in \cite{fe} and \cite{jo}, the sequence of empirical measures 
$$\mu_N = \frac{1}{N}\sum_{i = 1}^N \delta \left(\frac{\ell_i}{N}\right),$$
converges weakly in probability to a measure $\phi_V \in \mathcal{A}_{\lM + 1}^1$. Let $b_V$ denote the right endpoint of the support of $\phi_V$, fix $t > b_V$ and $\epsilon \in (0, [t- b_V]/4)$. Let $A_t = \{ \vec{\ell} \in \mathbb{W}_{N}^{1, M} : \ell_1 \geq t N\}$, which is the set of particle configurations whose rightmost particle exceeds $tN$, and for $\rho \geq 0$ we let $k_{\rho}(\vec{\ell}) = |\{i \in \{1, \dots, N\} : \ell_i \geq \rho N\}$, which counts the number of particles to the right of $\rho N$. 

Let $\tilde{\tau} : A_t \rightarrow \mathbb{W}_{N}^{1, M}$ be defined by erasing $\ell_1$ from $\vec{\ell}$, inserting a new particle at the first free integer site to the right of $(b_V + \epsilon)N$ and relabeling the resulting particle configuration so that the particles are again in decreasing order, see Figure \ref{S1_1}. We let $\ell_{\operatorname{new}}$ denote the location of the new particle we inserted.

\begin{figure}[ht]
\begin{center}
  \includegraphics[scale = 0.8]{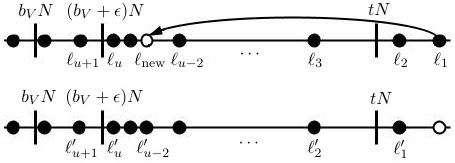}
  \vspace{-2mm}
  \caption{The figure represents $\vec{\ell} \in A_t$ and $\vec{\ell}' = \tilde{\tau}(\vec{\ell})$. Here $u = k_{b_V + \epsilon}(\vec{\ell})$ and the first two integer sites to the right of $(b_V+\epsilon)N$ are occupied by $\ell_u$ and $\ell_{u-1}$. }
  \label{S1_1}
  \end{center}
\end{figure}

For any $\vec{\ell} \in A_t $ and $\vec{\ell}' = \tilde{\tau}(\vec{\ell})$ we have by (\ref{Coulomb}) that 
$$\mathbb{P}_N(\vec{\ell}) = \mathbb{P}_N(\vec{\ell}') \cdot \mathfrak{L}_1(\vec\ell)\mathfrak{L}_2(\vec\ell), \mbox{ where } \mathfrak{L}_2(\vec\ell) =  \exp \left( - N V_N\left(\frac{\ell_1}{N}\right) + N V_N\left(\frac{\ell_{\operatorname{new}}}{N}\right) \right), \mbox{ and }$$
$$\mathfrak{L}_1(\vec\ell) = \exp \left( \beta \sum_{i = 2}^N \log \left| \frac{\ell_i}{N} - \frac{\ell_1}{N}\right| -  \beta \sum_{i = 2}^N \log \left| \frac{\ell_i}{N} - \frac{\ell_{\operatorname{new}}}{N}\right| \right).$$
The fact that $\mu_N$ converge to $\phi_V$ implies that with high probability $k_{b_V + \epsilon}(\vec{\ell}) = o(N)$ and so $\ell_{\operatorname{new}} \approx  (b_V + \epsilon)N$. In addition, since $\mu_N$ is close to $\phi_V$ one expects that with high probability
\begin{equation}\label{LogApprox}
 \beta \sum_{i = 2}^N \log \left| \frac{\ell_i}{N} - z\right|   \approx  N \beta \int_{0}^{b_V} \log|x - z|  \phi_V(x) dx \mbox{ for both $z = \ell_1/N$ and $z =\ell_{\operatorname{new}}/N$}.
\end{equation}
Combining the latter statements and the fact that $V_N(x) \approx V(x)$ we get for $\vec{\ell} \in A_t$ that
\begin{equation}\label{S1Approx}
\begin{split}
\mathbb{P}_N(\vec{\ell}) &\approx \mathbb{P}_N(\vec{\ell}') \cdot \exp \left( - N G_V(\ell_1/N) +  N G_V(b_V + \epsilon) \right) \\
&\leq  \mathbb{P}_N(\vec{\ell}') \cdot \exp \left( - N \inf_{y \geq t} G_V(y) +  N G_V(b_V + \epsilon) \right),
\end{split}
\end{equation}
where $G_V(z) =- \beta \int_{0}^{b_V} \log|x - z|  \phi_V(x) dx + V(z)$ and the second inequality used that $\ell_1/N \geq t$. Summing (\ref{S1Approx}) over $\vec{\ell} \in A_t$ we get 
$$\mathbb{P}_N( \ell_1 \geq t N) \leq  (\lM  N + 1) \cdot \exp \left( - N \inf_{y \geq t} G_V(y) +  N G_V(b_V + \epsilon) \right) \sum_{\vec{\ell'} \in \tilde{\tau}( A_t)}\mathbb{P}_N(\vec{\ell}') ,$$
where the extra $(\lM  N + 1) $ term comes from the fact that $|\tilde{\tau}^{-1}(\vec{\ell}')| \leq (\lM  N + 1) $. Bounding the last sum by $1$, taking logarithms on both sides, dividing by $N$ and letting $N$ tend to infinity we get
$$\limsup_{N\rightarrow \infty} \frac{1}{N} \log \mathbb{P}_N( \ell_1 \geq t N) \leq - \inf_{y \geq t} G_V(y) + G_V(b_V + \epsilon).$$
Letting $\epsilon \rightarrow 0+$ we obtain the analogue of the first inequality in (\ref{LINFSUP}) for the measures in (\ref{Coulomb}).

The above argument is of course only a heuristic. To complete the argument one needs to properly quantify the statement ``$k_{b_V + \epsilon}(\vec{\ell}) = o(N)$ with high probability'' and equation (\ref{LogApprox}). The savvy reader might notice that in (\ref{LogApprox}) we are considering a pairing of the empirical measure with the logarithm function, which has a singularity at $0$; however, as this singularity is very mild one can appropriately truncate the logarithm near $0$ and still get a statement as in (\ref{LogApprox}).

While only heuristic, we believe that one can fill in the details of the above sketch and obtain the analogue of the first inequality in (\ref{LINFSUP}) assuming only that $V_N$ converge uniformly to $V$ and the latter is continuous, which are the assumptions in \cite{fe} and \cite{jo}. I.e., one can significantly relax the $r_N = o(N^{-3/4})$ assumption to $r_N = o(1)$ and completely remove the differentiability assumptions on $V$ in Definition \ref{Assumptions}. \\

We next explain what goes wrong with the above argument when $\mathbb{P}_N$ is as in (\ref{PDef}), as opposed to (\ref{Coulomb}). As before we assume that $M = \lM N$ and $w(\cdot; N) = \exp\left( - N V_N(x/N)\right)$, where $|V_N(x) - V(x)| \leq r_N$ for some sequence $r_N$ that is converging to $0$. We also let as before $A_t = \{ \vec{\ell} \in \mathbb{W}_{N}^{\theta, M} : \ell_1 \geq t N\}$, which is the set of particle configurations whose rightmost particle exceeds $tN$, and for $\rho \geq 0$ we let $k_{\rho}(\vec{\ell}) = |\{i \in \{1, \dots, N\} : \ell_i \geq \rho N\}$, which counts the number of particles to the right of $\rho N$. 

The key player in our earlier argument was the map $\tilde{\tau}$, which mapped a particle configuration whose rightmost particle is far to the right to a new particle configuration with a much more favorable probability, implying that the first configuration was unlikely. The way this map was defined was by taking the rightmost particle, {\em transporting} it to a more favorable location and then {\em relabeling} the particles in the resulting configuration. Both the transportation and relabeling step heavily rely on the symmetry of the state space and the measure, and neither of these steps is possible when we are working with $\mathbb{P}_N$ as in (\ref{PDef}), since the particles occupy {\em different} lattices. To overcome this difficulty, we need a new map $\tau$, which mimics the transportation step while respecting the structure of our state space.

\begin{figure}[ht]
\begin{center}
  \includegraphics[scale = 0.8]{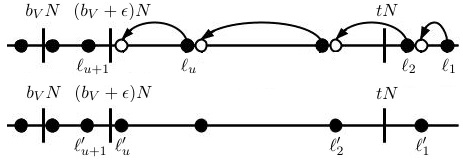}
  \vspace{-2mm}
  \caption{The figure represents $\vec{\ell} \in A_t$ and $\vec{\ell}' = \tau(\vec{\ell})$. Here $u = k_{b_V + \epsilon}(\vec{\ell}) = 4$. }
  \label{S1_2}
  \end{center}
\end{figure}
Let us define our new map $\tau$. From Proposition \ref{gldp} we have that the sequence of empirical measures $\mu_N = \frac{1}{N}\sum_{i = 1}^N \delta \left(\frac{\ell_i}{N}\right),$ concentrates to a measure $\phi_V \in \mathcal{A}_{\lM + \theta}^{\theta}$. Let $b_V$ denote the right endpoint of the support of $\phi_V$, fix $t > b_V$ and $\epsilon \in (0, [t- b_V]/4)$. Let $\vec{\ell} \in A_t$ and $u = k_{b_V + \epsilon }(\vec{\ell})$ be the index of the first particle to the right of $ ( b_V + \epsilon)N$. We set 
$$\nu(\vec{\ell}) := \min_{ a \in \mathbb{Z}_{\geq 0}} \{ \ell_{u+1} + \theta + a  \mid \ell_{u+1} + \theta + a \geq (b_V + \epsilon ) N \},$$
and let $\tau(\vec{\ell}) = \vec{\ell}'$, where 
$$\ell'_k = \ell_k \mbox{ for $k = u+1, \dots, N$, } \ell'_u = \nu(\vec{\ell}) \mbox{, and } \ell'_k = \ell_{k+1} + \theta \mbox{ for $k = 1, \dots, u-1$}.$$
In words, $\tau$ is taking the rightmost particle $\ell_1$ and pushing it to the closest location on $\mathbb{Z} + (N-1) \theta$ that is to the right of $\ell_2$, then it takes $\ell_2$ and pushes it to the closest location on $\mathbb{Z} + (N-2) \theta$ that is to the right of $\ell_3$ and so on until it reaches the $u$-th particle which is pushed to the closest location on $\mathbb{Z} + (N-u) \theta$ that is to the right of $(b_V + \epsilon)N$, see Figure \ref{S1_2}. As can be seen from Figure \ref{S1_2} the particle configuration $\vec{\ell}'$ closely resembles $\vec{\ell}$ with $\ell_1$ being moved to $(b_V + \epsilon)N$; however, in the process we have {\em perturbed} the locations of all the particles between $\ell_1$ and $\ell_u$ by $\theta$.

For any $\vec{\ell} \in A_t $ and $\vec{\ell}' = \tau(\vec{\ell})$ we have by (\ref{PDef}) that 
$$\P_N(\vec\ell) = \P_N(\vec\ell') \cdot \mathfrak{L}_1(\vec\ell)\mathfrak{L}_2(\vec\ell), \mbox{ where }$$
where
$$\mathfrak{L}_1(\vec\ell)  = \frac{\prod_{j =2}^N Q_\theta(\ell_1 - \ell_j)}{\prod_{i = 1}^{u-1} Q_{\theta} (\ell_{i+1} + \theta - \nu(\vec{\ell})) \prod_{j = u+1}^N Q_\theta(\nu(\vec{\ell}) - \ell_j) } \cdot \frac{\prod_{i = 2}^u \prod_{j = u+1}^N Q_\theta(\ell_i - \ell_j) }{\prod_{i = 1}^{u-1} \prod_{j = u+1}^N Q_\theta(\ell_{i+1} + \theta - \ell_j)}$$
and 
$$\mathfrak{L}_2(\vec\ell)   =\exp \left( - N \sum_{i = 1}^u  V_N(\ell_i/N)  + N \sum_{i =2}^u V_N (\ell_i/N + \theta/N) + N V_N(\nu(\vec{\ell})/N)\right).$$
We are interested in obtaining an estimate of the form (\ref{S1Approx}) using the above three equations. 

Focusing on $\mathfrak{L}_2(\vec\ell) $ we wish to replace $V_N$ with $V$ but when we do this we accumulate an error of $O(u N r_N)$ and this error needs to be $o(N)$ so as not to contribute to the rate function once we take logarithms and divide by $N$. The latter demands that we obtain with high probability that $ur_N = o(1)$. In Proposition \ref{gldp}, which is our global large deviations estimate for $\mu_N$, we provide strong enough bounds to show that $u = O(N^{3/4})$ with very high probability, and if we know that $N^{3/4} r_N$ converges to $0$ this would allow us to neglect the error $O(u N r_N)$. This is the origin of the assumption $\lim_{N \rightarrow \infty}N^{3/4} r_N = 0$ in Definition \ref{Assumptions}. We mention that one can use Proposition \ref{gldp} to show that $u = O(N^{1/2 + \delta})$ with high enough probability for any $\delta > 0$, which would allow us to relax $\lim_{N \rightarrow \infty}N^{3/4} r_N = 0$ to $\lim_{N \rightarrow \infty}N^{1/2 + \delta} r_N = 0$, but we choose $\delta = 1/4$ to make the arguments in the text a bit cleaner. We also mention that we did not have the same issue of estimating $u = k_{b_V + \epsilon }(\vec{\ell})$ for $\mathfrak{L}_2(\vec\ell)$ when we were working with (\ref{Coulomb}), since the map $\tilde{\tau}$ allowed for an almost perfect cancellation of the parts in $\mathbb{P}_N(\vec{\ell})/ \mathbb{P}_N(\tilde{\tau}(\vec{\ell}))$ involving $V_N$ (due to symmetry), and that is no longer occurring for us with the map $\tau$.

Replacing $V_N$ with $V$ in $\mathfrak{L}_2(\vec\ell) $ we get 
$$\mathfrak{L}_2(\vec\ell)   \approx \exp \left( - N \sum_{i = 1}^u  V(\ell_i/N)  + N \sum_{i =2}^u V (\ell_i/N + \theta/N) + N V(\nu(\vec{\ell})/N)\right),$$
and here we see that we are forced to estimate $V(x) - V(x + \theta/N)$, for which we require {\em some} information about the modulus of continuity of $V$ on short scales. This is the origin of  our assumption that $V$ is differentiable in Definition \ref{Assumptions}. Again, this issue was not present when we were working with (\ref{Coulomb}) from the almost perfect cancellation of the parts in $\mathbb{P}_N(\vec{\ell})/ \mathbb{P}_N(\tilde{\tau}(\vec{\ell}))$ involving $V_N$. 

The above two paragraphs explain why we need to make additional assumptions compared to \cite{fe} and \cite{jo} to push our argument through. Once these assumptions are in place one can argue using $\nu(\vec{\ell})/N \approx b_V + \epsilon$ that
$$ \mathfrak{L}_2(\vec\ell)  \approx \exp \left( - N V(\ell_1/N) + N V(b_V + \epsilon) \right).$$
Using that $\mathcal{Q}_{\theta}(\ell_i - \ell_j) \approx (\ell_i - \ell_j)^{2\theta}$ (see Lemma \ref{InterApprox}) we also have
$$ \mathfrak{L}_1(\vec\ell)  \approx \exp \left( 2\theta \int_{0}^{b_V} \log|x - \ell_1/N|  \phi_V(x) dx - 2\theta \int_{0}^{b_V} \log|x - b_V - \epsilon|  \phi_V(x) dx \right).$$
The latter two (approximate with high probability) equalities allow us to conclude 
\begin{equation*}
\begin{split}
\mathbb{P}_N(\vec{\ell}) &\approx \mathbb{P}_N(\vec{\ell}') \cdot \exp \left( - N G_V(\ell_1/N) +  N G_V(b_V + \epsilon) \right) \\
&\leq  \mathbb{P}_N(\vec{\ell}') \cdot \exp \left( - N \inf_{y \geq t} G_V(y) +  N G_V(b_V + \epsilon) \right),
\end{split}
\end{equation*}
where $G_V(z) = -2\theta \int_{0}^{b_V} \log|x - z|  \phi_V(x) dx + V(x)$. The last inequality is the analogue of (\ref{S1Approx}) and leads to the first inequality in (\ref{LINFSUP}) the same way we discussed before. 
 
The above description of the main argument for proving (\ref{LINFSUP}) is of course quite reductive, and the full argument, presented in Section \ref{Section5.2}, relies on various technical estimates and statements that are discussed in Section \ref{Section2} with our global large deviation estimate, Proposition \ref{gldp}, playing an indispensable role.

%
\subsection{Outline and notation} The rest of the article is organized as follows. In Section \ref{Section2} we prove a global large deviation estimate for the measures $\mathbb{P}_N^{\theta, M}$ from (\ref{PDef}) when $M$ is finite and scales linearly with $N$ -- the precise statement is given as Proposition \ref{gldp}. In Section \ref{Section3} we consider the measures $\mathbb{P}_N^{\theta, \infty}$ and show that the rightmost particle $\ell_1$ belongs to a window of order $N$ with exponentially high probability --the precise statement is given as Proposition \ref{exptight}. In Section \ref{Section4} we prove analogues of Theorem \ref{emp} and Theorem \ref{TMain}(a) for the measures $\mathbb{P}_N^{\theta, M}$ when $M$ is finite and scales linearly with $N$, by utilizing Proposition \ref{gldp}. The precise statements are given as Proposition \ref{S4emp} and Proposition \ref{FinLLDP}, and are used together with Proposition \ref{exptight} to establish Theorem \ref{emp} and Theorem \ref{TMain}(a) in Section \ref{Section4.1}. In Section \ref{Section5} we prove an analogue of Theorem \ref{TMain}(b) for the measures $\mathbb{P}_N^{\theta, M}$ when $M$ is finite and scales linearly with $N$ -- see Proposition \ref{FinULDP}. In the same section, we prove Theorem \ref{TMain}(b) by using Proposition \ref{FinULDP} and Proposition \ref{exptight}. In Section \ref{Section6} we apply the main results of the paper to two classes of measures that are related to Jack symmetric functions. In Section \ref{Section7} we prove various technical results that are used throughout the paper.\\

Throughout the paper, we will use the following notation. For three sequences $a_N, b_N, c_N$, with $c_N \geq 0$ the equation $a_N = b_N + O(c_N)$ and the inequality $a_N \leq b_N + O(c_N)$ respectively mean that there exists a constant $C > 0$ such that {\em for all } $N \in S \subset \mathbb{Z}$ we have 
\begin{equation}\label{BigODef}
|a_N - b_N| \leq C \cdot c_N \mbox{ and } a_N \leq b_N + C \cdot c_N.
\end{equation}
Whenever we use this notation the constant $C$ will depend on a particularly specified set of parameters, but its value may change from line to line. In addition, the subset $S \subset \mathbb{N}$ will be clear from the context, usually $S = \{n \in \mathbb{Z}: n \geq N_0\}$ for some explicitly defined $N_0$. The important point here is that while the $C$ changes from line to line the set $S$ is the same and the inequalities in (\ref{BigODef}) hold {\em for all } $N \in S$ as opposed to {\em for all large enough $N$}, which is how the big $O$ notation is typically used in the literature. Some of the results we prove, specifically in Sections \ref{Section2} and \ref{Section3}, are established {\em uniformly} in $N$ with very detailed descriptions of the errors that are {\em uniform} in some prescribed set of parameters. For example, Proposition \ref{gldp} holds for all $N \in S=\{n\in \Z: n\ge 2\}$ as opposed to $N \geq N_0$, with $N_0$ depending on the constants $\theta, \gamma, p, a_0, A_0, \dots, A_4$ that appear in that proposition. We have tried to obtain detailed estimates on various quantities of interest not just for the purposes of the present paper, but with an outlook to future applications.

If $A_N, B_N > 0$ and $c_N \geq 0$, the equation $A_N = B_N \exp (O(c_N))$ and the inequality $A_N \leq  B_N \exp (O(c_N))$ respectively mean $a_N = b_N + O(c_N)$ and $a_N \leq b_N + O(c_N)$ with $a_N = \log A_N$ and $b_N = \log B_N$. 

Analogously, for three functions $f(x), g(x), h(x)$ with $h(x) \geq 0$ the equation $f(x) = g(x) + O(h(x))$ and the inequality $f(x) \leq g(x) + O(h(x))$ respectively mean that there exists a constant $C > 0$ such that {\em for all } $x \in S \subset \mathbb{R}$ we have 
\begin{equation*}
|f(x) - g(x)| \leq C \cdot h(x) \mbox{ and } f(x)  \leq g(x) + C \cdot h(x),
\end{equation*}
where $S$ will be explicitly specified and fixed and $C$ will depend on a particularly specified set of parameters, but its value may change from line to line. If $F(x), G(x) > 0$ and $h(x) \geq 0$ the equation $F(x) = G(x) \exp (O(h(x)))$ and the inequality $F(x) \leq  G(x) \exp (O(h(x)))$ respectively mean $f(x) = g(x) + O(h(x))$ and $f(x) \leq g(x) + O(h(x))$ with $f(x) = \log F(x)$ and $g(x) = \log G(x)$.

\subsection{Acknowledgement.} We are grateful to Ivan Corwin, Vadim Gorin and Kurt Johansson for useful comments on earlier drafts of the paper. We are indebted to Alisa Knizel, and the anonymous referees for their many valuable suggestions. We thank Promit Ghosal, Yier Lin, Sumit Mukherjee, and Shalin Parekh for fruitful discussions. SD's research was partially supported from Ivan Corwin's NSF grant DMS:1811143 as well as the Fernholz Foundation's ``Summer Minerva Fellows" program. ED is partially supported by the Minerva Foundation Fellowship and NSF grant DMS:2054703.

%
\section{Global large deviation estimate}\label{Section2}
The aim of this section is to prove a global large deviation estimate (LDE) for the measures $\mathbb{P}_N^{\theta, M}$ from Section \ref{Section1.1} when $M$ is finite and scales linearly with $N$ -- the precise statement is given as Proposition \ref{gldp}. In Section \ref{Section2.1} we state the necessary assumptions under which the global LDE holds and formulate it. In Section \ref{Section2.2} we give the proof of Proposition \ref{gldp} by utilizing Lemma \ref{lemma:prob_bound}, which in turn is proved in Section \ref{Section2.3}. The proof of Lemma \ref{lemma:prob_bound} relies on Lemma \ref{const}, which itself is proved in Section \ref{Section2.4}. Throughout the section we will need various technical lemmas, whose proofs are deferred to Section \ref{Section7}.

%
\subsection{Assumptions}\label{Section2.1} We continue with the same notation as in Section \ref{Section1} and assume that $M$ is finite. We make the following assumptions about the scaling of the weights $w(\cdot;N)$ and $M$ as $N \rightarrow \infty$.

\begin{assumption}\label{as1} Assume that we are given parameters $\theta > 0$, $ A_0 \geq a_0 > 0$ and $\lM \in [a_0, A_0]$. In addition, assume that we have a sequence of parameters $M_N \in \mathbb{Z}_{\geq 0}$ and $0 \leq \qn\leq 1 $ such that
\begin{equation}\label{GenPar}
\left| N^{-1} M_N - \lM\right| \leq A_1 \cdot \qn \mbox{ for some $A_1> 0$}.
\end{equation}
\end{assumption}
\begin{assumption}\label{as2} We assume that $w(x;N)$ in the interval $[0, M_N + (N-1) \cdot \theta]$ has the form
$$w(x;N) = \exp\left( - N V_N(x/N)\right),$$
for a function $V_N$ that is continuous in the interval $I_N = [0, M_N \cdot N^{-1} + (N-1) \cdot N^{-1} \cdot \theta]$. In addition, we assume that we have a continuous function $V$ on $I = [0, \lM + \theta]$ and we extend $V$ to a continuous function on $[0,\infty)$ by setting $V(x) = V(\lM + \theta)$ for $x \geq \lM + \theta$. We assume that there is a sequence $\fn \geq 0$ such that 
\begin{equation}\label{GenPot}
\left| V_N(s) - V(s) \right| \leq A_2 \cdot \fn, \mbox{ for $s \in I_N$, and $|V(s)| \leq A_3$ for $s \in I$, }
\end{equation}
for some constants $A_2,A_3 > 0$. We also require that $V(s)$ is differentiable on $(0, \lM + \theta)$ and for some $A_4 > 0$ we have
\begin{equation}\label{DerPot}
\left| V'(s) \right| \leq A_4 \cdot \left[ 1 + \left| \log |s | \right|  + | \log |s -  \lM - \theta||  \right], \mbox{ for } s \in \left(0, \lM + \theta \right).
\end{equation}
\end{assumption}
\begin{remark}
In applications we will typically have that $\qn, \fn$ are explicit sequences converging to $0$ as $N \rightarrow \infty$. In this case Assumption \ref{as1} would state that $M_N \sim N \lM$ for a fixed positive constant $\lM$ and Assumption \ref{as2} would state that the weights $w(x;N)$ underlying the discrete model $\mathbb{P}^{\theta, M_N}_N$ in (\ref{PDef}) asymptotically look like $e^{-NV(x/N)}$ for some function $V$ that plays the role of an external potential in our model. This external potential is assumed to be differentiable on the interval $(0, \lM + \theta)$, but its derivative is allowed to have logarithmic singularities near the endpoints $0$ and $\lM + \theta$ -- some of the applications we have in mind satisfy this condition.
\end{remark}

\begin{definition}\label{S2PDef} We let $\mathbb{P}^{\theta, M_N}_N$ be as in (\ref{PDef}) for $ N \in \mathbb{N}$, $\theta > 0$, and $w(\cdot; N)$ all satisfying Assumptions \ref{as1} and \ref{as2}. In particular, we have fixed $\lM, V, a_0, A_0,A_1, A_2, A_3, A_4$ and $\fn, \qn$ as in these two assumptions. When $M_N$ and $\theta$ are clear from the context we will write $\mathbb{P}_N$ in place of $\mathbb{P}^{\theta, M_N}_N$. If $\vec{\ell} = (\ell_1, \dots, \ell_N)$ is distributed according to $\mathbb{P}_N$ we recall from (\ref{emp-meas}) that  
\begin{equation*}
\mu_N(\vec{\ell}) =\frac{1}{N}\sum_{i=1}^N\delta\left(\frac{\ell_i}{N}\right),
\end{equation*}
denotes the (random) empirical measure. For a fixed $\vec{\ell} \in \mathbb{W}^{\theta,M_N}_{N} $ as in (\ref{GenState}) we will still write $\mu_N(\vec{\ell})$ for the above (now deterministic) empirical measure. The distinction between these two will be clear from the context.
\end{definition}

Before we state the main result of this section we introduce a bit of notation. Let 
$$\hat{f}(\xi) = \int_{\R} f(x) e^{- i \xi x} dx$$
denote the Fourier transform of $f$. For a compactly supported Lipschitz function $g$ on $\mathbb{R}$ we define
\begin{equation}\label{S2Norms}
\norm{g}_{1/2}:=\left(\int_{\R} |t||\widehat{g}(t)|^2dt\right)^{1/2}, \quad \norm{g}_{\operatorname{Lip}}=\sup_{x\neq y}\left|\frac{g(x)-g(y)}{x-y}\right|.
\end{equation}
Since $g$ is Lipschitz, it is clear that $\norm{g}_{\operatorname{Lip}}$ is finite. The next lemma explains why $\norm{g}_{1/2}$ is finite when $g$ is compactly supported and Lipschitz. It's proof is given in Section \ref{Section7} (see Lemma \ref{S7tech}).
\begin{lemma}\label{tech} Let $g : \mathbb{R} \rightarrow \mathbb{R}$ be a compactly supported Lipschitz function. Then $\norm{g}_{1/2} < \infty$.
\end{lemma}

We may now state our global large deviation estimate for the empirical measures $\mu_N$ from Definition \ref{S2PDef}.
\begin{proposition}\label{gldp} Suppose that $\P_{N}$ is as in Definition \ref{S2PDef} and let $\phi_V^{\theta, \lM + \theta}$ be the equilibrium measure from Lemma \ref{iv}. Then for all $\gamma > 0$, $p\geq 2$, $N \geq 2$, and all compactly supported Lipschitz function $g$ we have
\begin{equation}\label{eq:gldp1}
\begin{split}
& \P_N  \left(\left|  \int_{\R} \hspace{-1mm} g(x)\mu_N(dx)- \int_{\R} g(x)\phi_V^{\theta, \lM + \theta} (x)dx \right|\ge \gamma \norm{g}_{1/2}+\frac{\theta \norm{g}_{\operatorname{Lip}}}{{N}^p} \right)  \\
& \le  \exp \left(-2\pi^2\gamma^2\theta N^2+O\left(N^2 \cdot\fn+ N^2\cdot \qn +N\log N\right) \right),
\end{split}
\end{equation}
where the constant in the big $O$ notation depends on $\theta, p$ and the constants $a_0,A_0,\ldots, A_4$ from Assumptions \ref{as1} and \ref{as2}.
\end{proposition}
\begin{remark}\label{gldpRem}
Proposition \ref{gldp} as well as most of the ideas behind its proof are adapted from \cite[Corollaries 2.17 and 5.7]{bgg}. We remark that the constants in \cite[Corollaries 2.17 and 5.7]{bgg} are not precise due to a misapplication of Plancherel's theorem, see \cite[Corollary 8.3]{DK2020} and the discussion after. We also mention that Proposition \ref{gldp} is more quantified than \cite[Corollaries 2.17 and 5.7]{bgg}, in that it is formulated for general sequences $\fn$ and $\qn$ (in \cite{bgg} the authors consider $\fn = \qn = O(N^{-1} \log (N))$) and in that the constant in the big $O$ notation is shown to depend only on $a_0,A_0,\ldots, A_4$. 
\end{remark}
\begin{remark}
As long as $\lim_{N \rightarrow \infty} \fn = 0 = \lim_{N \rightarrow \infty} \qn$, Proposition \ref{gldp} shows that the empirical measures $\mu_N$ concentrate at rate $N^2$ to the equilibrium measure $\phi_V^{\theta, \lM + \theta}$ from \eqref{fs}. This is why we refer to this proposition as a global large deviation estimate.
\end{remark}

%
\subsection{Proof of Proposition \ref{gldp}} \label{Section2.2} In this section we give the proof of Proposition \ref{gldp}. We continue with the same notation as in Sections \ref{Section1} and \ref{Section2.1}. Before we go into the proof we summarize some notation and results, which will be required.

We start by defining a notion of distance between two measures. For any two compactly supported absolutely continuous probability measures with uniformly bounded densities $\nu(dx)=\nu(x)dx$ and $\rho(dx)=\rho(x)dx$ we define $\mathcal{D}(\nu,\rho)$ as
\begin{align*}
\mathcal{D}^2(\nu,\rho)=-\int_{\R}\int_{\R}\log|x-y|(\nu(x)-\rho(x))(\nu(y)-\rho(y))dx dy.
\end{align*}
There is another representation of $\mathcal{D}$ in terms of Fourier transforms, cf. \cite{BenGui}:
\begin{align}
\mathcal{D}^2(\nu,\rho)  = \int_{0}^{\infty} \frac{1}{\xi}|\widehat{\nu}(\xi)-\widehat{\rho}(\xi)|^2d\xi \label{d2}.
\end{align}

We will require the following lemma, whose proof is given in Section \ref{Section2.3}.
\begin{lemma}\label{lemma:prob_bound} Suppose that $\P_{N}$ is as in Definition \ref{S2PDef}. For any $N \geq 2$ and $\vec\ell\in \mathbb{W}_N^{\theta,M_N}$ as in (\ref{GenState}) one has
\begin{align}\label{up-bound}
\P_N(\vec\ell) \le \exp \left(N^2(F_{V}^{\theta, \lM + \theta}-I^{\theta}_{V_N}(\mu_N)) + O\left(N^2 \cdot\fn+ N^2\cdot \qn +N\log N\right)\right),
\end{align}
where $F_{V}^{\theta, \lM + \theta}$ is defined in \eqref{fs}, $I_{V_N}^\theta$ is defined in (\ref{IV}) and the constant in the big $O$ notation depends on $\theta$ and the constants $a_0,A_0,\ldots, A_4$ from Assumptions \ref{as1} and \ref{as2}.
\end{lemma}	

We will also require the following two lemmas. The first is a straightforward computation and so we omit its proof and the second is proved in Section \ref{Section7} (see Lemma \ref{S7LemmaTech2}).
\begin{lemma}\label{LemmaTech1}
Let $a, b \in \mathbb{R}$ be given such that $a < b$ and put $r = b - a$. Then we have 
\begin{equation}\label{S2Tech2}
\int_{a}^{b} \int_{a}^{b}   \log \left|w - v \right| dw dv =  r^2 \log (r) - 3r^2/2,
\end{equation}
For any $c \in \mathbb{R}$ we have 
\begin{equation}\label{S2Tech22}
\int_{a}^{b}    \log \left|v- c \right|  dv = (b-c) \log|b-c| +(c-a)\log|a-c| + a- b. 
\end{equation}
If $c \geq a \geq 0$ we have 
\begin{equation}\label{S2Tech23}
\int_{0}^{a} \int_0^a \log (c + x -y) dx dy = \frac{(c-a)^2 \log (c-a)}{2} + \frac{(c+a)^2 \log (c+a)}{2} - c^2 \log (c) - \frac{3a^2}{2}, 
\end{equation}
where we have the convention $0 \cdot \log 0 = 0$. 
\end{lemma}

\begin{lemma}\label{LemmaTech2}
Suppose that $\lM, \theta, V$ satisfy the conditions in Assumptions \ref{as1} and \ref{as2}. Let $\phi_V^{\theta, \lM + \theta}$ be the equilibrium measure from \eqref{fs} and define for $x \in [0, \infty)$ the function
\begin{equation}\label{S2DefG}
G^{\theta, \lM + \theta}_V(x) = -2\theta\int_{\R}\log|x-t|\phi_V^{\theta, \lM + \theta}(t)dt + V(x).
\end{equation}
Then the function $G^{\theta, \lM + \theta}_V$ is continuous, $\sup_{ x \in [0, \lM + \theta]} |G^{\theta, \lM + \theta}_V(x)| = O(1)$, and moreover for any $N \geq 2$ one has
$$\sup_{x,y \geq 0, |x-y| \leq \theta/N} |G^{\theta, \lM + \theta}_V(x) - G^{\theta, \lM + \theta}_V(y)|  = O(N^{-1} \log N)$$
where the constants in the big $O$ notations depend on $\theta, A_0, A_3, A_4$ from Assumptions \ref{as1} and \ref{as2}. 
\end{lemma}

With the above results in place, we can proceed with the proof of Proposition \ref{gldp}.
\begin{proof}[Proof of Proposition \ref{gldp}] We introduce a few notations. Fix a parameter $p \geq 2$, and $\vec\ell\in \mathbb{W}_{N}^{\theta,M_N}$ as in (\ref{GenState}). Let $\til{\mu}_N$ denote the convolution of the empirical measure $\mu_N(\vec{\ell})$ from Definition \ref{S2PDef} with the uniform measure on $[0, \theta N^{-p}]$.  We denote the density of $\til\mu_N$ by $\psi_N(x)$. In the proof below the constants in all big $O$ notations will depend on $\theta, p$ and the constants $a_0,A_0,\ldots, A_4$ from Assumptions \ref{as1} and \ref{as2} -- we will not mention this further.\\
	
\noindent\textbf{Step 1.} We claim that the following inequality holds for all and $\vec\ell\in \mathbb{W}_N^{\theta,M_N}$ 
\begin{align}\label{ter}
F_V^{\theta, \lM + \theta}-I_{V_N}^{\theta}(\mu_N(\vec{\ell})) \le -\theta \cdot \calD^2(\psi_N,\phi_V^{\theta, \lM + \theta })+O(\fn+\qn+N^{-1}\log N).
\end{align}
We wil establish (\ref{ter}) in the steps below. Here we assume its validity and conclude the proof of the proposition. Combining \eqref{ter} and \eqref{up-bound} we get for all $\vec\ell\in \mathbb{W}_N^{\theta,M_N}$
\begin{align}\label{bd}
\P_N(\vec{\ell})\le \exp(-\theta  N^2 \cdot \calD^2(\psi_N, \phi_V^{\theta, \lM + \theta})+O(N^2 \cdot \fn+ N^2 \cdot \qn +N\log N)).
\end{align}
Note that by the triangle inequality we have
\begin{align*}
& \left|\int_{\R} g(x)\mu_N(dx)-\int_{\R} g(x)\phi_V^{\theta, \lM + \theta}(x)dx \right| \le \left|\int_{\R} g(x)\mu_N(dx)-\int_{\R} g(x)\til\mu_N(dx) \right|  \\ 
& +  \left|\int_{\R} g(x)\til\mu_N(dx)-\int_{\R} g(x)\phi_V^{\theta, \lM + \theta}(x)dx \right| \leq \theta N^{-p} \norm{g}_{\operatorname{Lip}} + \left|\int_{\R} g(x)\psi_N(x) dx -\int_{\R} g(x)\phi_V^{\theta, \lM + \theta}(x)dx \right|.		
\end{align*}
For the second term, note that $g$ and $\psi_N(x) -\phi_V^{\theta, \lM + \theta}(x)$ are bounded and belong to $L^1(\mathbb{R})\cap L^2(\mathbb{R})$. Thus by the Plancherel's formula (see e.g. \cite[Theorem 7.1.6]{Hor}) and the Cauchy-Schwarz inequality we have
\begin{align*}
&\left|\int_{\R} g(x)\psi_N(x) dx -\int_{\R} g(x)\phi_V^{\theta, \lM + \theta}(x)dx \right| = \frac{1}{2\pi} \left|\int_{\R} \left(\sqrt{|\xi|}\widehat{g}(\xi)\right)\frac{\widehat{\psi_N}(\xi)-\widehat{\phi_V^{\theta, \lM + \theta}}(\xi)}{\sqrt{|\xi|}}d\xi\right| \\ 
& \le  \frac1{2\pi}\norm{g}_{1/2}\sqrt{\int_{\R}\frac{|\widehat{\psi_N}(\xi)-\widehat{\phi_V^{\theta, \lM + \theta}}(\xi)|^2}{|\xi|}d\xi} = \frac1{\pi\sqrt2}\norm{g}_{1/2}\mathcal{D}(\psi_N,\phi_V^{\theta, \lM + \theta}),
\end{align*}
where the last equality follows from the second representation of $\mathcal{D}$ as noted in \eqref{d2}. 

Combining the above two inequalities we get
\begin{align*}
&\P_N \hspace{-1mm} \left(\left| \hspace{-0.5mm} \int_{\R} \hspace{-1mm} g(x)\mu_N(dx)- \hspace{-1mm}\int_{\R} \hspace{-2mm} g(x)\phi_V^{\theta, \lM + \theta} (x)dx \right|\ge \gamma \norm{g}_{1/2}+\frac{\theta \norm{g}_{\operatorname{Lip}}}{{N}^p}\hspace{-1mm} \right)  \le \P_N(\mathcal{D}(\psi_N,\phi_V^{\theta, \lM + \theta})\ge \pi\sqrt{2}\gamma)  \\
&= \sum_{\vec\ell \in  \mathbb{W}_{N}^{\theta,M_N} :  \mathcal{D}^2(\psi_N,\phi_V^{\theta, \lM + \theta})\ge 2\pi^2\gamma^2} \P_N (\vec\ell) \leq \exp(-2\pi^2\gamma^2 \theta N^2+O(N^2 \cdot \fn+ N^2 \cdot \qn +N\log N) ),
\end{align*}
where the last inequality holds by (\ref{bd}) and $|\mathbb{W}_{N}^{\theta,M_N}| = e^{O(N\log N)}$. The last equation gives \eqref{eq:gldp1}.\\

\noindent\textbf{Step 2.} In the remainder we focus on proving (\ref{ter}). In this step we show that for all $\vec{\ell} \in \mathbb{W}_{N}^{\theta,M_N}$
\begin{align}\label{1ter}
I_{V_N}^{\theta}(\mu_N(\vec\ell))=I_V^{\theta}(\til\mu_N)+O\left(\fn+\qn+N^{-1}\log N \right).
\end{align}
Towards this end, note that by \eqref{GenPot}, we have 
\begin{equation}\label{1terR1}
I^{\theta}_{V_N}(\mu_N)=I^{\theta}_{V}(\mu_N)+O\left( \fn\right),
\end{equation}
where we recall that $V$ was continuously extended to $[0,\infty)$ by setting $V(x) = V(\lM + \theta)$ for $x \geq \lM + \theta$. Furthermore, if $u, u'$ are independent uniform random variables on $[0, \theta N^{-p}]$, we have
\begin{align*}
I_V^{\theta}(\til\mu_N) & = -\theta \cdot \Ex_{u,u'}\left[ \int_0^{\infty}\int_0^{\infty} \log| x+ u - y - u'| \mu_N(dx) \mu_N(dy)\right]+\Ex_{u}\left[\frac1N\sum_{i=1}^N V\left(\frac{\ell_i}{N}+u \right)\right]  \\ 
& = I_V^{\theta}(\mu_N)-\theta \cdot \frac{1}{N} \cdot \Ex_{u,u'}\left[  \log| u -  u'|  \right] - \theta \cdot \Ex_{u,u'}\left[\frac1{N^2}\sum_{1\leq i\neq j \leq N} \log\left|1+\frac{N(u-u')}{\ell_i-\ell_j}\right|\right]\\
&+ \frac1N\sum_{i=1}^N\Ex_{u}\left[ V\left(\frac{\ell_i}{N}+u\right)-V\left(\frac{\ell_i}{N}\right)\right].
\end{align*}
We now have that 
$$ \frac{1}{N} \cdot \Ex_{u,u'}\left[ \log| u -  u'| \right] = \theta^{-2} N^{2p - 1} \int_0^{\theta N^{-p}} \int_0^{\theta N^{-p}} \log |w-v|dw dv = O(N^{-1} \log N),$$
where in the last equality we used (\ref{S2Tech2}). In addition, since $p \geq 2$ and $|\ell_i - \ell_j| \geq \theta$  for $i\neq j$
$$ \Ex_{u,u'}\left[\frac1{N^2}\sum_{1\leq i\neq j \leq N} \log\left|1+\frac{N(u-u')}{\ell_i-\ell_j}\right|\right] = O(N^{1-p}).$$
Combining the last three equations and using that $p \geq 2$ we get
\begin{equation}\label{1terR2}
I_V^{\theta}(\til\mu_N)  = I_V^{\theta}(\mu_N) + \frac1N\sum_{i=1}^N\Ex_{u}\left[ V\left(\frac{\ell_i}{N}+u\right)-V\left(\frac{\ell_i}{N}\right)\right] +O(N^{-1} \log N). 
\end{equation}
Finally, to bound the sum in (\ref{1terR2}) we consider different bounds for the summands depending on whether $\ell_i /N$ is in $[\theta/N, \lM + \theta - \theta/N]$ or not. Since the distance between $\ell_i$ and $\ell_{i+1}$ is at least $\theta$, we have by Assumption \ref{as1} that the number of $\ell_i$'s outside of $[\theta/N, \lM + \theta - \theta /N]$ is at most $O(1 + N \qn)$, and for all such $i$ we have by (\ref{GenPot}) the trivial bound 
$$ \left| \Ex_{u}\left[ V\left(\frac{\ell_i}{N}+u\right)-V\left(\frac{\ell_i}{N}\right)\right]\right| \leq 2 A_3 .$$
For $\ell_i$ inside $[\theta/N, \lM + \theta - \theta /N]$ we can apply the mean value theorem so that for any $u \in [0, \theta N^{-p}]$ 
$$ \left| V\left(\frac{\ell_i}{N}+u\right)-V\left(\frac{\ell_i}{N}\right) \right|  = |u| |V'(t)| = O( N^{-p} \log (N)),$$
where $t$ is a point in the interval with endpoints $\frac{\ell_i}{N}+u$ and $\frac{\ell_i}{N}$, and the last equality used (\ref{DerPot}). Combining the last two estimates, we see that 
\begin{equation}\label{1terR3}
\frac1N\sum_{i=1}^N\Ex_{u}\left[ V\left(\frac{\ell_i}{N}+u\right)-V\left(\frac{\ell_i}{N}\right)\right] = O(N^{-1} + \qn) +  O( N^{1-p} \log (N)).
\end{equation}
Since $p \geq 2$, we see that (\ref{1terR1}), (\ref{1terR2}) and (\ref{1terR3}) together imply (\ref{1ter}).\\

\noindent\textbf{Step 3.} In this step we invoke some of the properties of the equilibrium measure to show that
\begin{align}\label{s2}
F_V^{\theta, \lM + \theta}-I_{V}^{\theta}(\til{\mu}_N(\vec{\ell}))  \le  -\theta \cdot \mathcal{D}^2(\psi_N,\phi_V^{\theta, \lM + \theta})+O(N^{-1}\log N + \qn),
\end{align}
Where we recall that $\psi_N$ is the density of $\til\mu_N(\vec{\ell})$. Notice that (\ref{1ter}) and (\ref{s2}) together imply (\ref{ter}), and so we have reduced the proof of the proposition to showing (\ref{s2}).

Let $G_V^{\theta, \lM+\theta}$ be as in (\ref{S2DefG}). By \cite[Theorem 2.1 (c)]{ds97} and the Lebesgue differentiation theorem \cite[Chapter 3, Theorem 1.3]{SteinReal} it follows that there exists a constant $\kappa$ such that 
\begin{align}\label{DSGI}
G_V^{\theta, \lM+\theta}(x)-\kappa \begin{cases} = 0 & \mbox{a.e. }x \in [0, \lM + \theta] \mbox{ such that }\phi_V^{\theta, \lM + \theta}(x) \in (0,\theta^{-1}), \\
 \leq 0 & \mbox{a.e. }x \in [0, \lM + \theta] \mbox{ such that }\phi_V^{\theta, \lM + \theta}(x)=\theta^{-1}, \\ 
\geq  0 & \mbox{a.e. }x \in [0, \lM + \theta] \mbox{ such that }\phi_V^{\theta, \lM + \theta}(x)=0.
\end{cases}
\end{align}
A simple calculation gives 
\begin{align*}
F_V^{\theta, \lM + \theta}-I_V^{\theta}(\til{\mu}_N(\vec{\ell})) & = -\theta \cdot \calD^2(\phi_V^{\theta, \lM + \theta},\psi_N) - \int_{0}^{\infty} G_V^{\theta, \lM+\theta}(x)\left[\psi_N(x)-\phi_V^{\theta, \lM + \theta}(x)\right]\d x \\ 
& = -\theta \cdot \calD^2(\phi_V^{\theta, \lM + \theta},\psi_N) - \int_{0}^{\infty} (G_V^{\theta, \lM+\theta}(x)-\kappa) \left[\psi_N(x)-\phi_V^{\theta, \lM + \theta}(x) \right]\d x \\
&=  - \theta \cdot \calD^2(\phi_V^{\theta, \lM + \theta},\psi_N) - \int_{0}^{\lM + \theta} \hspace{-2mm} (G_V^{\theta, \lM+\theta}(x)-\kappa) \left[\psi_N(x)-\phi_V^{\theta, \lM + \theta}(x) \right]\d x + O( \qn),
\end{align*}
where in the second equality we used that $\psi_N$ and $\phi_V^{\theta, \lM + \theta}$ are probability density functions, while in the last equality we used Lemma \ref{LemmaTech2} and the fact that $\til{\mu}_N(\vec{\ell})$ has mass at most $O(\qn)$ outside the interval $[0, \lM + \theta]$.

Let $R:=\{x \in (0, \lM + \theta) \mid G_V^{\theta, \lM+\theta}(x) < \kappa\}.$  Note that for a.e. $x \in [0, \lM + \theta] \setminus R$ we have
\begin{equation}\label{S2RT1}
(G_V^{\theta, \lM+\theta}(x)-\kappa) \left[\psi_N(x)-\phi_V^{\theta, \lM + \theta}(x) \right] \geq 0.
\end{equation}
Indeed, we know that if $x \in (0, \lM + \theta)$ and $G_V^{\theta, \lM+\theta}(x) - \kappa =0$ then (\ref{S2RT1}) trivially holds. Thus we only need to show (\ref{S2RT1}) when $x \in (0, \lM + \theta)$ and $G_V^{\theta, \lM+\theta}(x) - \kappa > 0$. By (\ref{DSGI}) a.e. such $x$ satisfies $\phi_V^{\theta, \lM + \theta}(x)  =0$ and so (\ref{S2RT1}) clearly holds.

Combining (\ref{S2RT1}) with the fact that for a.e. $x \in R$ we have $\phi_V^{\theta, \lM + \theta}(x)=\theta^{-1}$ (by (\ref{DSGI})), we get 
\begin{align}\label{intt}
F_V^{\theta, \lM + \theta }-I_V^{\theta}(\til{\mu}_N(\vec{\ell}))  \le - \theta \cdot \calD^2(\phi_V^{\theta, \lM + \theta},\psi_N) - \int_{R } (G_V^{\theta, \lM+\theta}(x)-\kappa)(\psi_N(x)-\theta^{-1})\d x + O(\qn).
\end{align}
We claim that 
\begin{align}\label{intt2}
\int_{R } (\kappa - G_V^{\theta, \lM+\theta} (x))(\psi_N(x)-\theta^{-1})\d x \leq O(N^{-1} \log N).
\end{align}
If true, then (\ref{intt}) and (\ref{intt2}) together will imply (\ref{s2}). We will prove (\ref{intt2}) in the next step.\\

\noindent\textbf{Step 4.} In this final step we prove (\ref{intt2}). From Lemma \ref{LemmaTech2} we know that $G_V^{\theta, \lM+\theta}(x)$ is a continuous function and so $R$ is an open set. Consequently, we can write $R$ as the union of countably many disjoint open intervals: $\cup_{i\in J} (s_i,t_i)$. If $t_i-s_i\ge \theta/ N$, we divide the interval into further segments of length exactly $\theta/N$, leaving out edge segments $(s_i,x)$ and $(y,t_i)$ with length at most $\theta/N$. Hence
\begin{align*}
R= \left[\bigcup_{i\in J_1} [r_i,r_i+ \theta/N]\right] \cup \left[\bigcup_{i\in J_2} (s_i,t_i)\right]. 
\end{align*}
Here $J_1$ is a finite collection, with $|J_1| \le N \theta^{-1} (\lM + \theta)$. For each $i \in J_2$ the interval $(s_i,t_i)$ is of length at most $\theta/N$ and has at least one end-point on the boundary of $R$. At a boundary point $x$ of $R$ we have $G_V^{\theta, \lM+\theta}(x)=\kappa$ (using the continuity of $G_V^{\theta, \lM+\theta}$ from Lemma \ref{LemmaTech2} again), and so
\begin{align*}
&\sum_{i\in J_2} \int_{s_i}^{t_i} (\kappa -G_V^{\theta, \lM+\theta}(x))(\psi_N(x)-\theta^{-1})\d x  \le  \sup_{x,y \geq 0, |x-y|\le \theta /N} |G_V^{\theta, \lM+\theta}(x)-G_V^{\theta, \lM+\theta}(y)| \\
& \times \sum_{i\in J_2} \int_{s_i}^{t_i} (\psi_N(x)+ \theta^{-1})\d x  \le  \left(2 + \theta^{-1} \lM  \right)\sup_{x,y \geq 0, |x-y|\le  \theta /N} |G_V^{\theta, \lM+\theta}(x)-G_V^{\theta, \lM+\theta}(y)|.
\end{align*}
Since $\ell_i/N$ are at least $\theta/N$ apart, we have for each interval $[w, w+ \theta/N]$ that
$$ 0 \leq  \int_{w}^{w+ \theta/N} \psi_N(x) \le \frac{1}{N}, \mbox{ and so } \int_{w}^{w+ \theta/N} (\psi_N(x)-\theta^{-1}) \le 0.$$ 
The latter implies that for $i \in J_1$ we have
 \begin{align*}
\int_{r_i}^{r_i+ \theta/N}  (\kappa - G_V^{\theta, \lM+\theta}(x))(\psi_N(x)- \theta^{-1})\d x & \le  \int_{r_i}^{r_i+\theta/N} (G_V^{\theta, \lM+\theta}(r_i) - G_V^{\theta, \lM+\theta}(x))(\psi_N(x)- \theta^{-1})\d x    \\
& \le  \frac{2}{N}\sup_{x,y \geq 0, |x-y|\le \theta/N} |G_V^{\theta, \lM+\theta}(x)-G_V^{\theta, \lM+\theta}(y)|,
 \end{align*}
where in the first inequality we used that $G_V^{\theta, \lM+\theta}(r_i) - \kappa < 0$, which is true as $r_i \in R$ by construction. Combining the last few statements we conclude that 
\begin{equation*}
\begin{split}
& \int_{R } (\kappa - G_V^{\theta, \lM+\theta}(x))(\psi_N(x)-\theta^{-1})\d x  \\
&\leq \sup_{x,y \geq 0, |x-y|\le  \theta /N} |G_V^{\theta, \lM+\theta}(x)-G_V^{\theta, \lM+\theta}(y)| \cdot \left( 2 + \theta^{-1} \lM + \frac{2|J_1|}{N} \right) = O(N^{-1} \log N),
\end{split}
\end{equation*}
where in the last equality we used Lemma \ref{LemmaTech2}  and the fact that $|J_1| \le N \theta^{-1} (\lM + \theta)$. The last equation gives (\ref{intt2}), which concludes the proof of the proposition.
\end{proof}

%
\subsection{Proof of Lemma \ref{lemma:prob_bound}} \label{Section2.3} In this section we give the proof of Lemma \ref{lemma:prob_bound}. We continue with the same notation as in Sections \ref{Section1}, \ref{Section2.1} and \ref{Section2.2}. Before we go into the proof we summarize some notation and results, which will be required.

\begin{lemma}\label{ivi} Fix $\theta>0$, $s \in [0, \infty)$, $\alpha \in [\theta/2,\theta]$  and let $V$ be a continuous function on $[0, s+\theta]$. Let $I_V^{\theta}$ be as in (\ref{IV}) and $\mathcal{A}_{s + \theta}^\alpha$ as in (\ref{ainf}). Then for each $\alpha \in [\theta/2,\theta]$ there is a unique $\Phi_{\alpha} \in \mathcal{A}_{s + \theta}^\alpha$ such that 
$$  \inf_{\phi \in \mathcal{A}_{s + \theta}^{\alpha}}  I^{\theta}_V(\phi) = I^{\theta}_V(\Phi_{\alpha}) =: F_\alpha \in \mathbb{R}.$$
Moreover, $\Phi_{\alpha}$ has compact support and for all $\alpha \in [\theta/2,\theta]$ we have 
\begin{equation}\label{S2SuppBound}
\sup \left( \operatorname{supp}\left( \Phi_\alpha \right) \right) \leq \sup \left( \operatorname{supp}\left( \Phi_\theta \right) \right).
\end{equation}
\end{lemma}
\begin{remark} We remark that $\Phi_{\alpha}$ depends on $V, s$ and $\theta$ even though this dependence is not reflected in the notation. In addition, when $\alpha = \theta$ the function $\Phi_\alpha$ agrees with $\phi_V^{\theta, \lM + \theta}$ from (\ref{fs}).
\end{remark}
\begin{proof}
The existence and uniqueness of $\Phi_{\alpha}$, the compactness of its support and the fact that $F_\alpha \in \mathbb{R}$ follows from \cite[Theorem 2.1]{ds97}. Equation (\ref{S2SuppBound}) follows from \cite[Theorem 2.8]{ds97}.
\end{proof}

The following lemma asserts that $Q_{\theta}(x)$ as in (\ref{PDef}) behaves like $x^{2\theta}$ and provides a quantitative error bound. Its proof is postponed until Section \ref{Section7} (see Lemma \ref{S7InterApprox}).
\begin{lemma}\label{InterApprox} Fix $\theta > 0$. Then for any $x \geq \theta$ we have
\begin{equation}\label{Sandwich1}
Q_\theta(x):=\frac{\Gamma(x + 1)\Gamma(x+ \theta)}{\Gamma(x)\Gamma(x +1-\theta)} = x^{2\theta} \cdot \exp(O(x^{-1})),
\end{equation}
where the constant in the big $O$ notation can be taken to be $(1 + \theta)^3$. 
\end{lemma}

\begin{lemma}\label{pmf-exact} Fix $N\ge 2$ and  $M\in [0,\infty]$. Suppose that $\P_{N}^{\theta,M}$ is as in \eqref{PDef} and $\mathbb{W}_{N}^{\theta,M}$ is as in \eqref{GenState}. For every $\vec\ell\in  \mathbb{W}_{N}^{\theta,M}$ we have
\begin{equation}\label{pf-bound}
\mathbb{P}_N^{\theta,M}(\vec\ell) = Z_N^{-1} \cdot \exp \left(\theta N(N-1) \log N -N^2 I^{\theta}_{V_N}(\mu_N (\vec{\ell}))+ O(N \log N) \right),
\end{equation}
where $\mu_N (\vec{\ell})$ is empirical measure corresponding to $\vec\ell$ defined in \eqref{emp-meas}, $I_{V_N}^{\theta}$ is defined in \eqref{IV} and the constant in the big $O$ notation depends only on $\theta$.
\end{lemma}

\begin{proof} 
	From the definition of $\mathbb{P}_N^{\theta,M}$ and $I_V^\theta$ we see that it suffices to show that 
	\begin{equation}\label{S11}
		\prod_{1 \leq i < j \leq N}  \frac{\Gamma(\ell_i - \ell_j + 1)\Gamma(\ell_i - \ell_j + \theta)}{\Gamma(\ell_i - \ell_j)\Gamma(\ell_i - \ell_j +1-\theta)}   = \prod_{1 \leq i < j \leq N} (\ell_i - \ell_j)^{2\theta} \cdot \exp \left( O \left(N \log N \right) \right).
	\end{equation}
	Since $\ell_i - \ell_j \geq \theta$, we may apply Lemma \ref{InterApprox}, which gives
	$$\prod_{1 \leq i < j \leq N}  \frac{\Gamma(\ell_i - \ell_j + 1)\Gamma(\ell_i - \ell_j + \theta)}{\Gamma(\ell_i - \ell_j)\Gamma(\ell_i - \ell_j +1-\theta)}   = \prod_{1 \leq i < j \leq N} (\ell_i - \ell_j)^{2\theta} \cdot \exp \left( O \left(\frac{1}{\ell_i - \ell_j}\right) \right).$$
	Using $\ell_i -\ell_j \geq (j-i) \theta$ we get that for a large enough universal constant $C > 0$ and all $N \geq 2$
	$$ 0 \leq \sum_{1 \leq i < j \leq N} \frac{1}{\ell_i - \ell_j}  \leq \sum_{1 \leq i < j \leq N} \frac{1}{\theta (j-i)} \leq \frac{1}{\theta} \sum_{i = 1}^N \sum_{j = 1}^N \frac{1}{j} \leq C \theta^{-1} N \log N .$$
	The last two statements imply (\ref{S11}) and hence (\ref{pf-bound}).
\end{proof}
To state the next lemma we require some additional notation. For $N,r \in \mathbb{N}$ with $N/2\le r\le N$ we set $\alpha =r\theta/N$. For any $\vec{\ell}\in \mathbb{W}_{r}^{\theta, M_N}$ (recall that this was defined in \eqref{GenState}) we define
\begin{equation}\label{munr}
\mu_{N,r}(\vec{\ell}):=\frac1r\sum_{i=1}^r \delta \left(\frac{\ell_i}{N}\right).
\end{equation} 
to be the corresponding empirical measure.

We will require the following lemma, whose proof is given in Section \ref{Section2.4}.
\begin{lemma}\label{const} Suppose that Assumptions \ref{as1} and \ref{as2} from Section \ref{Section2.1} hold. For any $N \geq 2$, $r \in \mathbb{N}$ with $r \in [N/2, N]$ there exists $\vec{\ell}' = ({\ell}'_1, \dots, {\ell}'_r) \in \mathbb{W}_r^{\theta, M_N}$ such that
\begin{align}\label{clo}
r^2I_V^{\theta}\left(\mu_{N,r}(\vec\ell') \right) \le r^2 \inf_{ \phi \in \mathcal{A}^{\alpha}_{\lM + \alpha}} I_V^{\theta}(\phi) +O(N\log N + N^2 \cdot \qn),
\end{align}
where $\mu_{N,r}$ is as in (\ref{munr}) and $\alpha = \theta r/N$. The constant in the big $O$ notation depends on $\theta$ and the constants $a_0,A_0, \ldots, A_4$ from Assumptions \ref{as1} and \ref{as2}.
\end{lemma}	

With the above results in place, we can proceed with the proof of Lemma \ref{lemma:prob_bound}.
\begin{proof}[Proof of Lemma \ref{lemma:prob_bound}] In the proof below the constants in all big $O$ notations will depend on $\theta$ and the constants $a_0, A_0,\ldots, A_4$ from Assumptions \ref{as1} and \ref{as2} -- we will not mention this further. 

Let $\vec\ell' \in \mathbb{W}_N^{\theta,M_N}$ be as in Lemma \ref{const} for $r=N$ and note that 
\begin{align*}
Z_N \ge Z_N \cdot \P_{N}^{\theta,M_N}(\vec\ell') = \exp \left(\theta N(N-1) \log N -N^2 I^{\theta}_{V_N}(\mu_N(\vec{\ell}')) + O(N \log N) \right),
\end{align*}
where in the last equality we used Lemma \ref{pmf-exact} with $M=M_N$. The latter and \eqref{GenPot} imply 
\begin{align*}
Z_N \ge \exp \left(\theta N(N-1) \log N -N^2 I^{\theta}_{V}(\mu_N(\vec{\ell}')) + O(N \log N + N^2 \cdot \fn) \right).
\end{align*}
Notice that from (\ref{clo}) we have that
$$-N^2I^{\theta}_{V}(\mu_N(\vec{\ell}')) \geq- N^2F_{V}^{\theta, \lM + \theta} + O(N \log N + N^2 \cdot \qn),$$
where we used that $\inf_{ \phi \in \mathcal{A}^{\alpha}_{s + \alpha}} I_V^{\theta}(\phi)$ from Lemma \ref{const} equals $F_{V}^{\theta, \lM + \theta}$, since $\alpha = \theta$ when $r = N$. Consequently,
\begin{equation}\label{ZNLB}
Z_N \geq \exp \left(\theta N(N-1) \log N -N^2 F_{V}^{\theta, \lM + \theta}+ O\left(N^2 \cdot\fn+ N^2\cdot \qn +N\log N\right) \right).
\end{equation}
The last equation and Lemma \ref{pmf-exact} with $M=M_N$ imply \eqref{up-bound}. This suffices for the proof.
\end{proof}

%
\subsection{Proof of Lemma \ref{const}}\label{Section2.4}  In this section we give the proof of Lemma \ref{const}. We continue with the same notation as in Sections \ref{Section1}, \ref{Section2.1}, \ref{Section2.2} and \ref{Section2.3}. Before we go into the proof we formulate two lemmas, which will be required. Their proofs are given in Section \ref{Section7} (see Lemmas \ref{S7LemmaTech3} and \ref{S7LemmaTech4}).

\begin{lemma}\label{LemmaTech3} Suppose that Assumption \ref{as2} from Section \ref{Section2.1} holds. Fix $A > 0$ and $\alpha \in [\theta/2, \theta]$. There exist $K_1, K_2 > 0$, depending on $A$ and the constants $\theta, A_3$ from Assumption \ref{as2}, such that for all $0 \leq s_1 \leq s_2 \leq A$ one has 
\begin{equation}\label{contV2}
\left|\inf_{\phi \in \mathcal{A}^{\alpha}_{s_1 + \alpha}} I_V^{\theta}(\phi)\right| \leq K_1 \mbox{ and }0 \leq \inf_{\phi \in \mathcal{A}^{\alpha}_{s_1 + \alpha}} I_V^{\theta}(\phi)  - \inf_{\phi \in \mathcal{A}^{\alpha}_{s_2 + \alpha}}I_V^{\theta}(\phi)  \leq K_2 (s_2 -s_1),
\end{equation} 
where we recall that $\mathcal{A}_{s }^\alpha$ was defined in (\ref{ainf}).
\end{lemma}

\begin{lemma}\label{LemmaTech4}
Suppose that Assumption \ref{as2} from Section \ref{Section2.1} holds. Fix $A > 0$, $N \geq 2$, $ r\in [N/2, N]$ and $\alpha = \theta r/N$. For each $\vec{\ell} \in \mathbb{W}_r^{\theta, A N}$ let $\mu_{N,r}(\vec{\ell})$ be as in (\ref{munr}) and let $\til\mu_{N,r}(\vec{\ell})$ be the measure with density
\begin{align}\label{cont}
\psi_{N,r}(x) = \sum_{i=1}^r f_{\ell_i/N}(x), \mbox{  where $f_{a}(x)=\frac{N}{r\theta} \cdot\ind_{x\in [a,a+\frac{\theta}{N})}$ }.
\end{align}
There are $K_3, K_4 > 0$, depending on $A$ and $\theta, A_3, A_4$ from Assumption \ref{as2}, such that
\begin{align}\label{EQMol}
\left| I_V^{\theta}(\mu_{N,r}(\vec\ell))\right| \leq K_3 \mbox{ and }\left| I_V^{\theta}(\mu_{N,r}(\vec\ell)) - I_V^{\theta}(\til\mu_{N,r}(\vec\ell)) \right| \leq K_4 N^{-1}\log N.
\end{align}
\end{lemma}

With the above results in place, we can proceed with the proof of Lemma \ref{const}.
\begin{proof}[Proof of Lemma \ref{const}] In the proof below the constants in all big $O$ notations will depend on $\theta$ and the constants $a_0, A_0,\ldots, A_4$ from Assumptions \ref{as1} and \ref{as2} -- we will not mention this further. 

We first observe that by Lemma \ref{LemmaTech3} with $A = A_0$ we have 
$$r^2 \inf_{ \phi \in \mathcal{A}^{\alpha}_{\lM + \alpha}} I_V^{\theta}(\phi)  = O(N^{2}).$$
In addition, by Lemma \ref{LemmaTech4} with $A = A_0 + A_1$ we have that if $\vec{\ell} \in \mathbb{W}_r^{\theta, M_N}$ is such that ${\ell}_i = (r-i)\theta$ for $i = 1, \dots, r$ then
$$r^2 I_V^{\theta}(\mu_{N,r}(\vec\ell)) = O(N^{2}).$$
We conclude that (\ref{clo}) holds with $\vec{\ell}' = \vec{\ell}$ as long as $N \leq 4 \theta a_0^{-1}$ or $ q_N \geq \frac{a_0}{4A_1}$ by making the constant in the big $O$ notation sufficiently large. Consequently, we only need to prove the lemma if $a_0/4 \geq \theta/N$ and $a_0/4 \geq A_1q_N $, which we assume in the sequel.

For clarity, the rest of the proof is split into three steps.\\

\noindent\textbf{Step 1. Construction of $\vec{\ell}'$.} Let $L_N = N^{-1} M_N  - \theta/N$. Notice that from our assumptions $a_0/4 \geq \theta/N$ and $a_0/4 \geq A_1q_N $ we have 
$$A_0 + A_1 \geq L_N = N^{-1} M_N  - \theta/N \geq \lM - N^{-1} A_1 q_N - \theta/N \geq a_0/2 \geq 0.$$
Let $\Phi$ denote the unique minimizer of $I_V^{\theta}$ over $\mathcal{A}^{\alpha}_{L_N + \alpha}$ as afforded by Lemma \ref{ivi}. We let $y_i$, $i = 1, \dots, r$ be the quantiles of $\Phi$, defined as the smallest positive numbers such that
$$\int_0^{y_i} \Phi(x)dx = \frac{i - 1/2}{r}.$$
Since $\Phi$ is supported on $[0, L_N + \alpha]$ and is bounded we have that $y_i$'s are all well-defined and $y_i \in [0, L_N + \alpha]$ for all $i = 1, \dots ,r$. 

We now let $\vec{\ell}'_i$ denote the largest element in $\mathbb{Z} + (r-i) \theta$, which is less than or equal to $N y_{r-i+1}$.  We claim that ${\ell}' = ({\ell}'_1, \dots, {\ell}'_r) \in \mathbb{W}_r^{\theta, M_N}$, or equivalently we want
$$M_N \geq {\lambda}'_1 \geq \cdots \geq {\lambda}'_r \geq 0, \mbox{ where }{\lambda}'_i = {\ell}'_i - (r-i)\theta.$$
To see the latter, notice that $y_1 \geq 0$, which implies ${\lambda}'_r \geq 0$. In addition, since $y_N \le L_N + \alpha$, we must have 
$${\ell}'_1 \leq Ny_1 \le M_N + N \alpha  - \theta  \implies \lambda_1' +(r-1) \theta \leq M_N + N \alpha  - \theta  = M_N + (r-1)\theta \implies \lambda_1' \leq M_N.$$
Suppose, for the sake of contradiction, that ${\lambda}'_i - {\lambda}'_{i-1} \geq 1 \mbox{ for some $i \in \{2, \dots, r\}$.}$ Then
$${\ell}'_{i-1} + 1 = {\lambda}'_{i-1} + 1 + (r-i + 1) \theta \leq {\ell}'_i + \theta \leq Ny_{r-i+1} + \theta =  Ny_{r-i + 2} + N(y_{r-i + 1} - y_{r-i+2}) + \theta.$$ 
On the other hand, as $\Phi(x) \in [0, \alpha^{-1}]$ and $\alpha = \theta r /N$, we have 
\begin{align}\label{y_gap}
\frac1r = \int_{y_{r-i + 1}}^{y_{r-i + 2}} \Phi(x)dx \leq \frac{N}{r\theta} \cdot (y_{r-i + 2} - y_{r-i+1})  \implies N(y_{r-i + 1} - y_{r-i+2})  \le -\theta.
\end{align}
Combining the last two inequalities we get ${\ell}'_{i-1} + 1  \leq Ny_{r-i + 2},$
which contradicts the maximality of ${\ell}'_{i-1}$ and so we conclude that $\vec{\ell}'$ as constructed is in $\mathbb{W}_r^{\theta, M_N}$.

In the steps below we prove that the $\vec{\ell}'$ we constructed satisfies (\ref{clo}). \\
	
\noindent\textbf{Step 2: Upper bound on logarithms.} By the definition of $\Phi$ and Lemma \ref{LemmaTech3} applied to $A = A_0 + A_1$ we have 
\begin{equation}\label{S2RR1}
 \inf_{\phi \in \mathcal{A}^{\alpha}_{\lM + \alpha}} I_V^{\theta}(\phi)  - \inf_{\phi \in \mathcal{A}^{\alpha}_{L_N + \alpha}}I_V^{\theta}(\phi)  =  \inf_{\phi \in \mathcal{A}^{\alpha}_{\lM + \alpha}} I_V^{\theta}(\phi)  - I_V^{\theta}(\Phi)= O(N^{-1} +  q_N),
\end{equation}
where we used the fact that $|L_N - \lM| \leq \theta/N + A_1 q_N$. 

We next claim that 
\begin{align}\label{log_bd}
r^2\iint\limits_{x_1>x_2}\log(x_1-x_2)\Phi(x_2)\Phi(x_1)\d x_2 \d x_1 & \le \sum_{1 \leq i < j \leq r}\log\left(\frac{\ell'_i}{N}-\frac{\ell'_j}{N}\right)+O\left(N\log N\right),
\end{align}
and also
\begin{align}\label{v_bdd2}
r\sum_{i=1}^r V\left(\frac{\ell'_i}{N}\right) = r^2\int_{\mathbb{R}} V(t)\Phi(t)dt + O(N\log N + N^2 \cdot \qn ).
\end{align}
Assuming (\ref{log_bd}) and (\ref{v_bdd2}) we see that 
$$ r^2I_V^{\theta}(\Phi) \leq r^2 I_V^{\theta}(\mu_{N,r}(\vec\ell')) + O(N \log N + N^2 \cdot \qn ),$$
which together with (\ref{S2RR1}) implies (\ref{clo}). Consequently, we have reduced the proof of the lemma to showing that (\ref{log_bd}) and (\ref{v_bdd2}) both hold. In this step we prove (\ref{log_bd}) and in the next and final step we prove (\ref{v_bdd2}).\\

Using the fact that $x \ge \log(1+x)$ for $x \geq 0$ we get
\begin{align*}
&\sum_{1\le i<j \le r} \log\left(\frac{\ell'_i}{N}-\frac{\ell'_j}{N}\right)+\sum_{1\le i<j \le r} \frac{1}{\ell'_i-\ell'_j} \ge \sum_{1\le i<j \le r} \log\left(\frac{\ell'_i}{N}-\frac{\ell'_j}{N}\right) +\sum_{1\le i<j \le r} \log\left(1+\frac{1}{\ell'_i-\ell'_j}\right) \nonumber \\ 
& = \sum_{1\le i<j \le r} \log\left(\frac{\ell'_i}{N}-\frac{\ell'_j}{N}+\frac1N\right)  \ge \sum_{1\le i<j \le r} \log\left(y_{r-i+1}-y_{r-j+1}\right),
\end{align*}
where the last inequality follows from the construction of $\vec{\ell}'$.
On the other hand, since for $i<j$ we know that $\ell'_i -\ell'_j \ge (j-i)\theta$ we have
\begin{align*}
0 \leq \sum_{1\le i<j \le N}\frac{1}{\ell'_i-\ell'_j} \le \sum_{1\le i<j\le N}\frac{1}{(j-i)\theta}=O(N\log N).
\end{align*}
Thus we get that
\begin{align}\label{0_bd}
\sum_{1\le i<j \le r}\log(y_j-y_i) \leq \sum_{1 \leq i < j \leq r}\log\left(\frac{\ell'_i}{N}-\frac{\ell'_j}{N}\right)+O(N\log N).
\end{align}

In the remainder of this step we show that 
\begin{align}\label{0_bdV2}
r^2\iint_{x_1>x_2}\log(x_1-x_2)\Phi(x_2)\Phi(x_1)\d x_2\d x_1 \leq \sum_{1\le i<j\le r} \log(y_j-y_i)+O(N).
\end{align}
If true, then (\ref{0_bd}) and (\ref{0_bdV2}) would imply (\ref{log_bd}). \\

Define $y_0=0$ and $y_{r+1}= L_N +\alpha $ and for $i = 1, \dots, r+1$ define $B_i:=[y_{i-1},y_{i})$. Observe that for $i,j\in [1,r+1]\cap \mathbb{N}$  with $i<j$ we have
\begin{equation*} 
\begin{split}
&\int_{B_{i}}\int_{ B_j}\log(x_1-x_2)\Phi(x_2)\Phi(x_1)\d x_1 \d x_2  \\
&\le \log(y_j-y_{i-1})\int_{B_{i}}\int_{B_j} \Phi(x_2)\Phi(x_1)\d x_1 \d x_2    =\frac1{r^2}\log(y_j-y_{i-1}).
\end{split}
\end{equation*}
On the other hand, for each $i = 1, \dots, r+1$ we have
\begin{align*}
& \iint_{\substack{x_2\in B_{i},x_1\in B_{i} \\ x_1>x_2}}\log(x_1-x_2)\Phi(x_2)\Phi(x_1)\d x_1 \d x_2 \\ & \le \iint_{\substack{x_2\in B_{i},x_1\in B_{i} \\ x_1>x_2}}\log (y_{i} - y_{i-1})\Phi(x_2)\Phi(x_1)\d x_1 \d x_2  = \frac1{2r^2}\log (y_{i} - y_{i-1})
\end{align*}	
Combining the above two inequalities, we get
\begin{align*}
r^2\iint_{x_1>x_2}\log(x_1-x_2)\Phi(x_2)\Phi(x_1)\d x_2\d x_1 & \leq \sum_{ i = 1}^{r+1} \sum_{j = i + 1}^{r+1} \log(y_j-y_{i-1}) + \sum_{ i = 1}^{r+1} \frac1{2}\log (y_{i} - y_{i-1}) \\
&\leq \sum_{1\le i<j\le r} \log(y_j-y_i)+O(N),
\end{align*}
where in the last equality we used that $\log(y_j-y_i) \leq O(1)$ for all $1 \leq i < j \leq r+1$.\\
	
\noindent\textbf{Step 3. Handling the potentials.} In this step we prove (\ref{v_bdd2}). Let $I$ be the largest index $i \in [1, r-1] \cap \mathbb{N}$ such that $y_{i+1} \leq \lM - \theta/N$. If no such $I$ exists, then we set $I = 0$. Notice that $r - I = O(1 + N q_N)$ from (\ref{GenPar}). Combining the latter with \eqref{GenPot}, we conclude that
\begin{equation}\label{S2GR1}
\begin{split}
&r\sum_{i=1}^r V\left(\frac{\ell'_i}{N}\right)  = r V\left(\frac{\ell'_r}{N}\right) +  r^2\sum_{i=2}^{I}V\left(\frac{\ell'_{r-i + 1}}{N}\right)\int_{y_{i}}^{y_{i+1}} \Phi(t)\d t   +  r \sum_{ i = I+1}^r V\left(\frac{\ell'_{r-i + 1}}{N}\right) \\ 
& = O(N + N^2 \cdot q_N)+r^2\sum_{i=2}^{I} \int_{y_{i}}^{y_{i+1}}  \left[ V\left(\frac{\ell'_{r-i + 1}}{N}\right) - V(t)\right]\Phi(t)\d t + r^2\int_{y_2}^{y_{I+1}} V(t) \Phi(t) \d t. 
\end{split}
\end{equation}
In addition, since $r - I = O(1 + N \cdot q_N)$ we have 
\begin{equation}\label{S2GR2}
r^2\int_{y_2}^{y_{I+1}} V(t) \Phi(t) \d t = r^2\int_{ 0}^{\infty} V(t) \Phi(t) \d t + O(N +N^2 \cdot q_N).
\end{equation}

Finally, we have for each $i = 2, \dots, I$ and $t \in [y_i, y_{i+1}]$ that
$$  V\left(\frac{\ell'_{r-i + 1}}{N}\right) - V(t) = V'( \kappa(t)) \cdot \left[ t - \frac{\ell'_{r-i + 1}}{N} \right].$$
where $\kappa(t) \in (\ell'_{r-i + 1}/N, t)$. Note that $\kappa(t) \in [y_2, y_{I+1} ]$ and thus $\kappa(t) \in [\theta/N, \lM - \theta/N]$ as follows from the definition of $y_2$ and $I$. This implies by (\ref{DerPot}) that $V'( \kappa(t)) = O(\log N)$ and so
$$r^2\sum_{i=2}^{I} \int_{y_{i}}^{y_{i+1}}  \left[ V\left(\frac{\ell'_{r-i + 1}}{N}\right) - V(t)\right]\Phi(t)\d t  = O(N^2 \log N) \cdot \sum_{i=2}^{I} \int_{y_{i}}^{y_{i+1}}  \left[y_{i+1} - y_i + N^{-1}\right]\Phi(t)\d t,$$
where we used that $\ell'_{r-i + 1}/N \geq y_i - 1/N$ by definition. Observe that
$$ \sum_{i=2}^{I} \int_{y_{i}}^{y_{i+1}}  \left[y_{i+1} - y_i + N^{-1}\right]\Phi(t)\d t = \frac{1}{r} \sum_{i=2}^{I} \left[y_{i+1} - y_i + N^{-1}\right] = O(N^{-1}).$$
The last two equations imply that 
\begin{equation}\label{S2GR3}
r^2\sum_{i=2}^{I} \int_{y_{i}}^{y_{i+1}}  \left[ V\left(\frac{\ell'_{r-i + 1}}{N}\right) - V(t)\right]\Phi(t)\d t = O(N  \log N).
\end{equation}
Combining (\ref{S2GR1}), (\ref{S2GR2}) and (\ref{S2GR3}) we arrive at (\ref{v_bdd2}). This suffices for the proof.

\end{proof}

%
\section{Exponential tightness} \label{Section3} In this section, we consider the measures $\mathbb{P}_N^{\theta,M}$ from Section \ref{Section1.1} when $M = \infty$ and show that under certain technical conditions the rightmost particle $\ell_1$ belongs to a window of order $N$ with exponentially high probability. The precise statement is given in Proposition \ref{exptight}.

%
\subsection{Assumptions}\label{Section3.1} We continue with the same notation as in Section \ref{Section1} and assume that $M = \infty$. We make the following assumptions about the scaling of the weights $w(\cdot;N)$ as $N \rightarrow \infty$.

\begin{assumption}\label{as3} We assume that we are given parameters $\theta, \xi, \Bc > 0$, and increasing functions $\F_1, \F_2 : (0,\infty) \rightarrow (0, \infty)$. We assume that for $N \in \mathbb{N}$, $w(x;N)$ in the interval $[0, \infty)$ has the form
$$w(x;N) = \exp\left( - N V_N(x/N)\right),$$
for a function $V_N$ that is continuous on $[0, \infty)$ and satisfies
\begin{equation}\label{S3VPot}
V_N(t) \geq (1 + \xi) \cdot \theta \cdot \log(1 + t^2), \quad \mbox{ for }t\ge 0.
\end{equation}

In addition, we assume that there is a continuous function $V(t)$ on $[0, \infty)$ and a sequence $\gn \rightarrow 0$ such that the following holds for all $a > 0$ and $N \in \mathbb{N}$
\begin{align}\label{S3conv-rate}
\sup_{x\in [0,a]}|V_N(x)-V(x)|\le \F_1(a) \cdot  \gn. 
\end{align}
Finally, we assume that $V$ is differentiable on $(0,\infty)$ and that for all $a \geq 0$ and $x \in (0, a]$ we have
\begin{align}\label{der-rate}
|V'(x)|\le \Bc \left( \F_2(a) + |\log (x)| \right).
\end{align}
\end{assumption}
\begin{remark}
In applications we will typically have that $\gn$ is an explicit sequence converging to $0$ as $N \rightarrow \infty$. In this case Assumption \ref{as3} would state that the weights $w(x;N)$ underlying the discrete model $\mathbb{P}^{\theta, \infty}_N$ in (\ref{PDef}) asymptotically look like $e^{-NV(x/N)}$ for some function $V$ that plays the role of an external potential in our model. This external potential is assumed to be differentiable on the interval $(0, \infty)$, but its derivative is allowed to have a logarithmic singularity near $0$ -- some of the applications we have in mind satisfy this condition.
\end{remark}

\begin{definition}\label{S3PDef} We let $\mathbb{P}^{\theta, \infty}_N$ be as in (\ref{PDef}) for $\theta > 0$, $N \geq N_0 = \max(2,\theta^{-1} \xi^{-1})$, and $w(\cdot; N)$ all satisfying Assumption \ref{as3}. The condition $N \geq \theta^{-1} \xi^{-1}$ is required to ensure that the measure is well-defined, cf. Remark \ref{RemDecay}, and the condition $N \geq 2$ is included for convenience. In particular, we have fixed $V, \theta, \xi, \Bc, \F_1, \F_2$ and $\gn$ as in this assumption. When $\theta$ is clear from the context we will write $\mathbb{P}_N$ in place of $\mathbb{P}^{\theta,\infty}_N$. If $\vec{\ell} = (\ell_1, \dots, \ell_N)$ is $\mathbb{P}_N$-distributed we recall from (\ref{emp-meas}) that  
\begin{equation*}
\mu_N(\vec{\ell}) =\frac{1}{N}\sum_{i=1}^N\delta\left(\frac{\ell_i}{N}\right),
\end{equation*}
denotes the (random) empirical measure. For a fixed $\vec{\ell} \in \mathbb{W}^{\theta,\infty}_{N} $ as in (\ref{GenState}) we still write $\mu_N(\vec{\ell})$ for the above (now deterministic) empirical measure. The distinction between these two will be clear from the context.
\end{definition}

We may now state the main result of the section.
\begin{proposition} \label{exptight} Suppose that $\P_{N}$ is as in Definition \ref{S3PDef} and that there exists $R_0 > 0$ such that $r_N \leq R_0 N^{-3/4}$. For any $A > 0$ there exists a constant $\MA> 0$ such that for all $N \geq N_0$
\begin{align}\label{expTightEqn}
\P_N(\ell_1\ge \MA N) \le e^{-AN+O(1)},
\end{align}
where $\MA$ and the constant in the big $O$ notation both depend on $R_0, A$ and $V, \theta, \xi, \Bc, \F_1, \F_2$ in Assumption \ref{as3}.
\end{proposition}
\begin{remark} In words, Proposition \ref{exptight} states that for any $A >  0$ we can find a large enough $\MA$ so that for all large $N$ the empirical measures $\mu_N(\vec{\ell})$ from Definition \ref{S3PDef} are supported on $[0, \MA]$ at least with probability $1 - e^{-AN}.$ I.e., outside of an $e^{-AN }$ probability event, the measures $\mu_N(\vec{\ell})$ are tight -- for this reason we refer to the result as {\em exponential tightness for the empirical measures}.
\end{remark}
\begin{remark} We mention that a similar result to Proposition \ref{exptight} (under quite different assumptions) appears as \cite[Theorem 10.1]{bgg}.
\end{remark}

%
\subsection{Proof of Proposition \ref{exptight}} \label{Section3.2} In this section we give the proof of Proposition \ref{exptight}. We continue with the same notation as in Sections \ref{Section1} and \ref{Section3.1}. Before we go into the proof we summarize some notation and results, which will be required.

For each $\rho>0$ and $\vec{\ell} \in \mathbb{W}_N^{\theta,\infty}$ as in (\ref{GenState}) we define the function
\begin{align}\label{kp}
k_{\rho}(\vec\ell):=|\{j \in \{1, \dots, N \} \mid \ell_j\ge \rho N\}|,
\end{align}
that counts the number of particles exceeding $\rho N$. We require the following lemma, whose proof is given in Section \ref{Section7} (see Lemma \ref{S7tail-estimate}).

\begin{lemma}\label{tail-estimate} Let $B, \theta,\xi> 0$ and $N\in\mathbb{N}$ be such that $N\theta \xi\ge 1$. For any $i \in \{1, \dots N\}$ we have
\begin{equation} \label{tail}
\begin{aligned}
\sum_{\substack{\ell \in \mathbb{Z}+(N-i)\theta \\ \ell \ge (B+\theta+1)N}} \frac1{\left(\ell^2/ N^2 +1\right)^{N\theta\xi}} 
\le \frac{N\pi/2}{(B^2+1)^{N\theta\xi-1}}.
\end{aligned}
\end{equation}
\end{lemma}

We will also require the following lemma, whose proof is given in Section \ref{Section3.3}.
\begin{lemma}\label{mid-region} Suppose that $\P_{N}$ is as in Definition \ref{S3PDef}. There exists $\MB > 0$ such that for all $N \geq N_0$ and $N \geq n \geq 1$
\begin{align}\label{S3EqMid}
\P_N\left( n \le k_{\MB} (\vec{\ell}) \right) \le \exp \left(-Nn+O(N^2 \cdot \gn +N\log N) \right),
\end{align}
where $\MB$ and the constant in the big $O$ notation depend on $V, \theta, \xi, \Bc, \F_1, \F_2$ from Assumption \ref{as3}.
\end{lemma}

With the above results in place, we can proceed with the proof of Proposition \ref{exptight}.
\begin{proof}[Proof of Proposition \ref{exptight}] In the sequel we assume that we have fixed $R_0 > 0, A > 0$ as well as $V, \theta, \xi, \Bc, \F_1, \F_2$ as in Definition \ref{S3PDef}. In the proof below, unless otherwise specified, all constants (including those in big $O$ notations) will depend on $A, R_0, V, \theta, \xi, \Bc, \F_1, \F_2$  -- we will not mention this further.\\

We start by introducing some relevant notation and fixing $\MA$ as in the statement of the proposition. Let $R_1 > 0$ be sufficiently large so that for all $N \geq 2$
\begin{equation}\label{R1const}
\max_{w = 1, \dots, N} \left( \frac{N^{N- w}}{\theta^{N-w} (N-w)!} \right)^{2\theta}  \leq \exp (R_1 N).
\end{equation}
The existence of $R_1$ follows from Stirling's formula, e.g. from \cite[Equation (1)]{Rob} we have
\begin{equation}\label{EqnRob}
n! = \sqrt{2\pi} n^{n+1/2} e^{-n} \cdot e^{r_n}, \mbox{ where } \frac{1}{12 n +1 } < r_n < \frac{1}{12n} \mbox{  for all $n \in \mathbb{N}$}.
\end{equation}
Let $\MB$ be as in Lemma \ref{mid-region} and let $R_2 > 0$ be sufficiently large so that 
\begin{equation}\label{R2const}
R_2 \geq \sup_{x \in [0, \MB + \theta]} |V(x)|.
\end{equation}
Let $T$ be sufficiently large so that for $N \geq 2 \cdot \max \left(1, \xi^{-1} \theta^{-1}\right)$ 
\begin{equation}\label{S3S2tC}
(T^2 + 1)^{-N\theta\xi +1} \leq \exp( - N R_1 - N R_2 - 2A N )\mbox{ and  set } \MA = \max( T + \theta + 1, \MB).
\end{equation}
We will prove the proposition for the above choice of $\MA$ and we split the proof into four steps. \\

{\bf \raggedleft Step 1.} Fix $B > 0$, $w= \lfloor N^{1/2} \rfloor$ and $0 \leq w' \leq w$. With this data we define
\begin{equation}\label{caladef}
\begin{split}
\mathcal{A}^{(\MB,B)}_{w,w'} & =\{\vec\ell\in \mathbb{W}_{N}^{\theta,\infty} \mid k_{\MB}(\vec\ell)\le w, k_{B}(\vec\ell)=w'\} \\ & = \{\vec\ell\in \mathbb{W}_{N}^{\theta,\infty} \mid \ell_{w+1}<\MB N, \ell_{w'} \ge BN,\mbox{ and }\ell_{w'+1}< BN\} ,
\end{split}
\end{equation}
where we recall that $k_\rho(\vec{\ell})$ was defined in (\ref{kp}). 

We claim that for any $N \geq 2 \cdot \max \left(1, \xi^{-1} \theta^{-1}\right)$, $B \in [\MA, \MA + \theta]$, $w= \lfloor N^{1/2} \rfloor $ and $w \geq w' \geq 1$ 
\begin{equation}\label{S3S2R1}
\P_N\left(\vec\ell\in \mathcal{A}^{(\MB,B)}_{w,w'}\right)  \le \P_N\left( \vec\ell\in \mathcal{A}^{(\MB,B+\theta/N)}_{w,w'-1}\right) \cdot \exp\left(-2AN+O(N^{3/4})\right).
\end{equation}
We will prove (\ref{S3S2R1}) in the next steps. Here we assume its validity and conclude the proof of (\ref{expTightEqn}).\\

Equation (\ref{S3S2R1}) implies that there exists $N_1 \geq 2 \cdot \max \left(1, \xi^{-1} \theta^{-1}\right)$ such that for $N \geq N_1$, $B \in [\MA, \MA + \theta]$, $w=\lfloor N^{1/2} \rfloor$ and $w \geq w' \geq 1$ we have
$$\P_N\left(\vec\ell\in \mathcal{A}^{(\MB,B)}_{w,w'}\right)  \le \P_N\left( \vec\ell\in \mathcal{A}^{(\MB,B+\theta/N)}_{w,w'-1}\right) \cdot \exp\left(-AN\right).$$
Iterating the last inequality $w'$ times we have for any $N \geq N_1$,  $w=  \lfloor N^{1/2} \rfloor$ and $w \geq w' \geq 1$ that
\begin{equation*}
\begin{split}
\P_N\left(\vec\ell\in \mathcal{A}^{(\MB,\MA)}_{w,w'}\right)  & \le  \P_N \left(\vec\ell\in  \mathcal{A}^{(\MB,\MA+ w'\theta/N)}_{w,0} \right)\cdot \exp\left(-AN w'\right) \leq \exp\left(-AN w'\right) .
\end{split}
\end{equation*}
In particular, we see that for $N \geq N_1$
\begin{equation*}
\begin{split}
\P_N\left( \ell_1 \geq \MA N \mbox{ and } k_{\MB}(\vec\ell)\le w \right) \leq \sum_{w' = 1}^w \P_N\left(\vec\ell\in \mathcal{A}^{(\MB,\MA)}_{w,w'}\right)  \le   \exp\left(-AN+O(1)\right) .
\end{split}
\end{equation*}
Note that by making the constant in the big $O$ notation big enough we can ensure that the right side is larger than $1$ whenever $N \leq N_1$ and so the last inequality holds for all $N \geq N_0$. On the other hand, by Lemma \ref{mid-region} we have for all $N \geq N_0$
\begin{equation*}
\begin{split}
\P_N\left( \ell_1 \geq \MA N \mbox{ and } k_{\MB}(\vec\ell)\geq w + 1\right) \leq \P_N\left(  k_{\MB}(\vec\ell)\geq w + 1\right) \leq \exp \left(-N^{3/2}+O(N^{5/4}) \right),
\end{split}
\end{equation*}
where we used that $r_N \leq R_0N^{-3/4}$. The last two equations imply (\ref{expTightEqn}).\\

{\bf \raggedleft Step 2.} In this step we prove (\ref{S3S2R1}). In the sequel we fix $N \geq 2 \cdot \max \left(1, \xi^{-1} \theta^{-1}\right)$, $w =\lfloor N^{1/2} \rfloor$ and $w \geq w' \geq 1$. We begin by introducing a bit of notation.

Given $\vec{\ell} \in \mathbb{W}_{N}^{\theta,\infty}$ we let $\vec\ell\mid_{w'}$ denote the vector in $\mathbb{R}^{N-1}_{\geq 0}$
$$\vec\ell\mid_{w'} = (\ell_1,\ldots,\ell_{w'-1},\ell_{w'+1},\ldots,\ell_N).$$
We will also write $(\vec\ell\mid_{w'};\ell_{w'}):=(\ell_1,\ldots,\ell_{w'-1},\ell_{w'},\ell_{w'+1},\ldots,\ell_N)\in \mathbb{W}_{N}^{\theta,\infty}$. We consider the set
$$\til{\mathcal{A}}_{w,w'}^{(\MB,B)}:= \left\{\vec\ell\mid_{w'} \in \mathbb{R}^{N-1}_{\geq 0}\ : \ (\vec\ell\mid_{w'};\ell_{w'})\in \mathcal{A}_{w,w'}^{(\MB,B)} \mbox{ for some } \ell_{w'}\in \mathbb{Z}_{\ge0}+(N-w')\theta \right\}.$$	
Finally, we define $\tau: \til{\mathcal{A}}_{w,w'}^{(\MB,B)} \to \mathcal{A}_{w,w'-1}^{(\MB,B+ \theta/N)}$ as follows. For $\vec{\ell} \in \til{\mathcal{A}}_{w,w'}^{(\MB,B)}$ we let $\tau(\vec{\ell}) = \vec{\ell}'$, where 
\begin{align}\label{defpr}
\ell'_{k}  =\ell_{k+1}+\theta \mbox{ for } k=w',w'+1,\ldots w, \qquad \ell'_k  =\ell_k \mbox{ for } k\ge w+1\mbox{ or } k \le w'-1.
\end{align}
One readily observes from the definition of $\til{\mathcal{A}}_{w,w'}^{(\MB,B)}$ that $\vec\ell'\in \mathbb{W}_N^{\theta,\infty}$. Furthermore, from the definition of $\vec\ell'$ from \eqref{defpr}, we see that  $\ell'_{w+1}=\ell_{w+1}<\MB N$, $\ell'_{w'-1}=\ell_{w'-1} \ge \ell_{w'}+\theta \ge BN+\theta$, $\ell'_{w'}=\ell'_{w'+1}+\theta < BN+\theta$. Hence, we see that $\vec\ell' \in \mathcal{A}^{(\MB,B+\theta/N)}_{w,w'-1}$ as claimed. \\

We claim that if $B \in [\MA, \MA + \theta]$, $\vec{\ell} \in {\mathcal{A}}_{w,w'}^{(\MB,B)}$ and $\vec{\ell}'$ is as in (\ref{defpr}) we have
\begin{equation}\label{S3S2R2}
\P_N(\vec\ell)\le \exp\left(R_1N + R_2 N + O( N^{3/4})\right) \cdot \left(1+\ell_{w'}^2/N^2\right)^{-N\theta\xi} \cdot \P_N(\vec\ell'),
\end{equation}
where we recall that $R_1$ and $R_2$ were defined in (\ref{R1const}) and (\ref{R2const}). We will prove (\ref{S3S2R2}) in the next steps. Here we assume its validity and conclude the proof of (\ref{S3S2R1}).\\

We now observe that we have the following tower of inequalities
\begin{equation*}
\begin{split}
&\P_N\left(\vec\ell\in \mathcal{A}^{(\MB,B)}_{w,w'}\right)   = \sum_{\vec{\ell} \in \mathcal{A}^{(\MB,B)}_{w,w'}} \P_N (\vec{\ell}) =  \sum_{ \vec{u} \in \til{\mathcal{A}}_{w,w'}^{(\MB,B)}} \sum_{ \vec{\ell} \in  \mathcal{A}^{(\MB,B)}_{w,w'}: \vec\ell\mid_{w'} = \vec{u} }\P_N (\vec{\ell})  \\
&\leq \exp\left(R_1N + R_2 N + O( N^{3/4})\right)  \cdot \sum_{\vec{u} \in \til{\mathcal{A}}_{w,w'}^{(\MB,B)}} \sum_{ \vec{\ell} \in  \mathcal{A}^{(\MB,B)}_{w,w'}: \vec\ell\mid_{w'} = \vec{u} } \left(1+\ell_{w'}^2/N^2\right)^{-N\theta\xi} \P_N( \tau( \vec{u}))  \\
&\leq \exp\left(R_1N + R_2 N + O( N^{3/4})\right) \cdot \sum_{ \vec{u} \in \til{\mathcal{A}}_{w,w'}^{(\MB,B)}}  \P_N( \tau( \vec{u})) \sum_{  \substack{\ell \in \mathbb{Z}_{\ge 0}+(N-w')\theta \\ \ell \ge B N} } \left(1+\ell^2/N^2\right)^{-N\theta\xi} \\
& \leq \exp\left(R_1N + R_2 N + O( N^{3/4})\right)  \cdot\frac{N\pi/2}{(T^2+1)^{N\theta\xi-1}} \cdot  \sum_{ \vec{u} \in \til{\mathcal{A}}_{w,w'}^{(\MB,B)}}  \P_N( \tau( \vec{u}))   \\
& \leq \exp\left(R_1N + R_2 N + O( N^{3/4})\right) \cdot\frac{N\pi/2}{(T^2+1)^{N\theta\xi-1}} \cdot  \sum_{\vec{\ell} \in {\mathcal{A}}_{w,w'}^{(\MB,B + \theta/N)}}  \P_N( \vec{\ell})   \\
& \leq \P_N\left( \vec\ell\in \mathcal{A}^{(\MB,B+\theta/N)}_{w,w'-1}\right) \cdot \exp\left(-2AN+O( N^{3/4})\right).
\end{split}
\end{equation*}
Let us elaborate on the last equation briefly. The equalities on the first line follow from the additivity of $\P_N$. In going from the first to the second line we used  (\ref{S3S2R2}). In going from the second to the third line we used that for all $\vec{\ell} \in  \mathcal{A}^{(\MB,B)}_{w,w'}$ we have $\ell_{w'} \geq BN$ and $\ell_{w'} \in \mathbb{Z}_{\ge 0}+(N-w')\theta$. In going from the third to the fourth line we used that $B \geq \MA \geq T+ \theta + 1$ and (\ref{tail}). In going from the fourth to the fifth line we used that $\tau$ is injective (as can be seen from the definition in (\ref{defpr})). The last inequality follows from the additivity of $\P_N$ and (\ref{S3S2tC}). Since the above tower implies (\ref{S3S2R1}), this completes our work in this step.\\

{\bf \raggedleft Step 3.} In this step we fix $B \in [\MA, \MA + \theta]$ and $\vec{\ell} \in {\mathcal{A}}_{w,w'}^{(\MB,B)}$ and prove (\ref{S3S2R2}). From \eqref{PDef} 
\begin{align*}
\frac{\P_N(\vec\ell)}{\P_N(\vec\ell')}& =\mathfrak{L}_1(\vec\ell)\mathfrak{L}_2(\vec\ell) \mbox{ where }  \mathfrak{L}_1(\vec\ell) = \frac{\prod\limits_{1\le i<j\le N} Q_{\theta}(\ell_i-\ell_j)}{\prod\limits_{1\le i<j\le N} Q_{\theta}(\ell'_i-\ell'_j)} \mbox{ and } \mathfrak{L}_2(\vec\ell) = \frac{\prod\limits_{i=1}^N \exp(-NV_N(\ell_i/N))}{\prod\limits_{i=1}^N \exp(-NV_N(\ell'_i/N))}.
\end{align*}
Using the last equation we would be able to deduce (\ref{S3S2R2}) if we can show that 
\begin{equation}\label{S3S2R3}
\begin{split}
&\mathfrak{L}_1(\vec\ell) \leq \exp\left(R_1 N + O(N^{3/4})\right)  \cdot  \left(1+\ell_{w'}^2/N^2\right)^{N\theta} \mbox{, and }\\
& \mathfrak{L}_2(\vec\ell)  \leq \exp\left(R_2 N + O(N^{3/4})\right) \cdot  \left(1+\ell_{w'}^2/N^2\right)^{-N\theta(\xi + 1)}. 
\end{split}
\end{equation}
We will prove the second inequality in (\ref{S3S2R3}) in this step, and postpone the proof of the first inequality in (\ref{S3S2R3}) to the next (and final) step.\\

Using the definition of $\vec\ell'$ from \eqref{defpr}, $\mathfrak{L}_2(\vec\ell)$ can be simplified as follows
\begin{align*}
\mathfrak{L}_2(\vec\ell)  = \exp(-NV_N(\ell_{w'}/N)+NV_N((\ell_{w+1}+\theta)/N))\prod\limits_{i=w'+1}^{w} \exp(-N(V_N(\ell_i/N)-V_N((\ell_{i}+\theta)/N))). 
\end{align*}
 Note that in the above expression, except $\ell_{w'}/N$, all the other particles lie inside $[0,B+\theta/N] \subset [0, \MA + 2\theta]$. Consequently, by Assumption \ref{as3} (specifically (\ref{S3conv-rate})) we have
\begin{equation*}
\begin{split}
&\mathfrak{L}_2(\vec\ell)  = \exp(O(w N \gn) -NV_N(\ell_{w'}/N)+NV_N((\ell_{w+1}+\theta)/N)) \\
&\times \prod\limits_{i=w'+1}^{w} \exp(-N(V(\ell_i/N)-V((\ell_{i}+\theta)/N))). 
\end{split}
\end{equation*}
Furthermore,  by Assumption \ref{as3} (specifically (\ref{der-rate})) we get
$$\mathfrak{L}_2(\vec\ell)  = \exp(O(w N \gn) + O(w) -NV_N(\ell_{w'}/N)+NV((\ell_{w+1}+\theta)/N)). $$
Finally, since $\ell_{w+1} \leq \MB N$ (from our assumption that $\vec{\ell} \in {\mathcal{A}}_{w,w'}^{(\MB,B)}$) and (\ref{S3VPot}), we see that 
$$\mathfrak{L}_2(\vec\ell)  \leq \exp\left(O(w N \gn) + O(w) +N \sup_{x \in [0, \MB + \theta]} V(x)  \right) \cdot \left(1+ \ell_{w'}^2/N^2\right)^{-N(\xi+1)\theta}. $$
The last inequality implies the second inequality in (\ref{S3S2R3}), once we use the definition of $R_2$ from (\ref{R2const}) and the fact that $w N r_N = O(N^{3/4})$, while $w = O(N^{1/2})$.\\

{\bf \raggedleft Step 4.} In this step we prove the first inequality in (\ref{S3S2R3}). Using the definition of $\vec{\ell}'$ from \eqref{defpr}, simple but tedious calculations show that
\begin{equation}\label{sMA}
\begin{split}
\mathfrak{L}_1(\vec\ell)  = \frac{\prod_{ j =w'+1}^w Q_{\theta}(\ell_{w'}-\ell_j)}{\prod_{i = w'+1}^w  Q_{\theta}(\ell_{i}-\ell_{w+1})}\cdot \prod_{i = w'}^w \prod_{j = w+1}^N\frac{Q_{\theta}(\ell_i-\ell_j)}{Q_{\theta}(\ell_{i+1}+\theta-\ell_j)} \cdot \prod_{i = 1}^{w'-1} \prod_{ j = w'}^w\frac{Q_{\theta}(\ell_i-\ell_j)}{Q_{\theta}(\ell_{i}-\ell_{j+1}-\theta)}.
\end{split}
\end{equation}

 Using \eqref{Sandwich1} we get
\begin{equation*}
\begin{split}
&\mathfrak{L}_1(\vec\ell)   \leq  \frac{\prod_{ j =w'+1}^w (\ell_{w'}-\ell_j)^{2\theta}}{\prod_{i = w'+1}^w  (\ell_{i}-\ell_{w+1})^{2\theta}}\cdot \prod_{i = w'}^w \prod_{j = w+1}^N\frac{(\ell_i-\ell_j)^{2\theta}}{(\ell_{i+1}+\theta-\ell_j)^{2\theta}} \cdot \prod_{i = 1}^{w'-1} \prod_{ j = w'}^w\frac{(\ell_i-\ell_j)^{2\theta}}{(\ell_{i}-\ell_{j+1}-\theta)^{2\theta}} \\
& \times \exp \left(    (1+ \theta)^3 \cdot \left( \sum_{i = w'+1}^w \frac{1}{\ell_{w'}-\ell_j} + \frac{1}{\ell_{i}-\ell_{w+1}}     \right) + \sum_{i = w'}^w \sum_{j = w+1}^N \frac{1}{\ell_i-\ell_j} + \frac{1}{\ell_{i+1}+\theta-\ell_j}  \right)\\
&\times  \exp \left( (1+ \theta)^3 \cdot  \sum_{i = 1}^{w'-1} \sum_{ j = w'}^w\frac{1}{\ell_i-\ell_j} + \frac{1}{\ell_{i}-\ell_{j+1}-\theta}  \right).
\end{split}
\end{equation*}
Using the fact that $|\ell_i-\ell_j|\ge |i-j|\theta$ and $w \leq N^{1/2}$ we see that 
\begin{equation*}
\mathfrak{L}_1(\vec\ell)   \leq  e^{O( N^{1/2} \log N)} \cdot\frac{\prod_{ j =w'+1}^w (\ell_{w'}-\ell_j)^{2\theta}}{\prod_{i = w'+1}^w  (\ell_{i}-\ell_{w+1})^{2\theta}}\cdot \prod_{i = w'}^w \prod_{j = w+1}^N\frac{(\ell_i-\ell_j)^{2\theta}}{(\ell_{i+1}+\theta-\ell_j)^{2\theta}} \cdot \prod_{i = 1}^{w'-1} \prod_{ j = w'}^w\frac{(\ell_i-\ell_j)^{2\theta}}{(\ell_{i}-\ell_{j+1}-\theta)^{2\theta}}.
\end{equation*}
As $\ell_j \ge \ell_{j+1}+\theta$ we see that all the factors in the last product are at most $1$ and thus
\begin{equation*}
\mathfrak{L}_1(\vec\ell)   \leq  \exp \left(O( N^{1/2} \log N)\right) \cdot\frac{\prod_{ j =w'+1}^w (\ell_{w'}-\ell_j)^{2\theta}}{\prod_{i = w'+1}^w  (\ell_{i}-\ell_{w+1})^{2\theta}}\cdot \prod_{i = w'}^w \prod_{j = w+1}^N\frac{(\ell_i-\ell_j)^{2\theta}}{(\ell_{i+1}+\theta-\ell_j)^{2\theta}}.
\end{equation*}
Since $\ell_{i+1} + \theta - \ell_j\ge \ell_{i+1}  - \ell_j$, we see that 
\begin{equation*}
\begin{split}
&\mathfrak{L}_1(\vec\ell)   \leq  \exp \left(O( N^{1/2} \log N) \right) \cdot\frac{\prod_{ j =w'+1}^w (\ell_{w'}-\ell_j)^{2\theta}}{\prod_{i = w'+1}^w  (\ell_{i}-\ell_{w+1})^{2\theta}}\cdot \prod_{j = w+1}^N\frac{(\ell_{w'}-\ell_j)^{2\theta}}{(\ell_{w+1}+\theta-\ell_j)^{2\theta}} \\
& =\exp \left(O( N^{1/2} \log N) \right) \cdot\frac{\prod_{ j =w'+1}^w (\ell_{w'}/N-\ell_j/N)^{2\theta}}{\prod_{i = w'+1}^w  (\ell_{i}/N-\ell_{w+1}/N)^{2\theta}}\cdot \prod_{j = w+1}^N\frac{(\ell_{w'}/N-\ell_j/N)^{2\theta}}{(\ell_{w+1}/N+\theta/N-\ell_j /N)^{2\theta}} \\
& \leq \exp \left(O( N^{1/2} \log N)\right) \cdot \left( \frac{N^{N- w}}{\theta^{N-w} (N-w)!} \right)^{2\theta} \cdot \left(1+\ell_{w'}^2/N^2\right)^{N\theta},
\end{split}
\end{equation*}
where in the last inequality we used that  $(\ell_{w'} - \ell_{i})/N \le \sqrt{1+\ell_{w'}^2/ N^2}$, $|\ell_i-\ell_j|\ge |i-j|\theta$ and that $w \leq N^{1/2}$. The last equation implies the first inequality in (\ref{S3S2R3}), once we use the definition of $R_1$ from (\ref{R1const}). This completes the proof of (\ref{S3S2R3}) and hence the proposition.
\end{proof}

%
\subsection{Proof of Lemma \ref{mid-region}} \label{Section3.3} In this section, we prove Lemma \ref{mid-region}. We will require three preliminary results -- Lemmas \ref{MinEq}, \ref{pr-bound} and \ref{large-region}, after which we will present the proof of Lemma \ref{mid-region}.\\

\begin{lemma}\label{MinEq} Assume the same notation as in Lemma \ref{iv} and let $b_V^{\theta, s} \in [\theta, \infty)$ be the rightmost point of the support of $\phi_V^{\theta, s}$. If $\theta \leq s_1 < s_2 \leq \infty$ and $b_V^{\theta, s_2} \leq s_1$ then $\phi_V^{\theta, s_1} = \phi_V^{\theta, s_2}$. In particular, $F_V^{\theta, s_1} = F_V^{\theta, s_2}$ and $b_V^{\theta, s_1} = b_V^{\theta, s_2}$.
\end{lemma}
\begin{proof}
As $\mathcal{A}_{s_1}^{\theta} \subset \mathcal{A}_{s_2}^{\theta}$ we have $F_V^{\theta,s_2}\le F_V^{\theta,s_1}$ by definition (see Lemma \ref{iv}). However, as $\phi_V^{\theta,s_2} \in \mathcal{A}_{s_1}^{\theta}$ by assumption,  we have $F_V^{\theta,s_2} \ge F_V^{\theta,s_1}$ as well. The latter implies that $F_V^{\theta, s_1} = F_V^{\theta, s_2}$, which by the uniqueness of the minimizer of $I_V^{\theta}$ in $\mathcal{A}_{s_2}^\theta$ implies $\phi_V^{\theta, s_1} = \phi_V^{\theta, s_2}$.
\end{proof}

\begin{lemma}\label{pr-bound} Suppose that $\P_{N}$ is as in Definition \ref{S3PDef}. For any $N \geq N_0$ and $\vec\ell\in \mathbb{W}_N^{\theta,\infty}$ as in (\ref{GenState}) one has the following inequality
\begin{align}\label{u-bound}
\P_N(\vec\ell) \le \exp \left(N^2(F_{V}^{\theta,\infty}-I^{\theta}_{V_N}(\mu_N(\vec{\ell}))) + O\left(N^2 \cdot \gn+N\log N\right)\right),
\end{align}
where $F_{V}^{\theta, \infty}$ is defined in \eqref{fs}, $I_{V_N}^\theta$ is defined in (\ref{IV}) and the constant in the big $O$ notation depends on $V, \theta, \xi, \Bc, \F_1, \F_2$ from Assumption \ref{as3}.
\end{lemma}
\begin{proof} The idea of the proof is to reduce the problem to the finite $M$ setting and apply the results from Section \ref{Section2}. In order to accomplish this we need to check that Assumptions \ref{as1} and \ref{as2} are satisfied. In the proof below, unless otherwise specified, all constants in the big $O$ notations will depend on $V, \theta, \xi, \Bc, \F_1, \F_2$ -- we will not mention this further.\\

Recall the definition of the equilibrium measure $\phi_V^{\theta,\infty}$ and $F_V^{\theta,\infty}$ from Lemma \ref{iv}. By Lemma \ref{iv}, $\phi_V^{\theta,\infty}$ has a compact support and we let $b_V \in [\theta, \infty)$ be the right-most point of its support. Set $M_N=\lfloor N (b_V - \theta +1) \rfloor$ and $\lM=b_V -\theta +1 $. Note that $M_N$ and $\lM$ satisfy Assumption \ref{as1} with $a_0 = A_0 = b_V -\theta +1$, $q_N=1/N$ and $A_1=1$. In addition, we note that Assumption \ref{as2} is satisfied for $\fn = \gn$, $A_2 = F_1( b_V + 1)$, $A_3 = \sup_{x \in [0,  b_V + 1]} V(x)$ and $A_4 = B_0 \cdot F_2(b_V + 1) + B_0$.

Let $\vec\ell'\in \mathbb{W}_N^{\theta,M_N} \subset \mathbb{W}_N^{\theta,\infty}$ be as in Lemma \ref{const} for $r=N$. Note that
\begin{equation*}
Z_N \ge Z_N\cdot\P_N(\vec\ell') = \exp\big(\theta N(N-1)\log N-N^2I_{V_N}^{\theta}(\mu_N(\vec\ell'))+O(N\log N)\big)
\end{equation*}
where the last equality follows from Lemma \ref{pmf-exact} with $M=\infty$. The latter and \eqref{S3conv-rate} imply
\begin{align*}
Z_N \ge  \exp\big(\theta N(N-1)\log N-N^2I_{V}^{\theta}(\mu_N(\vec\ell'))+O(N^2 \cdot \gn + N\log N)\big).
\end{align*} 
Notice that from \eqref{clo} we have
\begin{align*}
-N^2I_{V}^{\theta}(\mu_N(\vec\ell')) \ge -N^2F_V^{\theta,b_V+1}+O(N\log N),
\end{align*}
where we recognize $F_V^{\theta,b_V+1}=\inf_{\phi\in \mathcal{A}_{b_V+ 1}^{\theta}} I_V^{\theta}(\phi)$ by Lemma \ref{ivi}. 

From Lemma \ref{MinEq} applied to $s_1 = b_V + 1$ and $s_2 = \infty$ we know $F_V^{\theta, \infty} = F_V^{\theta, b_V + 1}$. Hence, the above two equations imply
\begin{align*}
Z_N \ge \exp\left(\theta N(N-1)\log N-N^2F_V^{\theta,\infty}+O(N^2 \cdot \gn + N\log N)\right).
\end{align*}
The last equation and Lemma \ref{pmf-exact} with $M=\infty$ together imply \eqref{u-bound}, completing the proof.
\end{proof}

\begin{lemma}\label{large-region} Suppose that $\P_{N}$ is as in Definition \ref{S3PDef}. There exists $\MC > 0$ such that for $N \geq N_0$
\begin{align}\label{eq:lar-reg}
\P_N\left(k_{\MC} (\vec{\ell})\ge N/2\right)\le \exp \left(-N^2+ O(N^2 \cdot \gn+N\log N) \right),
\end{align}
where $\MC$ and the constant in the  big $O$ notation depend on $V, \theta, \xi, \Bc, \F_1, \F_2$ from Assumption \ref{as3}.
\end{lemma}
\begin{proof}In the proof below the constants in all big $O$ notations will depend on $V, \theta,\xi, B_0, F_1, F_2$ from Assumptions \ref{as3} -- we will not mention this further. For clarity we split the proof into two steps.\\

{\bf \raggedleft Step 1.} We note that by making the constant in the big $O$ notation in (\ref{eq:lar-reg}) big enough we can ensure that the right side is larger than $1$ whenever $N \leq 2 \cdot \max \left(1, \xi^{-1} \theta^{-1}\right)$ and thus we may assume that $N  \geq 2 \cdot \max \left(1, \xi^{-1} \theta^{-1}\right)$. Let $T$ be sufficiently large so that for all $N \geq 2 \cdot \max \left(1, \xi^{-1} \theta^{-1} \right)$ 
\begin{align}\label{t_choice}
(T^2+1)^{N(N\theta\xi-1)/2} \ge \exp(N^2(\fni+1)),
\end{align} 
where $F_{V}^{\theta, \infty}$ is as in \eqref{fs}. We let $\MC = T + \theta + 1$ and show below that (\ref{eq:lar-reg}) holds with this choice of $\MC$. \\

For any $k \in \{1, \dots, N\}$ define 
\begin{align}\label{amk}
A^{\MC}_k :=\{\vec\ell\in\mathbb{W}_N^{\theta,\infty} \mid \ell_{k}\ge \MC N, \ell_{k+1} < \MC N  \}
\end{align} 
We claim that for any $N \geq 2 \cdot \max \left(1, \xi^{-1} \theta^{-1} \right)$  and $k \in \{1, \dots, N\}$ we have
\begin{equation}\label{S3L37R1}
\sum_{\ell\in A^{\MC}_k}\P_N(\vec\ell) \leq  \exp(N^2\fni+O( N^2 \cdot \gn + N\log N))\left(\frac{N\pi/2}{(T^2+1)^{N\theta\xi-1}}\right)^k.
\end{equation}
We will prove (\ref{S3L37R1}) in the next step. Here we assume its validity and conclude the proof of (\ref{eq:lar-reg}).\\

In view of (\ref{S3L37R1}) we have for any $N \geq 2 \cdot \max \left(1, \xi^{-1} \theta^{-1} \right)$ 
\begin{equation*}
\begin{split}
&\P_N(k_{\MC}\ge {N}/{2})  =\sum_{k=\lceil N/2\rceil}^N \sum_{\ell\in A^{\MC}_k} \P_N(\vec\ell)  \le \exp(N^2\fni+O( N^2 \cdot \gn + N\log N)) 
\\& \times  (T^2+1)^{-N(N\theta\xi-1)/2}  \le \exp\left(-N^2+O(N^2 \cdot \gn + N\log N) \right).
\end{split}
\end{equation*}
where the last inequality follows from \eqref{t_choice}. The last equation implies (\ref{eq:lar-reg}). \\

{\bf \raggedleft Step 2.} In this step we prove (\ref{S3L37R1}) and in the sequel we fix $N \geq 2 \cdot \max \left(1, \xi^{-1} \theta^{-1} \right)$  and $k \in \{1, \dots, N\}$. From \eqref{u-bound} we have for each $\vec{\ell} \in \mathbb{W}_N^{\theta,\infty}$ that
\begin{align*}
\P_N(\vec\ell)& \le \exp(N^2\fni+O(N^2 \cdot \gn + N\log N))\prod_{1\le i<j\le N}\left|\frac{\ell_i}{N}-\frac{\ell_j}{N}\right|^{2\theta} \prod_{i=1}^N \exp(-NV_N(\ell_i/N)).
\end{align*}
Using the inequalities $|x-y|\le \sqrt{(1+x^2)(1+y^2)}$ and $V_N(t)\ge \theta(1+\xi)\log(1+t^2)$ (the latter is true by \eqref{S3VPot}) we get for each $\vec\ell \in \mathbb{W}_N^{\theta,\infty}$ that
\begin{align*} 
\P_N(\vec\ell) \le \exp(N^2\fni+O(N^2 \cdot \gn + N\log N)) \prod_{i=1}^N \frac{1}{\left(\ell_i^2/N^2+1\right)^{N\theta\xi}} .
\end{align*}

Consider the map $\tau : A^{\MC}_k \to \mathbb{R}_{\ge 0}^k$ that sends $\vec\ell \mapsto (\ell_{1},\ell_{2},\ldots,\ell_{k})$. From the last equation we get
\begin{equation}\label{SumAk}
\begin{split}
&\sum_{\vec\ell\in A^{\MC}_k}\P_N(\vec\ell)  \le\exp(N^2\fni+O(N^2 \cdot \gn + N\log N)) \cdot \sum_{\vec\ell\in A^{\MC}_k} \prod_{i=1}^k \frac{1}{\left(\ell_i^2/N^2+1\right)^{N\theta\xi}} \\
&= \exp(N^2\fni+O(N^2\cdot \gn + N\log N)) \cdot \hspace{-10mm} \sum_{\vec{u} \in  \mathbb{R}_{\ge 0}^k: \substack{u_i\in \mathbb{Z}+(N-i)\theta \\ u_i\ge (T+\theta+1)N}} \sum_{\vec\ell\in \tau^{-1}(u_1, \dots, u_k)} \prod_{i=1}^k \frac{1}{\left(u_i^2/N^2+1\right)^{N\theta\xi}}  \\
& \leq  \exp(N^2\fni+O(N^2 \cdot \gn + N\log N)) \cdot \left(\frac{N\pi/2}{(T^2+1)^{N\theta\xi-1}}\right)^k,
\end{split}
\end{equation}
where in the last inequality we used (\ref{tail}) and the fact that $|\tau^{-1}(u_1, \dots, u_k)| = O(e^{N \log N})$. Equation (\ref{SumAk}) implies (\ref{S3L37R1}), which concludes the proof of the lemma.
\end{proof}

With the above results in place, we can proceed with the proof of Lemma \ref{mid-region}.
\begin{proof}[Proof of Lemma \ref{mid-region}] In the proof below the constants in all big $O$ notations will depend on $V, \theta,\xi, B_0, F_1, F_2$ from Assumptions \ref{as3} -- we will not mention this further.\\

We begin by defining the appropriate constants required in our proof. Recall the definition of equilibrium measure $\phi_V^{\theta,\infty}$ and $F_V^{\theta,\infty}$ from Lemma \ref{iv}. By Lemma \ref{iv}, $\phi_V^{\theta,\infty}$ has a compact support and we let $b_V \in [\theta, \infty)$ be the right-most point of its support. Let $\til{B} \geq 0$ be sufficiently large, depending on $V$, so that for all $N \geq 2$ and $1 \leq k \leq N/2$ we have
\begin{align}\label{new-con}
\theta \cdot \log \left(\prod_{i = 1}^{k-1} \frac{(i-1)! (k-i)! \theta^{k-1}}{N^{k-1}}\right) - 2k N \sup_{x \in [0, b_V + 3 + \theta]}|V(x)| \geq  - \tilde{B} k N.
\end{align} 
To see why such a choice of $\til{B}$ is possible note that 
\begin{equation*}
\begin{split}
&\log \left(\prod_{i = 1}^{k-1} \frac{(i-1)! (k-i)! \theta^{k-1}}{N^{k-1}}\right)  = (k-1)^2 \log \theta +   \log \left(\prod_{i = 1}^{k-1} \frac{(i-1)! (k-i)! }{N^{k-1}}\right)  \geq  (k-1)^2 \log \theta  \\
&+   \log \left(\prod_{i = 1}^{k-1} \frac{(i-1)! (N-i)! }{N^{N-1}}\right) = (k-1)^2 \log \theta - \sum_{i = 1}^{k-1}  \log \binom{N-1}{i}+ (k-1) \log \left( \frac{(N-1)! }{N^{N-1}}\right) \\
& \geq (k-1)^2 \log \theta  - (k-1) \log (2^{N-1}) - (k-1) N,
\end{split}
\end{equation*}
where in the last inequality we used that $\binom{N-1}{i} \leq 2^{N-1}$ and $n^n/ n! \geq e^n$, see (\ref{EqnRob}).

Let $T$ be sufficiently large so that for $N \geq 2 \cdot \max \left(1, \xi^{-1} \theta^{-1}\right)$ we have
\begin{equation}\label{tC2}
(T^2 + 1)^{-N\theta\xi +1} \leq \exp( - N \til{B} - N )\mbox{ and  set } \MB = \max( T + \theta + 1, \MC),
\end{equation}
where $\MC$ is as in Lemma \ref{large-region}. We will prove the lemma for the above choice of $\MB$ and for clarity we split the proof into five steps. \\

{\bf \raggedleft Step 1.} We note that by making the constant in the big $O$ notation in (\ref{S3EqMid}) big enough we can ensure that the right side is larger than $1$ whenever $N \leq 2 \cdot \max \left(1, \xi^{-1} \theta^{-1}\right)$ and thus we may assume that $N  \geq 2 \cdot \max \left(1, \xi^{-1} \theta^{-1}\right)$.

Recall from (\ref{amk}) that 
$$A^{\MB}_k :=\{\vec\ell\in\mathbb{W}_N^{\theta,\infty} \mid \ell_{k}\ge \MB N, \ell_{k+1} < \MB N  \}.$$
We claim that for all $N \geq 2 \cdot \max \left(1, \xi^{-1} \theta^{-1}\right)$, $1 \leq k \leq N/2$ and $\vec{\ell} \in A^{\MB}_k$ we have
\begin{equation}\label{S33R1}
\P_N(\vec\ell)  \le \exp(O(N^2 \cdot \gn+N\log N)+\til{B}\cdot Nk)\prod_{i=1}^k \frac1{\left(\ell_i^2/N^2+1\right)^{\theta \xi N}} .
\end{equation}
We will prove (\ref{S33R1}) in the next steps. Here we assume its validity and conclude the proof of (\ref{S3EqMid}).\\

Arguing as in (\ref{SumAk}) we see that (\ref{S33R1}) implies for all $N \geq 2 \cdot \max \left(1, \xi^{-1} \theta^{-1}\right)$ and $1 \leq k \leq N/2$
$$\sum_{\vec\ell\in A^{\MB}_k} \P_N(\vec\ell)  \leq  \exp(O(N^2 \cdot \gn+N\log N)+\til{B}\cdot Nk) \cdot  \left(\frac{N\pi/2}{(T^2+1)^{N\theta\xi-1}}\right)^k.$$
The last equation and the first inequality in (\ref{tC2}) imply that for $N \geq 2 \cdot \max \left(1, \xi^{-1} \theta^{-1}\right)$ and $N/2 \geq n \geq 1$
$$\P_N\left( n \le k_{\MB} (\vec{\ell}) \leq N/2 \right) = \sum_{N/2 \geq k \geq n} \sum_{\vec\ell\in A^{\MB}_k} \P_N(\vec\ell) \leq \exp(O(N^2 \cdot \gn+N\log N) -N n).$$
On the other hand, since $\MB \geq \MC$ we have by Lemma \ref{large-region} that for all $N \geq N_0$
$$\P_N\left(k_{\MB} (\vec{\ell})\ge N/2\right)\le \exp \left(-N^2+ O(N^2 \cdot \gn+N\log N) \right).$$
The last two equations imply (\ref{S3EqMid}).\\

{\bf \raggedleft Step 2.} In this step we prove (\ref{S33R1}) and in the sequel we fix $N \geq 2 \cdot \max \left(1, \xi^{-1} \theta^{-1}\right)$ and $1 \leq k \leq N/2$. Let us set $r = N- k$ and for any $ \vec{\ell} \in \mathbb{W}_N^{\theta,\infty}$ as in (\ref{GenState}) define the atomic measure 
$$\mu_{N,r}(\vec\ell)=\frac1{r}\sum_{i=k+1}^N \delta\left(\frac{\ell_{i}}{N}\right).$$

We claim that for any $\vec{\ell} \in A^{\MB}_k$ we have 
\begin{align}\label{eq:aimV2}
N^2I_{V_N}^{\theta}(\mu_N(\vec\ell)) \ge r^2I_{V_N}^{\theta}(\mu_{N,r}(\vec\ell))+\theta\xi N\sum_{i=1}^k\log\left(1+\ell_i^2/N^2\right).
\end{align}
We also claim that there exists $\vec\ell'\in \mathbb{W}_r^{\theta,\infty}$ such that $\ell'_1\le N (b_V + 1 + \theta)$ and $\vec{\ell}'$ satisfies  
\begin{align}\label{st3V2}
r^2I_V^{\theta}(\mu_{N,r}(\vec\ell')) \ge N^2\fni+O(N\log N)-\til{B}\cdot Nk.
\end{align}
and for all $\vec{\ell} \in A^{\MB}_k$
\begin{align}\label{passV2}
I_{V_N}^{\theta}(\mu_{N,r}(\vec{\ell})) \ge I_V^{\theta}(\mu_{N,r}(\vec\ell'))+O(\gn+N^{-1}\log N).
\end{align}
We will prove (\ref{eq:aimV2}), (\ref{st3V2}) and (\ref{passV2}) in the next steps. Here we assume their validity and conclude the proof of (\ref{S33R1}).\\

In view of (\ref{u-bound}) and (\ref{eq:aimV2}) we have 
$$\P_N(\vec\ell) \le \exp \left(N^2 F_{V}^{\theta,\infty}-r^2 I^{\theta}_{V_N}(\mu_{N,r}(\vec{\ell})) + O\left(N^2 \cdot \gn+N\log N\right)\right) \cdot \prod_{i=1}^k \frac1{\left(\ell_i^2/N^2+1\right)^{\theta \xi N}}.$$
Combining the last inequality with (\ref{st3V2}) and (\ref{passV2}) we get
$$\P_N(\vec\ell) \le \exp \left(\tilde{B} \cdot Nk + O\left(N^2 \cdot \gn+N\log N\right)\right) \cdot \prod_{i=1}^k \frac1{\left(\ell_i^2/N^2+1\right)^{\theta \xi N}},$$
which gives (\ref{S33R1}).\\

{\bf \raggedleft Step 3.} Our goal in this step is to show (\ref{eq:aimV2}). Towards this end we begin by separating out the logarithmic interaction terms.
\begin{equation*}
\begin{aligned}
\sum_{1\le i<j\le N} \hspace{-3mm} \log \left(\frac{\ell_i}{N}-\frac{\ell_j}{N}\right) & =\sum_{1\le i<j\le k} \hspace{-3mm} \log \left(\frac{\ell_i}{N}-\frac{\ell_j}{N}\right) + \sum_{\substack{1\le i\le k \\ k+1\le j\le N}} \hspace{-3mm}\log \left(\frac{\ell_i}{N}-\frac{\ell_j}{N}\right) + \hspace{-2mm}\sum_{k+1\le i<j\le N} \hspace{-4mm} \log \left(\frac{\ell_i}{N}-\frac{\ell_j}{N}\right). 
\end{aligned}
\end{equation*}
We keep the third sum on the right as it is. Using $|x-y| \le \sqrt{(1+x^2)(1+y^2)}$ we get
\begin{align*}
\sum_{1\le i<j\le k} \log \left(\frac{\ell_i}{N}-\frac{\ell_j}{N}\right) & \le \sum_{1\le i<j\le k} \log \sqrt{1+\ell_i^2/ N^2}\sqrt{1+\ell_j^2/N^2}.
\end{align*}
Using the fact that for $x>y\ge 0$ we have $(x-y)\le x \le \sqrt{1+x^2}$ we get
\begin{align*}
\sum_{\substack{1\le i\le k \\ k+1\le j\le N}} \log \left(\frac{\ell_i}{N}-\frac{\ell_j}{N}\right) \le (N-k)\sum_{i=1}^k \log \sqrt{1+\ell_i^2/ N^2}.
\end{align*}
Combining the above three equations wtih \eqref{S3VPot} we conclude
\begin{equation*}
\begin{aligned}
&N^2I_{V_N}^{\theta}(\mu_N(\vec\ell))  = -2\theta\sum_{1\le i<j\le N} \log \left(\frac{\ell_i}{N}-\frac{\ell_j}{N}\right)+N\sum_{i=1}^N V_N\left(\frac{\ell_i}{N}\right)   \\ 
&\geq -2\theta\sum_{k+1\le i<j\le N} \log \left(\frac{\ell_i}{N}-\frac{\ell_j}{N}\right)+N\sum_{i=k+1}^N V_N\left(\frac{\ell_i}{N}\right) +\theta\xi N\sum_{i=1}^k  \log\left(1+\frac{\ell_i^2}{N^2}\right)   \\
& \geq r^2I_{V_N}^{\theta}(\mu_{N,r}(\vec\ell)) +\theta \xi N\sum_{i=1}^k  \log\left(1+\frac{\ell_i^2}{N^2}\right),
\end{aligned}
\end{equation*}
where in the second inequality we used that $V_N \geq 0$ by \eqref{S3VPot}. The last equation implies (\ref{eq:aimV2}).\\
		
{\bf \raggedleft Step 4.} In this step we construct $\vec\ell'\in \mathbb{W}_r^{\theta,\infty}$ such that $\ell'_1\le N (b_V+1 + \theta)$  and show that it satisfies (\ref{passV2}). The idea of the construction and the proof is to use Lemma \ref{const}; however to accomplish this we need to set ourselves in the finite $M$ case and check that Assumptions \ref{as1} and \ref{as2} are satisfied. 

Set $M_N=\lfloor N (b_V +1) \rfloor$ and $\lM=b_V +1  $. Note that $M_N$ and $\lM$ satisfy Assumption \ref{as1} with $a_0 = A_0 = b_V +1$, $q_N=1/N$ and $A_1=1$. In addition, we note that Assumption \ref{as2} is satisfied for $p_N = r_N$, $A_2 = F_1( b_V + 1 + \theta)$, $A_3 = \sup_{x \in [0,  b_V + 1 + \theta]} |V(x)|$ and $A_4 = B_0 \cdot F_2(b_V + 1 + \theta) + B_0$. In view of Lemma \ref{const} we know that there exists $\vec\ell'\in \mathbb{W}_r^{\theta, M_N} $ such that 
\begin{equation}\label{passV2R1}
r^2I_V^{\theta}\left(\mu_{N,r}(\vec\ell') \right) \le r^2 \inf_{ \phi \in \mathcal{A}^{\alpha}_{\lM + \alpha}} I_V^{\theta}(\phi) +O(  N\log N),
\end{equation}
where we recall that $\alpha =  \theta r/N$ and $\mathcal{A}^{\alpha}_{\lM + \alpha}$ is as in \eqref{ainf}. This will be our choice of $\vec{\ell}'$ and we note that $\ell'_1 \leq M_N + (r-1) \theta \leq N(b_V + 1 +\theta )$. In the remainder of this step we fix $\vec{\ell} \in A^{\MB}_k$ and show that $\vec{\ell}'$, as constructed above, satisfies (\ref{passV2}). 

Note that by Assumption \ref{as3}, namely (\ref{S3conv-rate}), we know that 
$$I_{V_N}^{\theta}(\mu_{N,r}(\vec{\ell}))=I_{V}^{\theta}(\mu_{N,r}(\vec{\ell}))+O(\gn).$$
Furthermore, if $\til\mu_{N,r}(\vec{\ell})$ is the measure with density as in \eqref{cont}, we have by Lemma \ref{LemmaTech4} that 
$$I_{V}^{\theta}(\mu_{N,r}(\vec{\ell}))=I_V^{\theta}(\til{\mu}_{N,r}(\vec{\ell}))+O(N^{-1}\log N).$$ 	
We remark that in applying  Lemma \ref{LemmaTech4} we used that for all $\vec{\ell} \in A^{\MB}_k$ we have $(\ell_{k+1}, \dots, \ell_{N}) \in \mathbb{W}_r^{\theta, \MB N}$. Consequently, the conditions of the lemma hold for $A = \MB$, $\theta$ as in the present lemma, $A_3 = \sup_{x \in [0, \MB + \theta]}|V(x)|$ and $A_4 = B_0 \cdot F_2(\MB + \theta) + B_0$. 

Combining the last two equations with (\ref{passV2R1}) we see that to show (\ref{passV2}) it suffices to prove that 
\begin{align}\label{passV2R2}
I_V^{\theta}(\til{\mu}_{N,r}(\vec{\ell})) \ge \inf_{ \phi \in \mathcal{A}^{\alpha}_{\lM + \alpha}} I_V^{\theta}(\phi).
\end{align}

Notice that the minimizer of $I_V^\theta$ over $\mathcal{A}^{\theta}_{s + \theta}$ for any $s \geq b_V - \theta$ is precisely $\phi_V^{\theta,\infty}$ since the support of $\phi_V^{\theta,\infty}$ is in $[0,b_V]$. The latter and the second part of Lemma \ref{ivi} imply that if $\Phi^s_\alpha$ is the unique minimizer of $I_V^\theta$ over $\mathcal{A}^{\alpha}_{s + \theta}$ for any $s \geq b_V - \theta$ then $\Phi^s_\alpha$ is supported in $[0, b_V]$. In particular, we see that $\Phi^s_\alpha = \Phi^{b_V - \theta}_\alpha$ for all $s \geq b_V - \theta$, which implies that 
$$  \inf_{ \phi \in \mathcal{A}^{\alpha}_{\lM + \alpha}} I_V^{\theta}(\phi) = I_V^{\theta}(\Phi^{\lM - \theta + \alpha}_\alpha) = I_V^{\theta}(\Phi^{\MB + b_V - \theta + 1 }_\alpha)   =  \inf_{ \phi \in \mathcal{A}^{\alpha}_{\MB + b_V + 1}} I_V^{\theta}(\phi) \leq I_V^{\theta}(\til{\mu}_{N,r}(\vec{\ell})),$$
where we used that $\lM + \alpha - \theta \geq \lM - \theta \geq b_V - \theta$  and that $\til{\mu}_{N,r}(\vec{\ell}) \in \mathcal{A}^{\alpha}_{\MB + b_V + 1 }$ by construction (recall that $\vec{\ell} \in A^{\MB}_k$). The last equation proves (\ref{passV2R2}) and hence completes our work in this step. \\

{\bf \raggedleft Step 5.} In this final step we prove that the $\vec{\ell}'$ we constructed in Step 4 satisfies (\ref{st3V2}). Starting from $\vec{\ell}'$ we define $\vec{\ell}'' \in \mathbb{W}_N^{\theta,\infty}$ as follows
\begin{equation}\label{S33S5E1}
 \ell_j'' =\lceil N (b_V + 2)\rceil +(N-j) \cdot \theta, \mbox{ for } j=1,\ldots,k, \mbox{ and } \ell_j'' = \ell_{j-k}'  \mbox{, for }  j=k+1,\ldots, N.
\end{equation}
We remark that by construction we have $\vec{\ell}'' \in \mathbb{W}_N^{\theta,\infty}$ since $\ell'_1 - (r -1) \theta \leq M_N = \lfloor N (b_V +1) \rfloor$. We claim that $\vec{\ell}''$ satisfies the following two inequalities
\begin{equation}\label{S33S5E2}
N^2I_V^{\theta}(\mu_N(\vec\ell'')) \ge N^2\fni+O(N\log N) \mbox{ and } r^2I_V^{\theta}(\mu_{N,r}(\vec\ell')) \ge N^2I_V^{\theta} (\mu_{N}(\vec\ell''))-\til{B} \cdot Nk.
\end{equation}
The inequalities in (\ref{S33S5E2}) together imply (\ref{st3V2}), and so we only need to show (\ref{S33S5E2}).\\

If $\til\mu_{N}(\vec{\ell}'')$ is the measure with density as in \eqref{cont} for $r = N$, we have by Lemma \ref{LemmaTech4} that 
$$I_{V}^{\theta}(\mu_{N}(\vec{\ell}''))=I_V^{\theta}(\til{\mu}_{N}(\vec{\ell}''))+O(N^{-1}\log N).$$ 	
We remark that in applying  Lemma \ref{LemmaTech4} we used that $\ell''_1  = \lceil N (b_V + 2)\rceil  + (N-1)\theta$ so that $\vec{\ell}'' \in \mathbb{W}_N^{\theta, N(b_V + 3) }$. Consequently, the conditions of the lemma hold for $A = b_V + 3$, $\theta$ as in the present lemma, $A_3 = \sup_{x \in [0, b_V + 3 + \theta]}|V(x)|$ and $A_4 = B_0 \cdot F_2(b_V + 3 + \theta) + B_0$. On the other hand, since $I_V^{\theta}(\til{\mu}_{N}(\vec{\ell}'')) \in \mathcal{A}^{\theta}_{\infty}$ by construction, we have 
$$I_V^{\theta}(\til{\mu}_{N}(\vec{\ell}''))  \geq \fni.$$
The last two equations imply the first inequality in (\ref{S33S5E2}).

By the definition of $\vec{\ell}''$ and $I_V^{\theta}$ we have
\begin{equation*}
\begin{split}
 & r^2I_V^{\theta}(\mu_{N,r}(\vec\ell'))  =N^2I_V^{\theta} (\mu_{N}(\vec\ell''))+\theta\sum_{1\le i\le k, k+1\le j\le N} \log\left|\frac{\ell_i''}{N}-\frac{\ell''_j}N\right| + \theta\sum_{1\le i\neq j \le k}\log\left|\frac{\ell_i''}{N}-\frac{\ell''_j}N\right| \\ 
&-k\sum_{i=k+1}^N V(\ell_i''/N)- N\sum_{i=1}^k V(\ell_i''/ N) \\
&\geq N^2I_V^{\theta} (\mu_{N}(\vec\ell'')) + \theta\sum_{1\le i\neq j \le k}\log\left|\frac{\ell_i''}{N}-\frac{\ell''_j}N\right| - 2k N \sup_{x \in [0, b_V + 3 + \theta]}|V(x)|, 
\end{split}
\end{equation*}
where in the second line we used that $\ell''_k - \ell''_{k+1} \geq N$ and that $\ell_1'' \leq N(b_V + 3 + \theta)$ (by the construction in (\ref{S33S5E1})). Since $\ell_i'' - \ell_j'' = (j-i) \theta$ for $1 \leq i \neq j \leq k$  (again by the construction in (\ref{S33S5E1})) we conclude that 
$$ r^2I_V^{\theta}(\mu_{N,r}(\vec\ell'))  \geq N^2I_V^{\theta} (\mu_{N}(\vec\ell''))  + \theta \cdot \log \left(\prod_{i = 1}^{k-1} \frac{(i-1)! (k-i)! \theta^{k-1}}{N^{k-1}}\right) - 2k N \sup_{x \in [0, b_V + 3 + \theta]}|V(x)|.$$
In view of the definition of $\tilde{B}$ in (\ref{new-con}) we see that the last equation implies the second inequality in (\ref{S33S5E2}). This completes the proof of (\ref{S33S5E2}) and hence the lemma.

\end{proof}

%
\section{Lower tail LDP} \label{Section4} In this section we prove analogues of Theorems \ref{emp} and \ref{TMain}(a) for the measures from Section \ref{Section2.1} -- these are Propositions \ref{S4emp} and \ref{FinLLDP} below. In Section \ref{Section4.1} we use Proposition \ref{S4emp} to prove Theorem \ref{emp} and Proposition \ref{FinLLDP} to prove Theorem \ref{TMain}(a). The two propositions are proved in Section \ref{Section4.2}. \\

The main results of the section are as follows.

\begin{proposition}\label{S4emp} Suppose that $\P_{N}$ are as in Definition \ref{S2PDef} and that $\lim_{N \rightarrow\infty} \fn = \lim_{N \rightarrow \infty} \qn = 0$. If $\mu_N(\vec{\ell})$ are the empirical measures in Definition \ref{S2PDef} then $\mu_N(\vec{\ell})$ converges weakly in probability to $\phi_V^{\theta, \lM + \theta}$ from Lemma \ref{iv} in the sense that for any bounded real continuous function $f$ on $[0,\infty)$ the sequence of random variables
\begin{align*}
 \int_{0}^{\infty} f(x)\mu_N(dx) - \int_{0}^{\infty} f(x)\phi_V^{\theta, \lM+\theta}(x)dx
\end{align*}
converges to $0$ in probability.
\end{proposition}

\begin{proposition}\label{FinLLDP}Suppose that $\P_{N}$ are as in Definition \ref{S2PDef} and that $\lim_{N \rightarrow\infty} \fn = \lim_{N \rightarrow \infty} \qn = 0$. Then for any $t \in [\theta, \lM + \theta]$ we have
\begin{equation}\label{S4Lim}
\lim_{N\rightarrow \infty} \frac{1}{N^2} \log \P_N ( \ell_1 \leq tN) =  F_V^{\theta, \lM + \theta} - F_V^{\theta,t} ,
\end{equation}
where we recall that $F_{V}^{\theta,t}$ is defined in \eqref{fs}. Moreover, if $b_V$ is the rightmost point of the support of $\phi_V^{\theta, \lM + \theta}$ as in Lemma \ref{iv} then $ F_V^{\theta, t} > F_V^{\theta, \lM + \theta} $ for $t \in [\theta, b_V)$ and $ F_V^{\theta,t } = F_V^{\theta, \lM + \theta} $ for $t \in[b_V, \lM + \theta]$. 
\end{proposition}

%
\subsection{Proof of Theorems \ref{emp} and \ref{TMain}(a)}\label{Section4.1} The idea of both proofs is to reduce the problem to the finite $M$ setting and then apply Propositions \ref{S4emp} and \ref{FinLLDP}. In order to accomplish this we need to apply Proposition \ref{exptight} and subsequently check that Assumptions \ref{as1} and \ref{as2} are satisfied. \\

\begin{proof}[Proof of Theorem \ref{emp}] We continue with the same notation as in Theorem \ref{emp}.  For convenience we denote the random variables
$$X^f_N = \int_{0}^{\infty} f(x)\mu_N(dx) - \int_{0}^{\infty} f(x)\phi_V^{\theta,  \infty}(x)dx = \sum_{i = 1}^N f(\ell_i/N) - \int_{0}^{\infty} f(x)\phi_V^{\theta,  \infty}(x)dx.$$
To prove the theorem, we need to show that for any $\epsilon > 0$ we have
\begin{equation}\label{S40}
\lim_{N \rightarrow \infty} \P^{\theta, \infty}_N( |X^f_N| > \epsilon) = 0.
\end{equation}

Recall the definition of the equilibrium measures $\phi_V^{\theta,s}$ and $F_V^{\theta,s}$ from Lemma \ref{iv}. By Lemma \ref{iv}, $\phi_V^{\theta,\infty}$ has a compact support and we let $b_V \in [\theta, \infty)$ be the rightmost point of its support. 

Observe that the measures $\P_N^{\theta, \infty}$ satisfy the conditions of Definition \ref{S3PDef} and that $\gn = O(N^{-3/4})$ by assumption. By Proposition \ref{exptight} applied to $A = 2$ we can find integers $R_1 > b_V - \theta$ and $N_1 \geq \max(2, \xi^{-1} \theta^{-1}) $ such that for $N \geq N_1$ we have 
\begin{equation}\label{S40E1}
\P_N^{\theta, \infty} ( \ell_1 - (N-1) \theta > R_1 N ) \leq e^{-N}. 
\end{equation} 
 Set $M_N= N R_1 $ and $\lM= R_1$. Note that $M_N$ and $\lM$ satisfy Assumption \ref{as1} with $a_0 = A_0 = R_1$, $\qn= 0$ and $A_1=1$. In addition, we note that Assumption \ref{as2} is satisfied for $\fn = \gn$, $A_2 = F_1( R_1 + \theta)$, $A_3 = \sup_{x \in [0, R_1 + \theta]} V(x)$ and $A_4 = B_0 \cdot F_2( R_1 + \theta) + B_0$. In particular, we observe that $\P_N^{\theta, \infty}$ conditional on $\{\ell_1 - (N-1) \theta \leq  R_1 N \}$ is precisely $\P_N^{\theta, R_1N}$ and the latter satisfies the conditions of Definition \ref{S2PDef} with the above constants and sequences.

From Lemma \ref{MinEq} applied to $s_1 = R_1 + \theta$ and $s_2 = \infty$ we know that $\phi_V^{\theta, \infty} =  \phi_V^{\theta, R_1 + \theta}$ and so from Proposition \ref{S4emp} we conclude that for any $\epsilon > 0$
\begin{equation}\label{S40E2}
\lim_{N \rightarrow \infty} \frac{\P^{\theta, \infty}_N ( \{ |X_N^f| > \epsilon \} \cap \{ \ell_1 \leq R_1 N + (N-1) \theta\} )}{\P^{\theta, \infty}_N(\ell_1 \leq R_1 N + (N-1) \theta)} = \lim_{N \rightarrow \infty} \P_N^{\theta, R_1 N} (|X_N^f| > \epsilon) = 0.
\end{equation}
Equations (\ref{S40E1}) and (\ref{S40E2}) together imply (\ref{S40}).
\end{proof}

\begin{proof}[Proof of Theorem \ref{TMain}(a)]
We continue with the same notation as in Theorem \ref{TMain} and proceed to prove (\ref{LeftTail}). Recall the definition of the equilibrium measures $\phi_V^{\theta,s}$ and $F_V^{\theta,s}$ from Lemma \ref{iv}. By Lemma \ref{iv}, $\phi_V^{\theta,\infty}$ has a compact support and we let $b_V \in [\theta, \infty)$ be its rightmost point. 

Observe that the measures $\P_N^{\theta, \infty}$ satisfy the conditions of Definition \ref{S3PDef} with $\gn = O(N^{-3/4})$. In particular, by Proposition \ref{exptight} applied to $A = 2$ we can find integers $R_1 > b_V - \theta$ and $N_1 \geq \max(2, \xi^{-1} \theta^{-1})$ such that for $N \geq N_1$ we have 
\begin{equation}\label{S41E1}
\P_N^{\theta, \infty} ( \ell_1 - (N-1) \theta > R_1 N ) \leq e^{-N} \leq 1/2. 
\end{equation} 
Set $M_N=NR_1$ and $\lM=R_1 $. Note that $M_N$ and $\lM$ satisfy Assumption \ref{as1} with $a_0 = A_0 = R_1$, $\qn=0$ and $A_1=1$. In addition, we note that Assumption \ref{as2} is satisfied for $\fn = \gn$, $A_2 = F_1(  R_1 + \theta)$, $A_3 = \sup_{x \in [0,  R_1 + \theta ]} V(x)$ and $A_4 = B_0 \cdot F_2( R_1 + \theta) + B_0$.

From Proposition \ref{FinLLDP} we conclude that for all $t \in [\theta, R_1 + \theta]$ we have
\begin{equation}\label{S41E2}
\lim_{N\rightarrow \infty} \frac{1}{N^2} \log \P^{\theta, N R_1}_N ( \ell_1 \leq tN) = F_V^{\theta,  R_1 + \theta } - F_V^{\theta, t}.
\end{equation} 
From (\ref{S41E1}) we know that for all $t \in [\theta, R_1 + \theta]$ and $N \geq N_1$ 
\begin{equation}\label{S4Sth}
(1/2) \cdot \P^{\theta, N R_1}_N ( \ell_1 \leq tN) \leq  \P^{\theta, \infty}_N ( \ell_1 \leq tN)  \leq \P^{\theta, N R_1}_N ( \ell_1 \leq tN),
\end{equation}
which in view of (\ref{S41E2}) implies that for all $t \in [\theta, R_1 + \theta]$
\begin{equation}\label{S41E3}
\lim_{N\rightarrow \infty} \frac{1}{N^2} \log \P^{\theta, \infty}_N ( \ell_1 \leq tN) = F_V^{\theta,  R_1 + \theta} - F_V^{\theta, t}.
\end{equation} 

From Lemma \ref{MinEq} applied to any $s_1 \in[ b_V, \infty)$ and $s_2 = \infty$ we have $F_V^{\theta,  s_1} = F_V^{\theta,  \infty}$. In particular, $F_V^{\theta,  R_1 + \theta} = F_V^{\theta,  \infty}$, which together with (\ref{S41E3}) imply (\ref{LeftTail}) for $t \in [\theta,  R_1 + \theta]$. Furthermore, using (\ref{S41E1}) we have for $t \geq  R_1 + \theta$ and all large $N$ that the terms on the right side of (\ref{S41E3}) are lower bounded by $- \log(2)/N^2$ and upper bounded by $0$ so that (\ref{LeftTail}) holds for $t \geq  R_1 + \theta$ as well (here we used that $F_V^{\theta,  t} = F_V^{\theta,  \infty}$ for $t \geq R_1 + \theta$). This proves (\ref{LeftTail}) for $t \in [\theta, \infty)$. \\

The above paragraph explained why $F_V^{\theta,t} - F_V^{\theta, \infty} = 0$ for all $t \geq b_V$. What is left is to show that $F_V^{\theta,t} - F_V^{\theta, \infty}  > 0$ for $t \in [\theta, b_V)$. Since $\mathcal{A}_{t}^{\theta} \subset \mathcal{A}_{\infty}^{\theta}$ we have from Lemma \ref{iv} that $F_V^{\theta,t} - F_V^{\theta, \infty}  \geq 0$. If $F_V^{\theta,t} - F_V^{\theta, \infty} = 0$ for some $t \in  [\theta, b_V)$ we would have by the uniqueness of the minimizer of $I_V^{\theta}$ over $\mathcal{A}_{\infty}^\theta$ as in Lemma \ref{iv} that $\phi_V^{\theta, t} = \phi_V^{\theta, \infty}$, which would imply $b_V \leq t$ -- an obvious contradiction. Thus $F_V^{\theta,t} - F_V^{\theta, \infty} $ is strictly positive on $[\theta, b_V)$. This suffices for the proof.
\end{proof}

%
\subsection{Proofs of Propositions \ref{S4emp} and \ref{FinLLDP}}\label{Section4.2} 

\begin{proof}[Proof of Proposition \ref{S4emp}] Let $f$ be a bounded continuous function on $[0,\infty)$. Let $R > 0$ be sufficiently large so that $N^{-1} \cdot [M_N + (N-1) \theta] \subset [0, R]$ and $[0, \lM + \theta] \subset R$. We also let $\phi$ be a smooth function that is equal to $1$ on $[0, R]$ is supported in $[-1, R + 1]$ and $\phi(x) \in [0, 1]$ for all $x \in \mathbb{R}$. We also fix $\epsilon > 0$.

By the Stone-Weierstrass theorem, see e.g. \cite[Theorem 7.26]{Rudin} there exists a polynomial $P(x)$ such that $\sup_{x \in [0,R]} |P(x) - f(x)| < \epsilon/4$. From Proposition \ref{gldp} applied to $p = 2$ we have for any $\gamma > 0$ and $N \geq 2$
\begin{equation*}
\begin{split}
& \P_N  \left(\left|  \int_{\R}  P(x) \mu_N(dx)- \int_{\R} P(x)\phi_V^{\theta, \lM + \theta} (x)dx \right|\ge \gamma \norm{P(x) \phi(x)}_{1/2}+\frac{\theta \norm{P(x) \phi(x)}_{\operatorname{Lip}}}{{N}^2} \right)  \\
& \le  \exp \left(-2\pi^2\gamma^2\theta N^2+O\left(N^2 \cdot\fn+ N^2\cdot \qn +N\log N\right) \right),
\end{split}
\end{equation*}
where we used that $\mu_N$ and $\phi_V^{\theta, \lM + \theta}$ are both supported in $[0, R]$ and $\phi(x) = 1$ there by construction. We recall that $\norm{g}_{1/2}, \norm{g}_{\operatorname{Lip}}$ were defined in (\ref{S2Norms}).

Fixing $\gamma > 0$ sufficiently small so that $\gamma \norm{P(x) \phi(x)}_{1/2} < \epsilon/6$, and using the fact that $\lim_{N \rightarrow\infty} \fn = \lim_{N \rightarrow \infty} \qn = 0$, we conclude that there is a constant $c_{\epsilon} > 0$ such that for all large $N$
\begin{equation*}
\begin{split}
& \P_N  \left(\left|  \int_{\R} P(x) \mu_N(dx)- \int_{\R} P(x)\phi_V^{\theta, \lM + \theta} (x)dx \right|\ge \epsilon/4 \right)  \leq e^{-c_\epsilon N^2}.
\end{split}
\end{equation*}
By the triangle inequality and our assumption that $\sup_{x \in [0,R]} |P(x) - f(x)| < \epsilon/4$ we get
$$\P_N  \left(\left|  \int_{\R} f(x) \mu_N(dx)- \int_{\R} f(x)\phi_V^{\theta, \lM + \theta} (x)dx \right|\ge 3\epsilon/4 \right)  \leq e^{-c_\epsilon N^2},$$
which implies the statement of the proposition. This suffices for the proof.
\end{proof}

We require one preliminary result -- Lemma \ref{LemmaS42}, after which we give the proof of Proposition \ref{FinLLDP}.
\begin{lemma}\label{LemmaS42}Suppose that $\P_{N}$ is as in Definition \ref{S2PDef} and $\lim_{N \rightarrow\infty} \fn = \lim_{N \rightarrow \infty} \qn = 0$. Then
\begin{equation}
\lim_{N\rightarrow \infty} \frac{1}{N^2} \left[ \log Z_N - \theta N(N-1) \log N\right] =  -F_{V}^{\theta, \lM + \theta}.
\end{equation}
\end{lemma}
\begin{proof} In the proof below the constants in all big $O$ notations will depend on $\theta$ and the constants $a_0, A_0,\ldots, A_4$ from Assumptions \ref{as1} and \ref{as2} -- we will not mention this further.

From equation (\ref{ZNLB}) we have
\begin{equation*}
Z_N \geq \exp \left(\theta N(N-1) \log N -N^2 F_{V}^{\theta, \lM + \theta}+ O\left(N^2 \cdot \fn+ N^2 \cdot \qn +N\log N\right) \right),
\end{equation*}
which from our assumption that $\lim_{N \rightarrow\infty} \fn = \lim_{N \rightarrow \infty} \qn = 0$ implies that 
$$\liminf_{N\rightarrow \infty} \frac{1}{N^2} \left[ \log Z_N - \theta N(N-1) \log N\right] \geq  -F_{V}^{\theta, \lM + \theta}.$$
Consequently, we only need to show that 
\begin{equation}\label{S42E1}
\limsup_{N\rightarrow \infty} \frac{1}{N^2} \left[ \log Z_N - \theta N(N-1) \log N\right] \leq  -F_{V}^{\theta, \lM + \theta}.
\end{equation}

For each $N$, let $\vec{\ell}' \in \mathbb{W}^{\theta,M_N}_{N} $ be such that 
$$\P_N( \vec{\ell}') = \max_{\vec{\ell} \in \mathbb{W}^{\theta,M_N}_{N} } \P_N (\vec{\ell}).$$
Since $|\mathbb{W}^{\theta,M_N}_{N}| = \exp(O(N \log N))$ we see that to prove (\ref{S42E1}) it suffices to show that 
\begin{equation}\label{S42E2}
\limsup_{N\rightarrow \infty} \frac{1}{N^2} \left[ \log \left(Z_N \cdot \P_N( \vec{\ell}') \right)  - \theta N(N-1) \log N\right] \leq  -F_{V}^{\theta, \lM + \theta}.
\end{equation}
Let us denote 
$$\mu_N(\vec{\ell}') =  \frac{1}{N}\sum_{i=1}^N\delta\left(\frac{\ell_i'}{N}\right) \mbox{ and } \psi_N(x) = \sum_{i = 1}^N f_{\ell_i'/N}(x) \mbox{, where } f_a(x) = (N/\theta) {\bf 1}_{x \in [a,a+\theta/N)}.$$
From (\ref{pf-bound}), the definition of $\fn$ and (\ref{EQMol}) applied to $r = N$ we have
\begin{equation*}
\begin{split}
&\log \left(Z_N \cdot \P_N( \vec{\ell}') \right)  - \theta N(N-1) \log N= -N^2 I^{\theta}_{V_N}(\mu_N) + O(N \log N)  \\
& = -N^2 I^{\theta}_{V}(\mu_N) + O(N^2 \cdot \fn + N \log N) =  -N^2 I^{\theta}_{V}(\psi_{N,N}) + O(N^2 \cdot \fn + N \log N) \\
&\leq  -N^2 \inf_{\phi \in \mathcal{A}_{M_N /N + \theta}^{\theta}}  I^{\theta}_{V}(\phi) + O(N^2 \cdot \fn + N \log N),
\end{split}
\end{equation*}
where the last inequality used that $ \psi_{N,N}(x) \in \mathcal{A}_{M_N /N + \theta}^{\theta}$ and we recall that $\mathcal{A}_{s + \theta}^\theta$ was defined in (\ref{ainf}). The last inequality, the definition of $\qn$ and (\ref{contV2}) imply
$$\log \left(Z_N \cdot \P_N( \vec{\ell}') \right)  - \theta N(N-1) \log N \leq -N^2 \inf_{\phi \in \mathcal{A}_{\lM + \theta}^{\theta}}  I^{\theta}_{V}(\phi) + O(N^2 \cdot \fn + N^2 \cdot \qn + N \log N).$$
Since $\inf_{\phi \in \mathcal{A}_{\lM + \theta}^{\theta}}  I^{\theta}_{V}(\phi)= F_{V}^{\theta, \lM + \theta}$ and $\lim_{N \rightarrow\infty} \fn = \lim_{N \rightarrow \infty} \qn = 0$ we see that the last equation implies (\ref{S42E2}). This suffices for the proof.
\end{proof}

\begin{proof}[Proof of  Proposition \ref{FinLLDP}] By assumption we have that $\P_N$ satisfies Definition \ref{S2PDef} with constants $\theta$, $a_0$, $A_0,\ldots, A_4$ and sequences $\fn, \qn$ both converging to $0$ as $N \rightarrow \infty$. For $t\in (\theta,\lM + \theta)$ define $Z_N(Nt):=Z_N\P_N(\ell_1\le tN)$. Notice that for all large enough $N$ (so that $(t -\theta)N + \theta\leq M_N$) $Z_N(Nt)$ is the partition function of the measure $\P_N^{\theta, M_N^t}$, where $M^t_N = \lfloor (t- \theta)N + \theta \rfloor$. In particular, $\P_N^{\theta, M_N^t}$ satisfies the conditions of Definition \ref{S2PDef} with $M_N = M_N^t$, $\lM = a_0 = A_0  = t - \theta$, $A_1 = 1$ sequence $\qn = (\theta + 1)/N$ and the same constants $\theta, A_2, A_3, A_4$ and sequence $\fn$ as above. From Lemma \ref{LemmaS42} we conclude that 
\begin{equation*}
\begin{split}
&\lim_{N \rightarrow \infty}  \frac{1}{N^2} \left[ \log Z_N - \theta N(N-1) \log N\right] =  -F_{V}^{\theta, \lM + \theta}, \mbox{ and } \\
&\lim_{N \rightarrow \infty}  \frac{1}{N^2} \left[ \log Z_N(Nt) - \theta N(N-1) \log N\right] = -F_{V}^{\theta, t},
\end{split}
\end{equation*}
which in view of $Z_N(Nt)=Z_N\P_N(\ell_1\le tN)$ implies (\ref{S4Lim}) for $t\in (\theta,\lM + \theta)$. 

If $t = \lM + \theta$ then for any $s \in(\theta, \lM + \theta)$ we have that 
\begin{equation*}
\begin{split}
0 &\geq \limsup_{N \rightarrow \infty}  \frac{1}{N^2} \log \P_N ( \ell_1 \leq (\lM + \theta)N) \geq \liminf_{N \rightarrow \infty}  \frac{1}{N^2} \log \P_N ( \ell_1 \leq (\lM + \theta)N) \\
& \geq \liminf_{N \rightarrow \infty}  \frac{1}{N^2} \log \P_N ( \ell_1 \leq s N) =F_{V}^{\theta, \lM + \theta} - F_{V}^{\theta, s}.
\end{split}
\end{equation*}
Sending $s \rightarrow \lM + \theta$ we see that $F_{V}^{\theta, \lM + \theta} - F_{V}^{\theta, s} \rightarrow 0$ from Lemma \ref{LemmaTech3}, and so we conclude that 
$$0 = \lim_{N \rightarrow \infty}  \frac{1}{N^2} \log \P_N ( \ell_1 \leq (\lM + \theta)N),$$
which implies (\ref{S4Lim}) for $t= \lM + \theta$ as well. 

Suppose now that $t = \theta$. Observe that in this case $A = \{\vec{\ell} \in \mathbb{W}_N^{\theta, M_N}: \ell_1 \leq tN \} = \{\vec{\ell} \in \mathbb{W}_N^{\theta, M_N}: \ell_1 = (N-1) \theta +k \mbox{ for }k = 0, \dots, \lfloor \theta \rfloor \}$. From Lemma \ref{pmf-exact} we see that for all $\vec{\ell} \in A$ and $N \geq 2$ 
$$Z_N \cdot \mathbb{P}_N (\vec{\ell}) = \exp \left( \theta N (N-1) - N^2 I_V^{\theta}(\rho) + O (N \log N + N^2 \fn) \right),$$
where $\rho(x) = \theta^{-1} \cdot {\bf 1}[0, \theta]$ and the constant in the big $O$ notation depends on $\theta$, $A_3$ and $A_4$.

Using the latter, the fact that  $|A| = \exp( O(\log N))$ for $N \geq 2$, where the constant in the big $O$ notation depends on $\theta$ alone and Lemma \ref{LemmaS42} we conclude that 
$$  \lim_{N\rightarrow \infty} \frac{1}{N^2} \log \P_N ( \ell_1 \leq \theta N)  = \lim_{N\rightarrow \infty} \frac{1}{N^2} \log \P_N ( \vec{\ell} \in A) =  F_V^{\theta,  \lM + \theta } - I_V^{\theta}(\rho).$$
The last equality proves (\ref{S4Lim}) for $t=  \theta$ since $ I_V^{\theta}(\rho) = F_V^{\theta, \theta}$. (Notice that $\mathcal{A}_{\theta}^{\theta}$ has the single element $\rho$ so that $F_V^{\theta, \theta} =  I_V^{\theta}(\rho) $.) This concludes the proof of (\ref{S4Lim}) for all $t \in [\theta, \lM + \theta]$.\\

We next turn to the second part of the lemma. From Lemma \ref{MinEq} applied to any $s_1 \in[ b_V, \lM + \theta]$ and $s_2 = \lM + \theta$ we conclude that $F_V^{\theta,\lM + \theta} = F_V^{\theta,t}$ for $t \in [b_V, \lM + \theta]$. Also by Lemma \ref{iv} $F_V^{\theta, t} \geq F_V^{\theta, \lM + \theta}$ for all $t \in [\theta, b_V)$. If for some $t \in [\theta, b_V)$ we have $F_V^{\theta, t} = F_V^{\theta, \lM + \theta}$ then by the uniqueness of the minimizer of $I_V^{\theta}$ over $\mathcal{A}_{\lM + \theta}^\theta$ as in Lemma \ref{iv} we get $\phi_V^{\theta, t} = \phi_V^{\theta, \lM+ \theta}$, which implies $b_V \leq t$ -- an obvious contradiction. Thus $F_V^{\theta, t} > F_V^{\theta, \lM + \theta}$ for all $t \in [\theta, b_V)$. This suffices for the proof.
\end{proof}

%
\section{Upper tail LDP} \label{Section5} In this section we prove an analogue of Theorem \ref{TMain}(b) for the measures from Section \ref{Section2.1} -- this is Proposition \ref{FinULDP} below. In Section \ref{Section5.1} we use Proposition \ref{FinULDP} to prove Theorem \ref{TMain}(b). The first part of Proposition \ref{FinULDP} is proved in Section \ref{Section5.2} and the second in Section \ref{Section5.3}. 

Before we state the main result of the section, we need the following lemma, whose proof is postponed until Section \ref{Section7} (see Lemma \ref{S7ContGJ}).
\begin{lemma}\label{S5ContGJ}
Fix $\lM, \theta > 0$ and a continuous function $V: [0, \lM + \theta] \rightarrow \mathbb{R}$. Let $\phi^{\theta, \lM + \theta}_V$ be as in Lemma \ref{iv} for this choice of $V$ and $b_V \in [\theta, \lM + \theta]$ be the rightmost point of its support. Recall the function $G^{\theta, \lM + \theta}_V$ on $[0, \infty)$ from (\ref{S2DefG})
\begin{equation}\label{S5DefG}
G^{\theta, \lM + \theta}_V(x) = - 2\theta \int_{0}^{\lM + \theta} \log |x-t| \phi_V^{\theta, \lM + \theta}(t) dt + V(x),
\end{equation}
which is continuous by  Lemma \ref{LemmaTech2}. We define the function
\begin{equation}\label{S5DefL}
J^{\theta, \lM + \theta}(x) =\begin{cases} 0 &\mbox{ if $x \in [0, b_V)$, } \\ \ds \inf_{y \in [x, \lM + \theta]} G^{\theta, \lM + \theta}_V(y)  - G^{\theta, \lM + \theta}_V(b_V)   &\mbox{ if $x \in [b_V, \lM + \theta]$}. \end{cases}
\end{equation}
Then the function $J^{\theta, \lM + \theta}$ is continuous on $[0, \lM + \theta]$.
\end{lemma}

\begin{proposition} \label{FinULDP} Suppose that $\P_{N}$ is as in Definition \ref{S2PDef} and that $ 0 = \lim_{N \rightarrow \infty} \fn N^{3/4} =\lim_{N \rightarrow \infty} \qn N^{3/4}$. Recall the definition of the equilibrium measure $\phi_V^{\theta,\lM + \theta}$ and $F_V^{\theta,\lM + \theta}$ from Lemma \ref{iv} and let $b_V \in [\theta, \lM + \theta]$ be the rightmost point of its support. The following statements hold.
\begin{enumerate}[label=(\alph*), leftmargin=15pt]
\item For any $t\in [\theta,\lM + \theta)$, we have
\begin{equation}\label{UTailUB}
\limsup_{N \rightarrow \infty} \frac{1}{N} \log \P_N(\ell_1\geq tN) \leq - J^{\theta, \lM + \theta}_V(t),
\end{equation}
\item For $t \in [\theta, b_V)$, we have
\begin{equation}\label{UTailLB}
\liminf_{N \rightarrow \infty} \frac{1}{N} \log \P_N(\ell_1\geq tN) \geq - J^{\theta, \lM + \theta}_V(t).
\end{equation}
In addition, if $J^{\theta, \lM + \theta}_V(t) > 0$ for $t \in (b_V, \lM + \theta)$ then (\ref{UTailLB}) holds for any $t\in [\theta,\lM + \theta)$.
\end{enumerate}
In equations (\ref{UTailUB}) and (\ref{UTailLB}) the function $J^{\theta, \lM + \theta}_V$ is as in Lemma \ref{S5ContGJ}.
\end{proposition}

%
\subsection{Proof of Theorem \ref{TMain}(b)}\label{Section5.1} The idea of the proof is to reduce the problem to the finite $M$ setting and then apply Proposition \ref{FinULDP}. In order to accomplish this we need to apply Proposition \ref{exptight} and subsequently check that Assumptions \ref{as1} and \ref{as2} are satisfied. \\

We continue with the same notation as in Theorem \ref{TMain} and proceed to prove (\ref{RightTail}). Recall the definition of the equilibrium measure $\phi_V^{\theta,\infty}$ from Lemma \ref{iv}. By Lemma \ref{iv}, $\phi_V^{\theta,\infty}$ has a compact support and we let $b_V \in [\theta, \infty)$ be the rightmost point of its support.

Let us fix $t \in [\theta, \infty)$. By Lemma \ref{S1GJ} we know that $J_V^{\theta, \infty}$ is continuous on $[\theta, \infty)$ and so we can find $A >0 $ such that $2 + |J_V^{\theta, \infty}(t)| \leq A $. Observe that the measures $\P_N$ satisfy the conditions of Definition \ref{S3PDef} with $\gn \leq R_0N^{-3/4}$ for some $R_0 > 0$ and all $N \geq N_0$ (here we used that $\lim_{N \rightarrow \infty} \gn N^{3/4} = 0$). In particular, by Proposition \ref{exptight} applied to $A$ as above we can find integers $R_1 >  \max(b_V - \theta, t - \theta)$ and $N_1 \geq N_0$ such that for $N \geq N_1$ 
\begin{equation}\label{S51E1}
\P_N^{\theta, \infty} ( \ell_1 - (N-1) \theta > R_1 N ) \leq \exp \left( - (A-1) N\right) \leq 1/2. 
\end{equation} 
From Lemma \ref{MinEq} applied to $s_1 = R_1 + \theta$ and $s_2 = \infty$ we know that $\phi_V^{\theta, \infty} =  \phi_V^{\theta, R_1 + \theta}$, $b_V = b_V^{\theta, R_1 + \theta}$ and also $J^{\theta,R_1 + \theta}_V(t) = J^{\theta, \infty}_V(t)$.

Set $M_N=  NR_1$ and $\lM=R_1$. Note that $M_N$ and $\lM$ satisfy Assumption \ref{as1} with $a_0 = A_0 = R_1$, $\qn = 0$ and $A_1=1$. In addition, we note that Assumption \ref{as2} is satisfied for $\fn = \gn$, $A_2 = F_1(  R_1 + \theta)$, $A_3 = \sup_{x \in [0,  R_1 + \theta ]} V(x)$ and $A_4 = B_0 \cdot F_2( R_1 + \theta) + B_0$.

From Proposition \ref{FinULDP} (using that $J^{\theta,R_1 + \theta}_V(t) = J^{\theta,\infty}_V(t)> 0$ for $t > b_V$ by assumption) we get
\begin{equation}\label{S51E2}
\lim_{N\rightarrow \infty} \frac{1}{N} \log \P^{\theta, M_N}_N ( \ell_1 \geq tN) = - J^{\theta,  \infty}_V(t).
\end{equation} 
From (\ref{S51E1}) we know that for all $N \geq N_1$ 
$$(1/2) \cdot \P^{\theta, M_N}_N ( \ell_1 \geq tN) \leq  \P^{\theta, \infty}_N ( \ell_1 \geq tN)  \leq \P^{\theta, M_N}_N ( \ell_1 \geq tN) + \exp \left( - (A - 1) N\right),$$
which in view of (\ref{S51E2}) and the fact that $2 + |J_V^{\theta, \infty}(t)| \leq A $ implies
\begin{equation*}
\lim_{N\rightarrow \infty} \frac{1}{N} \log \P^{\theta, \infty}_N ( \ell_1 \geq tN) = - J^{\theta, \infty}_V(t).
\end{equation*} 
The last equation is precisely (\ref{RightTail}). As $t \in [\theta, \infty)$ was arbitrary, this completes the proof.
	
%
\subsection{Proof of Proposition \ref{FinULDP}(a)}\label{Section5.2}  In this section we prove  Proposition \ref{FinULDP}(a). We require one preliminary result -- Lemma \ref{LemmaS52}, whose proof is postponed until Section \ref{Section7} (see Lemma \ref{S7LemmaS52}), after which we present the proof of Proposition \ref{FinULDP}(a).

\begin{lemma}\label{LemmaS52} Let $\mathcal{A}_{c,M}$ be the set of all $C^1$ functions $f: \mathbb{R} \to \mathbb{R}$, supported on $[-M,M]$ with $|f'(x)| \le c$. Then
\begin{align*}
\sup_{f\in \mathcal{A}_{c,M}} \norm{f}_{1/2} \le 2c M.
\end{align*}
\end{lemma}

\begin{proof}[Proof of Proposition \ref{FinULDP}(a)] For clarity we split the proof into several steps.\\

{\bf \raggedleft Step 1.} We claim that for any $t \in (b_V,  \lM + \theta)$ we have 
\begin{equation}\label{S52E1}
\limsup_{N\to\infty}\frac1{N}\log\P_N(\ell_1 \ge tN) \le -J^{\theta, \lM + \theta}_V(t).
\end{equation}
We will prove (\ref{S52E1}) in the steps below. Here we assume its validity and conclude the proof of (\ref{UTailUB}).\\

Note that (\ref{S52E1}) implies (\ref{UTailUB}) for all $t \in (b_V,  \lM + \theta)$. Also, since $\P_N(\ell_1 \ge tN)  \leq 1$, we have
$$0 \geq \limsup_{N\to\infty}\frac1{N}\log\P_N(\ell_1 \ge tN) ,$$
which proves (\ref{UTailUB}) for $t \in [\theta, b_V]$ in view Lemma \ref{S5ContGJ}. Thus we have reduced the proof of the proposition to establishing (\ref{S52E1}). \\

{\bf \raggedleft Step 2.} In this step we summarize the notation we will require in the proof of (\ref{S52E1}). In the sequel we fix $t \in  (b_V,  \lM + \theta)$ and set $\alpha = t - b_V$. We also fix $\epsilon \in (0, \alpha/8)$. 

Let $q,g$ be a compactly supported smooth functions such that
\begin{itemize}
\item  $g(x) = 1$ for $x \in [0, b_V ]$ and $q(x) = 1$ for $x \in [0, b_V + 2\epsilon]$;
\item $g(x) = 0$ for $x \in \mathbb{R} \setminus [- \epsilon, b_V+ \epsilon]$ and $q(x) = 0$ for $x \in \mathbb{R} \setminus [- \epsilon, b_V + 3\epsilon]$;
\item $0 \leq g(x), q(x) \leq 1$ for all $x \in \mathbb{R}$. 
\end{itemize}
In addition, we let for any $y \geq b_V + 4\epsilon$ 
$$h_y(x) = q(x) \cdot \log| y -x|,$$
where we use the convention $0 \cdot \log 0 = 0$. One readily observes that 
$$\sup_{ x \in \mathbb{R}} |g'(x)| < \infty \mbox{, and } \sup_{ x \in \mathbb{R}} \sup_{ y \in [b_V + 4\epsilon, \lM + \theta + 1]}   |h_y'(x)| < \infty.$$
From Lemma \ref{LemmaS52} we can find a large enough constant $G > 0$, depending on $\epsilon$ and $V$ such that 
\begin{align}\label{eq:g}
G \geq \max\left\{\norm{g}_{1/2}, \norm{g}_{\operatorname{Lip}},\sup_{y\in [b_V +4\epsilon, \lM + \theta + 1]}\max\{\norm{h_y}_{1/2}, \norm{h_y}_{\operatorname{Lip}}\}\right\},
\end{align} 
where we recall that $\norm{g}_{1/2}, \norm{g}_{\operatorname{Lip}}$ were defined in (\ref{S2Norms}).

We next define several subsets of $\mathbb{W}_N^{\theta,M_N}$, which will be required in our analysis. We recall from \eqref{kp} the function $k_{\rho}$ on $\mathbb{W}_N^{\theta,M_N}$, which is given by
$$k_{\rho}(\vec\ell):=|\{j \in \{1, \dots, N \} \mid \ell_j\ge \rho N\}|,$$
and define the set
\begin{align}\label{defrc}
\mathcal{R}^N:=\{\ell\in \mathbb{W}^{\theta,M_N}_{N} \mid k_{b_V +2\epsilon}(\vec\ell) < 2GN^{3/4}\}.
\end{align}
Let $\mathcal{E}_N = \{ \lambda + (i-1) \theta \mid \lambda \in \mathbb{Z}, 0 \leq \lambda \leq M_N \mbox{ and }i = 1, \dots, N\} \cap [b_V + 4\epsilon, \lM + \theta + 1]$ and for $y \in \mathcal{E}_N $
\begin{align}\label{defR}
\mathcal{R}^N_{y}:=\left\{\vec\ell\in \mathbb{W}^{\theta,M_N}_{N} \mid \left| \int_{\R} h_y(x)\mu_N(dx)-\int_{\R} h_y(x)\phi_V^{\theta, \lM + \theta}(x)dx\right| < 2G N^{-1/4}\right\},
\end{align}
where we recall that $\mu_N$ is the empirical measure 
$$\mu_N(\vec{\ell}) = \frac{1}{N}\sum_{i=1}^N\delta\left(\frac{\ell_i}{N}\right) .$$
With the above notation we set
\begin{align}\label{eq:t}
\mathcal{T}:= \{\vec{\ell} \in \mathbb{W}^{\theta,M_N}_{N} \mid k_{t}(\vec{\ell})  \geq 1\} \cap  \mathcal{R}^N \bigcap \cap_{y \in \mathcal{E}_N} \mathcal{R}^N_{y}.
\end{align}

We finally recall that by assumption we have that $\P_N$ satisfies Definition \ref{S2PDef} with constants $\theta$, $a_0$, $A_0,\ldots, A_4$ and sequences $\fn, \qn$ such that $\lim_{N \rightarrow \infty} \fn N^{3/4} =\lim_{N \rightarrow \infty} \qn N^{3/4} = 0$. We will assume that $N$ is sufficiently large so that 
\begin{equation}\label{S52Large}
\fn \leq N^{-3/4}, \hspace{5mm}\qn \leq N^{-3/4},\hspace{5mm} N-1 \geq 2GN^{3/4}, \hspace{5mm} |M_N - \lM N| \leq N.
\end{equation}
 In the proof below the constants in all big $O$ notations will depend on $\theta, V, G, \epsilon$ and the constants $a_0, A_0,\ldots, A_4$ from Assumptions \ref{as1} and \ref{as2} -- we will not mention this further. \\

{\bf \raggedleft Step 3.} We claim that
\begin{align}\label{u3}
\P_N(\vec{\ell} \in \mathcal{T}) \le  \exp\left(O(N^{7/4} \cdot \fn+N^{3/4}\log N) -N \inf_{y \in [t, \lM + \theta]}[G_V^{\theta, \lM + \theta}(y)-G_V^{\theta, \lM + \theta}(b_V+4\epsilon)]\right),
\end{align}
where we recall that $G_V^{\theta, \lM + \theta}$ was defined in (\ref{S5DefG}). We will prove (\ref{u3}) in the steps below. Here we assume its validity and conclude the proof of (\ref{S52E1}).\\

From Proposition \ref{gldp} applied to $p = 2$, $\gamma = N^{-1/4} $ we have for all large $N$ and each $y \in \mathcal{E}_N$ that  
\begin{equation}\label{IneqRy}
 \P_N (\vec{\ell} \not \in \mathcal{R}_y^N) \leq \exp \left(-2\pi^2 N^{3/2} + O(N^{2} \cdot \fn + N^2 \cdot \qn + N \log N) \right),
\end{equation}
and also using the definition of $g$ from Step 2 we have
\begin{equation*}
\begin{split}
&  \P_N (\vec{\ell} \not \in \mathcal{R}^N) \leq \P_N  \left(\left|  \int_{\R}  g(x)\mu_N(dx)- \int_{\R}  g(x)\phi_V^{\theta, \lM + \theta} (x)dx \right|\ge 2G N^{-1/4}  \right)    \\
& \leq \exp \left(-2\pi^2 N^{3/2} + O(N^{2} \cdot \fn + N^2 \cdot \qn + N \log N) \right).
\end{split}
\end{equation*}
The last two equations and the fact that $|\mathcal{E}_N| = O(N^2)$
$$\P_N(\ell_1 \ge tN) \leq \P_N (\vec{\ell} \in \mathcal{T}) + \exp \left(-2\pi^2 N^{3/2} + O(N^{2} \cdot \fn + N^2 \cdot \qn + N \log N) \right).$$
The last equation, (\ref{u3}) and the fact that $\lim_{N \rightarrow \infty} \fn N^{3/4} =\lim_{N \rightarrow \infty} \qn N^{3/4} = 0$ together imply 
$$\limsup_{ N \rightarrow \infty} \frac1{N}\log\P_N(\ell_1 \ge tN)  \leq - \inf_{y \in [t, \lM + \theta]} [G_V^{\theta, \lM + \theta}(y)-G_V^{\theta, \lM + \theta}(b_V+4\epsilon)].$$
Since the latter is true for all $\epsilon \in (0, \alpha/8)$, we can take $\epsilon \rightarrow 0+$ above, and use the continuity of $G_V^{\theta, \lM + \theta}$ from Lemma \ref{LemmaTech2} to get (\ref{S52E1}).\\

{\bf \raggedleft Step 4.} We claim that there exists a map $\tau: \mathcal{T} \rightarrow \mathbb{W}_N^{\theta, M_N}$ such that:
\begin{enumerate}
\item for each $\vec{\ell} \in \mathcal{T}$ we have $|\tau^{-1} (\tau (\vec{\ell}))| \leq (M_N+1) N$, where $|\tau^{-1}(\tau(\vec{\ell}))|$ is the size of the preimage set $\tau^{-1}(\tau(\vec{\ell}))$;
\item if $\vec{\ell}' = \tau(\ell)$ then
\begin{equation}\label{S52E3}
\P_N(\vec\ell) \leq  \exp \left( O(N^{7/4} \cdot  \fn+N^{3/4}\log N) -N \hspace{-4mm} \inf_{y \in [t, \lM + \theta]}[G_V^{\theta, \lM + \theta}(y)-G_V^{\theta, \lM + \theta}(b_V+4\epsilon)]  \right) \P_N(\vec\ell').
\end{equation}
\end{enumerate}
We construct $\tau$ in the next steps. Here we assume its existence and conclude the proof of (\ref{u3}). \\

Notice that from (\ref{S52E3}) we have
\begin{equation*}
\begin{split}
& \sum_{ \vec{\ell} \in \mathcal{T}} \P_N (\vec{\ell}) \leq \exp \left( O(N^{7/4}\fn+N^{3/4}\log N) -N \hspace{-4mm}\inf_{y \in [t, \lM + \theta]}[G_V^{\theta, \lM + \theta}(y)-G_V^{\theta, \lM + \theta}(b_V+4\epsilon)]  \right) \sum_{ \vec{\ell} \in \mathcal{T}} \P_N(\tau(\vec{\ell}))  \\
& \leq \exp \left( O(N^{7/4}\cdot \fn+N^{3/4}\log N) -N \inf_{y \in [t, \lM + \theta]}[G_V^{\theta, \lM + \theta}(y)-G_V^{\theta, \lM + \theta}(b_V+4\epsilon)]  \right) \sum_{ \vec{\ell}' \in \tau(\mathcal{T})} \P_N(\vec{\ell}'),
\end{split}
\end{equation*}
where in the second inequality we used that $|\tau^{-1} (\tau (\vec{\ell}))| \leq (M_N+1) N = \exp( O (\log N))$. Since the last sum in the above equation is at most $1$, we see that (\ref{u3}) holds.\\

{\bf \raggedleft Step 5.} In this step we construct a map $\tau: \mathcal{T} \rightarrow \mathbb{W}_N^{\theta, M_N}$, which satisfies the conditions in Step 4. 

Suppose that $\vec{\ell} \in \mathcal{T}$ and let $u = k_{b_V + 4\epsilon }(\vec{\ell}).$ Since $\vec{\ell} \in \mathcal{T}$ we note that $2GN^{3/4} \geq u\geq 1$. We set
$$\nu(\vec{\ell}) := \min_{ a \in \mathbb{Z}_{\geq 0}} \{ \ell_{u+1} + \theta + a  \mid \ell_{u+1} + \theta + a \geq (b_V + 4\epsilon ) N \},$$
which is well-defined since $N - 1 \geq 2GN^{3/4}$, cf. (\ref{S52Large}). We now define $\vec{\ell}'$ via 
$$\ell'_k = \ell_k \mbox{ for $k = u+1, \dots, N$, } \ell'_u = \nu(\vec{\ell}) \mbox{, and } \ell'_k = \ell_{k+1} + \theta \mbox{ for $k = 1, \dots, u-1$}.$$
By construction, we have that $\ell_i' \in \mathbb{Z}_{\geq 0} + (N-i) \theta$ for $i = 1, \dots, N$ and also $\ell_{i}' \geq \ell_{i+1}' + \theta$ for all $i = 1, \dots, N-1$. The latter is clear when $i \neq u-1$ and $i \neq u $. If $i = u-1$ it follows from the fact that $\nu(\vec{\ell}) \leq \ell_{u}$ and if $i = u$ it follows from the fact that $\nu(\vec{\ell}) \geq \ell_{u+1} + \theta + 0$ by definition. Finally, $\ell'_1 \leq \ell_1$ and so $\vec{\ell}' \in \mathbb{W}_N^{\theta, M_N}$. We set $\tau(\vec{\ell}) = \vec{\ell}'$ and then $\tau: \mathcal{T} \rightarrow \mathbb{W}_N^{\theta, M_N}$ is well-defined.

If $\tau(\vec{x}) = \tau (\vec{y})$ for $\vec{x}, \vec{y} \in \mathcal{T}$ we see that $\vec{x} = \vec{y}$, provided that $x_1 = y_1$ and $k_{b_V + 4\epsilon }(\vec{x}) = k_{b_V + 4\epsilon }(\vec{y}) $. Since we have at most $M_N + 1$ possible choices for $x_1$ and at most $N$ choices for $k_{b_V + 4\epsilon }(\vec{x}) $ we conclude that $|\tau^{-1} (\tau (\vec{x}))| \leq N( M_N+1)$ so that condition (1) in Step 4 holds.

By the definition of $\P_N$ we note that if $\vec{\ell} \in \mathcal{T}$ and $\vec{\ell}' = \tau(\vec{\ell})$ we have
\begin{equation}\label{rat1}
\frac{\P_N(\vec\ell)}{\P_N(\vec\ell')} = \mathfrak{L}_1(\vec\ell)\mathfrak{L}_2(\vec\ell), \mbox{ where } \mathfrak{L_1}(\vec\ell) = \frac{\prod\limits_{1\le i<j\le N} Q_{\theta}(\ell_i-\ell_j)}{\prod\limits_{1\le i<j\le N} Q_{\theta}(\ell'_i-\ell'_j)} \mbox{ and } \mathfrak{L_2}(\vec\ell) = \frac{\prod_{i=1}^N e^{-NV_N(\ell_i / N)}}{\prod_{i=1}^Ne^{-NV_N(\ell'_i/N)}}.
\end{equation}
We claim that 
\begin{equation}\label{rat2}
 \mathfrak{L}_1(\vec\ell) \leq \exp\left(O(N^{3/4}\log N) + 2\theta \hspace{-1mm}\int_{\R} \left(\log|\ell_1/N-x|- \log|b_V+4\epsilon-x|\right)\phi^{\theta, \lM + \theta}_V(x)dx  \hspace{-1mm} \right),
\end{equation}
\begin{equation}\label{rat3}
 \mathfrak{L}_2(\vec\ell) \leq  \exp \left( O(N^{7/4}\cdot \fn+N^{3/4} \log N) + NV(b_V+4\epsilon)-NV(\ell_1/N) \right).
\end{equation}
We will prove (\ref{rat2}) and (\ref{rat3}) in the next step. Here we assume their validity and show that condition (2) in Step 4 holds.\\

In view of (\ref{rat1}), (\ref{rat2}) and (\ref{rat3}) and the definition of $G_V^{\theta, \lM + \theta}$ from (\ref{S5DefG}) we have
\begin{equation*}
\P_N(\vec\ell)  \le \exp\left(O(N^{7/4} \cdot \fn+N^{3/4}\log N) -N[G_V^{\theta, \lM + \theta} (\ell_1/ N)-G_V^{\theta, \lM + \theta}(b_V+4\epsilon)]\right)\P_N(\vec\ell').
\end{equation*}
Observe that we have the following tower of inequalities
\begin{equation*}
\begin{split}
& -[G_V^{\theta, \lM + \theta} (\ell_1/ N)-G_V^{\theta, \lM + \theta}(b_V+4\epsilon)] \leq - \inf_{y \in[t, N^{-1}(M_N +(N-1) \theta)]} [G_V^{\theta, \lM + \theta} (y)-G_V^{\theta, \lM + \theta}(b_V+4\epsilon)] \\ 
& \leq O((N^{-1}+\qn)\cdot \log N)  - \inf_{y \in[t, \lM + \theta]} [G_V^{\theta, \lM + \theta} (y)-G_V^{\theta, \lM + \theta}(b_V+4\epsilon)] \\
& \leq O(N^{-3/4} \log N) - \inf_{y \in[t, \lM + \theta]} [G_V^{\theta, \lM + \theta} (y)-G_V^{\theta, \lM + \theta}(b_V+4\epsilon)],
\end{split}
\end{equation*}
where in the first line we used that $M_N + (N-1) \theta \geq \ell_1 \geq Nt$ and in going from the first to the second line we used the second part of Lemma \ref{LemmaTech2} and Assumption \ref{as1}. In the last inequality we used that $\qn \leq N^{-3/4}$, cf (\ref{S52Large}). The last two equations imply condition (2) in Step 4.\\

{\bf \raggedleft Step 6.} In this step we prove (\ref{rat2}). Using the definition of $\vec{\ell}'$  we have
\begin{equation*}
\begin{split}
&\mathfrak{L}_1(\vec\ell)  = \frac{\prod_{j =2}^N Q_\theta(\ell_1 - \ell_j)}{\prod_{i = 1}^{u-1} Q_{\theta} (\ell_{i+1} + \theta - \nu(\vec{\ell})) \prod_{j = u+1}^N Q_\theta(\nu(\vec{\ell}) - \ell_j) } \cdot \frac{\prod_{i = 2}^u \prod_{j = u+1}^N Q_\theta(\ell_i - \ell_j) }{\prod_{i = 1}^{u-1} \prod_{j = u+1}^N Q_\theta(\ell_{i+1} + \theta - \ell_j)}  \\
& \leq  \frac{\exp \left( O( u \log N) \right) \cdot \prod_{j =2}^N |\ell_1 - \ell_j|^{2\theta}}{\prod_{i = 1}^{u-1} |\ell_{i+1} + \theta - \nu(\vec{\ell})|^{2\theta} \prod_{j = u+1}^N |\nu(\vec{\ell})  - \ell_j|^{2\theta}} \cdot \frac{\prod_{i = 2}^u \prod_{j = u+1}^N |\ell_i - \ell_j|^{2\theta} }{\prod_{i = 1}^{u-1} \prod_{j = u+1}^N |\ell_{i+1} + \theta - \ell_j|^{2\theta}} \\
&\leq   \exp \left( O( N^{3/4} \log N) \right) \frac{ \prod_{j =2}^N |\ell_1/N - \ell_j/N|^{2\theta}}{\prod_{i = 1}^{u-1} |\ell_{i+1}/N + \theta/N - \nu(\vec{\ell})/N|^{2\theta} \prod_{j = u+1}^N |\nu(\vec{\ell})/N - \ell_j/N|^{2\theta}} ,
\end{split}
\end{equation*}
where in the first inequality we used Lemma \ref{InterApprox} and the fact that $|\ell_i - \ell_j| \geq |i - j| \theta$. In the second inequality we bounded the second product by $1$ and used that $u = O(N^{3/4})$. 

Since $0 \leq \ell_1 \leq M_N + (N-1)\theta$, $|\ell_i - \ell_j| \geq |i - j| \theta$ and $u = O(N^{3/4})$ we see that 
\begin{equation*}
\begin{split}
 \mathfrak{L}_1(\vec\ell) &\leq   \exp \left( O( N^{3/4} \log N) \right)\frac{\prod_{j =u+1}^N |\ell_1/N - \ell_j/N|^{2\theta}}{ \prod_{j = u+1}^N |\nu(\vec{\ell})/N - \ell_j/N|^{2\theta}}   \\
& =  \exp \left( O( N^{3/4} \log N) + 2\theta \sum_{j = k_{b_V + 2\epsilon}(\vec{\ell})+1}^N \left(h_{\ell_1/N} (\ell_j/N) - h_{\nu(\vec{\ell})/N} (\ell_j/N )\right) \right) \\
& \times \exp \left(  2\theta \sum_{j = u+1}^{k_{b_V + 2\epsilon}(\vec{\ell})+1} \log|\ell_1/N - \ell_j/N| -  \log|\nu(\vec{\ell})/N - \ell_j/N|  \right) \\
&\leq  \exp \left( O( N^{3/4} \log N) + 2\theta \sum_{j = k_{b_V + 2\epsilon}(\vec{\ell})+1}^N \left(h_{\ell_1/N} (\ell_j/N) - h_{\nu(\vec{\ell})/N} (\ell_j/N )\right) \right) ,
\end{split}
\end{equation*}
where in the equality on the second line we used the definition of $h_y$ from Step 2, the fact that $\ell_j/N \in [0, b_V + 2\epsilon]$ for all $j = k_{b_V + 2\epsilon}(\vec{\ell})+1, \dots, N$ and that $\ell_1, \nu(\vec{\ell}) \geq b_V + 4\epsilon$. In the last inequality we used that $|\log|(x-y)/N|| = O(\log N)$ for all $x,y \in [0, M_N + (N-1) \theta]$ such that $|x-y| \geq \theta$ and the fact that $k_{b_V + 2\epsilon}(\vec{\ell})+1 \leq 2G N^{3/4}$ (as $\vec{\ell} \in \mathcal{T}$).

Using the fact that for all $x,y \in [0, M_N + (N-1) \theta]$ such that $|x-y| \geq \theta$ we have $|h_{y/N}(x/N)| = O(\log N)$, the fact that $h_{\ell_1/N}(\ell_1/N) = 0 = h_{\nu(\vec{\ell})/N}(\nu(\vec{\ell})/N)$ (since $h_y(x)$ is supported on $[0, b_V + 3\epsilon]$ by construction) and the fact that $k_{b_V + 2\epsilon}(\vec{\ell})+1 \leq 2G N^{3/4}$ (as $\vec{\ell} \in \mathcal{T}$) we conclude
\begin{equation}\label{lbd2}
\begin{split}
\mathfrak{L}_1(\vec\ell) &\leq\exp \left( O( N^{3/4} \log N) + 2\theta \sum_{j = 1}^N \left(h_{\ell_1/N} (\ell_j/N) - h_{\nu(\vec{\ell})/N} (\ell_j/N )\right) \right)   \\ 
& = \exp \left( O( N^{3/4} \log N) + 2\theta\int_{\R} \left(h_{\ell_1/N} (x) - h_{\nu(\vec{\ell})/N} (x )\right) \phi^{\theta, \lM + \theta}_V(x) dx \right) \\
& = \exp \left( O( N^{3/4} \log N) + 2\theta\int_{\R} \left( \log| \ell_1/N - x|- \log|\nu(\vec{\ell})/N - x |\right) \phi^{\theta, \lM + \theta}_V(x) dx \right),
\end{split}
\end{equation}
where in going from the first to the second line we used that $\vec{\ell} \in \mathcal{T}$ (and so in $\mathcal{R}^N_y$ for all $y \in \mathcal{E}_N$ -- see (\ref{defR})) as well as the fact that $\ell_1/N, \nu(\vec{\ell})/N \in \mathcal{E}_N$. Indeed, we have by construction that $\ell_1/N \geq \nu(\vec{\ell})/N  \geq b_V + 4\epsilon$ and also $\ell_1/N \leq \lM + \theta + 1$, which is true from our assumption that $|M_N - N \lM| \leq N$, cf. (\ref{S52Large}). In going from the second to the third line we used that $h_y(x) = \log|y-x|$ on $[0, b_V]$, where $ \phi^{\theta, \lM + \theta}_V(x)$ is supported. Since $\nu(\vec{\ell})/N \geq b_V + 4\epsilon$ we see that (\ref{lbd2}) implies (\ref{rat2}).\\

{\bf \raggedleft Step 7.} In this step we prove (\ref{rat3}). Using the definition of $\vec{\ell}'$ we have
$$\mathfrak{L}_2(\vec\ell)   =\exp \left( - N \sum_{i = 1}^u  V_N(\ell_i/N)  + N \sum_{i =2}^u V_N (\ell_i/N + \theta/N) + N V_N(\nu(\vec{\ell})/N)\right).$$
Using the fact that $u = O(N^{3/4})$, $|V_N(x)-V(x)| = O(\fn)$ and $\sup_{y \in [0, \theta/N]}|V(x + y) - V(x)| = O( N^{-1}\log N)$  (from Assumption \ref{as2}) we conclude that 
$$\mathfrak{L}_2(\vec\ell)   \leq \exp \left(  O( N^{7/4} \cdot \fn + N^{3/4} \log N) - NV( \ell_1/N) + NV(\nu(\vec{\ell})/N)\right).$$
Finally, we note that by definition we have $b_V + 4\epsilon + 1/N + \theta/N \geq \nu(\vec{\ell})/N \geq b_V + 4\epsilon$, which together with $\sup_{y \in [0, \theta/N]}|V(x + y) - V(x)| = O( N^{-1}\log N)$ and the above inequality imply 
$$\mathfrak{L}_2(\vec\ell)   \leq \exp \left(  O( N^{7/4} \cdot \fn + N^{3/4} \log N) - NV( \ell_1/N) + NV(b_V + 4\epsilon)\right).$$
The last equation proves (\ref{rat3}), and hence the proposition.
\end{proof}

%
\subsection{Proof of Proposition \ref{FinULDP}(b)}\label{Section5.3}  In this section we prove  Proposition \ref{FinULDP}(b). For clarity we split the proof into five steps.\\

 {\bf \raggedleft Step 1.} By Proposition \ref{FinLLDP} we have for all $t \in [\theta, b_V)$ and all large enough $N$ 
$$\P_N(\ell_1 \ge tN) = 1 - \P_N(\ell_1 < tN) \geq 1 - \exp \left( (N^{2}/2) \cdot(F_V^{\theta, \lM + \theta} - F_V^{\theta, t}) \right) \geq 1/2,$$
where in the last inequality we used that $F_V^{\theta, \lM + \theta} - F_V^{\theta, t} < 0$. This proves (\ref{UTailLB}) for $t \in [\theta, b_V)$.

In the sequel we assume that $J^{\theta, \lM + \theta}_V(t) > 0$ for $t \in (b_V, \lM + \theta)$. We claim that for any $t \in (b_V,  \lM + \theta)$ we have 
\begin{equation}\label{S53E1}
\liminf_{N\to\infty}\frac1{N}\log\P_N(\ell_1 \ge tN) \geq  -J^{\theta, \lM + \theta}_V(t).
\end{equation}
We will prove (\ref{S53E1}) in the steps below. Here we assume its validity and conclude the proof of (\ref{UTailLB}).\\

Note that (\ref{S53E1}) implies (\ref{UTailLB}) for all $t \in (b_V,  \lM + \theta)$. Also in the beginning of the step we showed (\ref{UTailLB}) for $t \in [\theta, b_V)$. Finally, suppose that $t = b_V \in [\theta, \lM + \theta)$. Then for any $\epsilon \in (0, \lM + \theta - b_V)$ we have by (\ref{S53E1}) that 
$$ \liminf_{N\to\infty}\frac1{N}\log\P_N(\ell_1 \ge b_V   N ) \geq\liminf_{N\to\infty}\frac1{N}\log\P_N(\ell_1 \ge (b_V  + \epsilon) N) \geq -J^{\theta, \lM + \theta}_V(b_V + \epsilon).$$
Letting $\epsilon \rightarrow 0+$ above and using the continuity of $J^{\theta, \lM + \theta}_V$ from Lemma \ref{S5ContGJ} we conclude (\ref{UTailLB}) for $t = b_V \in [\theta, \lM + \theta)$. Thus we have reduced the proof of the proposition to establishing (\ref{S53E1}). \\

{\bf \raggedleft Step 2.} In this step we summarize the notation we will require in the proof of (\ref{S53E1}). We adopt the same notation as in Step 2 of the proof of Proposition \ref{FinULDP}(a). Namely, we will fix $t \in  (b_V,  \lM + \theta)$, set $\alpha = t - b_V$, fix $\epsilon \in (0, \min(\alpha, \theta)/8)$ and have the same definitions for the functions $q,g, h_y$ for $y \geq b_V + 4\epsilon$, $G$, and $\mathcal{R}_y^N$ as in that step. 

We also set 
\begin{align}\label{eq:s}
\mathcal{S}:= \{\vec{\ell} \in \mathbb{W}^{\theta,M_N}_{N} \mid  |\ell_1 - N b_V | \leq N \epsilon\}  \bigcap \cap_{y \in \mathcal{E}_N} \mathcal{R}_y^N.
\end{align}
We finally recall that by assumption we have that $\P_N$ satisfies Definition \ref{S2PDef} with constants $\theta$, $a_0$, $A_0,\ldots, A_4$ and sequences $\fn, \qn$ such that $\lim_{N \rightarrow \infty} \fn N^{3/4} =\lim_{N \rightarrow \infty} \qn N^{3/4} = 0$. In the proof below the constants in all big $O$ notations will depend on $\theta, V,G, \epsilon$ and the constants $a_0, A_0,\ldots, A_4$ from Assumptions \ref{as1} and \ref{as2} -- we will not mention this further. \\

{\bf \raggedleft Step 3.} Let $x^{\epsilon}_0 \in [t + \epsilon, \lM + \theta]$ be such that 
$$\inf_{y \in [t + \epsilon, \lM + \theta]} G^{\theta, \lM + \theta}_V(y)  = G^{\theta, \lM + \theta}_V(x^{\epsilon}_0 ) ,$$
which exists by the continuity of $G^{\theta, \lM + \theta}_V$ from Lemma \ref{LemmaTech2} and the compactness of $[t + \epsilon, \lM + \theta]$. Notice that $x^{\epsilon}_0 $ need not be unique but we pick one minimizer $x^{\epsilon}_0 $ for each $\epsilon > 0$.

We claim that for all large $N$ there exists a map $\tau: \mathcal{S} \rightarrow \mathbb{W}_N^{\theta, M_N}$ such that 
\begin{enumerate}
\item for each $\vec{\ell}' \in \tau(\mathcal{S})$ we have $|\tau^{-1} (\vec{\ell}')| \leq M_N+1$ and also $\ell'_1 \geq (x^{\epsilon}_0 -\epsilon) N$;
\item if $\vec{\ell}' = \tau(\ell)$ then
\begin{equation}\label{S53E3}
\begin{split}
\P_N(\vec\ell)\leq  & \P_N(\vec\ell') \cdot \exp \left( O(N^{3/4}  + N \cdot \fn)  + N G_V^{\lM + \theta} (x^{\epsilon}_0 - \epsilon) \right) \times \\
&\exp\left(2\theta  N \int_{\R}\log| b_V + 4\epsilon - x|  \phi^{\theta, \lM + \theta}_V(x) dx - N \inf_{ y: |y - b_V| \leq \epsilon} V(y)   \right),
\end{split}
\end{equation}
\end{enumerate}
We will construct $\tau$ in the next steps. Here we assume its existence and conclude the proof of (\ref{S53E1}). \\

From our assumption that $J^{\theta, \lM + \theta}_V(t) > 0$ for $t \in (b_V, \lM + \theta)$ and (\ref{UTailUB}), which was proved in Section \ref{Section5.2}, we have
$$\lim_{N \rightarrow \infty} \P_N (\ell_1\leq  (b_V + \epsilon) N) = 1.$$
From Proposition \ref{FinLLDP} we have 
$$\lim_{N \rightarrow \infty} \P_N (\ell_1 \geq (b_V - \epsilon) N) = 1.$$
Also from (\ref{IneqRy}) and the fact that $|\mathcal{E}_N| = O(N^2)$ we have 
$$ \P_N (\vec{\ell} \not \in \cup_{ y \in \mathcal{R}_y^N} \mathcal{R}_y^N) \leq \exp \left(-2\pi^2 N^{3/2} + O(N^{2} \cdot \fn + N^2 \cdot \qn + N \log N) \right).$$
The last three equations and the fact that $\lim_{N \rightarrow \infty} \fn N^{3/4} =\lim_{N \rightarrow \infty} \qn N^{3/4} = 0$ together imply 
\begin{equation}\label{S53SLim}
\lim_{N \rightarrow \infty} \P_N( \vec{\ell} \in \mathcal{S}) = 1.
\end{equation}

On the other hand, we have by conditions (1) and (2) above that
\begin{equation*}
\begin{split}
&\P_N (\ell_1 \geq t N) \geq \sum_{ \vec{\ell}' \in \tau (\mathcal{S}) } \P_N(\vec{\ell}') \geq \frac{1}{M_N + 1}  \sum_{ \vec{\ell} \in \mathcal{S} } \P_N( \tau(\vec{\ell}) ) \geq  \exp \left( - O(N^{3/4} + N \cdot \fn )  \right)  \\
& \\
&\times \exp\left(- N G_V^{\lM + \theta} (x^{\epsilon}_0 - \epsilon)  - 2\theta  N \int_{\R}\log|b_V + 4\epsilon - x|  \phi^{\theta, \lM + \theta}_V(x) dx +N \inf_{ y: |y - b_V| \leq \epsilon} V(y)   \right) \sum_{ \vec{\ell} \in \mathcal{S} } \P_N( \vec{\ell} ).
\end{split}
\end{equation*}
The last inequality, $\lim_{N \rightarrow \infty} N^{3/4} \fn = 0$ and (\ref{S53SLim}) together imply that 
\begin{equation*}
\begin{split}
&\liminf_{N\to\infty}\frac1{N}\log\P_N(\ell_1 \ge tN)  \geq  -  G_V^{\lM + \theta} (x^{\epsilon}_0 - \epsilon) - 2\theta   \int_{\R}\log|b_V+ 4\epsilon - x|  \phi^{\theta, \lM + \theta}_V(x) dx + \inf_{ y: |y - b_V| \leq \epsilon} V(y).
\end{split}
\end{equation*}
Sending $\epsilon \rightarrow 0+$ in the last equation, and using the continuity of $V$, $G_V^{\lM + \theta} $ and the definition of $x_0^{\epsilon}$ we conclude 
\begin{equation*}
\begin{split}
&\liminf_{N\to\infty}\frac1{N}\log\P_N(\ell_1 \ge tN)  \geq  -  \inf_{y \in [t, \lM + \theta]} G_V^{\lM + \theta} (y ) - 2\theta   \int_{\R}\log|b_V  - x|  \phi^{\theta, \lM + \theta}_V(x) dx +  V(b_V),
\end{split}
\end{equation*}
which implies (\ref{S53E1}) by the definition of $J_V^{\theta, \lM + \theta}$ from Lemma \ref{S5ContGJ}. \\

{\bf \raggedleft Step 4.} In this step we fix $N$ large so that 
\begin{equation}\label{S52Large2}
M_N +(N-1) \theta \geq N(\lM + \theta - \epsilon) + 1/N \mbox{, } N \epsilon \geq \theta + 1 \mbox{, and } |M_N - \lM N | \leq N, 
\end{equation}
and construct a map $\tau: \mathcal{S} \rightarrow \mathbb{W}_N^{\theta, M_N}$, which satisfies the properties in the beginning of Step 3. Notice that (\ref{S52Large2}) is satisfied for all large $N$ by Assumption \ref{as1} and $\lim_{N \rightarrow \infty} \qn N^{3/4} = 0$.

Let $\nu_N = \min_{a \in \mathbb{Z}_{\geq 0}} \{ a + (N-1) \theta \mid a + (N-1) \theta \geq N(x_0^{\epsilon} - \epsilon) \},$ and for $\vec{\ell} \in \mathcal{S}$ we set $\tau(\vec{\ell}) = \vec{\ell}'$, where $\vec{\ell}'$ is given by
$$\ell_k' = \ell_k \mbox{ for $k = 2, \dots, N$ and }\ell_1' = \nu_N.$$
Notice that $\nu_N \geq N \cdot (b_V + 8 \epsilon) \geq N(b_V + 7 \epsilon) + \theta \geq  \ell_2 + \theta$ -- here we used that $\vec{\ell} \in \mathcal{S}$ and (\ref{S52Large2}). Notice that (\ref{S52Large2}) also implies
$$\nu_N - (N-1) \theta \leq N(x_0^{\epsilon} - \epsilon)+ 1/N - (N-1)\theta \leq N(\lM + \theta - \epsilon) + 1/N - (N-1)\theta \leq M_N.$$
In particular, we indeed have that $\tau : \mathcal{S} \rightarrow \mathbb{W}_N^{\theta, M_N}$ is well-defined. Furthermore, it is clear that if $\tau(\vec{x}) = \tau(\vec{y})$ and $x_1 = y_1$ for $\vec{x}, \vec{y} \in \mathcal{S}$ then $\vec{x} = \vec{y}$. Since we have at most $M_N+1$ choices for $x_1$ we see that condition (1) in Step 3 is satisfied. 

By the definition of $\P_N$ we note that if $\vec{\ell} \in \mathcal{S}$ and $\vec{\ell}' = \tau(\vec{\ell})$ we have 
\begin{equation}\label{S53rat1}
\frac{\P_N(\vec\ell)}{\P_N(\vec\ell')} =  \frac{\prod_{j = 2}^N Q_{\theta}(\ell_1-\ell_j)}{\prod_{j = 2}^N Q_{\theta}(\nu_N-\ell_j)}\cdot \exp \left(- NV_N(\ell_1/N) + N V_N(\nu_N/N) \right).
\end{equation}
We claim that 
\begin{equation}\label{S53rat2}
\frac{\prod_{j = 2}^N Q_{\theta}(\ell_1-\ell_j)}{\prod_{j = 2}^N Q_{\theta}(\nu_N-\ell_j)} \leq \exp\left( \hspace{-1mm} O( N^{3/4})  + 2\theta N  \hspace{-1mm} \int_{\R}  \log \left|\frac{b_V + 4\epsilon - x}{ x_0^{\epsilon} - \epsilon - x}\right| \cdot \phi^{\theta, \lM + \theta}_V(x) dx \hspace{-1mm}  \right), 
\end{equation}
and 
\begin{equation}\label{S53rat3}
 \exp \left(- NV_N(\ell_1/N) + N V_N(\nu_N/N) \right)  \leq  \exp \left( O(N \cdot \fn + 1) - N \hspace{-4mm} \inf_{ y: |y - b_V| \leq \epsilon} \hspace{-4mm} V(y) +  N V(x_0^{\epsilon}- \epsilon) \right).
\end{equation}
Since equations (\ref{S53rat1}), (\ref{S53rat2}) and (\ref{S53rat3}) together imply (\ref{S53E3}), we are left with proving (\ref{S53rat2}) and (\ref{S53rat3}). We accomplish this in the next (and final) step.\\

{\bf \raggedleft Step 5.} In this step we prove  (\ref{S53rat2}) and (\ref{S53rat3}). 

Let $\nu_N(\vec{\ell}) = \min_{a \in \mathbb{Z}_{\geq 0}} \{ a + (N-1) \theta \mid a + (N-1) \theta \geq N (b_V + 4\epsilon) + \theta \}$ and note that $\nu_N(\vec{\ell}) \leq \ell_1 \leq M_N + (N-1)\theta$ (here we used that $N \epsilon \geq 1 + \theta$, cf. (\ref{S52Large2}) and $M_N + (N-1)\theta \geq \ell_1 \geq t N \geq (b_V + 8\epsilon) N$). Using Lemma \ref{InterApprox} and the fact that $|\ell_i -\ell_j| \geq |i-j|\theta$ we have 
\begin{equation*}
\begin{split}
&\frac{\prod_{j = 2}^N Q_{\theta}(\ell_1-\ell_j)}{\prod_{j = 2}^N Q_{\theta}(\nu_N-\ell_j)}  = \exp \left( O( \log N) + 2\theta \sum_{ j = 2}^N \log| \ell_1/ N - \ell_j/N| - \log| \nu_N / N - \ell_j/N|  \right)   \\
&\leq \exp \left( O( \log N) + 2\theta \sum_{ j = 2}^N \log| \nu_N(\vec{\ell})/ N - \ell_j/N| - \log| \nu_N / N - \ell_j/N|  \right)  \\
& = \exp \left( O( \log N) + 2\theta \sum_{ j = 1}^N h_{\nu_N(\vec{\ell})/ N}(\ell_j/N)  - h_{\nu_N/N}(\ell_j/N) \right),
\end{split}
\end{equation*}
where in going from the first to the second line we used that $\log$ is increasing and the definition of $h_y$ (note that $\ell_j/N \in [0, b_V +2\epsilon]$ for $j = 2, \dots,N$, $\nu_N(\vec{\ell})/ N \geq b_V + 4\epsilon$ and $\nu_N/N \geq b_V + 4\epsilon$ by construction). In the last equality we used that $h_{\nu_N(\vec{\ell})/ N}(\ell_1/N) = 0 = h_{\nu_N/N}(\ell_1/N)$ since $\ell_1/N,\nu_N(\vec{\ell})/ N \geq b_V + 4\epsilon$ and $h_y$ vanishes there. 

We next note that $\nu_N/N, \nu_N(\vec{\ell})/N \in \mathcal{E}_N$. Indeed, by construction we have $\nu_N(\vec{\ell})/ N \geq b_V + 4\epsilon$ and $\nu_N/N \geq b_V + 4\epsilon$. Also  as shown earlier we have
$$N^{-1} \max( \nu_N(\vec{\ell}), \nu_N) \leq N^{-1} (M_N + (N-1) \theta) \leq \lM + \theta + 1,$$
where the last inequality follows from (\ref{S52Large2}). 

Using the fact that $\vec{\ell} \in \cap_{y \in \mathcal{E}_N} \mathcal{R}_y^N$ (by the definition of $\mathcal{S}$) we see that the above work implies
$$\frac{\prod_{j = 2}^N Q_{\theta}(\ell_1-\ell_j)}{\prod_{j = 2}^N Q_{\theta}(\nu_N-\ell_j)}  \leq \exp \left( O( N^{3/4}) + 2\theta N\int_{\R} \left( \log|\nu_N(\vec{\ell})/N - x|- \log|\nu_N/N - x |\right) \phi^{\theta, \lM + \theta}_V(x) dx \right).$$
Using that $b_V +4 \epsilon + 1/N + \theta/ N \geq \nu_N(\vec{\ell})/N \geq b_V +4 \epsilon$ and $x_0^{\epsilon} - \epsilon + 1/N \geq \nu_N/N \geq x_0^{\epsilon} - \epsilon,$ we see that the last equation implies (\ref{S53rat2}).

Next, from Assumption \ref{as2} and the fact that $\vec{\ell} \in \mathcal{S}$ we have 
$$- NV_N(\ell_1/N) \leq O(N \cdot \fn ) - N V(\ell_1/N) \leq O(N \cdot \fn )  - N \inf_{ y: |y - b_V| \leq \epsilon} V(y),$$
while since $x_0^{\epsilon} - \epsilon + 1/N \geq \nu_N/N \geq x_0^{\epsilon} - \epsilon,$ we have
$$ NV_N(\nu_N/N) = O(N \cdot \fn ) + N V(\nu_N/N) = O(N \cdot \fn + 1 ) + N V(x_0^{\epsilon} - \epsilon).$$
The last two equations imply (\ref{S53rat3}), and hence the proposition.

%
\section{Applications}\label{Section6} In this section we present two applications of the results from the previous sections. As both of our examples originate from {\em Jack probability measures}, we summarize the definition and basic properties of the latter in Section \ref{Section6.1}. The two examples we investigate are presented in Sections \ref{Section6.2} and \ref{Section6.3}. The one discussed in Section \ref{Section6.2} is a general $\theta > 0$ extension of the classical {\em Krawtchouk orthogonal polynomial ensemble}. The models in Section \ref{Section6.3} were previously studied in \cite{misha}.

%
\subsection{Jack measures}\label{Section6.1} In this section we introduce a certain class of probability measures on integer partitions, related to Jack polynomials. In Section \ref{Section6.1.1} we introduce the Jack measures, and in Section \ref{Section6.1.2} we derive a few formulas for Jack symmetric functions with different specializations.

%
\subsubsection{Definition of Jack measures}\label{Section6.1.1}
In this section we introduce a class of measures on partitions, which are related to Jack polynomials. We begin by introducing some relevant notation, following \cite[Section 2]{misha} and \cite{mcd}.

A {\em partition} of size $n$, or a {\em Young diagram} with $n$ boxes, $\lambda$ is a sequence of non-negative integers $\lambda_1\ge \lambda_2 \ge \cdots \ge 0$ with $|\lambda|:=\sum_i \lambda_i=n$. It is usually viewed as a diagram with $n$ boxes: $\lambda_1$ left justified boxes on the top row, $\lambda_2$ in the second row and so on. The {\em length} of a partition $\lambda$, denoted by $\ell(\lambda)$, is the number of non-zero $\lambda_i$ in $\lambda$. The conjugate of a Young diagram $\lambda$ is the Young diagram $\lambda'$ obtained by transposing the diagram $\lambda$. In particular, we have the formula $\lambda'_i=|\{j\in \mathbb{Z}_{>0}\mid \lambda_j\ge i\}|$. For a box $\square=(i,j)$ of a Young diagram $\lambda$ we let $a(\square), \ell(\square)$ denote the {\em arm} and {\em leg lengths} respectively, i.e. 
$$a(\square)=\lambda_i-j, \quad \ell(\square)=\lambda'_j-i.$$
Further, we let $a'(\square)$  and $\ell'(\square)$ denote the {\em co-arm} and {\em co-leg lengths}:
$$a'(\square)=j-1, \quad \ell'(\square)=i-1.$$

Let $\Lambda$ be the $\mathbb{Z}_{\ge 0}$ graded algebra over $\mathbb{C}$ of symmetric polynomials in countably many variables $X=(x_1,x_2,\dots)$, or symmetric functions. An element of $\Lambda$ is a formal symmetric power series of bounded degree in the variables $x_1, x_2, \dots$. One way to view $\Lambda$ is as an algebra of polynomials in Newton power sums $p_k=\sum_i x_i^k$. Denote by $\Lambda_N$ the algebra of symmetric polynomials in $N$ variables. There exists a canonical projection $\pi_N : \Lambda \to \Lambda_N$, which sets all variables except for $x_1, x_2, \ldots , x_N$ to zero, and it defines an algebra homomorphism.
	
We denote by $J_{\lambda}(X;\theta)$ the {\em Jack symmetric polynomials (functions)}, which are indexed by Young diagrams $\lambda\in \mathbb{Y}$ and $\theta\in \mathbb{R}_{>0}$. They form a linear basis of $\Lambda$ and enjoy many remarkable properties. We refer to \cite[Section 10, Chapter VI]{mcd} for more details (we remark that the notation we use corresponds to setting $\theta = \alpha^{-1}$ in \cite{mcd}). Setting $x_{N+1}=x_{N+2}=\cdots=0$, $J_{\lambda}$ can be viewed as an element of $\Lambda_N$. The leading term of $J_{\lambda}(X;\theta)$ and $J_{\lambda}(x_1,x_2,\ldots,x_N;\theta)$ (for $N\ge \ell(\lambda)$) is given by $x_{1}^{\lambda_1}x_2^{\lambda_2}\cdots x_{\ell(\lambda)}^{\lambda_{\ell(\lambda)}}$. For finite $N$, the polynomials $J_{\lambda}(x_1,x_2,\ldots,x_N; \theta)$ are known to be the eigenfunctions of the Sekiguchi differential operator \cite{forrester,mcd,seki}:
\begin{equation} \label{eigen}
\begin{aligned}
& \frac1{\prod\limits_{i<j}(x_i-x_j)}\det\left[x_i^{N-j}\left(x_i\frac{\partial}{\partial x_i}+(N-j)\theta+u\right)\right]J_{\lambda}(x_1,\ldots,x_N;\theta) \\ & \hspace{7cm} =\left(\prod_{i=1}^N (\lambda_i+(N-i)\theta+u)\right)J_{\lambda}(x_1,\ldots,x_N;\theta).
\end{aligned}   
\end{equation}
The eigenrelation \eqref{eigen} along with the form of the leading term uniquely define $J_{\lambda}(x_1,\ldots,x_N;\theta)$ and $J_{\lambda}(X;\theta)$. 

We also make use of the {\em dual Jack polynomials} $\til{J}_{\lambda}$, defined as
\begin{equation}\label{dualJack}
\til{J}_{\lambda}={J}_{\lambda}\cdot \prod_{\square\in\lambda}\frac{a(\square)+\theta\ell(\square)+\theta}{a(\square)+\theta \ell(\square)+1}.
\end{equation}

A {\em specialization} $\rho$ is an algebra homomorphism from $\Lambda$ to $\mathbb{C}$. We say that a specialization $\rho$ of $\Lambda$ is {\em Jack-positive} if takes non-negative values on all Jack polynomials (i.e., $J_{\lambda}(\rho;\theta)\ge 0$ for all $\lambda \in \mathbb{Y}$). The set of all Jack-positive specializations are characterized by the following statement.

\begin{proposition} \cite[Theorem A]{koo}\label{PKOO} For any fixed $\theta>0$, Jack-positive specializations can be parametrized by triplets $(\alpha,\beta,\gamma)$, where $\alpha,\beta$ are sequences of real numbers satisfying
$$\alpha_1\ge \alpha_2 \ge \cdots \ge 0, \quad \beta_1\ge \beta_2\ge \cdots \ge 0, \quad \sum_{i=1}^{\infty}(\alpha_i+\beta_i)<\infty$$
and $\gamma \ge 0$. The specialization corresponding to a triplet $(\alpha,\beta,\gamma)$ is given by its values on the Newton power sums $p_k$, $k\ge 1$:
$$p_1 \mapsto p_1(\alpha,\beta,\gamma):=\gamma+\sum_{i=1}^{\infty}(\alpha_i+\beta_i),$$
$$p_k \mapsto p_k(\alpha,\beta,\gamma):=\sum_{i=1}^{\infty} \alpha_i^k+(-\theta)^{k-1}\sum_{i=1}^{\infty}\beta_i^k, \quad k\ge 2.$$
\end{proposition}
If we set $\alpha_1=\alpha_2=\cdots=\alpha_N= 1$, and all other parameters equal to zero, we obtain what is known as a {\em pure-$\alpha$} specialization, denoted as $1^N$. If we set $\beta_1=\beta_2=\cdots=\beta_N= 1$, and all other parameters equal to zero, we obtain what is known as a {\em pure-$\beta$} specialization, denoted as $1_{\beta}^N$. When $\gamma=s>0$ and all other parameters are zero, we obtain what is called the {\em Plancherel} specialization, denoted simply by $\mathfrak{r}_s$. Using the notion of Jack-positive specializations, one can define probability measures on $\mathbb{Y}$ as follows.
\begin{definition}\label{jackms} Let $\rho_1$ and $\rho_2$ be two Jack-positive specializations such that the (non-negative) series $\sum_{\lambda \in \mathbb{Y}} J_\lambda(\rho_1) \tilde{J}_{\lambda}(\rho_2)$ is finite. The Jack probability measure $\mathcal{J}_{\rho_1,\rho_2}(\lambda)$ on $\mathbb{Y}$ is defined through
\begin{align*}
\mathcal{J}_{\rho_1,\rho_2}(\lambda)=\frac{J_{\lambda}(\rho_1)\til{J}_{\lambda}(\rho_2)}{H_{\theta}(\rho_1;\rho_2)}
\end{align*} 
where the normalization constant is given by
\begin{align*}
H_{\theta}(\rho_1;\rho_2)=\sum_{\lambda \in \mathbb{Y}} J_\lambda(\rho_1) \tilde{J}_{\lambda}(\rho_2).
\end{align*}
\end{definition}
\begin{remark}
The construction of probability measures through specializations was first considered in \cite{okounkov} in the context of Schur measures. Since then, the construction has been extended to a much more general family of polynomials, called Macdonald polynomials (see \cite{bor-cor}), which includes the Jack polynomials as a special case.
\end{remark}
\begin{remark}\label{S6CauchyId} When $\rho_1 = 1^N$ and $\rho_2 = (\alpha, \beta, \gamma)$ as in Proposition \ref{PKOO} with $\alpha_i \in [0, 1)$ for $i \in \mathbb{N}$, we have the following formula for the normalization constant in Definition \ref{jackms}
\begin{equation}\label{NormConst}
H_{\theta}(\rho_1;\rho_2) = e^{N \theta \gamma } \prod_{i = 1}^{\infty} (1 + \theta \beta_i)^N \cdot \prod_{i = 1}^\infty \frac{1}{(1 - \alpha_i)^{N\theta}} = \exp \left( \sum_{k = 1}^\infty \frac{\theta p_k(\rho_1) p_k(\rho_2)}{k} \right),
\end{equation}
where the convergence of the first product is ensured when 
$$1 > \alpha_1\ge \alpha_2 \ge \cdots \ge 0, \quad \beta_1\ge \beta_2\ge \cdots \ge 0, \quad \sum_{i=1}^{\infty}(\alpha_i+\beta_i)<\infty, \quad \gamma \in [0, \infty),$$
while the last equality in (\ref{NormConst}) holds provided that the series $\sum_{k = 1}^\infty \frac{\theta p_k(\rho_1) p_k(\rho_2)}{k}$ is absolutely convergent. We mention that the formula in (\ref{NormConst}) can be deduced by setting $t = q^{\theta}$ in \cite[(2.23) and (2.31)]{bor-cor} and letting $q \rightarrow 1^-$. We also mention that the formula in (\ref{NormConst}) is slightly different from \cite[(2.23) and (2.31)]{bor-cor}, since one needs to multiply our $\beta$ and $\gamma$ variables by $\theta^{-1}$ as can be deduced from \cite{Matveev19}.
\end{remark}

%
\subsubsection{Jack symmetric functions with different specializations}\label{Section6.1.2} In this section we summarize several basic formulas for Jack symmetric functions, evaluated at pure-$\alpha$, $\beta$ and Plancherel specializations, which will be required for our examples in the next sections -- these are (\ref{jackpoly2}), (\ref{jackpoly3}) and (\ref{S6E5}) below.

From \cite[Chapter VI, (10.20)]{mcd} we have for $\lambda \in \mathbb{Y}_N$ as in (\ref{GenState}) (the latter can be interpreted as regular partitions by setting $\lambda_i = 0$ for $i \geq N+1$) the following formula
\begin{equation}\label{jackpoly}
J_\lambda(1^{N}) = \prod_{\square \in \lambda} \frac{N \theta + a'(\square) - \theta \ell'(\square)}{a(\square) + \theta \ell(\square) + \theta} = \prod_{i = 1}^{N} \prod_{j =1}^{\lambda_i} \frac{N \theta + (j-1) - (i-1) \theta}{\lambda_i - j + \theta(\lambda_j' - i) + \theta }.
\end{equation}
Using $\ell_i = \lambda_i + (N-i) \cdot \theta$ for $i = 1, \dots, N$ the denominator in (\ref{jackpoly}) can be rewritten as 
$$\prod_{1 \leq i \leq k \leq N} \prod_{j =\lambda_{k+1} + 1}^{\lambda_k} \frac{1}{\lambda_i - j + \theta(k - i + 1) } = \prod_{1 \leq i \leq k \leq N}  \frac{\Gamma(\lambda_i + \theta(k - i + 1) - \lambda_k)}{\Gamma(\lambda_i + \theta(k-i+1) - \lambda_{k+1})}  $$
$$ = \prod_{1 \leq i < k \leq N}\hspace{-2mm}  \frac{\Gamma(\lambda_i - \lambda_k + \theta(k - i + 1) )}{\Gamma(\lambda_i - \lambda_{k}+ \theta(k-i) )}  \prod_{i = 1}^{N} \frac{\Gamma(\theta)}{\Gamma(\lambda_i + \theta(N-i + 1))} = \hspace{-3mm} \prod_{1 \leq i < j \leq N} \frac{\Gamma(\ell_i - \ell_j + \theta)}{\Gamma(\ell_i - \ell_j)}  \prod_{i = 1}^{N}  \frac{ \Gamma(\theta)}{\Gamma( \ell_i + \theta)},$$
where $\lambda_{N+1} = 0$. Similarly, the numerator in (\ref{jackpoly}) can be rewritten as
$$ \prod_{i = 1}^{N} \prod_{j =1}^{\lambda_i}[N \theta + (j-1) - (i-1) \theta] = \prod_{i = 1}^{N}\frac{\Gamma(N\theta + \lambda_i - (i-1)\theta)}{\Gamma((N-i + 1)\theta )} = \prod_{i = 1}^{N}\frac{\Gamma(\ell_i + \theta)}{\Gamma((N-i + 1)\theta )}.$$
Overall, we have
\begin{equation}\label{jackpoly2}
J_\lambda(1^{N}) = \prod_{i = 1}^{N} \frac{\Gamma(\theta)}{\Gamma( i \theta)} \times \prod_{1 \leq i < j \leq N} \frac{\Gamma(\ell_i - \ell_j + \theta)}{\Gamma(\ell_i - \ell_j)} .
\end{equation}

We also have by \cite[Proposition 2.3]{misha} and (\ref{dualJack}) 
\begin{equation*}
\til{J}_\lambda(\mathfrak{r}_s) = {J}_{\lambda}(\mathfrak{r}_s)\cdot \prod_{\square\in\lambda}\frac{a(\square)+\theta\ell(\square)+\theta}{a(\square)+\theta \ell(\square)+1} =\prod_{\square\in\lambda}\frac{s\theta }{a(\square)+\theta \ell(\square)+1}  = \prod_{i = 1}^{N} \prod_{j =1}^{\lambda_i} \frac{s \theta}{\lambda_i - j + \theta(\lambda_j' - i) + 1 }.
\end{equation*}
The above can be rewritten as 
\begin{equation*}
\begin{split}
&(s \theta)^{|\lambda|}\prod_{1 \leq i \leq k \leq N} \prod_{j =\lambda_{k+1} + 1}^{\lambda_k} \frac{1}{\lambda_i - j + \theta(k - i ) + 1 } = (s \theta)^{|\lambda|} \prod_{1 \leq i \leq k \leq N}  \frac{\Gamma(\lambda_i + \theta(k - i ) + 1 - \lambda_k)}{\Gamma(\lambda_i + \theta(k-i) + 1 - \lambda_{k+1})} \\
&=(s \theta)^{|\lambda|}\prod_{1 \leq i < k \leq N}  \frac{\Gamma(\lambda_i - \lambda_k + \theta(k - i )+ 1 )}{\Gamma(\lambda_i - \lambda_{k}+ \theta(k-i - 1)+ 1 )} \cdot \prod_{i = 1}^{N} \frac{\Gamma(1)}{\Gamma(\lambda_i + \theta(N-i) + 1 )},
\end{split}
\end{equation*}
which gives
\begin{equation}\label{jackpoly3}
\til{J}_\lambda(\mathfrak{r}_s) = (s\theta)^{-\frac{N(N-1)}2} \prod_{1 \leq i < j \leq N} \frac{\Gamma(\ell_i - \ell_j + 1)}{\Gamma(\ell_i - \ell_j + 1 - \theta)}  \prod_{i = 1}^{N}  \frac{(s\theta)^{\ell_i}}{\Gamma( \ell_i + 1 )}.
\end{equation}

The last formula we require is for $\til{J}_\lambda(1^M_{\beta})$, i.e. specializing the dual Jack polynomial in a pure-$\beta$ specialization with $M$ variables, all equal to $1$. From \cite[Sections 5 and 10, Chapter VI]{mcd} 
\begin{equation*}
\tilde{J}_{\lambda}(1_{\beta}^M; \theta) = J_{\lambda'}(1^M; \theta^{-1}) = {\bf 1}\{ \lambda_1 \leq M\} \cdot  \prod_{i = 1}^N \prod_{j = 1}^{\lambda_i} \frac{M + \theta (i-1) - (j-1)}{\lambda_i - j + \theta (\lambda_i' - j) + 1},
\end{equation*}
where the second equality follows from (\ref{jackpoly}). We can rewrite the product above similarly to the displayed equation above (\ref{jackpoly3}), which gives
\begin{equation}\label{S6E5}
\tilde{J}_{\lambda}(1_{\beta}^M; \theta)  = {\bf 1}\{ \lambda_1 \leq M\} \cdot \prod_{1 \leq i < j \leq N} \frac{\Gamma(\ell_i - \ell_j + 1)}{\Gamma(\ell_i - \ell_j + 1 - \theta)} \prod_{i = 1}^N \frac{\Gamma(M + \theta (i-1) +1)}{\Gamma(\ell_i + 1) \Gamma(M + N\theta - \ell_i + 1 - \theta)}.
\end{equation}

%
\subsection{Application to the Jack measures with pure $\beta$-specialization}\label{Section6.2} In this section we consider a special case of the measures in Definition \ref{jackms}, corresponding to setting $\rho_1 = 1^N$ (i.e. a pure-$\alpha$ specialization in $N$ variables that are all equal to $1$) and $\rho_2 = 1_{\beta}^M$ (i.e. a pure-$\beta$ specialization in $M$ variables that are all equal to $1$), where $N,M \in \mathbb{N}$. In view of (\ref{NormConst}) we have that 
$$H_{\theta}(\rho_1, \rho_2) = (1 + \theta)^{NM}< \infty,$$
so that the measures in Definition \ref{jackms} on $\mathbb{Y}$ are indeed well-defined for this choice of $\rho_1, \rho_2$. In addition, using that $J_\lambda(1^N) = 0$ if $\lambda_{N+1} > 0$ and $\til{J}_{\lambda}(1_{\beta}^M) = 0$ if $\lambda_1 > M$, we see that $\mathcal{J}_{\rho_1,\rho_2}(\cdot )$ is supported on $\lambda \in \mathbb{Y}$ such that $M \geq \lambda_1$ and $\lambda_{N+1} = 0$. Thus we may think of $ \mathcal{J}_{\rho_1,\rho_2}$ as a measure on $N$-tuples $M \geq \lambda_1 \geq \cdots \geq \lambda_N \geq 0$, i.e. a measure on $\mathbb{Y}_N^M$ as in (\ref{GenState}). Explicitly, we have for $\lambda \in \mathbb{Y}_N^M$ that 
\begin{equation}\label{JackBeta}
\mathcal{J}_{\rho_1,\rho_2}(\lambda) =(1 + \theta)^{-NM} \cdot J_{\lambda}(1^N) \cdot \til{J}_{\lambda}(1_{\beta}^M).
\end{equation}

Our first task is to rewrite the measure in (\ref{JackBeta}), as a measure on $\vec{\ell} \in \mathbb{W}^{\theta,M}_{N} $ (the latter set was defined in (\ref{GenState})) using the relations $\ell_i = \lambda_i + (N-i) \cdot \theta$ for $i = 1, \dots, N$. In the process of doing this we will see that the measure in (\ref{JackBeta}) is of the form (\ref{PDef}) and then we will explain how our results can be used to study its asymptotics. \\

Combining (\ref{jackpoly2}), (\ref{S6E5}) and (\ref{JackBeta}) we see that the measure (\ref{JackBeta}) induces the measure on $\vec{\ell} \in \mathbb{W}_N^{\theta, M}$, denoted $\mathbb{P}_{\beta, N}^{\theta,M}$, of the form
\begin{equation}\label{S6Kr}
\begin{split}
&\mathbb{P}_{\beta, N}^{\theta,M}(\vec{\ell}) = \frac{1}{Z(M,N)} \prod_{1 \leq i < j \leq N} \frac{\Gamma(\ell_i - \ell_j + \theta)\Gamma(\ell_i - \ell_j + 1)}{\Gamma(\ell_i - \ell_j)\Gamma(\ell_i - \ell_j + 1 - \theta)} \prod_{i = 1}^N \tilde{w}(\ell_i;N) , \mbox{ where }\\
&\tilde{w}(\ell;N) = \frac{1}{\Gamma(\ell+1)\Gamma(M +N \theta -\ell +1 - \theta)}, \hspace{2mm} Z(M,N) =  \prod_{i = 1}^N \frac{(1 + \theta)^M\Gamma(i\theta) }{\Gamma(M + \theta(i-1) +1) \Gamma(\theta)}.
\end{split}
\end{equation}
We mention that when $\theta = 1$, the measure in (\ref{S6Kr}) is called the {\em Krawtchouk (orthogonal polynomial) ensemble}, and it has been studied in \cite[Section 5]{jo-en}. Part of the interest in the Krawtchouk ensemble stems from its connection to a certain simplified first passage percolation model introduced in \cite{sep}. The Krawtchouk ensemble also bears connections with zig-zag paths in random domino tilings of the Aztec diamond \cite{jonon}, stochastic systems of non intersecting paths \cite{bbdt, kor}, and with the representation theory of the infinite-dimensional unitary group $U(\infty)$ \cite[Section 4]{b1}, \cite[Section 5]{bo2}.\\

We aim to study the large deviation of the rightmost particle $\ell_1$ under the measure $\mathbb{P}_{\beta, N}^{\theta,M}$ as $N\to \infty$ and $M$ is scaled linearly with $N$. In particular, we set $M_N=\lfloor\lM N\rfloor$, where $\lM >  0$ is fixed. We have that (\ref{S6Kr}) is of the form (\ref{PDef}) with
\begin{equation}\label{K1}
w(\ell;N) = e^{- N V_N(\ell/N)}, \mbox{ and } V_N(x) :=\frac1N\log \frac{\Gamma(Nx+1)\Gamma(\lfloor \lM N\rfloor + N \theta -Nx + 1 - \theta)}{N^{M_N + N\theta + 2 - \theta} e^{- M_N - N\theta + \theta }},
\end{equation}
i.e. we have with a possibly different normalization constant $Z_N$ that 
\begin{equation}\label{K2}
\mathbb{P}_{\beta, N}^{\theta,M_N}(\vec{\ell})  = \frac{1}{Z_N} \cdot  \prod_{1 \leq i < j \leq N} \frac{\Gamma(\ell_i - \ell_j + \theta)\Gamma(\ell_i - \ell_j + 1)}{\Gamma(\ell_i - \ell_j)\Gamma(\ell_i - \ell_j + 1 - \theta)} \prod_{i=1}^N \exp( - N V_N(\ell/N)).
\end{equation}
We mention that, when $\theta = 1$, (\ref{K2}) is of the form (\ref{Coulomb}) with $\beta = 2$ and so the asymptotics of these measures could in principle be studied using the frameworks in \cite{fe,jo}. Our goal below is to illustrate how our results from Section \ref{Section2} apply to these models for general $\theta > 0$.

Our first task is to show that $\mathbb{P}_{\beta, N}^{\theta,M_N}$ as in (\ref{K2}) satisfy the conditions of Definition \ref{S2PDef}. Regarding Assumption \ref{as1}, we have that it is trivially satisfied with $\lM$ as above, $a_0 = A_0 = \lM$, $A_1 = 1$, while $q_N = 1/N$. Regarding Assumption 2.2, we observe that $V_N$ is continuous on $[0, M_N \cdot N^{-1} + (N-1) \cdot N^{-1} \cdot \theta ]$ by the continuity of the gamma function on $(0, \infty)$. If we set $V(x) = x\log x+(\lM + \theta -x)\log(\lM + \theta-x)$, we see that $V$ is continuous on $I = [0, \lM  + \theta]$ and hence bounded in absolute value by some finite $A_3$ depending on $\lM$ and $\theta$. In addition, $V$ is differentiable on $(0, \lM + \theta)$ and 
$$V'(x) = \log x - \log (\lM + \theta - x),$$
which satisfies (\ref{DerPot}) with $A_4 = 1$. What remains is to show for $s \in [0, M_N \cdot N^{-1} + (N-1) \cdot N^{-1} \cdot \theta ]$
\begin{equation}\label{K3}
|V_N(s) - V(s)| = O(N^{-1}\log (N+1)),
\end{equation}
where the constant in the big $O$ notation depends on $\lM$ alone. If true, then Assumption \ref{as2} would be satisfied with the above choice of constants $A_3, A_4$ and for $\fn = N^{-1}\log(N+1)$ and $A_2$ being the constant in the big $O$ notation in (\ref{K3}).

To see why (\ref{K3}) holds, we use \cite[Theorem 1]{LC07}, which states that for all $x \geq 1$ we have 
\begin{equation}\label{GammaULBound}
\frac{x^{x - \gamma}}{e^{x-1}} \leq \Gamma(x) \leq \frac{x^{x - 1/2}}{e^{x-1}},
\end{equation}
where $\gamma$ is the Euler-Mascheroni constant $\gamma = 0.577215\dots$. The latter inequality implies that 
\begin{equation*}
\begin{split}
&(x + 1/N)^{Nx + 1} \cdot \left( \frac{M_N + 1 -  \theta}{N}  + \theta - x \right)^{M_N + N\theta - Nx + 1 - \theta } \hspace{-5mm}  \cdot (Nx + 1)^{- \gamma}(M_N + N\theta - Nx + 1 - \theta)^{- \gamma}   \\ 
&\leq \frac{\Gamma(Nx+1)\Gamma(\lfloor \lM N\rfloor + N\theta -Nx + 1 - \theta )}{N^{M_N + N\theta + 2 - \theta} e^{- M_N - N\theta + \theta}} = \exp ( - N V_N(x))  \\
&\leq  (x + 1/N)^{Nx + 1 } \cdot \left( \frac{M_N + 1 - \theta }{N}  + \theta - x \right)^{M_N + N\theta - Nx + 1 - \theta }\hspace{-15mm} \cdot (Nx + 1)^{- 1/2}(M_N + N\theta - Nx + 1  - \theta)^{- 1/2} .
\end{split}
\end{equation*}
Since $(1/N) \cdot \log[(Nx + 1)(M_N + N\theta - Nx + 1  -\theta)]^{- a} = O(N^{-1}\log (N+1))$ for $x \in [0, M_N \cdot N^{-1} + (N-1) \cdot N^{-1} \cdot \theta]$ and any fixed $a > 0$ (the constant in the big $O$ notation depends on $\lM$ alone), we conclude that
\begin{equation*}
\begin{split}
&|V_N(x) - V(x)| \leq  O(\log N/N) + \left|x \log x - (x + 1/N) \log (x + 1/N)  \right| \\
&+ \left|(\lM + \theta - x) \log (\lM + \theta - x) - \left( \frac{M_N + 1 - \theta}{N}  + \theta - x \right) \log \left( \frac{M_N+1 - \theta}{N}  + \theta - x \right)  \right| \\
& = O(N^{-1}\log (N+1)),
\end{split}
\end{equation*}
where in the last equality we used that $M_N = \lfloor \lM N\rfloor$. We thus conclude that (\ref{K3}) holds.\\

Since the $\mathbb{P}_{\beta, N}^{\theta,M_N}$ from (\ref{K2}) satisfy the conditions of Definition \ref{S2PDef} with $\fn = N^{-1}\log (N+1)$ and $\qn = 1/N$, we see that Propositions \ref{S4emp}, \ref{FinLLDP} and \ref{FinULDP} are all applicable to this sequence of measures. Below we explain what each of those statements says about the sequence $\mathbb{P}_{\beta, N}^{\theta,M_N}$.

Proposition \ref{S4emp} gives a law of large numbers for the empirical measures $\mu_N = \frac{1}{N} \sum_{i = 1}^N \delta (\ell_i/N)$. The equilibrium measure $\phi^{\theta, \lM + \theta}_{\beta}$ (this equals $\phi^{\theta, \lM + \theta}_V$ as in Lemma \ref{iv} for $V(x) = x\log x+(\lM + \theta-x)\log(\lM + \theta-x)$) was computed in \cite[Example 4.2]{ds97} when $\theta = 1$. In addition, the law of large numbers for the Krawtchouk ensemble has been established earlier in \cite{bgg} and \cite{jonon}. As an immediate consequence of these works we deduce the following result for general $\theta > 0$.
\begin{proposition}\label{krawden}
Let $\lM, \theta > 0 $ and put $M_N = \lfloor \lM N \rfloor$. Let $\vec\ell$ be distributed according to $\mathbb{P}_{\beta, N}^{\theta,M_N}$ as in (\ref{K2}). The sequence of empirical measures $\mu_N:=\frac1N\sum_{i=1}^N \delta(\ell_i/N)$ converges weakly in probability to a measure with density $\phi^{\theta, \lM + \theta}_{\beta}$, given as follows. For $\lM\ge \theta$,
\begin{align*}
\phi^{\theta, \lM + \theta}_{\beta}(x)=\begin{cases}
(\theta \pi)^{-1} \cdot \operatorname{arccot}\left(\dfrac{\lM  - \theta}{2\sqrt{\lM \theta -(x-(\lM + \theta)/2)^2}}\right) &\mbox{ for } |x  - (\lM  +\theta)/2|<\sqrt{\lM \theta} , \\ 0 &\mbox{ for }  |x  - (\lM +\theta)/2|\geq \sqrt{\lM \theta} .
\end{cases}
\end{align*}
For $\lM \in (0, \theta)$,
\begin{align*}
\phi^{\theta, \lM + \theta}_{\beta}(x)=\begin{cases}
(\theta \pi)^{-1} \cdot \operatorname{arccot}\left(\dfrac{\lM-\theta}{2\sqrt{\lM \theta - (x-(\lM + \theta)/2)^2}}\right) &\mbox{for } |x-(\lM + \theta)/2|<\sqrt{\lM \theta},
\\ \theta^{-1} & \mbox{for } (\lM+\theta)/2\ge |x-(\lM+\theta)/2| \ge \sqrt{\lM \theta},
\\ 0 & \mbox{for }  |x-(\lM+\theta)/2|\ge (\lM+\theta)/2. 
\end{cases}
\end{align*}
\end{proposition}
\begin{proof} From Proposition \ref{S4emp} we see that it suffices to show that $\phi_V^{\theta, \lM + \theta}$ from Lemma \ref{iv} agrees with $\phi^{\theta, \lM + \theta}_{\beta}$ as in the statement of the proposition. When $\theta = 1$ this equality follows from \cite[Proposition 2.2]{bgg}. 

To get the result for general $\theta > 0$, we note that that the map $H: \mathcal{A}_{\lM + \theta}^{\theta} \rightarrow \mathcal{A}_{\lM \theta^{-1} + 1}^{1}$, given by 
$$H(\phi)(x) = \theta \phi(\theta x),$$
defines a bijection. In addition, if $V(x) = x\log x+(\lM + \theta -x)\log(\lM + \theta-x)$ and $\tilde{V}(x) = x \log x + ( \lM\theta^{-1} + 1 - x)  \log ( \lM\theta^{-1} + 1 - x),$ we have
$$I_V^{\theta}(\phi) = \theta I^1_{\tilde{V}}(H(\phi)) + \lM \log \theta.$$
The latter implies that $H(\phi_V^{\theta, \lM + \theta}) = \phi^{1, \lM \theta^{-1} + 1}_{\beta}$, from which we conclude the statement of the proposition for any $\theta > 0$.
\end{proof}

\begin{remark}\label{S6EndKr} As can be deduced from Proposition \ref{krawden} the support of the measure $\phi^{\theta, \lM + \theta}_{\beta}$ is $[a_{\lM}, b_{\lM}]$, where $a_{\lM} = 0$ and $b_{\lM} = \lM + \theta$ when $\lM \in (0,\theta)$ and $a_{\lM} = (\lM+\theta)/2 - \sqrt{\lM \theta} $, $b_{\lM} = (\lM+\theta)/2 + \sqrt{\lM \theta} $ when $\lM \geq \theta$. See Figure \ref{S6_1} for a pictorial description of $\phi^{\theta, \lM + \theta}_{\beta}$ in the cases $\lM \geq \theta$ and $\lM \in (0,\theta)$. 
\end{remark}

\begin{figure}[ht]
\begin{center}
  \includegraphics[scale = 0.7]{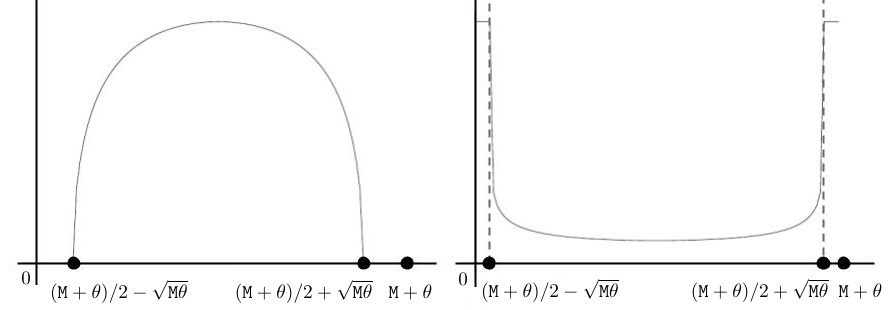}
  \vspace{-2mm}
  \caption{The figure on the left side depicts $\phi^{\theta, \lM + \theta}_{\beta}$ when $\lM \geq \theta$. The figure on the right side depicts $\phi^{\theta, \lM + \theta}_{\beta}$ when $\lM \in (0,\theta)$.}
\label{S6_1}
  \end{center}
\end{figure}

We next discuss the lower tail large deviation principle for the rightmost particle.
\begin{proposition}\label{krawLow}
Assume the same notation as in Proposition \ref{krawden}. For any $t \in [\theta, \lM + \theta]$
\begin{equation}\label{S6LimKrLow}
\lim_{N\rightarrow \infty} \frac{1}{N^2} \log \mathbb{P}_{\beta, N}^{\theta,M_N} ( \ell_1 \leq tN) = F_V^{\theta, \lM + \theta} - F_V^{\theta, t} ,
\end{equation}
where we recall that $F_{V}^{\theta,t}$ is defined in \eqref{fs} and $V(x) = x\log x+(\lM + \theta-x)\log(\lM + \theta-x) $. Furthermore, 
\begin{itemize}
\item If $\lM \in (0,\theta)$ then $ F_V^{\theta, t} > F_V^{\theta, \lM + \theta} $ for $t \in [\theta, \lM + \theta)$.  
\item If $\lM \geq \theta$ then $ F_V^{\theta, t} > F_V^{\theta, \lM + \theta} $ for $t \in [\theta, (\lM+\theta)/2 + \sqrt{\lM \theta} )$ and $ F_V^{\theta,t } = F_V^{\theta, \lM + \theta} $ for $t \in[ (\lM+\theta)/2 + \sqrt{\lM \theta}, \lM + \theta]$. 
\end{itemize}
\end{proposition}
\begin{proof} As explained earlier in the section, $ \mathbb{P}_{\beta, N}^{\theta,M_N}$ satisfy Definition \ref{S2PDef} with $\fn = N^{-1}\log (N+1)$ and $\qn = 1/N$ and so Proposition \ref{FinLLDP} is applicable. The latter gives the statements of the present proposition, once we note that the rightmost point of the support of $\phi^{\theta, \lM + \theta}_{\beta}$ from Proposition \ref{krawden} is $b_{\lM} = \lM + \theta$ when $\lM \in (0,\theta)$ and $b_{\lM} = (\lM+\theta)/2 + \sqrt{\lM \theta} $ when $\lM \geq \theta$, cf. Remark \ref{S6EndKr}.
\end{proof}

We finally turn our attention to the upper tail large deviation principle for the rightmost particle. We will split the result into two statement, depending on whether $\lM > \theta$ or $\lM \in (0, \theta]$.
\begin{proposition}\label{krawLow1}
Assume the same notation as in Proposition \ref{krawden}. Assume further that $\lM \in (0,\theta]$. Then for any $t \in [\theta, \lM + \theta)$ we have
\begin{equation}\label{S6LowEq1}
\lim_{N\rightarrow \infty} \frac{1}{N} \log \mathbb{P}_{\beta, N}^{\theta,M_N} ( \ell_1 \geq tN) = 0.
\end{equation}
\end{proposition}
\begin{proof}As explained earlier in the section, $\mathbb{P}_{\beta, N}^{\theta,M_N}$ satisfy Definition \ref{S2PDef} with $\fn = N^{-1}\log(N+1)$ and $\qn = 1/N$. This implies that Proposition \ref{FinULDP} is applicable. The latter gives the statements of the present proposition, once we note that the rightmost point of the support of $\phi^{\theta, \lM + \theta}_{\beta}$ from Proposition \ref{krawden} is $b_{\lM} = \lM + \theta$.
\end{proof}

\begin{proposition}\label{krawLow2}
Assume the same notation as in Proposition \ref{krawden}. Assume further that $\lM > \theta$ and set $a_{\lM} =(\lM+\theta)/2 - \sqrt{\lM \theta}$, $b_{\lM} =  (\lM+\theta)/2 + \sqrt{\lM \theta}$. For any $t \in [\theta, \lM + \theta)$ we have 
\begin{equation}\label{S6LowEq2}
\lim_{N\rightarrow \infty} \frac{1}{N} \log \mathbb{P}_{\beta, N}^{\theta,M_N} ( \ell_1 \geq tN) = - J^{\theta, \lM + \theta}_{\beta}(t),
\end{equation}
where $J^{\theta, \lM + \theta}_{\beta}(t) = J^{\theta, \lM + \theta}_V(t)$ is as in Lemma \ref{S5ContGJ} for $V(x) = x\log x+(\lM + \theta-x)\log(\lM + \theta-x) $. Moreover, we have the following explicit expression for $J^{\theta, \lM + \theta}_{\beta}(y)$
\begin{equation}\label{k-exact-up}
\begin{aligned}
J^{\theta, \lM + \theta}_{\beta}(y)=2\Lambda(y)+V(y)-V(b_{\lM}) - (y-b_{\lM}) V'(b_{\lM}) , \quad y\in (b_{\lM},\lM + \theta),
\end{aligned}
\end{equation}
where
\begin{equation}\label{k-capl}
\begin{aligned}
\Lambda(y)  :=&-y\log\frac{\sqrt{y-a_{\lM}}+\sqrt{y-b_{\lM}}\sqrt{a_{\lM}/b_{\lM}}}{\sqrt{y-a_{\lM}}+\sqrt{y-b_{\lM}}\sqrt{b_{\lM}/a_{\lM}}} - (\lM + \theta)  \log\left(\sqrt{y-a_{\lM}}+\sqrt{y-b_{\lM}}\sqrt{b_{\lM}/a_{\lM}}\right) \\
& +(\lM - \theta)\log\left(\sqrt{y-a_{\lM}}+\sqrt{y-b_{\lM}}\right) +\theta \log(b_{\lM} - a_{\lM}).
\end{aligned}
\end{equation}	
In particular, we have
\begin{align}\label{k-asymp-up}
\lim_{\alpha \rightarrow 0+} \frac{J^{\theta, \lM + \theta}_{\beta}(b_{\lM}+\alpha)}{\alpha^{3/2}}=\frac{4\sqrt{b_{\lM}-a_{\lM}}}{3\sqrt{a_{\lM} b_{\lM} }}=\frac{8\sqrt{2\sqrt{\lM \theta}}}{3(\lM-\theta)}.
\end{align}	
\end{proposition}

Before we go into the proof we record the following lemma, which will be required. Its proof is postponed to Section \ref{Section7} (see Lemma \ref{S7logints}).
\begin{lemma}\label{logints} For $a,b,c,d\ge 0$ with $cd>0$ and $a+b>0$, we consider the integrals
\begin{align*}
\mathcal{I}^{\pm}_{a,b,c,d;n}:=\int_0^{\infty} \frac{\log|a^2 \pm b^2z^2|}{(c^2+d^2z^2)^n}\d z, \quad \mathcal{J}_{c,d;n} := \int_0^{\infty} \frac{\d z}{(c^2+d^2z^2)^n}.
\end{align*}
We have the following exact expressions for the above integrals for particular values.
\begin{multicols}{2}
\begin{enumerate}
\item $\ds \mathcal{I}_{a,b,c,d;1}^{-} = \frac{\pi\log|a^2+\frac{b^2c^2}{d^2}|}{2cd},$
\item $\ds \mathcal{I}_{a,b,c,d;1}^{+}=\frac{\pi}{cd}\log(a+\tfrac{bc}d)$,
\item $\ds \mathcal{I}_{a,b,1,1;2}^{-}= \frac{\pi}{4}\log(a^2+b^2)-\frac{\pi b^2}{2(a^2+b^2)},$
\item $\ds \mathcal{I}_{a,b,1,1;2}^{+}=\frac{\pi}{2}\log(a+b)-\frac{b\pi}{2a+2b},$ \vspace{1mm}
\item $\ds \mathcal{J}_{c,d;1}=\frac{\pi}{2{cd}},$ \vspace{2mm}
\item $\ds \mathcal{J}_{1,1;2}=\frac{\pi}{4}$.
\end{enumerate}
\end{multicols}
\end{lemma}

\begin{proof}[Proof of Proposition \ref{krawLow2}] Let us recall the notation from Lemma \ref{S5ContGJ}. We have that 
\begin{align}\label{eq:gvm}
G^{\theta, \lM + \theta}_{\beta}(y):=-2\theta \int_0^{\lM + \theta} \log|y-t|\phi^{\theta, \lM + \theta}_{\beta}(t)dt+V(y),
\end{align}
where $\phi^{\theta, \lM + \theta}_{\beta}$ is as in Proposition \ref{krawden} and $V(x) = x\log x+(\lM + \theta-x)\log(\lM + \theta-x)$. In addition,
\begin{equation}\label{S6JFun}
J^{\theta, \lM + \theta}_{\beta}(x) =\begin{cases} 0 &\mbox{ if $x \in [0, b_{\lM})$, } \\ 
\ds \inf_{y \in [x, \lM + \theta]} G^{\theta, \lM + \theta}_{\beta}(y)  - G^{\theta, \lM + \theta}_{\beta}(b_{\lM})   &\mbox{ if $x \in [b_{\lM}, \lM + \theta]$}. \end{cases}
\end{equation}
For clarity we split the proof into two steps.\\

{\bf \raggedleft Step 1.} In this step we assume that (\ref{k-exact-up}) holds and prove the other statements in the proposition. We will prove (\ref{k-exact-up}) in the next step.

One directly computes from (\ref{k-capl}) that
\begin{align}\label{lmbdapk}
\Lambda'(y) & =- \log\frac{\sqrt{y-a_{\lM}}+\sqrt{y-b_{\lM}}\sqrt{a_{\lM}/b_{\lM}}}{\sqrt{y-a_{\lM}}+\sqrt{y-b_{\lM}}\sqrt{b_{\lM}/a_{\lM}}} > 0 \mbox{ for } y>b_{\lM}.
\end{align}
In addition, 
\begin{align}\label{lmbdapk2}
V'(y) = \log(y) - \log(a_{\lM} + b_{\lM} - y)> 0 \mbox{ for } y \in (b_{\lM}, \lM + \theta), \mbox{ and }V'(b_{\lM}) = \log b_{\lM} - \log a_{\lM} > 0,
\end{align}
where we used that $a_{\lM} + b_{\lM} = \lM + \theta$. The latter two inequalities and (\ref{k-exact-up}) show that $dJ^{\theta, \lM + \theta}_{\beta}(y)/dy > 0$ for $y > b_{\lM}$ and so $J^{\theta, \lM + \theta}_{\beta}(y) > 0$ for $y  \in (b_{\lM}, \lM + \theta)$ (here we used $J^{\theta, \lM + \theta}_{\beta}(b_{\lM}) = 0$ by Lemma \ref{S5ContGJ}).

As explained earlier in the section, $\mathbb{P}_{\beta, N}^{\theta,M_N}$ satisfy Definition \ref{S2PDef} with $\fn = N^{-1}\log (N+1)$ and $\qn = 1/N$. This implies that Proposition \ref{FinULDP} is applicable. The latter gives (\ref{S6LowEq2}), once we note that the rightmost point of the support of $\phi^{\theta, \lM + \theta}_{\beta}$ is $b_{\lM}$ and that $J^{\theta, \lM + \theta}_{\beta}(y) > 0$ for all $y  \in (b_{\lM}, \lM + \theta)$. \\

We next prove (\ref{k-asymp-up}). Since $J^{\theta, \lM + \theta}_{\beta}$ is differentiable (from the explicit form in (\ref{k-exact-up})) and by direct computation $J^{\theta, \lM + \theta}_{\beta}(b_{\lM}) = 0$ we see that 
\begin{equation*}
\begin{split}
 \lim_{\alpha \rightarrow 0+} \frac{J^{\theta, \lM + \theta}_{\beta}(b_{\lM}+\alpha)}{\alpha^{3/2}} &  = \lim_{\alpha \rightarrow 0+} \frac{2\Lambda'(b_{\lM}+\alpha) + V'(b_{\lM} + \alpha) - V'(b_{\lM})}{(3/2)\alpha^{1/2}}  \\
& = \lim_{x \rightarrow 0+} \frac{- 2\log\frac{\sqrt{b_{\lM}-a_{\lM} + x^2}+x \sqrt{a_{\lM}/b_{\lM}}}{\sqrt{b_{\lM}-a_{\lM} + x^2}+x\sqrt{b_{\lM}/a_{\lM}}} + \log(1+ x^2/ b_{\lM}) - \log(1  - x^2/a_{\lM})}{(3/2)x} \\
& = \lim_{x \rightarrow 0+} \frac{4}{3} \cdot \frac{\log\frac{\sqrt{b_{\lM}-a_{\lM} + x^2}+x\sqrt{b_{\lM}/a_{\lM}}}{\sqrt{b_{\lM}-a_{\lM} + x^2}+x \sqrt{a_{\lM}/b_{\lM}}} }{x} = \frac{4}{3} \cdot \frac{\sqrt{b_{\lM}/a_{\lM}} - \sqrt{a_{\lM}/b_{\lM}}}{\sqrt{b_{\lM} - a_{\lM}}},
\end{split}
\end{equation*}
where we used L'H\^{o}pital's rule a few times and the fact that for any real $c$, we have $\lim_{x \rightarrow 0+} \frac{\log(1 + cx^2)}{x} = 0$. The last equation implies (\ref{k-asymp-up}).\\

{\bf \raggedleft Step 2.} In this step we prove (\ref{k-exact-up}). Define for $y > b_{\lM}$ the function
\begin{align}\label{eq:im}
\mathcal{I}_y & :=\int_{0}^{\lM + \theta} \log(y-x)\phi^{\theta, \lM + \theta}_{\beta}(x)\d x =\int_{a_{\lM}}^{b_{\lM}} \log(y-x)\phi^{\theta, \lM + \theta}_{\beta}(x)\d x.
\end{align}
Our first goal is to compute $\mathcal{I}_y$. Then we will express $G^{\theta, \lM + \theta}_{\beta}$ in terms of it and finally compute $J^{\theta, \lM + \theta}_{\beta}$ and show that (\ref{k-exact-up}) holds.

Noting that $\frac{d}{dx} [(x-y)\log (y-x)-x] = \log(y-x)$ and 
$$\frac{d}{dx}\phi^{\theta, \lM + \theta}_{\beta}(x)=\frac{(\lM-\theta)(\lM + \theta-2x)}{4\theta \pi(\lM + \theta -x)x\sqrt{(x-a_{\lM})(b_{\lM}-x)}},$$
we may apply integration by parts to get
\begin{align*}
\mathcal{I}_y & = \frac{1}{\theta \pi }\int_{a_\lM}^{b_\lM}[(y-x)\log(y-x)+x] \cdot\frac{(\lM-\theta)(\lM + \theta-2x)dx}{4(\lM + \theta -x)x\sqrt{(x-a_{\lM})(b_{\lM}-x)}}.
\end{align*}
We substitute $\lM + \theta=a _{\lM} + b_{\lM}$ and set
\begin{align*}
\mathcal{K}_y & :=\int_{a_{\lM}}^{b_{\lM}}\frac{[(y-x)\log(y-x)+x](a_{\lM}+b_{\lM}-2x)\d x}{2(a_{\lM}+b_{\lM} -x)x\sqrt{(x-a_{\lM})(b_{\lM}-x)}}, \mbox{ so that } \mathcal{I}_y = \frac{\lM-\theta}{2\theta \pi}\mathcal{K}_y.
\end{align*}
For $\mathcal{K}_y$ we use the substitution $x\mapsto \frac{a_{\lM}+b_{\lM}z^2}{z^2+1}$ to get
\begin{align*}
\mathcal{K}_y & = \int_{0}^{\infty}\left[\left(y-\frac{a_{\lM}+b_{\lM}z^2}{z^2+1}\right)\log\left(y-\frac{a_{\lM}+b_{\lM} z^2}{z^2+1}\right)+\frac{a_{\lM}+b_{\lM}z^2}{z^2+1}\right]\left[\frac{1}{a_{\lM}+b_{\lM}z^2}-\frac{1}{a_{\lM}z^2+b_{\lM}}\right]\d z \\
& = y\int_{0}^{\infty}\log\left(y-\frac{a_{\lM}+b_{\lM}z^2}{z^2+1}\right)\left[\frac{1}{a_{\lM}+b_{\lM}z^2}-\frac1{a_{\lM}z^2+b_{\lM}}\right]\d z \\ &  \hspace{2cm}+ \int_{0}^{\infty}\left[\log\left(y-\frac{a_{\lM}+b_{\lM}z^2}{z^2+1}\right)-1\right]\left[\frac{a_{\lM}+b_{\lM}}{a_{\lM}z^2+b_{\lM}}-\frac{2}{z^2+1}\right]\d z \\ 
& = y\left[\mathcal{I}^{+}_{\sqrt{y-a_{\lM}},\sqrt{y-b_{\lM}},\sqrt{a_{\lM}},\sqrt{b_{\lM}};1}-\mathcal{I}^{+}_{\sqrt{y-a_{\lM}},\sqrt{y-b_{\lM}},\sqrt{b_{\lM}},\sqrt{a_{\lM}};1}-\mathcal{I}^{+}_{1,1,\sqrt{a_{\lM}},\sqrt{b_{\lM}};1}+\mathcal{I}^{+}_{1,1,\sqrt{b_{\lM}},\sqrt{a_{\lM}};1}\right]  \\ 
& \hspace{1cm} +(a_{\lM}+b_{\lM})\left[\mathcal{I}^+_{\sqrt{y-a_{\lM}},\sqrt{y-b_{\lM}},\sqrt{b_{\lM}},\sqrt{a_{\lM}};1}-\mathcal{I}^+_{1,1,\sqrt{b_{\lM}},\sqrt{a_{\lM}};1}-\mathcal{J}_{\sqrt{b_{\lM}},\sqrt{a_{\lM}};1}\right]\\ 
& \hspace{1cm} -2\left[\mathcal{I}^+_{\sqrt{y-a_{\lM}},\sqrt{y-b_{\lM}},1,1;1}-\mathcal{I}^+_{1,1,1,1;1}-\mathcal{J}_{1,1;1}\right] \\
 & = \frac{y\pi}{\sqrt{a_{\lM} b_{\lM}}}\left[\log\frac{\sqrt{y-a_{\lM}}+\sqrt{y-b_{\lM}}\sqrt{a_{\lM}/b_{\lM}}}{\sqrt{y-a_{\lM}}+\sqrt{y-b_{\lM}}\sqrt{b_{\lM}/a_{\lM}}}+\frac12\log\frac{b_{\lM}}{a_{\lM}}\right]\\ 
& +\frac{(a_{\lM} + b_{\lM})\pi}{\sqrt{a_{\lM}b_{\lM}}}\left[\log\frac{\sqrt{y-a_{\lM}}+\sqrt{y-b_{\lM}}\sqrt{b_{\lM}/a_{\lM}} }{1+\sqrt{b_{\lM}/a_{\lM}} }-\frac12\right]  -2\pi\left[\log\frac{\sqrt{y-a_{\lM}}+\sqrt{y-b_{\lM}}}2-\frac12\right],
\end{align*}
where the last two equalities follow from Lemma \ref{logints}. 

Using the above expression and recalling \eqref{eq:gvm} and \eqref{eq:im} we get
\begin{equation*}
\begin{split}
&G^{\theta, \lM + \theta}_{\beta}(y) - G^{\theta, \lM + \theta}_{\beta}(b_{\lM}) = V(y) - V(b_{\lM})  -2\theta \mathcal{I}_y+ 2 \theta \mathcal{I}_{b_{\lM}}  = V(y) - V(b_{\lM}) -\frac{(a_{\lM} + b_{\lM} - 2\theta)}{\pi}[ \mathcal{K}_y - \mathcal{K}_{b_{\lM}}]\\
& = V(y) - V(b_{\lM}) -2y\log\frac{\sqrt{y-a_{\lM}}+\sqrt{y-b_{\lM}}\sqrt{a_{\lM}/b_{\lM}}}{\sqrt{y-a_{\lM}}+\sqrt{y-b_{\lM}}\sqrt{b_{\lM}/a_{\lM}}} - (y- b_{\lM}) \log\frac{b_{\lM}}{a_{\lM}}  \\ 
&- 2 (\lM + \theta) \log\left(\frac{\sqrt{y-a_{\lM}}+\sqrt{y-b_{\lM}}\sqrt{b_{\lM}/a_{\lM}}}{\sqrt{b_{\lM} - a_{\lM}}}\right) +2 (\lM - \theta)\log\left(\frac{\sqrt{y-a_{\lM}}+\sqrt{y-b_{\lM}}}{\sqrt{b_{\lM} - a_{\lM}}}\right)   \\
&= 2\Lambda(y)+V(y)-V(b_{\lM}) - (y-b_{\lM}) V'(b_{\lM}) ,
\end{split}
\end{equation*}
where in going from the first to the second line we used the fact that $a_{\lM} + b_{\lM} = \lM + \theta$ and $(a_{\lM} + b_{\lM} -2\theta) = 2 \sqrt{a_{\lM} b_{\lM}}$, and the last equality used the definition of $\Lambda$ in (\ref{k-capl}) and the fact that $V'(b_{\lM}) = \log b_{\lM} - \log a_{\lM}$.

Using (\ref{lmbdapk}) and (\ref{lmbdapk2}) we see that $G^{\theta, \lM + \theta}_{\beta}(y) - G^{\theta, \lM + \theta}_{\beta}(b_{\lM}) $ is strictly increasing on $(b_{\lM}, \lM + \theta)$ and so from (\ref{S6JFun}) we conclude that 
$$J^{\theta, \lM + \theta}_{\beta}(y) = G^{\theta, \lM + \theta}_{\beta}(y) - G^{\theta, \lM + \theta}_{\beta}(b_{\lM}) = 2\Lambda(y)+V(y)-V(b_{\lM}) - (y-b) V'(b_{\lM}) ,$$
establishing (\ref{k-exact-up}).
\end{proof}

%
\subsection{Application to the Jack measures with Plancherel specialization}\label{Section6.3} 
In Section \ref{Section6.3.1} we fix a particular one-parameter family of Jack measures, explain how they can be interpreted as discrete $\beta$-ensembles and what the results of the present paper say about the asymptotics of these measures. Section \ref{Section6.3.2} contains the proof of a technical result we require in Section \ref{Section6.3.1}.

%
\subsubsection{Asymptotics of $\P^{\theta, tN}_{\mathfrak{r}, N}$}\label{Section6.3.1}
In this section we consider a special case of the measures in Definition \ref{jackms}, corresponding to setting $\rho_1 = 1^N$ (i.e. a pure-$\alpha$ specialization in $N$ variables that are all equal to $1$) and $\rho_2 = \mathfrak{r}_s$ (i.e. a Plancherel specialization with parameter $\gamma = s > 0$). In view of (\ref{NormConst}), we have that 
$$H_{\theta}(\rho_1, \rho_2) = \exp\left( \sum_{k = 1}^\infty \frac{\theta p_k(\rho_1) p_k(\rho_2)}{k}\right) = \exp \left( \theta \cdot s N  \right) < \infty,$$
so that the measures in Definition \ref{jackms} on $\mathbb{Y}$ are indeed well-defined for this choice of $\rho_1, \rho_2$. In addition, using that $J_\lambda(1^N) = 0$ if $\lambda_{N+1} > 0$, we see that $\mathcal{J}_{\rho_1,\rho_2}(\cdot )$ is supported on $\lambda \in \mathbb{Y}$ such that $\lambda_{N+1} = 0$. Thus we may think of $ \mathcal{J}_{\rho_1,\rho_2}$ as a measure on $N$-tuples $\lambda_1 \geq \cdots \geq \lambda_N \geq 0$, i.e. a measure on $\mathbb{Y}_N$ as in (\ref{GenState}). Explicitly, we have for $\lambda \in \mathbb{Y}_N$ that 
\begin{equation}\label{JackLambda}
\mathcal{J}_{\rho_1,\rho_2}(\lambda) = e^{- \theta sN} \cdot J_{\lambda}(1^N) \cdot \til{J}_{\lambda}(\mathfrak{r}_s).
\end{equation}

Our first task is to rewrite the measure in (\ref{JackLambda}), as a measure on $\vec{\ell} \in \mathbb{W}^{\theta,\infty}_{N} $ (the latter set was defined in (\ref{GenState})) using the relations $\ell_i = \lambda_i + (N-i) \cdot \theta$ for $i = 1, \dots, N$. In the process of doing this we will see that the measure in (\ref{JackLambda}) is of the form (\ref{PDef}) and then we will explain how our results can be used to study its asymptotics. We mention here that the measure $\mathcal{J}_{\rho_1,\rho_2}(\lambda)$ in (\ref{JackLambda}) was previously studied in \cite{misha}, where it arises as the time $s$ distribution of a certain Markov process on $\mathbb{Y}_N$, which is a discrete version of $\beta$-Dyson Brownian motion (here $\beta = 2 \theta$).\\

Combining (\ref{jackpoly2}) and (\ref{jackpoly3}) we get that the induced measure on $\vec{\ell}$ from (\ref{JackLambda}) is given by
\begin{align}\label{jackM}
\mathcal{J}_{\rho_1,\rho_2}(\vec{\ell})=\frac{\Gamma(\theta)^Ne^{-s\theta N}(s\theta)^{-\frac{N(N-1)}2}}{\prod_{i=1}^N\Gamma(i\theta)} \prod_{1\le i<j\le N} \frac{\Gamma(\ell_i-\ell_j+\theta)\Gamma(\ell_i-\ell_j+1)}{\Gamma(\ell_i-\ell_j)\Gamma(\ell_i-\ell_j+1-\theta)}\prod_{i=1}^N \frac{(s\theta)^{\ell_i}}{\Gamma(\ell_i+1 )}.
\end{align}
Observe that (\ref{jackM}) is of the form (\ref{PDef}), as claimed earlier. 

We will consider the asymptotic behavior of the measure in (\ref{jackM}) when $s = t N$ for a fixed $t > 0$. For concreteness, we denote this by
\begin{align}\label{jackM2}
\P^{\theta, tN}_{\mathfrak{r}, N}(\vec{\ell})= \frac{1}{Z_N} \cdot \prod_{1\le i<j\le N} \frac{\Gamma(\ell_i-\ell_j+\theta)\Gamma(\ell_i-\ell_j+1)}{\Gamma(\ell_i-\ell_j)\Gamma(\ell_i-\ell_j+1-\theta)}\prod_{i=1}^N e^{- N V_N(\ell_i/N)},
\end{align}
where 
\begin{equation}\label{jackM3}
V_N(x) = A + \frac{1}{N}\log\frac{\Gamma(Nx+1 )}{(t\theta N)^{Nx}}, \mbox{ and }Z_N =e^{-N^2 A} \cdot \Gamma(\theta)^{-N}e^{s\theta N}(s\theta)^{\frac{N(N-1)}2} \cdot \prod_{i=1}^N\Gamma(i\theta).
\end{equation}
Here $A$ is a large enough constant depending on $t, \theta$ alone, which we will specify momentarily. 

Our next task is to show that the measures in (\ref{jackM2}) satisfy the conditions in Definition \ref{Assumptions}. In the process, we will specify the value of $A$. We remark that  that changing the value of $A$ does not affect the measure in (\ref{jackM2}) and our goal is to pick $A$ large enough so that $V_N(x)$ satisfy equation (\ref{VPot}) in Definition \ref{Assumptions}.

In view of (\ref{GammaULBound}) we have for all $x \geq 0$ 
\begin{equation}\label{CA1}
\begin{split}
&\left(x+\frac{1 - \gamma}{N}\right ) \log (Nx+1) -x- \log(t \theta N)x \leq \frac{1}{N}\log\frac{\Gamma(Nx+1 )}{(t\theta N)^{Nx}}\\
& \leq \left(x+\frac{1}{2N} \right) \log (Nx+1) -x - \log(t \theta N)x,
\end{split}
\end{equation}
where we recall that $\gamma = 0.577215\dots$ is the Euler-Mascheroni constant. Observe that for $a \in \{1/2, 1 - \gamma\}$ and $x \geq 0$ we have 
\begin{equation}\label{CA2}
  \left(x+\frac{a}{N} \right) \log (Nx+1) -x - \log(t \theta N)x = (x+ a/N) \log(x + 1/N) - \log(e t \theta) x + \frac{a \log N}{N}.
\end{equation}
Equations (\ref{CA1}) and (\ref{CA2}) imply that for $x \geq 1$ and all $N \geq 1$ we have
$$\frac{1}{N}\log\frac{\Gamma(Nx+1 )}{(t\theta N)^{Nx}} \geq x \log x - \log(e t \theta) x,$$
while for any fixed $n \in \mathbb{N}$ and $x \in [0,n]$ we have 
$$\frac{1}{N}\log\frac{\Gamma(Nx+1 )}{(t\theta N)^{Nx}}= O(1),$$
where the constant in the big $O$ notation depends on $n, \theta, t$. The last two equations imply that there exists $A > 0$, depending on $\theta$ and $t$ such that for all $N \geq 1$ and $x \geq 0$ we have
$$A + \frac{1}{N}\log\frac{\Gamma(Nx+1 )}{(t\theta N)^{Nx}} \geq 2 \theta \log (1 + x^2).$$
We fix this choice of $A$ in the remainder, and then the last inequality implies that $V_N$ as in (\ref{jackM3}) satisfy (\ref{VPot}) with $\xi = 1$. 

With the choice of $A$ as above we check that $\P^{\theta, tN}_{\mathfrak{r}, N}$ as in (\ref{jackM2}) satisfy Definition \ref{Assumptions} with 
\begin{equation}\label{S6DefV}
V(x) =  A + x\log x - \log(e t \theta) x,
\end{equation}
and $r_N = N^{-1}\log (N+1)$. The continuity of $V_N$ on $[0, \infty)$ is an immediate consequence of the continuity of the gamma function on $(0,\infty)$ and clearly $V$ is differentiable on $(0, \infty)$ with 
$$V'(x) = \log x + 1 - \log(e t \theta),$$
so that (\ref{VPotU}) is satisfied with $B_0 = 1$ and $F_2(a) = |1 - \log(e t \theta)|$. What remains to be shown is that (\ref{conv-rate}) holds for  some increasing function $F_1$ and $r_N = N^{-1}\log (N+1)$. In view of (\ref{jackM3}), (\ref{CA1}) and (\ref{CA2}) we know that for any $n \in \mathbb{N}$ we can find a constant $A_n > 0$ (depending on $n$, $\theta$ and $t$) such that for all $N \geq 1$
$$\sup_{x \in [0, n]} |V_N(x) - V(x)| \leq A_n N^{-1}\log (N+1).$$
 Thus we may set $F_1(a) = \max_{1 \leq n \leq \lceil a \rceil} A_n$ and this function would satisfy (\ref{conv-rate}). Overall, we conclude that $\P^{\theta, tN}_{\mathfrak{r}, N}$ as in (\ref{jackM2}) indeed satisfy Definition \ref{Assumptions} with $V(x) =  A + x\log x - \log(e t \theta) x$ and the above choices of $\xi, r_N, F_1, F_2$. \\

Having checked that $\P^{\theta, tN}_{\mathfrak{r}, N}$ satisfy Definition \ref{Assumptions}, we explain what Theorems \ref{emp} and \ref{TMain} imply for this sequence of measures as $N \rightarrow \infty$. To formulate the law of large numbers we require the following lemma, whose proof is postponed until Section \ref{Section6.3.2}.

\begin{lemma}\label{JackMin} Fix $\theta, t > 0$ and let $V(x) = A + x\log x - \log(e t \theta) x$ be as in (\ref{S6DefV}). We define the function $\phi^{\theta, t}_{\mathfrak{r}}$ on $[0, \infty)$ through the following equations. For $t\ge 1$,
\begin{align*}
\phi^{\theta, t}_{\mathfrak{r}}(x)=\begin{cases} (\theta\pi)^{-1}\operatorname{arccot}\left(\dfrac{x+\theta(t-1)}{\sqrt{4\theta tx-[x+\theta(t-1)]^2}}\right) &  \mbox{for } x\in(\theta(\sqrt{t}-1)^2,\theta(\sqrt{t}+1)^2), \\ 0 & \mbox{otherwise.}
\end{cases}
\end{align*}
For $t\in (0,1)$,
\begin{align*}
\phi^{\theta, t}_{\mathfrak{r}}(x)=\begin{cases}
(\theta\pi)^{-1}\operatorname{arccot}\left(\dfrac{x+\theta(t-1)}{\sqrt{4\theta tx-[x+\theta(t-1)]^2}}\right) &  \mbox{for } x\in(\theta(\sqrt{t}-1)^2,\theta(\sqrt{t}+1)^2),
\\ \theta^{-1} & \mbox{for }0 \leq x<\theta(\sqrt{t}-1)^2,  \\ 0 & \mbox{for }  x>\theta(\sqrt{t}+1)^2. 
\end{cases}
\end{align*} 
Then $\phi^{\theta, t}_{\mathfrak{r}} = \phi^{\theta, \infty}_V$ as in Lemma \ref{iv}. 

If we let $G^{\theta, t}_{\mathfrak{r}}$ be as in Lemma \ref{S1GJ}, i.e.
$$G^{\theta, t}_{\mathfrak{r}}(x) = -2\theta  \int_{0}^{\infty} \log |x - t|\phi^{\theta, t}_{\mathfrak{r}}(t)dt + V(x),$$
and $b = \theta(\sqrt{t}+1)^2$ be the rightmost endpoint of the support of $\phi^{\theta, t}_{\mathfrak{r}}$, then we have
\begin{equation}\label{S6GJ}
G^{\theta, t}_{\mathfrak{r}}(b+\alpha) - G^{\theta, t}_{\mathfrak{r}}(b) = 2\Lambda(\alpha) + (b + \alpha)\log \frac{b + \alpha}{b} - \alpha, \mbox{ for } \alpha \in [0, \infty),
\end{equation}
where 
\begin{align}\label{S6LambdaF}
\Lambda(\alpha)=(b+\alpha)\log\frac{\sqrt{\alpha}+\sqrt{4\sqrt{t}\theta+\alpha}}{\frac{\sqrt{t}-1}{\sqrt{t}+1}\sqrt{\alpha}
\hspace{-0.25mm}+ \hspace{-0.25mm}\sqrt{4\sqrt{t}\theta+\alpha}}-\frac{2\sqrt{t}\theta\sqrt{\alpha}}{\sqrt{\alpha}+\sqrt{4\sqrt{t}\theta+\alpha}} \hspace{-0.25mm}- \hspace{-0.25mm}2\theta\log\frac{\sqrt{\alpha}+\sqrt{4\sqrt{t}\theta+\alpha}}{\sqrt{4\sqrt{t}\theta}}.
\end{align}
\end{lemma}
\begin{remark}\label{DimPanRem}
The formula for $\phi^{\theta, t}_{\mathfrak{r}}$ in Lemma \ref{JackMin} was {\em guessed} in \cite{DLjack} by using what are known as {\em discrete loop equations} or {\em Nekrasov's equations} (see \cite{bgg}) for the measures $\P^{\theta, tN}_{\mathfrak{r}, N}$. However, the authors in \cite{DLjack} did not establish Lemma \ref{JackMin} and so we will provide a proof of it in Section \ref{Section6.3.2} by using a variational characterization of the equilibrium measure $\pv$ from \cite{ds97}.
\end{remark}
\begin{remark}\label{IndepA}
We mention that for any constant $a \in \mathbb{R}$ we have that $I_{V+a}^{\theta}(\phi) = a + I_{V}^{\theta}(\phi)$ and so the minimizer $\pv$  of $I_V^{\theta}$ over $\mathcal{A}^{\theta}_\infty$ in Lemma \ref{iv} does not depend on $A$, but only on $t$ and $\theta$.
\end{remark}

With the above result in place, we can state the law of large numbers result for $\P^{\theta, tN}_{\mathfrak{r}, N}$.
\begin{proposition}\label{jckden} Fix $\theta, t > 0$. Let $\vec{\ell}$ be distributed according to $\P^{\theta, tN}_{\mathfrak{r}, N}$ as in (\ref{jackM2}). The sequence of empirical measures $\mu_N:= \frac{1}{N} \sum_{i = 1}^N \delta(\ell_i/N)$ converges weakly in probability to the measure with density $\phi^{\theta, t}_{\mathfrak{r}}$ as in Lemma \ref{JackMin}.
\end{proposition}
\begin{proof}As explained earlier in the section, $\P^{\theta, tN}_{\mathfrak{r}, N}$ satisfy Definition \ref{Assumptions}. This implies that Theorem \ref{emp} is applicable and the latter implies the present proposition in view of Lemma \ref{JackMin}.
\end{proof}

We note that the equilibrium measure $\phi^{\theta, t}_{\mathfrak{r}}$, describing the law of large numbers for $\P^{\theta, tN}_{\mathfrak{r}, N}$ depends on $t$ and $\theta$ alone (see Remark \ref{IndepA}) and has a different form depending on whether $t \in (0,1)$ or $t \geq 1$ -- see Figure \ref{fig:jckm}.
\begin{figure}[ht]
\begin{center}
  \includegraphics[scale = 0.8]{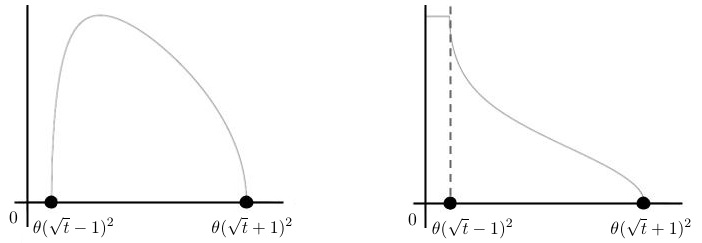}
  \vspace{-2mm}
  \caption{The figure on the left side depicts the graph of $\phi^{\theta, t}_{\mathfrak{r}}$ for $t>1$. The right side depicts the graph $\phi^{\theta, t}_{\mathfrak{r}}$ for $t<1$.}
\label{fig:jckm}
  \end{center}
\end{figure}

We next discuss the lower tail large deviation principle for the rightmost particle.

\begin{proposition}\label{JackLow}
Assume the same notation as in Proposition \ref{jckden}. For any $s \in [\theta, \infty)$
\begin{equation}\label{S6LimJackLow}
\lim_{N\rightarrow \infty} \frac{1}{N^2} \log \P^{\theta, tN}_{\mathfrak{r}, N}( \ell_1 \leq  s N) = F_V^{\theta, \infty} - F_V^{\theta, s} ,
\end{equation}
where we recall that $F_{V}^{\theta,s}$ is defined in \eqref{fs} and $V(x) = A + x\log x - \log(e t \theta) x$ is as in (\ref{S6DefV}). Furthermore, we have that $ F_V^{\theta, s} > F_V^{\theta, \infty}$ for $s \in [\theta,  \theta (\sqrt{t} + 1)^2)$ and $ F_V^{\theta, s} = F_V^{\theta, \infty} $ for $s \geq  \theta (\sqrt{t} + 1)^2$. 
\end{proposition}
\begin{remark} We mention that $ F_V^{\theta, \infty} - F_V^{\theta, s}$ depends on $s,t$ and $\theta$ alone (i.e. it does not depend on $A$), cf. Remark \ref{IndepA}.
\end{remark}
\begin{proof} As explained earlier in the section, $\P^{\theta, tN}_{\mathfrak{r}, N}$ satisfy Definition \ref{Assumptions} and so Theorem \ref{TMain}(a) is applicable. The latter gives the statements of the present proposition, once we note that the rightmost point of the support of $\phi^{\theta, t}_{\mathfrak{r}}$ from Proposition \ref{jckden} is $b = \theta (\sqrt{t} + 1)^2$.
\end{proof}

We finally turn our attention to the upper tail large deviation principle for the rightmost particle.
\begin{proposition}\label{JackHigh}
Assume the same notation as in Proposition \ref{jckden} and let $b = \theta (\sqrt{t} + 1)^2$ be the rightmost endpoint of the support of $\phi^{\theta, t}_{\mathfrak{r}}$. For any $s \in [\theta, \infty)$ we have
\begin{equation}\label{S6LimJackHigh}
\lim_{N\rightarrow \infty} \frac{1}{N} \log \P^{\theta, tN}_{\mathfrak{r}, N}( \ell_1 \geq s N) = - J^{\theta, t}_{\mathfrak{r}}(s),
\end{equation}
where $J^{\theta, t}_{\mathfrak{r}}(s) = J^{\theta, \infty}_V(s)$ is as in Lemma \ref{S1GJ} for $V(x) =A + x\log x - \log(e t \theta) x$  as in (\ref{S6DefV}).
Moreover, we have the following explicit expression for $J^{\theta, t}_{\mathfrak{r}}(y)$ for $y\ge b$:
\begin{equation}\label{j-exact-up}
J^{\theta, t}_{\mathfrak{r}}(b + \alpha)=2\Lambda(\alpha) + (b + \alpha)\log \frac{b + \alpha}{b} - \alpha,
\end{equation}
where $\Lambda(\alpha)$ is defined in \eqref{S6LambdaF}.
In particular, we have
\begin{align}\label{asymp-up}
\lim_{\alpha \rightarrow 0+} \frac{J^{\theta, t}_{\mathfrak{r}}(b+\alpha)}{\alpha^{3/2}}=\frac4{3\sqrt{\theta\sqrt{t}}(\sqrt{t}+1)}.
\end{align}	 
\end{proposition}
\begin{proof} For clarity we split the proof into two steps.

{\bf \raggedleft Step 1.} In this step we prove (\ref{j-exact-up}). In the next step we show the other parts of the proposition.
Recall from \eqref{S1GJ} that for $\alpha \geq 0$ we have
$$ J^{\theta, t}_{\mathfrak{r}}(b + \alpha) = \inf_{y \ge \alpha} (G^{\theta, t}_{\mathfrak{r}}(b+y)-G^{\theta, t}_{\mathfrak{r}}(b)),$$
where $G^{\theta, t}_{\mathfrak{r}}$ is as in Lemma \ref{JackMin}.
From (\ref{S6GJ}) we have that $G^{\theta, t}_{\mathfrak{r}}(b+y)-G^{\theta, t}_{\mathfrak{r}}(b) =  2\Lambda(y) + (b + y)\log \frac{b + y}{b} - y$ and by a direct computation we have 
\begin{equation}\label{S6LambdaDer}
\Lambda'(y)=\log\frac{\sqrt{y}+\sqrt{4\sqrt{t}\theta+y}}{\frac{\sqrt{t}-1}{\sqrt{t}+1}\sqrt{y}+\sqrt{4\sqrt{t}\theta+y}},
\end{equation}
and so
\begin{equation}\label{S6GDer}
 \frac{dG^{\theta, t}_{\mathfrak{r}}(b+y)}{dy} = 2 \cdot \log\frac{\sqrt{y}+\sqrt{4\sqrt{t}\theta+y}}{\frac{\sqrt{t}-1}{\sqrt{t}+1}\sqrt{y}+\sqrt{4\sqrt{t}\theta+y}} + \log \frac{b+y}{b}.
\end{equation}
Furthermore, we have 
\begin{equation}\label{S6GDer2}
\frac{d^2G^{\theta, t}_{\mathfrak{r}}(b+y)}{dy^2} = \frac{b + y + \theta - t\theta}{(y+b) \sqrt{y (y + 4 \sqrt{t} \theta)}} > 0, \mbox{ if }y > 0.
\end{equation}

Equation (\ref{S6GDer2}) implies that $\frac{dG^{\theta, t}_{\mathfrak{r}}(b+y)}{dy} $ is strictly increasing on $(0, \infty)$ and since from (\ref{S6GDer}) it equals $0$ at $y = 0$, we conclude that $\frac{dG^{\theta, t}_{\mathfrak{r}}(b+y)}{dy} > 0$ for $y \in (0, \infty)$. This means that $G^{\theta, t}_{\mathfrak{r}}$ is strictly increasing on $(b,\infty)$ and so
\begin{equation}\label{S6GJRel}
 J^{\theta, t}_{\mathfrak{r}}(b + \alpha) = \inf_{y \ge \alpha} (G^{\theta, t}_{\mathfrak{r}}(b+y)-G^{\theta, t}_{\mathfrak{r}}(b)) = G^{\theta, t}_{\mathfrak{r}}(b+\alpha)-G^{\theta, t}_{\mathfrak{r}}(b),
\end{equation}
which establishes (\ref{j-exact-up}) in view of (\ref{S6GJ}). We record for future use that the above computations imply that for all $\alpha > 0$
\begin{equation}\label{S6FU}
  2\Lambda(\alpha) + (b + \alpha)\log \frac{b + \alpha}{b} - \alpha > 0.
\end{equation}

{\bf \raggedleft Step 2.} As explained earlier in the section, $\P^{\theta, tN}_{\mathfrak{r}, N}$ satisfy Definition \ref{Assumptions} and so Theorem \ref{TMain}(b) is applicable. Here we also used $J^{\theta, t}_{\mathfrak{r}}(b + \alpha)  > 0$ for $\alpha > 0$ as follows from (\ref{j-exact-up}) and (\ref{S6FU}). Theorem \ref{TMain}(b) implies (\ref{S6LimJackHigh}). 

In the remainder we show (\ref{asymp-up}). Combining (\ref{S6GDer}), (\ref{S6GDer2}) and (\ref{S6GJRel}) we see that 
\begin{equation*}
\begin{split}
\lim_{\alpha \rightarrow 0+} \frac{J^{\theta, t}_{\mathfrak{r}}(b+\alpha)}{\alpha^{3/2}} = &\lim_{\alpha \rightarrow 0+} \frac{(dJ^{\theta, t}_{\mathfrak{r}}/d \alpha)(b + \alpha) }{(3/2)\alpha^{1/2}} = \lim_{\alpha \rightarrow 0+} \frac{4 \alpha^{1/2} }{3} \cdot  \frac{d^2J^{\theta, t}_{\mathfrak{r}}}{d \alpha^2}(b + \alpha) \\
=&\lim_{\alpha \rightarrow 0+} \frac{4 \alpha^{1/2} }{3} \cdot \frac{b + \alpha + \theta - t\theta}{(\alpha+b) \sqrt{\alpha (\alpha + 4 \sqrt{t} \theta)}} = \frac{4}{3} \cdot  \frac{b + \theta - t\theta}{b \sqrt{4 \sqrt{t} \theta }},
\end{split}
\end{equation*}
where we used L'H\^{o}pital's rule a few times. The last equation implies (\ref{asymp-up}).
\end{proof}

%
\subsubsection{Proof of Lemma \ref{JackMin}}\label{Section6.3.2} In this section we prove Lemma \ref{JackMin}. Recall that $k^{\theta}_V(\cdot,\cdot)$ and $\mathcal{A}_{s}^{\theta}$ were defined in \eqref{kv} and \eqref{ainf}, respectively. The following proposition is the key to obtaining the exact expression of $\pv$ as stated in Lemma \ref{JackMin}.
\begin{proposition} \label{ch-pv} \cite[Theorem 2.1 (d)]{ds97}.  Suppose that $V(x) : [0, \infty) \rightarrow \mathbb{R}$ is continuous and satisfies (\ref{VPot}) for some $\theta, \xi > 0$. Assume that there exists $\phi\in \mathcal{A}_{\infty}^{\theta}$ that satisfies
\begin{enumerate}[label=(\alph*), leftmargin=20pt]
\item $\ds \int_{0}^\infty k^{\theta}_V(x,y)\phi(x)\d x\ge \lambda$ if $\phi(y)=0$,
\item $\ds \int_{0}^\infty k^{\theta}_V(x,y)\phi(x)\d x\le \lambda$ if $\phi(y)=\theta^{-1}$,
\item $\ds \int_{0}^\infty k^{\theta}_V(x,y)\phi(x)\d x= \lambda$ if $0<\phi(y)<\theta^{-1}$,
\end{enumerate}
for some $\lambda \in \mathbb{R}$. Then $\phi=\phi_V^{\theta,\infty}$ from Lemma \ref{iv} .
\end{proposition}

\begin{proof}[Proof of Lemma \ref{JackMin}] For clarity we split the proof into several steps. In the first step we show that $\phi^{\theta, t}_{\mathfrak{r}} \in \mathcal{A}_{\infty}^{\theta}$. In the second and third steps we show that $\phi^{\theta, t}_{\mathfrak{r}}$ satisfies the three conditions of Proposition \ref{ch-pv} modulo a technical result, given in (\ref{S6KV}). The latter shows that $\phi^{\theta, t}_{\mathfrak{r}}$ is indeed equal to $\phi^{\theta, \infty}_V$ from Lemma \ref{iv}. In Step 4 we prove  (\ref{S6KV}) and in Step 5 we prove the second part of the lemma, namely equation (\ref{S6GJ}). For simplicity of the notation we will write $\phi$ in place of $\phi^{\theta, t}_{\mathfrak{r}}$ for the remainder of the proof. \\

{\bf \raggedleft Step 1.} In this step we prove that $\phi \in \mathcal{A}_{\infty}^{\theta}$. From the definition of $\phi$ it is clear that $\phi(x) \in [0, \theta^{-1}]$ for all $x \in [0, \infty)$. Thus we only need
\begin{equation}\label{PF1}
\int_{0}^{\infty} \phi(x) dx = 1.
\end{equation}
We will prove (\ref{PF1}) below when $t \in (0,\infty)$ and $t \neq 1$. The result for $t = 1$ can be obtained from the $t > 1$ case by sending $t \rightarrow 1+$ and invoking the bounded convergence theorem. 

Let $a =  \theta(\sqrt{t}-1)^2$ and $b = \theta(\sqrt{t}+1)^2$. Using the change of variables $x \rightarrow \frac{a + b z^2}{z^2 + 1}$ we have
$$\int_a^b \phi(x) dx = \frac{b-a}{\theta \pi}\int_0^{\infty} \frac{2z }{(z^2 + 1)^2}\operatorname{arccot}\left(\frac{(\sqrt{t}+1)z^2+(\sqrt{t}-1)}{2z}\right)\d z. $$
Performing integration by parts with $u = \operatorname{arccot}\left(\frac{(\sqrt{t}+1)z^2+(\sqrt{t}-1)}{2z}\right)$ and $v = -\frac{1}{1+z^2}$ we get
\begin{equation*}
\begin{split}
&\int_a^b \phi(x) dx = \frac{b-a}{\theta  }\cdot  {\bf 1}\{ t < 1\} + \frac{b-a}{\theta \pi} \int_0^{\infty} \frac{-2 z^2 (\sqrt{t}+1) +2(\sqrt{t} -1)}{(1+z^2)^2 [ (\sqrt{t} - 1)^2 + z^2 (\sqrt{t} +1)^2]} dz = \frac{b-a}{\theta  }\cdot  {\bf 1}\{ t < 1\}  \\
& + \frac{b-a}{\theta \pi}\cdot  \left[ \frac{(t-1) (\sqrt{t}+1) \operatorname{arctan} \left(\frac{(\sqrt{t}+1)z}{|\sqrt{t}-1|} \right)}{4\sqrt{t} |\sqrt{t}-1|} - \frac{(t + 2\sqrt{t} -1)(z^2 +1) \operatorname{arctan}(z) + 2\sqrt{t} z}{4 \sqrt{t} (z^2 + 1)} \right]_0^{\infty}\\
& =  \frac{b-a}{\theta }\cdot {\bf 1}\{ t < 1\} +  \frac{b-a}{4\theta \sqrt{t}  } \cdot {\bf 1}\{t > 1\} -  \frac{b-a}{4\theta \sqrt{t}} \cdot (2 \sqrt{t} +t) \cdot {\bf 1}\{ t < 1\} . 
\end{split}
\end{equation*}
If $t > 1$ we conclude that 
$$\int_{0}^{\infty} \phi(x) dx = \int_a^b \phi(x) dx = \frac{b-a}{4\theta \sqrt{t}  } =1, $$
where we used that $b - a = 4 \theta \sqrt{t}$ and also if $t < 1$ we get
$$\int_{0}^{\infty} \phi(x) dx =\frac{a \theta^{-1}}{\pi} + \int_a^b \phi(x) dx =a \theta^{-1}  +   \frac{b-a}{\theta }  -  \frac{b-a}{4\theta \sqrt{t}} \cdot (2 \sqrt{t} +t) =1, $$
where we used $a =  \theta(\sqrt{t}-1)^2$, $b = \theta(\sqrt{t}+1)^2$. The last two equations give (\ref{PF1}) and thus $\phi \in \mathcal{A}_{\infty}^{\theta}$.\\

{\bf \raggedleft Step 2.} In this step we specify our choice of $\lambda$, for which we will verify Proposition \ref{ch-pv} in the next step, and find expressions for $\int_{\R} k_V^{\theta}(x,y)\phi(x)\d x - \lambda $, where we recall from (\ref{kv}) that  
\begin{equation*}
k^{\theta}_V(x,y) = - \theta \cdot \log|x-y|+ \frac{1}{2} V(x) + \frac{1}{2}V(y).
\end{equation*}
We claim that for all $y \in [0,\infty)$ we have
\begin{equation}\label{S6KV}
\begin{split}
\int_{\R} k_V^{\theta}(x,y)\phi(x)\d x = & -\theta\mathcal{I}_y+\frac{1}{2}\int_{\R} V(x)\phi(x)\d x+\frac{1}{2} V(y)  - {\bf 1}\{t<1\}  \int_0^{a}\log|y-x|\d x ,
\end{split}
\end{equation}
where $\mathcal{I}_y =\frac{b-a }{\theta \pi} \cdot( \mathcal{I}_y^1 - \mathcal{I}_y^2 + \mathcal{I}_y^3)$ and if $ y \in [a,b]$ we have
\begin{equation}\label{S6intexp}
\begin{aligned}
\mathcal{I}_y^{1} & = \mathcal{J}_{1,1;2}+\left[\frac{\theta(t-1)-b+a+(y-a)\log|y-a|}{b-a}\right]\mathcal{J}_{1,1;1}\\ & \hspace{3cm}-\frac{\theta(t-1)}{b-a}(a+(y-a)\log|y-a|)\mathcal{J}_{\sqrt{a},\sqrt{b};1}, \\
\mathcal{I}_y^{2} & = \mathcal{I}^{-}_{\sqrt{|y-a|},\sqrt{|y-b|},1,1;2} +\frac{y+\theta(t-1)-b}{b-a}\mathcal{I}^{-}_{\sqrt{|y-a|},\sqrt{|y-b|},1,1;1} \\ & \hspace{5cm} -	\frac{\theta(t-1)y}{b-a}\mathcal{I}^{-}_{\sqrt{|y-a|},\sqrt{|y-b|},\sqrt{a},\sqrt{b};1}, \\
\mathcal{I}_y^{3} & = \mathcal{I}^{+}_{1,1,1,1;2}+\frac{y+\theta(t-1)-b}{b-a}\mathcal{I}^{+}_{1,1,1,1;1}-	\frac{\theta(t-1)y}{b-a}\mathcal{I}^{+}_{1,1,\sqrt{a},\sqrt{b};1},
\end{aligned}
\end{equation}
while if $y \not \in [a,b]$ we have the same formulas for $\mathcal{I}_y^1$ and $\mathcal{I}_y^3$, and for $\mathcal{I}_y^2$ the $\mathcal{I}^{-}$ expressions need to be replaced with $\mathcal{I}^{+}$. Here $\mathcal{I}^{\pm}_{a,b,c,d;n}$ and $\mathcal{J}_{c,d;n}$ are as in Lemma \ref{logints}. We will prove (\ref{S6KV}) in Step 4. Here we assume its validity and proceed to compute $\int_{\R} k_V^{\theta}(x,y)\phi(x)\d x - \lambda$, where
\begin{equation}\label{S6Lambda}
\lambda := \frac{1}{2} (\theta - t \theta - \theta\log(t \theta^2))+\frac12\int_{\R} V(x)\phi(x)\d x + \frac{A}{2},
\end{equation}
and we recall that $V(x) = A + x\log x - \log(e t \theta) x$. We will require expressions for $\int_{\R} k_V^{\theta}(x,y)\phi(x)\d x - \lambda$ when $y \in (0,a)$, $y \in [a,b]$ and $y \in (b,\infty)$. These expressions are given in equations (\ref{K1B}), (\ref{K1V}), (\ref{K1S}) and (\ref{K2S}) below and will be used in the next step to show that $\phi$ satisfies the conditions in Proposition \ref{ch-pv} with $\lambda$ as in (\ref{S6Lambda}).\\

From (\ref{S6intexp}) and Lemma \ref{logints} we have that  
\begin{equation}\label{Ieval1}
\begin{split}
&\mathcal{I}_y^1 = \frac{\pi}{4} + \left[ \frac{\theta (t-1) - b + a + (y-a) \log |y-a|}{b-a} \right] \cdot \frac{\pi}{2} - \frac{\theta \pi (t-1)(a + (y-a) \log |y-a|) }{2\sqrt{ab}(b-a)} ,\\
&\mathcal{I}_y^3 =  \frac{\pi \cdot \log 2}{2} + \frac{\pi \cdot\log 2 \cdot  (y + \theta (t-1) - b) }{b-a} - \frac{\pi \theta (t-1) y \log(1 + \sqrt{a/b})}{(b-a) \sqrt{ab} }  - \frac{\pi}{4},
\end{split}
\end{equation}
if in addition $y \in [a,b]$ we have 
\begin{equation}\label{Ieval2}
\begin{split}
&\mathcal{I}^2_y  = \frac{\pi}{4} \log(b-a) \hspace{-0.25mm} - \hspace{-0.25mm} \frac{\pi(b-y)}{2(b-a)} \hspace{-0.25mm}+ \hspace{-0.25mm}  \frac{\pi(y + \theta (t-1) - b)}{2(b-a)} \log (b-a) - \frac{\pi \theta (t-1) y\log (y (1 - a/b))}{2 (b-a) \sqrt{ab}}.
\end{split}
\end{equation}
and if $y \not \in [a,b]$ we instead have (recall that $\mathcal{I}^-$ get replaced with $\mathcal{I}^+$ in this case)
\begin{equation}\label{Ieval3}
\begin{split}
&\mathcal{I}^2_y =  \frac{\pi  \log(\sqrt{|y-a|} + \sqrt{|y-b|})}{2} - \frac{\pi\sqrt{|y-b|}}{2(\sqrt{|y-a|} + \sqrt{|y-b|})} \\
&+ \frac{\pi (y + \theta(t-1) - b) \log (\sqrt{|y-a|} + \sqrt{|y-b|})}{b-a} - \frac{\pi \theta (t-1) y \log  (\sqrt{|y-a|} + \sqrt{|y-b|a/b})}{(b-a) \sqrt{ab}} .
\end{split}
\end{equation}

Using (\ref{S2Tech22}) (when $t < 1$), (\ref{S6KV}), (\ref{S6Lambda}), (\ref{Ieval1}) and (\ref{Ieval2}) we get for $y \in [a,b]$ and $t > 0$
\begin{equation}\label{K1B}
\begin{split}
&\int_{\R} k_V^{\theta}(x,y)\phi(x)\d x - \lambda= -{\bf 1}\{ t < 1 \} \cdot \int_0^a \log |y - x| dx  - \frac{b-a }{\pi }(\mathcal{I}^1_y - \mathcal{I}_y^2 + \mathcal{I}_y^3  )  \\
&+\frac{1}{2}(y\log y - \log(e t \theta) y)-\frac{1}{2} (\theta - t \theta - \theta\log(t \theta^2)) = 0.
\end{split}
\end{equation}

Using (\ref{S2Tech22}) (when $t < 1$), (\ref{S6KV}), (\ref{S6Lambda}), (\ref{Ieval1}) and (\ref{Ieval3}) we get for $y \in (b,\infty)$ and $t >0$
\begin{equation}\label{K1V}
\begin{split}
&\int_{\R} k_V^{\theta}(x,y)\phi(x)\d x - \lambda = -{\bf 1}\{ t < 1 \} \cdot \int_0^a \log |y - x| dx +\frac{1}{2}(y\log y - \log(e t \theta) y)  \\
&- \frac{b-a }{\pi }(\mathcal{I}^1_y - \mathcal{I}_y^2 + \mathcal{I}_y^3  )   -\frac{1}{2} (\theta - t \theta - \theta\log(t \theta^2))  = \Lambda(y- b) + \frac{y}{2} \log \frac{y}{b} - \frac{y-b}{2}  ,\\
\end{split}
\end{equation}
where $\Lambda(\alpha)$ is as in (\ref{S6LambdaF}).

Using  (\ref{S6KV}), (\ref{S6Lambda}), (\ref{Ieval1}) and (\ref{Ieval3}) we get for $y \in (0,a)$ and $t \geq 1$
\begin{equation}\label{K1S}
\begin{split}
&\int_{\R} k_V^{\theta}(x,y)\phi(x)\d x - \lambda=  - \frac{b-a }{\pi }(\mathcal{I}^1_y - \mathcal{I}_y^2 + \mathcal{I}_y^3  )  \\
&+\frac{1}{2}(y\log y - \log(e t \theta) y)-\frac{1}{2} (\theta - t \theta - \theta\log(t \theta^2)) = \Delta_+(y),
\end{split}
\end{equation}
where 
\begin{equation}\label{S6Delta1}
\begin{split}
\Delta_+(y) = & - \frac{2\theta \sqrt{t} \sqrt{b-y}}{\sqrt{a-y} + \sqrt{b-y}} + 2 \theta\log \frac{\sqrt{b-a}}{\sqrt{b-y}+ \sqrt{a-y}} + y \log (\sqrt{b-y} + \sqrt{a-y}) \\
& + \frac{y}{2} \log (y/b) - y \log \left(\sqrt{a-y} + \sqrt{b-y}\sqrt{a/b} \right) +  \frac{b-y}{2}.
\end{split}
\end{equation}
Using  (\ref{S2Tech22}), (\ref{S6KV}), (\ref{S6Lambda}), (\ref{Ieval1}) and (\ref{Ieval3}) we get for $y \in (0,a)$ and $t \in(0,1)$
\begin{equation}\label{K2S}
\begin{split}
&\int_{\R} k_V^{\theta}(x,y)\phi(x)\d x - \lambda=-\int_0^a \log |y - x| dx  - \frac{b-a }{\pi }(\mathcal{I}^1_y - \mathcal{I}_y^2 + \mathcal{I}_y^3  )  \\
&+\frac{1}{2}(y\log y - \log(e t \theta) y)-\frac{1}{2} (\theta - t \theta - \theta\log(t \theta^2)) = \Delta_-(y),
\end{split}
\end{equation}
where 
\begin{equation}\label{S6Delta2}
\begin{split}
\Delta_-(y) = & - \frac{2\theta \sqrt{t} \sqrt{b-y}}{\sqrt{a-y} + \sqrt{b-y}} + 2 \theta\log \frac{\sqrt{b-a}}{\sqrt{b-y}+ \sqrt{a-y}} \\
&+ \log \left( \frac{(\sqrt{a} \sqrt{b-y} + \sqrt{b} \sqrt{a-y})(\sqrt{b-y} + \sqrt{a - y})}{\sqrt{y} (b-a)}\right) +  \frac{b-y}{2}.
\end{split}
\end{equation}

{\bf \raggedleft Step 3.} In this step we show that $\phi$ satisfies the three conditions of Proposition \ref{ch-pv}. Notice that $(a,b)$ is precisely the set of points where $\phi(x) \in (0, \theta^{-1})$ and $\phi(x) = 0$ for $x \geq b$. In addition, $\phi(x) = 0$ on $(0,a)$ if $t > 1$ and $\phi(x) = \theta^{-1}$ if $t < 1$. In view of this we see that (\ref{K1B}) implies condition (c) in Proposition \ref{ch-pv}.

We next show that condition (b) of Proposition \ref{ch-pv} is satisfied. What we need to show is that if $t < 1$, $y \in (0, a)$ we have 
\begin{equation}\label{S6Pb}
\int_{\R} k_V^{\theta}(x,y)\phi(x)\d x - \lambda \leq 0.
\end{equation}
If $\Delta_-(y)$ is as in (\ref{S6Delta2}) we have that
$$\Delta_-'(y)= - \frac{1}{2} \log (t\theta y) + \log \left(\frac{-y + \theta -t\theta + \sqrt{y^2 - 2 (1+t)y\theta +(t-1)^2 \theta^2 }}{2} \right),$$
while 
$$\Delta_-''(y)= \frac{-y + (t-1) \theta}{2y \sqrt{y^2 - 2(1+t)\theta y + (t-1)^2\theta^2}} < 0,$$
where the latter inequality holds for all $y \in (0,a)$. In particular, we see that $\Delta'_-$ is decreasing on $(0,a)$ and $\Delta'_-(a) = 0$ so that $\Delta_-'(y) > 0$ for $y \in (0,a)$. Since $\Delta_-(a) = 0$, we conclude that $\Delta_-(y) < 0$ for $y \in (0,a)$ and this implies (\ref{S6Pb}) in view of (\ref{K2S}).

We finally show that condition (a) of Proposition \ref{ch-pv} is satisfied. What we need to show is that if $t \geq 1$, $y \in (0, a)$ or if $t > 0$, $y \in (b,\infty)$  we have 
\begin{equation}\label{S6Pa}
\int_{\R} k_V^{\theta}(x,y)\phi(x)\d x - \lambda \geq 0.
\end{equation}
Equation (\ref{S6Pa}) holds when $y > b$ from (\ref{K1V}) and (\ref{S6FU}). In the remainder of this step we assume that $t > 1$, $y \in (0,a)$ and prove (\ref{S6Pa}). (Notice that $a= 0$  if $t = 1$ and there is nothing to prove when $t = 1$ -- that is why we assumed $t > 1$). 

If $\Delta_+(y)$ is as in (\ref{S6Delta1}) we have that
\begin{align*}
\Delta_+'(y)=  \log \frac{(z+ 1)\sqrt{y}}{(z-1) \sqrt{b}}, \mbox{ where } z = \sqrt{t} + (1+ \sqrt{t}) \cdot \frac{\sqrt{a-y}}{ \sqrt{b-y}}.
\end{align*}
We observe that
\begin{equation*}
\begin{split}
\frac{(z+ 1)\sqrt{y}}{(z-1) \sqrt{b}} < 1 \iff [  \sqrt{b-y} +  \sqrt{a-y}] \sqrt{y} < \sqrt{a} \sqrt{b-y} + \sqrt{b} \sqrt{a-y},
\end{split}
\end{equation*}
which clearly holds as $y < a < b$. This implies that $\Delta_+'(y) < 0$ for $y \in (0,a)$ and since $\Delta_+(a) = 0$, we conclude that $\Delta_+(y) > 0$ for $y \in (0,a)$. The latter and (\ref{K1S}) imply (\ref{S6Pa}), and so all conditions of Proposition \ref{ch-pv} are satisfied, proving the first part of the lemma modulo (\ref{S6KV}).\\

{\bf \raggedleft Step 4.} In this step we prove (\ref{S6KV}). By the definition of $\phi$ and $k_V^{\theta}$ in (\ref{kv}) it suffices to show
\begin{equation*}
\int_a^b \log |x-y| \phi(x) dx = \frac{b-a}{\theta \pi} \cdot ( \mathcal{I}_y^1 - \mathcal{I}_y^2 + \mathcal{I}_y^3).
\end{equation*}
Using the definition of $\phi$ and the change of variables $x \rightarrow \frac{a+bz^2}{z^2+1}$ we see that it suffices to prove
\begin{equation}\label{S6S4E1}
 \mathcal{I}_y^1 - \mathcal{I}_y^2 + \mathcal{I}_y^3=\int_{0}^{\infty} \frac{2z}{(z^2+1)^2}\log\left|\frac{(y-a)+(y-b)z^2}{z^2+1}\right|\operatorname{arccot}\left(\frac{(\sqrt{t}+1)z^2+(\sqrt{t}-1)}{2z}\right)\d z.
\end{equation}
In the remainder of this step we prove (\ref{S6S4E1}).

We wish to apply integration by parts to the right side of (\ref{S6S4E1}). Towards this end, we define a few functions for simplicity. Set $$h(z)=\frac1{z^2+1}-\frac{\theta(t-1)}{a+bz^2}, \quad f_{1,y}(z)=\frac{1}{z^2+1}-1+\frac{y-a}{b-a}\log|y-a|, \quad f_{2,y}(z)=\frac{1}{z^2+1}-\frac{b-y}{b-a},$$ and
\begin{align*}
f_{y}(z)=f_{1,y}(z)-f_{2,y}(z)\log\left|\frac{(y-a)+(y-b)z^2}{z^2+1}\right|.
\end{align*}
Observe that $f_y(0)=0$, 
$$f_y'(z)=\frac{2z}{(z^2+1)^2}\log\left|\frac{(y-a)+(y-b)z^2}{z^2+1}\right|, \mbox{ and } \frac{\d}{\d z} \operatorname{arccot}\left(\frac{(\sqrt{t}+1)z^2+(\sqrt{t}-1)}{2z}\right) = -h(z).$$

Using integration by parts with $u = \operatorname{arccot}\left(\frac{(\sqrt{t}+1)z^2+(\sqrt{t}-1)}{2z}\right)$ and $v = f_y(z)$ we get 
\begin{equation}\label{S6asd}
\begin{split}
&\int_{0}^{\infty} \frac{2z}{(z^2+1)^2}\log\left|\frac{(y-a)+(y-b)z^2}{z^2+1}\right|\operatorname{arccot}\left(\frac{(\sqrt{t}+1)z^2+(\sqrt{t}-1)}{2z}\right)\d z \\
&= \int_0^{\infty} f_y(z) h(z) dz = \mathcal{J}_y^1 - \mathcal{J}_y^2 + \mathcal{J}_y^3, \mbox{ where } \mathcal{J}_y^1= \int_0^{\infty} f_{1,y}(z) h(z) dz, \\
&\mathcal{J}_y^2 =  \int_0^\infty f_{2,y}(z) h(z) \log|(y-a) + (y-b)z^2|dz, \mbox{ and }\mathcal{J}_y^3 =  \int_0^\infty f_{2,y}(z) h(z) \log|z^2 + 1|dz.
\end{split}
\end{equation}
Using Lemma \ref{logints} and the identities 
\begin{align*}
f_{1,y}(z)h(z) & =\frac{1}{(z^2+1)^2}+\left[\frac{\theta(t-1)-b+a+(y-a)\log|y-a|}{b-a}\right]\frac{1}{z^2+1}\\ & \hspace{3cm}-\frac{\theta(t-1)}{b-a}(a+(y-a)\log|y-a|)\frac1{a+bz^2}, \\
f_{2,y}(z)h(z) & =\frac{1}{(z^2+1)^2}+\frac{y+\theta(t-1)-b}{b-a}\frac{1}{z^2+1}-\frac{\theta(t-1)y}{b-a}\frac{1}{a+bz^2},
\end{align*}
we see that $\mathcal{J}_y^i = \mathcal{I}_y^i$ for $i = 1,2,3$ and so (\ref{S6asd}) implies (\ref{S6S4E1}), which concludes the proof of (\ref{S6KV}).\\

{\bf \raggedleft Step 5.} In this final step we prove (\ref{S6GJ}). Recall from the statement of the lemma that
$$G^{\theta, t}_{\mathfrak{r}}(x) = -2\theta  \int_{0}^{\infty} \log |x - t|\phi(t)dt + V(x).$$
The latter equation, the definition of $k_V^{\theta}$ in (\ref{kv}), the fact that $\phi$ is a density as shown in Step 1, and the definition of $\lambda$ from (\ref{S6Lambda}) together imply that for $y \geq 0$ we have
$$G^{\theta, t}_{\mathfrak{r}}(y) = 2 \left(\int_{\R} k_V^{\theta}(x,y)\phi(x)\d x - \lambda  \right) + A + (\theta - t\theta - \theta \log (t \theta^2)).$$
The latter implies that for some constant $C$ (not depending on $y$) we have
$$G^{\theta, t}_{\mathfrak{r}}(y) - G^{\theta, t}_{\mathfrak{r}}(b) = 2 \left(\int_{\R} k_V^{\theta}(x,y)\phi(x)\d x - \lambda  \right) +C,$$
and since both sides vanish at $y = b$ (see (\ref{K1B})) we conclude that $C = 0$. The last equation and (\ref{K1V}) imply (\ref{S6GJ}).
\end{proof}

%
\section{Appendix} \label{Section7} In Sections \ref{Section7.1}-\ref{Section7.3} we prove various lemmas that were used throughout the paper. For the reader's convenience we recall the statements of these lemmas and indicate where they appeared in the text. In Section \ref{Section7.4} we discuss the results in \cite{fe} and \cite{jo} focusing on the errors in those papers and how they can be fixed.
	
%
\subsection{Technical estimates} \label{Section7.1}

\begin{lemma}\label{S7InterApprox}[Lemma \ref{InterApprox}] Fix $\theta > 0$. Then for any $x \geq \theta$ we have
\begin{equation}\label{S7Sandwich}
\frac{\Gamma(x + 1)\Gamma(x+ \theta)}{\Gamma(x)\Gamma(x +1-\theta)} = x^{2\theta} \cdot \exp(O(x^{-1})),
\end{equation}
where the constant in the big $O$ notation can be taken to be $(1 + \theta)^3$. 
\end{lemma}
\begin{proof}
We will essentially rely only on the functional equation $\Gamma(z+1) = z \Gamma(z)$ and  \cite[Equation (2.8)]{Qi}, which says that we have for all $y > 0$ and $s \in [0,1]$ 
\begin{equation}\label{RatGammaBound}
y^{1-s} \leq \frac{\Gamma(y+1)}{\Gamma(y+s)} \leq (y + s)^{1-s}.
\end{equation}
It will be convenient for us to treat the cases $\theta \in (0,1)$ and $\theta > 1$ separately. Observe that   
$$ \frac{\Gamma(x + 1)\Gamma(x+ \theta)}{\Gamma(x)\Gamma(x +1-\theta)} = x^{2\theta} \mbox{ if $\theta = 1$} ,$$
and there is nothing to prove in this case.
		
If $\theta \in (0,1)$ then we have
$$\frac{\Gamma(x+ \theta)}{\Gamma(x)} = \frac{\Gamma(x + \theta + 1)}{\Gamma(x+1)} \cdot \frac{x}{x+\theta},$$
and in view of (\ref{RatGammaBound}) with $y = x + \theta$, $s = 1 - \theta$ we conclude that 
\begin{equation*}
\frac{x}{x+\theta} \cdot (x+1)^{\theta}\geq \frac{\Gamma(x+\theta)}{\Gamma(x)} \geq \frac{x}{x+\theta} \cdot (x+\theta)^{\theta}
\end{equation*}
On the other hand, by (\ref{RatGammaBound}) with $y = x$, $s = 1- \theta$ we have
\begin{equation*}
(x+1 - \theta)^{\theta}\geq \frac{\Gamma(x+1)}{\Gamma(x + 1 - \theta)} \geq x^{\theta}.
\end{equation*}
Combining the last two estimates we see that if $\theta \in (0,1)$ then
\begin{equation}\label{theta01}
x^{2\theta} \cdot \left(1 + \frac{1}{x} \right)^{2\theta} = (x+1)^{2\theta} \geq \frac{\Gamma(x + 1)\Gamma(x+ \theta)}{\Gamma(x)\Gamma(x +1-\theta)} \geq x^{2\theta} \cdot \left(1 + \frac{\theta}{x} \right)^{\theta - 1}
\end{equation}
Using that $e^a \geq 1 + a$ for any $a \geq 0$ we get
$$e^{2\theta/x} \geq \left(1 + \frac{1}{x} \right)^{2\theta} \mbox{ and also } \left(1 + \frac{\theta}{x} \right)^{\theta - 1} \geq \left(1 + \frac{\theta}{x} \right)^{ - 1}  \geq e^{-\theta/ x},$$
which together with (\ref{theta01}) implies (\ref{S7Sandwich}) when $\theta \in (0,1)$. \\
		
We next suppose that $\theta > 1$. Then we have
$$\frac{\Gamma(x+\theta)}{\Gamma(x)} = \frac{\Gamma(x+ \theta ) \cdot x}{\Gamma(x+1)}  = x\cdot \prod_{i = 1}^{\lfloor \theta \rfloor} (x+ \theta -i) \cdot \frac{\Gamma(x+ \theta -\lfloor \theta \rfloor)}{\Gamma(x+1)}$$
and in view of (\ref{RatGammaBound}) with $y = x$, $s = \theta - \lfloor \theta \rfloor$ we conclude that 
\begin{equation*}
(x+ \theta)^{\lfloor \theta \rfloor} \cdot x^{\theta - \lfloor \theta \rfloor} \geq\frac{\Gamma(x+\theta)}{\Gamma(x)}  \geq x^{\lfloor \theta \rfloor } \cdot (x + \theta - \lfloor \theta \rfloor)^{\theta - \lfloor \theta \rfloor}.
\end{equation*}
Similarly, we have
$$\frac{\Gamma(x+1)}{\Gamma(x + 1 - \theta)} = \prod_{i = 1}^{\lfloor \theta \rfloor} (x- i + 1) \cdot (x + 1 - \theta) \cdot \frac{\Gamma(x - \lfloor \theta \rfloor + 1)}{\Gamma(x + 2- \theta)},$$ 
and in view of (\ref{RatGammaBound}) with $y = x+1 - \theta$, $s = \theta - \lfloor \theta \rfloor$ we conclude that 
\begin{equation*}
x^{\lfloor \theta \rfloor}\cdot(x+ 1 - \theta)^{\theta - \lfloor \theta \rfloor }\geq \frac{\Gamma(x+1)}{\Gamma(x + 1 - \theta)} \geq  (x + 1 - \theta)^{\lfloor \theta \rfloor} \cdot (x - \lfloor \theta\rfloor + 1)^{\theta - \lfloor \theta \rfloor }
\end{equation*}
Combining the above estimates we see that if $\theta > 1$ then
\begin{equation}\label{thetaBig}
x^{2\theta} \cdot \left( 1 + \frac{\theta}{x} \right)^{\theta} \geq x^{\theta} (x+ \theta)^{\theta} \geq \frac{\Gamma(x + 1)\Gamma(x+ \theta)}{\Gamma(x)\Gamma(x +1-\theta)} \geq x^{\theta} \cdot (x - \theta + 1)^{\theta} = x^{2\theta} \cdot \left( 1 - \frac{\theta - 1}{x} \right)^{\theta}.
\end{equation}
Since $e^a \geq 1 + a$ for any $a \geq 0$ we get 
$$ e^{\theta^2/x} \geq \left( 1 + \frac{\theta}{x} \right)^{\theta}.$$
We claim that 
\begin{equation}\label{lastTheta}
\left( 1 - \frac{\theta - 1}{x} \right)^{\theta} \geq \exp \left(- \theta^3/x \right)  \mbox{ if $x \geq \theta$}.
\end{equation}
If true then combining the last two inequalities with (\ref{thetaBig}) implies (\ref{S7Sandwich}) when $\theta > 1$. Upon taking logarithms and setting $u = x^{-1}$ we see that (\ref{lastTheta}) is equivalent to showing that
$$f(u):= \log (1 + (1 - \theta) u) + \theta^2 u \geq 0 \mbox{ if $u \in [0, \theta^{-1}]$}.$$
We notice that 
$$f''(u) = - \frac{(1- \theta)^2}{(1 + (1- \theta)u)^2} < 0,$$
which implies that $f$ is concave on $[0, \theta^{-1}]$ and hence it achieves its minimum at either $0$ or $\theta^{-1}$ or both. Since
$$f(0) = 0 \mbox{ and }f(\theta^{-1}) = \log(\theta^{-1}) + \theta \geq 0 \iff e^{\theta} \geq \theta,$$
we conclude that $f(u) \geq 0 $ on $[0, \theta^{-1}]$ as desired. 
\end{proof}

The following lemma shows that the measure in (\ref{PDef}) is well-defined when $M = \infty$ if (\ref{WDecay}) holds.
\begin{lemma}\label{S7WellDef} Fix $\theta > 0$, $N \in \mathbb{N}$. Let $Q_\theta(x) = \frac{\Gamma(x + 1)\Gamma(x + \theta)}{\Gamma(x)\Gamma(x +1-\theta)}$, and $\mathbb{W}_N^{\theta, \infty}$ be as in (\ref{GenState}). For each $\vec{\ell} \in\mathbb{W}_N^{\theta, \infty}$ define 
$$W(\vec{\ell}) = \prod_{1 \leq i < j \leq N} Q_{\theta}(\ell_i-\ell_j)  \prod_{i = 1}^N w(\ell_i; N),$$
where $w(\cdot; N): [0, \infty) \rightarrow (0, \infty)$ is continuous and for some $T \geq 0$ we have 
\begin{equation}\label{S7L2}
\log w(x; N) \leq - [\theta \cdot N + 1] \cdot \log (1 + (x/N)^2) \mbox{ for $x \geq T$.}
\end{equation}
 Then $W(\vec{\ell}) > 0$ and $Z_N : = \sum_{ \vec{\ell} \in \mathbb{W}_N^{\theta, \infty}} W(\vec{\ell})  \in (0, \infty)$ so that (\ref{PDef}) is a well-defined measure.
\end{lemma}
\begin{proof} The positivity of $ W(\vec{\ell}) $ follows from the positivity of the gamma function on $(0, \infty)$ and the positivity of $w(\cdot; N)$. Thus we only need to prove that 
\begin{equation}\label{S7L2R1}
\sum_{ \vec{\ell} \in \mathbb{W}_N^{\theta, \infty}} \prod_{1 \leq i < j \leq N} Q_{\theta}(\ell_i-\ell_j) \cdot  \prod_{i = 1}^N w(\ell_i; N) < \infty .
\end{equation}
The continuity of $\log w(x; N) $ and (\ref{S7L2}) imply that we can find $A > 0$ such that 
$$\log w(x; N) \leq A - [\theta \cdot N + 1] \cdot \log (1 + (x/N)^2) \mbox{ for $x \geq 0$.}$$
Combining the latter with Lemma \ref{S7InterApprox} we conclude that 
\begin{equation*}
\begin{split}
&\sum_{ \vec{\ell} \in \mathbb{W}_N^{\theta, \infty}} \prod_{1 \leq i < j \leq N} Q_{\theta}(\ell_i-\ell_j) \cdot  \prod_{i = 1}^N w(\ell_i; N)   \\
& \leq e^{AN + (1+ \theta)^3 N(N-1)/(2\theta)} N^{\theta N (N-1)}  \sum_{ \vec{\ell} \in \mathbb{W}_N^{\theta, \infty}} \left(\frac{\ell_i}{N} - \frac{\ell_j}{N} \right)^{2\theta}  \cdot  \prod_{i = 1}^N \exp \left( - [\theta \cdot N + 1] \cdot \log (1 + (\ell_i/N)^2) \right) \\
& \leq  e^{AN + (1+ \theta)^3 N(N-1)/(2\theta)} N^{\theta N (N-1)} \sum_{ \vec{\ell} \in \mathbb{W}_N^{\theta, \infty}}    \prod_{i = 1}^N \exp \left( - \log (1 + (\ell_i/N)^2) \right) \\
&\leq  e^{AN + (1+ \theta)^3 N(N-1)/(2\theta)} N^{\theta N (N-1)} \left( \sum_{j = 0}^{\infty} \frac{1}{1 + (j/N)^2} \right)^N < \infty
\end{split}
\end{equation*}
where in going from the second to the third line we used that $|x- y| \leq \sqrt{1 + x^2} \sqrt{1 + y^2}$ for $x,y \in \mathbb{R}$ and in going from the third to the fourth line we used that $\ell_i$ is summed over $(N-i) \theta + \mathbb{Z}_{\geq 0}$. The last inequality implies (\ref{S7L2R1}) and hence the lemma.
\end{proof}

\begin{lemma}\label{S7tail-estimate}[Lemma \ref{tail-estimate}] Let $B, \theta,\xi> 0$ and $N\in\mathbb{N}$ be such that $N\theta \xi\ge 1$. For any $i \in \{1, \dots N\}$
\begin{equation} \label{S7tail}
\begin{aligned}
\sum_{\substack{\ell \in \mathbb{Z}+(N-i)\theta \\ \ell \ge (B+\theta+1)N}} \frac1{\left(\ell^2/ N^2 +1\right)^{N\theta\xi}} 
\le \frac{N\pi/2}{(B^2+1)^{N\theta\xi-1}}.
\end{aligned}
\end{equation}
\end{lemma}
\begin{proof} Using the fact that $\frac{1}{(x^2+1)^{N\theta \xi}}$ is a decreasing function on $[0, \infty)$, we see that
\begin{align}\label{tail1}
\sum_{\substack{\ell \in \mathbb{Z}+(N-i)\theta \\ \ell \ge (B+\theta+1)N}} \hspace{-1mm} \frac1{\left(\frac{\ell^2}{N^2}+1\right)^{N\theta\xi}}  \le \hspace{-2mm}\sum_{\substack{\ell \in \mathbb{Z}+(N-i)\theta \\ \ell \ge (B+\theta+1)N}} \hspace{-1mm} \frac1{\left(\frac{(\ell-(N-i)\theta)^2}{N^2}+1\right)^{N\theta\xi}} \le \hspace{-1mm} \sum_{\ell=(B+1)N}^{\infty} \frac1{\left(\frac{\ell^2}{N^2}+1\right)^{N\theta\xi}}.
\end{align}
Splitting the last sum over blocks of size $N$ we get
\begin{align*}
\mbox{r.h.s.~of \eqref{tail1}} \le \sum_{\ell=B+1}^\infty \frac{N}{\left(\ell^2+1\right)^{N\theta\xi}}  & \le N\int_{B}^{\infty} \frac{\d x}{(x^2+1)^{N\theta\xi}} \\ &
\le \frac{N}{(B^2+1)^{N\theta\xi-1}}\int_{B}^{\infty} \frac{\d x}{x^2+1} \leq \frac{N\pi/2}{(B^2+1)^{N\theta\xi-1}}.
\end{align*}
\end{proof}

\begin{lemma}\label{S7tech}[Lemma \ref{tech}] Let $g : \mathbb{R} \rightarrow \mathbb{R}$ be a compactly supported Lipschitz function. Then $\norm{g}_{1/2} < \infty$.
\end{lemma}
\begin{proof}
Define 
$$[g]_{H^{1/2}(\mathbb{R})} := \int_{\mathbb{R}} \int_{\mathbb{R}} \frac{| g(x) -  g(y) |^2}{|x-y|^2}dxdy,$$
and note that the latter is finite as $g$ is Lipschitz and compactly supported. On the other hand, as can be deduced from the proof of \cite[Proposition 3.4]{NPV}, we have
\begin{equation}\label{S7Sob}
[g]_{H^{1/2}(\mathbb{R})}  = 2 C(1,1/2)^{-1} \cdot \|g \|_{1/2}^2, \mbox{ where }C(1,1/2) = \int_{\mathbb{R}} \frac{1 - \cos (x)}{x^2}dx \in (0, 2).
\end{equation}
Combining the last two equations gives the lemma.
\end{proof}

\begin{lemma}\label{S7LemmaS52}[Lemma \ref{LemmaS52}] Let $\mathcal{A}_{c,M}$ be the set of all $C^1$ functions $f: \mathbb{R} \to \mathbb{R}$, supported on $[-M,M]$ with $|f'(x)| \le c$. Then
\begin{align*}
\sup_{f\in \mathcal{A}_{c,M}} \norm{f}_{1/2} \le 2 c M.
\end{align*}
\end{lemma}
\begin{proof}
We consider $g \in \mathcal{A}_{c,M}$. Using (\ref{S7Sob}) we see that 
$$\|g \|_{1/2}^2 = 2^{-1} C(1,1/2)  \int_{\mathbb{R}} \int_{\mathbb{R}} \frac{| g(x) -  g(y) |^2}{|x-y|^2}dxdy \leq 2^{-1} C(1,1/2) c^2 \cdot 4 M^2 \leq 4 c^2 M^2.$$
\end{proof}

In the remainder of this section we prove Lemmas \ref{LemmaTech2}, \ref{S1GJ} and \ref{S5ContGJ}, for which we require the following two auxiliary results.

\begin{lemma}\label{S7LemmaTech1.1}
Let $a>0$. There exists a constant $C >0$, depending $a$, such that 
\begin{equation}\label{S7Tech2.1}
 \sup_{v \in [0, a]}\int_{0}^{a}   |\log \left|w - v \right|| dw  \le C.
\end{equation}
\end{lemma}
\begin{proof} By a direct computation we have for any $v \in [0, a]$ that
\begin{align*}
&\int_{0}^{a}   |\log \left|w - v \right||dw = \int_{0}^{v}  |\log \left|r \right|| dr+\int_{0}^{a-v}   |\log \left|r \right|| dr \\
&\le 2\int_0^a \hspace{-1mm} |\log \left|r \right|| dr = \begin{cases} 2a[ 1 - \log (a)] & \mbox{ if $a \in (0,1]$}, \\ 4 +2a( \log a -1) & \mbox{ if $a \geq 1$},  \end{cases}
\end{align*}	which gives (\ref{S7Tech2.1}).
\end{proof}

\begin{lemma}\label{S7LemmaTech1.5} Fix $\theta > 0$ and $s\in [\theta,\infty)$. Recall $\mathcal{A}_s^{\theta}$ from \eqref{ainf} and fix $\phi \in \mathcal{A}_{s}^{\theta}$. For $x\in [0,\infty)$ define the function
\begin{align}
H(x)=-2\theta\int_{0}^s \log|x-t|\phi(t)\d t.
\end{align}
Then the function $H$ is continuous $\sup_{x\in [0,s]} |H(x)|=O(1)$ and moreover for any $N \geq 2$ one has
\begin{align}\label{eq:hsup}
	\sup_{x,y \geq 0, |x-y| \leq \theta/N} |H(x) - H(y)|  = O(N^{-1} \log N)
\end{align}
where the constants in the big $O$ notations depend only on $\theta$ and $s$. 
\end{lemma}
\begin{proof}
Since $\phi(t) \le \theta^{-1}$ as $\phi\in \mathcal{A}_s^{\theta}$ we have
$$\sup_{x\in [0,s]} 2\theta\left|\int_{0}^s\log|x-t|\phi(t)dt\right| \le \sup_{x\in [0,s]} 2\int_{0}^{s}\left|\log|x-t|\right|dt = O(1),$$
where the last equality used Lemma \ref{S7LemmaTech1.1}. To show both continuity and \eqref{eq:hsup}, it suffices to prove \eqref{eq:hsup} when $N\ge \max(2, \theta^{-1})$ and $y \geq x \geq 0$. 

We observe that for $x \in [0, s+1]$, $x + \theta/N \geq y \geq x$ and $I_x = [x - \theta/N, x + \theta/ N]$ we have
\begin{equation*}
\begin{aligned}
|H(x)-H(y)| & \le 2\theta\left|\int_ {I^c_x \cap [0, x]} \log\left|1+\frac{y-x}{x-t}\right|\phi(t)dt\right|+2\theta\left|\int_{ I_x^c \cap [x, s + 1] }\log\left|1+\frac{y-x}{x-t}\right|\phi(t)dt\right| \\ 
& \hspace{2cm} + 2\theta\left|\int_{I_x} \log\left|\frac{y-t}{x-t}\right|\phi(t)dt\right| \\ 
& \le 2\int_{\theta/N}^{s + 1} \log\left(1+\frac{\theta}{Nt}\right)dt  -2\int_{\theta/N}^{s+1} \log\left(1-\frac{\theta}{Nt} \right)dt \\ 
& \hspace{2cm} +2\int_{I_x} |\log|y-t||dt+2\int_{I_x} |\log|x-t||dt = O(N^{-1} \log N),		
\end{aligned}
\end{equation*}
where in going from the first to the second line we used that $\phi(x) \leq \theta^{-1}$. If $x + \theta/N \geq y \geq x \geq s+1$
\begin{equation*}
\begin{aligned}
|H(x)-H(y)| =  2\theta \int_{0}^{s} \log \left(1+\frac{y-x}{x-t}\right)\phi(t)dt \leq   2 \int_{0}^{s} \log\left(1+\frac{\theta}{N}\right)dt = O(N^{-1}).
\end{aligned}
\end{equation*}
\end{proof}

\begin{lemma}\label{S7LemmaTech2}[Lemma \ref{LemmaTech2}]
Suppose that $\lM, \theta, V$ satisfy the conditions in Assumptions \ref{as1} and \ref{as2}. Let $\phi_V^{\theta, \lM + \theta}$ be the equilibrium measure from \eqref{fs} and define for $x \in [0, \infty)$ the function
$$G^{\theta, \lM + \theta}_V(x) = -2\theta\int_{\R}\log|x-t|\phi_V^{\theta, \lM + \theta}(t)dt + V(x).$$
Then the function $G^{\theta, \lM + \theta}_V$ is continuous, $\sup_{ x \in [0, \lM + \theta]} |G^{\theta, \lM + \theta}_V(x)| = O(1)$, and moreover for any $N \geq 2$ one has
\begin{align}\label{eq:gvsup}
	\sup_{x,y \geq 0, |x-y| \leq \theta/N} |G^{\theta, \lM + \theta}_V(x) - G^{\theta, \lM + \theta}_V(y)|  = O(N^{-1} \log N)
\end{align}
where the constants in the big $O$ notations depend on $\theta, A_0, A_3, A_4$ from Assumptions \ref{as1} and \ref{as2}. 
\end{lemma}
\begin{proof} Note that $\sup_{ x \in [0, \lM + \theta]} |G^{\theta, \lM + \theta}_V(x)| = O(1)$ follows from Assumption \ref{as2} and Lemma \ref{S7LemmaTech1.5} with $s=A_0+\theta$ and $\phi= \phi_V^{\theta,\lM+\theta}$. From Lemma \ref{S7LemmaTech1.5} the continuity of $G_V^{\theta,\lM+\theta}$ is also immediate as $V$ is continuous by Assumption \ref{as2}. 

We next prove \eqref{eq:gvsup}. Note that \eqref{DerPot} and the fact that $V(x) = V(\lM + \theta)$ for $x \geq \lM + \theta$  imply that $V'(x)$ is integrable and so $V$ is absolutely continuous. Applying \cite[Chapter 3, Theorem 3.11]{SteinReal} and \eqref{DerPot} we have for $\theta/N + x \geq y \geq x$ and $x \in [0, \lM + \theta]$ that
\begin{align*}
	|V(x)-V(y)| & \le \int_x^{\min\{y,\lM+\theta\}} |V'(s)|ds  \\ & \le A_4\int_x^{\min\{y,\lM+\theta\}}[1+|\log|s||+|\log|s-\lM-\theta|] ds = O(N^{-1}\log N).
\end{align*}
where the last line follows as $\min\{y,\lM+\theta\}-x\le \theta/ N$. If $x \geq \lM + \theta$ and $y \geq x$ we have $V(x) = V(y)$, and so the last equation holds in this case as well. The last equation and \eqref{eq:hsup} give \eqref{eq:gvsup}.
\end{proof}

\begin{lemma}\label{S7GJ} [Lemma \ref{S1GJ}] Suppose that $V(x) : [0, \infty) \rightarrow \mathbb{R}$ is continuous and satisfies (\ref{VPot}) for some $\theta, \xi > 0$. Let $\phi^{\theta, \infty}_V$ be as in Lemma \ref{iv} and let $b_V^{\theta, \infty}$ be the rightmost point of its support. Then the function 
\begin{equation}\label{S7DefG}
G_V^{\theta, \infty}(x) = -2\theta  \int_{0}^{\infty} \log |x - t|\phi_V^{\theta, \infty}(t)dt + V(x)
\end{equation}
is well-defined and continuous on $[0, \infty)$. In addition, the function 
\begin{equation}\label{S7DefJ}
J^{\theta, \infty}_V(x) =\begin{cases} 0 &\mbox{ if $x < b^{\theta, \infty}_V$, } \\  \inf_{y \geq x} G^{\theta, \infty}_V(y)  - G^{\theta, \infty}_V(b_V^{\theta, \infty})   &\mbox{ if $x \geq b_V^{\theta, \infty}$}. \end{cases}
\end{equation}
is continuous on $[0,\infty)$.
\end{lemma}
\begin{proof} Note that the continuity of the logarithmic integral in \eqref{S7DefG} follows from Lemma \ref{S7LemmaTech1.5} by taking $s=b_V^{\theta,\infty}+\theta$ and $\phi=\phi_V^{\theta,b_V^{\theta,\infty}+\theta}=\phi_V^{\theta,\infty}$ (here we used Lemma \ref{MinEq}). As $V$ is continuous, we conclude that $G_V^{\theta,\infty}$ is continuous as well. 

We next show that $\inf_{y \geq x} G^{\theta, \infty}_V(y)$ is finite and continuous on $(0, \infty)$. Let $x_0 \in (0, \infty)$ be given, and let $\delta_1 > 0$ be sufficiently small so that $x_0 - \delta_1 > 0$. From (\ref{VPot}) we have for $y \geq b_V^{\theta, \infty} +1$ 
$$G_V^{\theta, \infty}(y) = -2\theta  \int_{0}^{\infty} \log |y - t|\phi_V^{\theta, \infty}(t)dt + V(y) \geq -2\theta \log (y) + (1 + \xi) \cdot \theta \cdot \log(1 + y^2),$$
so that $\lim_{y \rightarrow \infty} G_V^{\theta, \infty}(y) = \infty$. The latter and the continuity of $G^{\theta, \infty}_V(x)$ imply that there exists $A > x_0 + \delta_1$ such that for all $y \geq A$ we have
$$G_V^{\theta, \infty}(y) \geq \sup_{x \in [x_0 - \delta_1, x_0 + \delta_1]} G_V^{\theta, \infty}(x).$$
In particular, we see that for $x \in [x_0 - \delta_1, x_0 + \delta_1]$ and all $s \geq 0$ we have 
\begin{equation}\label{S7GJU1}
\inf_{y \geq x} G^{\theta, \infty}_V(y) = \inf_{y \in [x, A+ s]} G^{\theta, \infty}_V(y),
\end{equation}
which by the continuity of $G^{\theta, \infty}_V(x)$ and compactness of $[x, A]$ implies that $\inf_{y \geq x} G^{\theta, \infty}_V(y)$ is finite for all $x\in [x_0 - \delta_1, x_0 + \delta_1]$.

We next show that $\inf_{y \geq x} G^{\theta, \infty}_V(y)$ is continuous at $x_0$. Suppose that $\epsilon > 0$ is given. Since $ G^{\theta, \infty}_V(x)$ is continuous and $[x_0 - \delta_1, A + 1]$ is compact we know that there exists $\delta_2 \in (0, \min (1, \delta_1))$ such that 
$$\left|G^{\theta, \infty}_V(x) - G^{\theta, \infty}_V(y)\right| < \epsilon \mbox{ if } x,y \in [x_0 - \delta_1, A + 1] \mbox{ and } |x-y| < \delta_2.$$
Combining the latter equation with (\ref{S7GJU1}) we conclude that for any $x \in (x_0 - \delta_2, x_0 + \delta_2)$ we have 
$$\left| \inf_{y \geq x} G^{\theta, \infty}_V(y) - \inf_{y \geq x_0} G^{\theta, \infty}_V(y) \right| = \left| \inf_{y \in [x, A + 1 + x - x_0]} G^{\theta, \infty}_V(y) - \inf_{y \in [x_0,A + 1]} G^{\theta, \infty}_V(y) \right|  $$
$$= \left| \inf_{y \in [x_0, A + 1]} G^{\theta, \infty}_V(y + x- x_0) - \inf_{y \in [x_0,A + 1]} G^{\theta, \infty}_V(y) \right| \leq \epsilon.$$
The latter shows that $\inf_{y \geq x} G^{\theta, \infty}_V(y)$ is indeed continuous at $x_0$ and since $x_0 > 0$ was arbitrary we see that $\inf_{y \geq x} G^{\theta, \infty}_V(y)$ is continuous on $(0, \infty)$. \\

The above work shows that $J^{\theta, \infty}_V(x)$ is continuous on $( b_V^{\theta, \infty}, \infty)$ and clearly it is continuous on $[0, b_V^{\theta, \infty})$. What remains is to show that $J^{\theta, \infty}_V(x)$ is continuous at $b_V^{\theta, \infty}$, for which we need
\begin{align}\label{S7I1}
G^{\theta, \infty}_V(b_V^{\theta, \infty}) = \inf_{y \geq b_V^{\theta, \infty}} G^{\theta, \infty}_V(y).
\end{align}
By \cite[Theorem 2.1 (c)]{ds97} and the Lebesgue differentiation theorem \cite[Chapter 3, Theorem 1.3]{SteinReal} it follows that there exists a constant $\kappa$ such that 
\begin{align}\label{S7DSGI}
G_V^{\theta, \infty}(x)-\kappa \begin{cases} = 0 & \mbox{a.e. }x \in [0, \infty) \mbox{ such that }\phi_V^{\theta, \infty}(x) \in (0,\theta^{-1}), \\
 \leq 0 & \mbox{a.e. }x \in [0, \infty) \mbox{ such that }\phi_V^{\theta, \infty}(x)=\theta^{-1}, \\ 
\geq  0 & \mbox{a.e. }x \in [0, \infty) \mbox{ such that }\phi_V^{\theta, \infty}(x)=0.
\end{cases}
\end{align}
The continuity of $G_V^{\theta, \infty}(x)$ and the fact that $b_V^{\theta, \infty}$ is on the support of $\phi_V^{\theta, \infty}$ imply that $\kappa = G_V^{\theta, \infty}(b_V^{\theta, \infty})$. Using the latter, the third inequality in (\ref{S7DSGI}) and the continuity of $ G_V^{\theta, \infty}(x)$ gives (\ref{S7I1}). This proves that $J^{\theta, \infty}_V(x)$ is continuous on $[0, \infty)$ as desired.
\end{proof}

\begin{lemma}\label{S7ContGJ}[Lemma \ref{S5ContGJ}]
Fix $\lM, \theta > 0$ and a continuous function $V: [0, \lM + \theta] \rightarrow \mathbb{R}$. Let $\phi^{\theta, \lM + \theta}_V$ be as in Lemma \ref{iv} for this choice of $V$ and $b_V \in [\theta, \lM + \theta]$ be the rightmost point of its support. Recall the function $G^{\theta, \lM + \theta}_V$ on $[0, \lM + \theta]$ from (\ref{S2DefG}) (or see Lemma \ref{S7LemmaTech2} above)
\begin{equation}\label{S7DefG2}
G^{\theta, \lM + \theta}_V(x) = - 2\theta \int_{0}^{\lM + \theta} \log |x-t| \phi_V^{\theta, \lM + \theta}(t) dt + V(x),
\end{equation}
which is continuous by Lemma \ref{S7LemmaTech2}. We define the function
\begin{equation}\label{S7DefL2}
J^{\theta, \lM + \theta}(x) =\begin{cases} 0 &\mbox{ if $x < b_V$, } \\ \ds \inf_{y \in [x, \lM + \theta]} G^{\theta, \lM + \theta}_V(y)  - G^{\theta, \lM + \theta}_V(b_V)   &\mbox{ if $x \geq b_V$}. \end{cases}
\end{equation}
Then the function $J^{\theta, \lM + \theta}$ is continuous on $[0, \lM + \theta]$.
\end{lemma}
\begin{proof} Clearly, $J^{\theta, \lM + \theta}$ is continuous and equals $0$ if $b_V = \lM + \theta$ and so we may assume that $b_V < \lM + \theta$. 

 Let $\epsilon > 0$ be given. Since $G^{\theta, \lM + \theta}_V(x)$ is continuous on $[0, \lM + \theta]$ and the latter is compact we know that it is uniformly continuous and so we can find $\delta_2 > 0 $ such that 
$$\left|G^{\theta, \lM + \theta}_V(x) - G^{\theta, \lM + \theta}_V(y)\right| < \epsilon \mbox{ if } x,y \in [0, \lM + \theta] \mbox{ and } |x-y| < \delta_2.$$
The latter implies that if $x,y \in [b_V, \lM + \theta]$ and $|x-y| < \delta_2$ we have 
$$|J^{\theta, \lM + \theta}(x) - J^{\theta, \lM + \theta}(y)| = \left|\inf_{z \in [x, \lM + \theta]} G^{\theta, \lM + \theta}_V(z) - \inf_{z \in [y, \lM + \theta]} G^{\theta, \lM + \theta}_V(z) \right| \leq \epsilon,$$
which proves that $J^{\theta, \lM + \theta}(x)$ is continuous on $(b_V, \lM + \theta]$. Also, we clearly have that $J^{\theta, \lM + \theta}(x)$ is continuous on $[0, b_V)$. What remains is to show that $J^{\theta, \lM + \theta}(x)$ is continuous at $b_V$ and for that we need to show that 
\begin{align}\label{S7I2}
G^{\theta, \lM + \theta}_V(b_V) = \inf_{y \in [ b_V, \lM + \theta]} G^{\theta, \lM + \theta}_V(y).
\end{align}
By \cite[Theorem 2.1 (c)]{ds97} and the Lebesgue differentiation theorem \cite[Chapter 3, Theorem 1.3]{SteinReal} it follows that there exists a constant $\kappa$ such that 
\begin{align}\label{S7DSGI2}
G_V^{\theta, \lM + \theta}(x)-\kappa \begin{cases} = 0 & \mbox{a.e. }x \in [0, \lM + \theta] \mbox{ such that }\phi_V^{\theta, \lM + \theta}(x) \in (0,\theta^{-1}), \\
 \leq 0 & \mbox{a.e. }x \in [0, \lM + \theta] \mbox{ such that }\phi_V^{\theta, \lM + \theta}(x)=\theta^{-1}, \\ 
\geq  0 & \mbox{a.e. }x \in [0, \lM + \theta] \mbox{ such that }\phi_V^{\theta, \lM + \theta}(x)=0.
\end{cases}
\end{align}
The continuity of $G_V^{\theta, \lM + \theta}(x)$ and the fact that $b_V$ is on the support of $\phi_V^{\theta, \lM + \theta}$ imply that $\kappa = G_V^{\theta, \lM + \theta}(b_V)$. Using the latter, the third inequality in (\ref{S7DSGI2}) and the continuity of $ G_V^{\theta, \lM + \theta}(x)$ gives (\ref{S7I2}). This proves that $J^{\theta, \lM + \theta}_V(x)$ is continuous on $[0, \lM + \theta]$ as desired.
\end{proof}

%
\subsection{Integral identities} \label{Section7.2} In this section we prove Lemma \ref{logints}, recalled here as Lemma \ref{S7logints}.

\begin{lemma}\label{S7logints}[Lemma \ref{logints}]	For $a,b,c,d\ge 0$ with $cd>0$ and $a+b>0$, we consider the integrals
\begin{align*}
\mathcal{I}^{\pm}_{a,b,c,d;n}:=\int_0^{\infty} \frac{\log|a^2 \pm b^2z^2|}{(c^2+d^2z^2)^n}\d z, \quad \mathcal{J}_{c,d;n} := \int_0^{\infty} \frac{\d z}{(c^2+d^2z^2)^n}.
\end{align*}
We have the following exact expressions for the above integrals for particular values.
\begin{multicols}{2}
\begin{enumerate}
\item $\ds \mathcal{I}_{a,b,c,d;1}^{-} = \frac{\pi\log|a^2+\frac{b^2c^2}{d^2}|}{2cd},$
\item $\ds \mathcal{I}_{a,b,c,d;1}^{+}=\frac{\pi}{cd}\log(a+\tfrac{bc}d)$,
\item $\ds \mathcal{I}_{a,b,1,1;2}^{-}= \frac{\pi}{4}\log(a^2+b^2)-\frac{\pi b^2}{2(a^2+b^2)},$
\item $\ds \mathcal{I}_{a,b,1,1;2}^{+}=\frac{\pi}{2}\log(a+b)-\frac{b\pi}{2a+2b},$
\item $\ds \mathcal{J}_{c,d;1}=\frac{\pi}{2cd},$\vspace{2mm}
\item $\ds \mathcal{J}_{1,1;2}=\frac{\pi}{4}$.
\end{enumerate}
\end{multicols}
\end{lemma}
\begin{proof}  We will prove the lemma assuming $a, b > 0$ -- the cases $a = 0$ or $b =0$ can be deduced from the statements when $a, b > 0$ by a limit transition after applying the dominated convergence theorem. For clarity we split the proof into four steps.\\

{\bf \raggedleft Step 1.} In this step we prove $(5)$ and $(6)$. By direct computations we have
$$\mathcal{J}_{c,d;1} = \int_{0}^\infty \frac{dz}{c^2 + d^2 z^2}  = \frac{\operatorname{arctan}(zd/c)}{cd} \Big{\vert}_0^{\infty} = \frac{\pi}{2cd},$$
which establishes $(5)$ and 
$$\mathcal{J}_{1,1;2} = \int_{0}^\infty \frac{dz}{(z^2 +1)^2}  = \frac{\operatorname{arctan}(z)}{2} + \frac{z}{2z^2 + 2} \Big{\vert}_0^{\infty} = \frac{\pi}{4},$$
which establishes $(6)$. \\

{\bf \raggedleft Step 2.} In this step we prove $(1)$ and $(3)$. Substituting $u=\frac{dz}{c}$ we get that $\mathcal{I}_{a,b,c,d;1}^{-}=\frac1{cd} \cdot \mathcal{I}_{a,\frac{bc}{d},1,1;1}^{-}$. So we only need to prove $(1)$ and $(3)$ when $c=d=1$. 
	
We consider the closed semi-circle contour (oriented counterclockwise) $\mathcal{C}_{R}$ with radius $R$ and center at the origin that resides in the lower half plane. We also let $\mathcal{C}^{\epsilon}_{R}$ be $\mathcal{C}_{R}$, shifted down by $\epsilon > 0$. We define
$$\mathcal{K}_{R ;n, \epsilon}^-=\int_{\mathcal{C}^{\epsilon}_{R}}\frac{\log(a^2-b^2z^2)}{(z^2+1)^n}dz ,$$
where $\log(x)$ denotes the principal branch of the logarithm. Clearly, for large enough $R$ and small enough $\epsilon$ it encloses only one pole at $z=-i$ of order $n$. By the Residue theorem, \cite[Chapter 3, Theorem 2.1]{Stein}, we have
$$\mathcal{K}_{R;1, \epsilon}^-=2\pi i \cdot\left.\frac{\log(a^2-b^2z^2)}{z-i}\right|_{z=-i} =-\pi\log(a^2+b^2),$$
and 
$$\mathcal{K}_{R;2, \epsilon}^-=-2\pi i \cdot\left[\frac{2b^2z}{(z-i)^2(a^2-b^2z^2)}+\frac{2\log(a^2-b^2z^2)}{(z-i)^3}\right|_{z=-i} =\frac{b^2\pi}{a^2+b^2}-\frac{\pi}{2}\log(a^2+b^2).$$
Taking real parts on both sides and letting $\epsilon \rightarrow 0+$ we conclude by the dominated convergence theorem that for $R \geq 2$
$$-\int_{-R}^R \frac{\log|a^2 - b^2z^2|}{1 + z^2}\d z = -\pi\log(a^2+b^2) + O( \log R/R), \mbox{ and }$$
$$ - \int_{-R}^R \frac{\log|a^2 - b^2z^2|}{(1 + z^2)^2}\d z = \frac{b^2\pi}{a^2+b^2}-\frac{\pi}{2}\log(a^2+b^2) + O( \log R/R^3),$$ 
where the constants in the big $O$ notations depend on $a,b$ alone. We mention that the extra negative signs on the left sides come from the fact that $\mathcal{C}_{R}$ traverses $[-R, R]$ from $R$ down to $-R$. Letting $R \rightarrow \infty$ in the last equation we conclude that 
$$2 \cdot \mathcal{I}_{a,b,1,1;1}^{-} = \pi\log(a^2+b^2) \mbox{ and } 2\cdot  \mathcal{I}_{a,b,1,1;2}^{-} = - \frac{b^2\pi}{a^2+b^2} +\frac{\pi}{2}\log(a^2+ b^2),$$
which proves $(1)$ and $(3)$.\\

{\bf \raggedleft Step 3.} In this and the next step we prove $(2)$ and $(4)$. Substituting $u=\frac{dz}{c}$ we get that $\mathcal{I}_{a,b,c,d;1}^{+}=\frac1{cd} \cdot \mathcal{I}_{a,\frac{bc}{d},1,1;1}^{+}$. So we only need to prove $(2)$ and $(4)$ when $c=d=1$. In this step we prove $(2)$ and $(4)$ when $a > b$.

Let $\mathcal{D}^{\epsilon}_R$ be a keyhole semicircular contour of radius $R$ (oriented counterclockwise), which omits the point $ai/b$ (see the left part of Figure \ref{fig:dr}). The width of the corridor is $2 \epsilon$ and the curve traverses a half-circular arc around $ai/b$ of radius $\epsilon$.

\begin{figure}[ht]
\begin{center}
  \includegraphics[scale = 0.3]{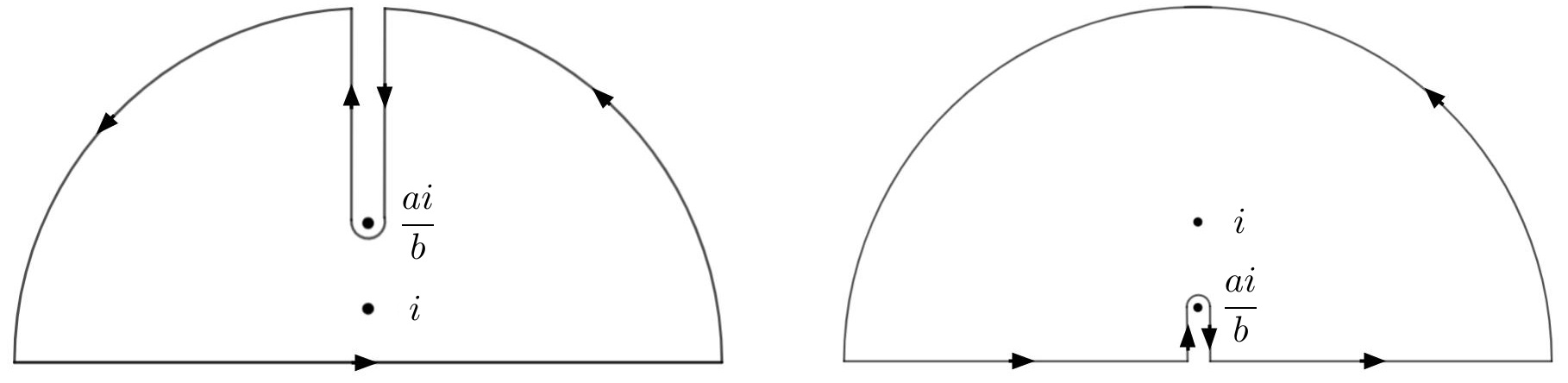}
  \vspace{-2mm}
  \caption{The contours $\mathcal{D}^{\epsilon}_R$ (on the left) and $\Gamma^{\epsilon}_R$ (on the right).}
  \label{fig:dr}
  \end{center}
\end{figure}

We define
$$\mathcal{K}^{+}_{R; n, \epsilon}=\int_{\mathcal{D}^{\epsilon}_R} \frac{\log(a^2+b^2z^2)}{(z^2+1)^n}dz.$$
For large enough $R$ and small enough $\epsilon$, we have that $\mathcal{D}^{\epsilon}_R$ encloses only one pole at $z =i$ of order $n$. 
 By the Residue theorem, \cite[Chapter 3, Theorem 2.1]{Stein}, we have
$$\mathcal{K}_{R;1, \epsilon}^+=2\pi i \cdot\left.\frac{\log(a^2+ b^2z^2)}{z+ i}\right|_{z=i } = \pi \log(a^2- b^2),$$
and 
$$\mathcal{K}_{R;2, \epsilon}^+=2\pi i \cdot\left[\frac{2b^2z}{(z+i)^2(a^2+ b^2z^2)}- \frac{2\log(a^2+b^2z^2)}{(z+i)^3}\right|_{z=i} =\frac{b^2\pi}{a^2-b^2}+ \frac{\pi}{2}\log(a^2-b^2).$$

We next let $\epsilon \rightarrow 0+$ and note that the integral along the vertical line to the right of $ai/b$ in $\mathcal{K}_{R;n, \epsilon}^+$ converges by the dominated convergence theorem to 
$$-i\int_{a/b}^{R} \frac{\log(b^2x^2-a^2)+\pi i}{(1-x^2)^n}dx,$$
while the integral along the vertical line to the left of $ai/b$ in $\mathcal{K}_{R;n, \epsilon}^+$ converges to 
$$i\int_{a/b}^{R} \frac{\log(b^2x^2-a^2)- \pi  i}{(1-x^2)^n}dx.$$
The integral over the small semi-circle around $ai/b$ is $O(\epsilon \log \epsilon^{-1})$ for all small enough $\epsilon$ and so does not contribute in the limit.
Putting it all together, we get for $R \geq 2 + a/b$
\begin{equation*}
\pi \log(a^2- b^2) = \lim_{\epsilon \rightarrow 0+} \mathcal{K}_{R;1,\epsilon}^+ = 2\pi\int_{a/b}^{R}\frac{dx}{1-x^2} + 2 \int_{0}^{R}  \frac{\log(a^2+b^2x^2)}{x^2+1}dx + O(\log R/ R),
\end{equation*}
and 
\begin{equation*}
\frac{b^2\pi}{a^2-b^2}+ \frac{\pi}{2}\log(a^2-b^2) = \lim_{\epsilon \rightarrow 0+} \mathcal{K}_{R;2, \epsilon}^+ = 2\pi\int_{a/b}^{R}\frac{dx}{(1-x^2)^2} + 2 \int_{0}^{R}  \frac{\log(a^2+b^2x^2)}{(x^2+1)^2}dx + O(\log R/ R^3),
\end{equation*}
where the constants in the big $O$ notations depend on $a,b$ alone. Letting $R \rightarrow \infty$ in the last two statements we get
\begin{equation}\label{HJ1}
\pi \log(a^2- b^2) = 2\pi\int_{a/b}^{\infty}\frac{dx}{1-x^2} + 2 \int_{0}^{\infty}  \frac{\log(a^2+b^2x^2)}{x^2+1}dx
\end{equation}
and 
\begin{equation}\label{HJ2}
\frac{b^2\pi}{a^2-b^2}+ \frac{\pi}{2}\log(a^2-b^2)  = 2\pi\int_{a/b}^{\infty}\frac{dx}{(1-x^2)^2} + 2 \int_{0}^{\infty}  \frac{\log(a^2+b^2x^2)}{(x^2+1)^2}dx.
\end{equation}
We finally note that 
$$\int_{a/b}^{\infty}\frac{dx}{1-x^2}=\frac12\log\frac{a-b}{a+b}, \mbox{ and }\int_{a/b}^{\infty} \frac{dx}{(1-x^2)^2}=\frac{ab}{2(a^2-b^2)}+\frac14\log\frac{a-b}{a+b}, $$
which together with (\ref{HJ1}) and (\ref{HJ2}) imply $(2)$ and $(4)$ in the case $a > b$.\\

{\bf \raggedleft Step 4.} In this step we show that $(2)$ and $(4)$ hold when $a \leq b$. The case $a = b$ can be obtained by letting $a \rightarrow b+$ in the $a > b$ case and invoking the dominated convergence theorem. We thus only focus on the case $a < b$.

Let $\Gamma^{\epsilon}_R$ be the keyhole semicircular contour of radius $R$ (oriented counterclockwise), which omits the point $ai/b$ (see the right part of Figure \ref{fig:dr}). As before, the width of the corridor is $2 \epsilon$ and the curve traverses a half-circular arc around $ai/b$ of radius $\epsilon$. We define
$$\mathcal{K}^{+}_{R,n,\epsilon}=\int_{\Gamma_{R}^{\e} + i \epsilon} \frac{\log(-a^2-b^2z^2)}{(z^2+1)^n}dz,$$
where $\log(x)$ denotes the principal branch of the logarithm. Note that as $b>a$, $-a^2-b^2z^2 \not\in \R_{\le 0}$ in a neighborhood of the region enclosed by $\Gamma_{R}^{\e} + i \epsilon$. Clearly, for large enough $R$ and small enough $\epsilon$, the contour $\Gamma_{R}^{\e} + i \epsilon$ encloses only one pole at $z=i$ of order $n$.
 By the Residue theorem, \cite[Chapter 3, Theorem 2.1]{Stein}, we have
$$\mathcal{K}_{R;1,\epsilon}^+=2\pi i \cdot\left.\frac{\log(-a^2- b^2z^2)}{z+ i}\right|_{z=i } = \pi \log(b^2-a^2),$$
and 
$$\mathcal{K}_{R;2,\epsilon}^+=2\pi i \cdot\left[\frac{2b^2z}{(z+i)^2(a^2+ b^2z^2)}- \frac{2\log(-a^2-b^2z^2)}{(z+i)^3}\right|_{z=i} =\frac{b^2\pi}{a^2-b^2}+ \frac{\pi}{2}\log(b^2-a^2).$$
We next let $\epsilon \rightarrow 0+$ and note that the integral along the vertical line to the right of $ai/b$ in $\mathcal{K}_{R;n, \epsilon}^+$ converges by the dominated convergence theorem to 
$$-i\int_{0}^{a/b} \frac{\log(a^2 - bx^2)-\pi i}{(1-x^2)^n}dx,$$
while the integral along the vertical line to the left of $ai/b$ in $\mathcal{K}_{R;n, \epsilon}^+$ converges to 
$$i\int_{0}^{a/b} \frac{\log(a^2 - bx^2)+ \pi  i}{(1-x^2)^n}dx.$$
The integral over the small semi-circle around $ai/b$ is $O(\epsilon \log \epsilon^{-1})$ for all small enough $\epsilon$ and so does not contribute in the limit.
Putting it all together, we get for $R \geq 2$
\begin{equation*}
\pi \log(b^2 -  a^2) = \lim_{\epsilon \rightarrow 0+}\operatorname{Re} \mathcal{K}_{R;1,\epsilon}^+ = \int_{0}^{a/b}\frac{-2\pi dx}{1-x^2} + 2 \int_{0}^{R}  \frac{\log(a^2+b^2x^2)}{x^2+1}dx + O(\log R/ R),
\end{equation*}
and 
\begin{equation*}
\frac{b^2\pi}{a^2-b^2}+ \frac{\pi}{2}\log(b^2-a^2) =  \lim_{\epsilon \rightarrow 0+}\operatorname{Re} \mathcal{K}_{R;2, \epsilon}^+ = \int_{0}^{a/b}\frac{-2\pi dx}{(1-x^2)^2} + 2 \int_{0}^{R}  \frac{\log(a^2+b^2x^2)}{(x^2+1)^2}dx + O(\log R/ R^3),
\end{equation*}
where the constants in the big $O$ notations depend on $a,b$ alone. Letting $R \rightarrow \infty$ in the last two statements we get 
\begin{equation}\label{HJ3}
	\pi \log(b^2- a^2) =  -2\pi\int_{0}^{a/b}\frac{dx}{1-x^2} + 2 \int_{0}^{\infty}  \frac{\log(a^2+b^2x^2)}{x^2+1}dx
\end{equation}
and 
\begin{equation}\label{HJ4}
	\frac{b^2\pi}{a^2-b^2}+ \frac{\pi}{2}\log(b^2-a^2) = -2\pi\int_{0}^{a/b}\frac{dx}{(1-x^2)^2} + 2 \int_{0}^{\infty}  \frac{\log(a^2+b^2x^2)}{(x^2+1)^2}dx.
\end{equation}
We finally note that 
$$\int_{0}^{a/b}\frac{dx}{1-x^2}=\frac12\log\frac{a+b}{b-a}, \mbox{ and }\int_{0}^{a/b} \frac{dx}{(1-x^2)^2}=\frac{ab}{2(b^2-a^2)}+\frac14\log\frac{a+b}{b-a}, $$
which together with (\ref{HJ3}) and (\ref{HJ4}) imply $(2)$ and $(4)$ for $a<b$.\\

\end{proof}

%
\subsection{Continuity of the $I_V^{\theta}$ functional} \label{Section7.3} In this section we prove Lemmas \ref{LemmaTech3} and \ref{LemmaTech4}, which are recalled as Lemmas \ref{S7LemmaTech3} and \ref{S7LemmaTech4}.

\begin{lemma}\label{S7LemmaTech3}[Lemma \ref{LemmaTech3}]
Suppose that Assumption \ref{as2} from Section \ref{Section2.1} holds. Fix $A > 0$ and $\alpha \in [\theta/2, \theta]$. There exist $K_1, K_2 > 0$, depending on $A$ and the constants $\theta, A_3$ from Assumption \ref{as2}, such that for all $0 \leq s_1 \leq s_2 \leq A$ one has 
\begin{equation}\label{S7contV2}
\left|\inf_{\phi \in \mathcal{A}^{\alpha}_{s_1 + \alpha}} I_V^{\theta}(\phi)\right| \leq K_1 \mbox{ and }0 \leq \inf_{\phi \in \mathcal{A}^{\alpha}_{s_1 + \alpha}} I_V^{\theta}(\phi)  - \inf_{\phi \in \mathcal{A}^{\alpha}_{s_2 + \alpha}}I_V^{\theta}(\phi)  \leq K_2 (s_2 -s_1),
\end{equation} 
where we recall that $\mathcal{A}_{s}^\alpha$ was defined in (\ref{ainf}).
\end{lemma}
\begin{proof} For clarity we split the proof into three steps. \\

{\bf \raggedleft Step 1.} In this step we prove the first part of (\ref{S7contV2}) and the first inequality in the the second part of (\ref{S7contV2}). Set 
\begin{equation}\label{S7DefK3}
K_3= \sup_{y\in [0,A+\theta]}\int_{[0,A+\theta]} |\log|x-y||dx, \mbox{ and } K_1 = 4 K_3  (A + \theta) \theta^{-1} + 2A_3 (A+ \theta) \theta^{-1},
\end{equation}
which are finite by Lemma \ref{S7LemmaTech1.1}. For any $\phi \in \mathcal{A}_{A+\theta}^{\theta/2}$ we have 
$$|I_V^{\theta} (\phi)| \leq \theta \cdot \int_0^{A+\theta} \int_0^{A+\theta}   |\log |x-y|| 4 \theta^{-2} dx dy + \int_{0}^{A + \theta} |V(x)| 2 \theta^{-1} \leq 4 K_3  (A + \theta) \theta^{-1} + 2A_3 (A+ \theta) \theta^{-1}.$$
Since $\mathcal{A}_{s_1+\alpha}^{\alpha} \subset \mathcal{A}_{A + \theta}^{\theta/2}$, we see that the last equation implies the first part of (\ref{S7contV2}) with $K_1$ as in (\ref{S7DefK3}). In addition, since $\mathcal{A}_{s_1+\alpha}^{\alpha} \subset \mathcal{A}_{s_2+\alpha}^{\alpha}$ we have
\begin{align*}
 \inf_{\phi \in \mathcal{A}_{s_2+\alpha}^{\alpha} } I_V^{\theta}(\phi) \le \inf_{\phi \in \mathcal{A}_{s_1+\alpha}^{\alpha} } I_V^{\theta}(\phi),
\end{align*}
which implies the first inequality in the the second part of (\ref{S7contV2}).\\

{\bf \raggedleft Step 2.} In this step we prove the second inequality of the second part of (\ref{S7contV2}). Let us fix $A\ge s_2>s_1\ge 0$ and $\alpha \in [\theta/2,\theta]$. Let $\Phi \in \mathcal{A}_{s_2+\alpha}^{\alpha}$ be such that $I_V^{\theta}(\Phi)=\inf_{\phi \in \mathcal{A}^{\alpha}_{s_2 + \alpha}} I_V^{\theta}(\phi).$ Let
\begin{align}\label{eq:a0}
	\alpha_0:=\inf\left\{t \in [0, s_1 + \alpha]\mid \alpha^{-1}t \ge \int_0^t \Phi(x)dx + \int_{s_1+\alpha}^{s_2+\alpha} \Phi(x)\d x \right\}.
\end{align}
Note that the above set under consideration is non-empty as $\alpha^{-1}(s_1+\alpha) \geq 1= \int_0^{s_2 + \alpha}\Phi(x)dx$. Thus $\alpha_0$ is well defined and moreover by continuity we have 
\begin{equation}\label{alphaID}
\alpha^{-1} \alpha_0 = \int_0^{\alpha_0} \Phi(x)dx + \int_{s_1+\alpha}^{s_2+\alpha} \Phi(x) dx.
\end{equation}

Consider a new density $\til\Phi$, given by
\begin{align}\label{tphi}
\til\Phi(x)=\begin{cases}
\alpha^{-1} & x\le \alpha_0 \\
\Phi(x) & x\in (\alpha_0,s_1+\alpha].
\end{cases}
\end{align}
From (\ref{alphaID}) and (\ref{tphi}) we see that $\til\Phi\in \mathcal{A}_{s_1+\alpha}^{\alpha}$.

We claim that we can find a constant $K_2 > 0$, depending on $A, \theta, A_3$ such that 
\begin{equation}\label{S7S3R1}
\left| I_V^{\theta} (\tilde{\Phi}) - I_V^{\theta} ({\Phi}) \right| \leq K_2 (s_2 -s_1).
\end{equation}
We will prove (\ref{S7S3R1}) in the next step. Here we assume its validity and conclude the proof of (\ref{S7contV2}).\\

Notice that since $\til\Phi\in \mathcal{A}_{s_1+\alpha}^{\alpha}$ we have 
$$\inf_{\phi \in \mathcal{A}^{\alpha}_{s_1 + \alpha}} I_V^{\theta}(\phi) \leq I_V^{\theta} (\tilde{\Phi}) ,$$
which together with (\ref{S7S3R1}) gives (\ref{S7contV2}).\\

{\bf \raggedleft Step 3.} For any $\phi_1,\phi_2\in \mathcal{A}_{A+\theta}^{\theta/2}$, we define
\begin{align}
\mathfrak{L}_1(\phi_1,\phi_2):=\iint_{\R^2} \log|x-y|\phi_1(x)\phi_2(y)dxdy, \quad \mathfrak{L}_2(\phi_1):=\int_{\R} V(x)\phi_1(x)dx,
\end{align}
and note that 
\begin{align}\label{eq:ivd}
I_V^{\theta}(\Phi)-I_V^{\theta}(\til{\Phi}) & = -\theta(\mathfrak{L}_1(\Phi,\Phi)-\mathfrak{L}_1(\til\Phi,\til\Phi))+(\mathfrak{L}_2(\Phi)-\mathfrak{L}_2(\til\Phi)).
\end{align}
We proceed to find appropriate bounds for the two differences on the right side. 

Using \eqref{alphaID}, \eqref{tphi} and \eqref{GenPot} we have 
\begin{equation}\label{eq:l2}
\begin{split}
|\mathfrak{L}_2(\Phi)-\mathfrak{L}_2(\til\Phi)|  & \le \int_{0}^{\alpha_0} |V(x)||\alpha^{-1}-\Phi(x)|dx+\int_{s_1+\alpha}^{s_2+\alpha}|V(x)|\Phi(x)dx \\ & \le A_3 \int_{0}^{\alpha_0} (\alpha^{-1}-\Phi(x))dx+A_3\int_{s_1+\alpha}^{s_2+\alpha}\Phi(x)dx \\ 
& = 2A_3\int_{s_1+\alpha}^{s_2+\alpha}\Phi(x)dx \le 4A_3\theta^{-1}(s_2-s_1) 
\end{split}
\end{equation}
where the last inequality follows from $\Phi\in \mathcal{A}_{s_2+\alpha}^{\alpha}$, forcing $\Phi(x)\le \alpha^{-1} \le 2\theta^{-1}$. 

Observe that for any $\phi\in \mathcal{A}_{A+\theta}^{\theta/2}$ we have
\begin{equation*}
\begin{aligned}
|\mathfrak{L}_1(\phi,\Phi)-\mathfrak{L}_1(\phi,\til{\Phi})| \le & \int_{0}^{A+\theta}\int_{s_1+\alpha}^{s_2+\alpha} |\log|x-y||\phi(x)\Phi(y)dydx \\ & \hspace{0.5cm}+\int_{0}^{A+\theta}\int_{0}^{\alpha_0} |\log|x-y||\phi(x)(\alpha^{-1}-\Phi(y))dydx.
\end{aligned}
\end{equation*}
Using (\ref{S7DefK3}), (\ref{alphaID}), \eqref{tphi}, the last equation and the fact that $\phi(x) \le 2\theta^{-1}$ we get
\begin{equation} \label{RED1}
\begin{split}
|\mathfrak{L}_1(\phi,\Phi)-\mathfrak{L}_1(\phi,\til{\Phi})|& \le  2\theta^{-1}K_3\int_{s_1+\alpha}^{s_2+\alpha}\Phi(y)dy +2\theta^{-1}K_3\int_0^{a_0} (\alpha^{-1}-\Phi(y))dy  \\ & =4\theta^{-1}K_3\int_{s_1+\alpha}^{s_2+\alpha} \Phi(y)dy \le 8\theta^{-2}K_3(s_2-s_1).
\end{split}
\end{equation}

Combining (\ref{eq:ivd}), (\ref{eq:l2}) and (\ref{RED1}) we obtain
$$\left| I_V^{\theta}(\Phi)-I_V^{\theta}(\til{\Phi})  \right| \leq 16\theta^{-2}K_3(s_2-s_1) + 4A_3\theta^{-1}(s_2-s_1),$$
which proves (\ref{S7S3R1}) with $K_2 = 16\theta^{-2}K_3 + 4A_3\theta^{-1}$. This suffices for the proof.
\end{proof}

\begin{lemma}\label{S7LemmaTech4}[Lemma \ref{LemmaTech4}]
Suppose that Assumption \ref{as2} from Section \ref{Section2.1} holds. Fix $A > 0$, $N \geq 2$, $ r\in [N/2, N]$ and $\alpha = \theta r/N$. For each $\vec{\ell} \in \mathbb{W}_r^{\theta, A N}$ let $\mu_{N,r}(\vec{\ell})$ be as in (\ref{munr}) and let $\til\mu_{N,r}(\vec{\ell})$ be the measure with density
\begin{align}\label{S7cont}
\psi_{N,r}(x) = \sum_{i=1}^r f_{\ell_i/N}(x), \mbox{  where $f_{a}(x)=\frac{N}{r\theta} \cdot\ind_{x\in [a,a+\frac{\theta}{N})}$ }.
\end{align}
There are $K_3, K_4 > 0$, depending on $A$ and $\theta, A_3, A_4$ from Assumption \ref{as2}, such that
\begin{align}\label{S7EQMol}
\left| I_V^{\theta}(\mu_{N,r}(\vec\ell))\right| \leq K_3 \mbox{ and }\left| I_V^{\theta}(\mu_{N,r}(\vec\ell)) - I_V^{\theta}(\til\mu_{N,r}(\vec\ell)) \right| \leq K_4 N^{-1}\log N.
\end{align}
\end{lemma}
\begin{proof} Throughout the proof the constants in all big $O$ notations will depend on $A$, $\theta$, $A_3$ and $A_4$ -- we will not mention this further. By definition we have
\begin{align}\label{S7mu}
I_V^{\theta}(\mu_{N,r}(\vec\ell)) & =-\frac{2\theta}{r^2} \sum_{1\le i< j\le r}\log\left|\frac{\ell_i}{N}-\frac{\ell_j}{N}\right|+\frac1r\sum_{i=1}^r V\left(\frac{\ell_i}{N}\right).
\end{align}
As $AN+N \theta \ge |\ell_i-\ell_j| \ge \theta$, $N \geq r\ge N/2$, and $|V(s)| \le A_3$ by \eqref{GenPot}, it follows that $|I_V^{\theta}(\mu_{N,r}(\vec\ell))| \le K_3$ for some $K_3>0$ depending only on $A$, $\theta$ and $A_3$. This proves the first inequality in (\ref{S7EQMol}).

In the remainder we seek to establish the second inequality in \eqref{S7EQMol}. Let us define $U_i=[\tfrac{\ell_i}{N},\tfrac{\ell_i}{N}+\tfrac{\theta}{N})$ for $i=1,2\ldots,r$. Note that $U_i$ and $U_j$ are disjoint as long as $i \neq j$. By definition 
\begin{align}\label{S7tilmu}
I_V^{\theta}(\til\mu_{N,r}(\vec{\ell}))=-\theta\sum_{i,j = 1}^r \int_{U_i}\int_{U_j}\frac{N^2}{r^2\theta^2} \log|x-y|\d x\d y +\sum_{i=1}^r\int_{U_i} V(x)\frac{N}{r\theta}\d x.
\end{align} 
We proceed to compare the expressions in \eqref{S7mu} and \eqref{S7tilmu}. 

Note that $V(x)=V(\lM+\theta)$ for $x\ge \lM+\theta$ and by \eqref{DerPot}, $V'(x)$ is integrable. Thus $V$ is absolutely continuous and by \cite[Chapter 3, Theorem 3.11]{SteinReal} we have for each $i = 1, \dots, r$ that
\begin{equation*}
\begin{split}
\left|\int_{U_i} V(x)\frac{N}{r\theta}\d x - \frac1rV(\tfrac{\ell_i}{N})\right| & \le \frac{N}{r\theta}\int_{\frac{\ell_i}{N}}^{\frac{\ell_i}{N}+\frac{\theta}{N}} \left|V(x)-V(\tfrac{\ell_i}{N})\right|\d x \\
 & \leq \frac{N}{r\theta}\int_{\frac{\ell_i}{N}}^{\frac{\ell_i}{N}+\frac{\theta}{N}}\int_{\min\{\frac{\ell_i}{N}, \lM+\theta\}}^{\min\{\frac{\ell_i}{N}+\frac{\theta}{N} , \lM+\theta\}} |V'(y)|\d y\d x =O(N^{-2}\log N),
\end{split}
\end{equation*}
where in the last equality we used \eqref{DerPot}. Summing the last inequality over $i = 1, \dots, r$ and using that $N \geq r \geq N/2$ we get
\begin{equation}\label{DIFFV}
\begin{split}
\left| \frac1r\sum_{i=1}^r V\left(\frac{\ell_i}{N}\right) - \sum_{i=1}^r\int_{U_i} V(x)\frac{N}{r\theta}\d x \right| = O(N^{-1} \log N).
\end{split}
\end{equation}

Note that by a change of variables we get for $i = 1, \dots, r$ that
$$\int_{U_i}\int_{U_i}\frac{N^2}{r^2\theta^2} \log|x-y|\d x\d y = \iint_{[\ell_i/N,(\ell_i+\theta)/N]^2} \log|x-y|\d x \d y=O(N^{-2}\log N),$$
where the latter follows from (\ref{S2Tech2}). Hence 
\begin{equation}\label{DIFF2}
\begin{split}
&-\theta\sum_{i,j = 1}^r \int_{U_i}\int_{U_j}\frac{N^2}{r^2\theta^2} \log|x-y|\d x\d y \\
&= -2\theta\sum_{1 \leq i < j \leq r} \int_{U_j}\int_{U_i}\frac{N^2}{r^2\theta^2} \log|x-y|\d x\d y + O(N^{-1} \log N).
\end{split}
\end{equation}

When $1 \leq i < j \leq r$ we have by changing variables a few times that
\begin{equation*}
\begin{split}
&\int_{U_j}\int_{U_i} \frac{N^2}{r^2\theta^2}\log (x-y)\d x \d y -  \frac{1}{r^2}\log \left( \frac{\ell_i}{N} - \frac{\ell_j}{N} \right)  = \int_{0}^{\theta/N}\int_{0}^{\theta/N} \frac{N^2}{r^2\theta^2}\log \left(1+ \frac{N(z-w)}{\ell_i - \ell_j} \right)\d z \d w\\
&= \frac{(\ell_i - \ell_j)^2}{r^2 \theta^2 } \int_{0}^{\theta/(\ell_i - \ell_j)}\int_{0}^{\theta/(\ell_i - \ell_j)} \log \left(1+ u - v \right)\d u \d v = O \left( \frac{1}{N^2 (j-i)} \right),
\end{split}
\end{equation*}
where in the last equality we used that $N \geq r \geq N/2$, $\ell_i - \ell_j \geq (j-i)\theta$ for $i < j$ and (\ref{S2Tech23}). Multiplying the last expression by $-2\theta$ and summing over $1 \leq i < j \leq r$ we conclude 
\begin{equation}\label{DIFF3}
\begin{split}
& \left|  -2\theta\sum_{1 \leq i < j \leq r} \int_{U_j}\int_{U_i}\frac{N^2}{r^2\theta^2} \log|x-y|\d x\d y  + \frac{2\theta}{r^2} \sum_{1\le i< j\le r}\log\left|\frac{\ell_i}{N}-\frac{\ell_j}{N}\right| \right|=O \left( N^{-1} \log N \right).
\end{split}
\end{equation}

Combining (\ref{DIFF2}) and (\ref{DIFF3}) we conclude 
$$\left|  -\theta\sum_{i,j = 1}^r \int_{U_i}\int_{U_j}\frac{N^2}{r^2\theta^2} \log|x-y|\d x\d y  + \frac{2\theta}{r^2} \sum_{1\le i< j\le r}\log\left|\frac{\ell_i}{N}-\frac{\ell_j}{N}\right| \right|=O \left( N^{-1} \log N \right),$$
which together with (\ref{S7mu}), (\ref{S7tilmu}) and (\ref{DIFFV}) gives the second inequality in \eqref{S7EQMol}.
\end{proof}

%
\subsection{Corrections to \cite{jo}} \label{Section7.4} As mentioned in Section \ref{Section1}, the results in the present paper overlap with the results from \cite{fe} and \cite{jo} when $\theta = 1$. We also mentioned that the large deviation principle for the upper tail in \cite{fe} and \cite{jo} was proved with an incorrect rate function $J$. In this section we explain precisely what the errors in \cite{fe} and \cite{jo} are, where they originate from and how they can be fixed. As the results in \cite{fe} and \cite{jo} are quite similar, we will focus on the latter. We begin by recalling the notation and results of interest to us from \cite{jo}.

Fix $N \in \mathbb{N}$, $\mathbb{A}_N = \{ m /N: m \in \mathbb{Z}_{\geq 0}\}$. In \cite[Section 2.2]{jo} the author considered measures on $\mathbb{A}_N^N$ of the form 
\begin{equation}\label{Jo1}
\mathbb{P}_N(x) = \frac{1}{Z_N} \cdot |\Delta_N(x) |^{\beta} \cdot \prod_{i = 1}^N \exp \left(- \frac{\beta N}{2} V_N(x_i) \right),
\end{equation}
where $\beta > 0$ is fixed, $Z_N$ is a normalization constant, $\Delta_N(x) = \prod_{1 \leq i < j \leq N} (x_j - x_i)$ is the Vandermonde determinant, and $V_N$ are continuous functions on $[0, \infty)$ such that 
\begin{equation}\label{Jo2}
V_N(t) \geq (1 + \xi) \log(t^2 + 1),
\end{equation}
for all $t \geq T$ and $N \geq N_0$, where $T , \xi, N_0 > 0$ are fixed positive constants. In addition, it was assumed that $V_N(t) \rightarrow V(t)$ uniformly over compact subsets of $[0, \infty)$ for some function $V$, which by the uniform convergence is also continuous and satisfies (\ref{Jo2}).

Let us define 
$$k_V(x,y) = \log |x-y|^{-1} + \frac{1}{2} V(x) + \frac{1}{2} V(y),$$
for $\phi \in \mathcal{A}_{\infty}^1$ (as in (\ref{ainf})) we let 
$$I_V[\phi] = \int_0^{\infty} \int_0^{\infty} k_V(x,y) \phi(x) \phi(y) dx dy.$$
By Lemma \ref{iv} we know that $I_V[\phi]$ has a unique minimizer on $\mathcal{A}_{\infty}^1$, which we denote by $\phi_V$ and we put $F_V = I_V[\phi_V]$. We also know that $\phi_V$ is compactly supported and we let $b_V$ denote the rightmost endpoint of its support. Finally, we define 
\begin{equation}\label{Jo3}
\begin{split}
&{J}(t) = \inf_{\tau \geq t} \int_0^{\infty} k_V(\tau,x) \phi_V(x) dx - F_V \mbox{, and } \\
&\tilde{J}(t) =\inf_{\tau \geq t} \int_0^{\infty} k_V(\tau,x) \phi_V(x) dx -\int_0^{\infty} k_V(b_V,x) \phi_V(x) dx .
\end{split}
\end{equation}
With the above notation, the following result appears as Theorem 2.2 in \cite{jo} (and also Theorem 4.2 in \cite{fe}).
\begin{theorem} If $J(t) > 0$ for $t > b_V$ then 
\begin{equation}\label{Jo4}
\lim_{N \rightarrow \infty} \frac{1}{N} \log  \mathbb{P}_N( \max_k x_k > t)  \mbox{ ``$=$''} - \beta J(t).
\end{equation}
\end{theorem}
We put quotation marks in (\ref{Jo4}), because the equation is not correct as written but does become correct if one replaces ${J}(t)$ with $\tilde{J}(t)$ from (\ref{Jo3}). In both \cite{fe} and \cite{jo} it was claimed that $J(t) = \tilde{J}(t)$ for $t \in [b_V, \infty)$: in \cite{jo} this happens in \cite[Equation (4.21)]{jo} and in \cite{fe} it happens in an unnumbered equation on page 37. In both papers the equality of $J(t)$ and $\tilde{J}(t)$ is claimed to be a consequence of the variational characterization of $\phi_V$ (this is Proposition \ref{ch-pv} in the present paper). Below we first explain why in general one does not have $J(t) = \tilde{J}(t)$ for $t \in [b_V, \infty)$, and why we believe the error was made. Afterwards, we explain how one needs to (very mildly) modify the arguments in \cite{jo} to prove (\ref{Jo4}) with $\tilde{J}$ in place of $J$. \\

The first thing to observe is that by the continuity of $\int_0^{\infty} k_V(\tau,x) \phi_V(x) dx$ (this is a consequence of Lemma \ref{S1GJ}) we have that the $\lambda$ in Proposition \ref{ch-pv} (see also \cite[Theorem 2.1]{ds97}) is in fact equal to $\int_0^{\infty} k_V(b_V,x) \phi_V(x) dx$. In \cite{jo} the analogue of Proposition \ref{ch-pv} can be found as Proposition 6.1 and there it is claimed that $\lambda = F_V$. The last statement is true if $\phi_V(x) \in [0, 1)$ (Lebesgue) almost everywhere, i.e. if the set of points $S = \{x \in [0, \infty): \phi(x) = 1 \}$ has measure $0$. To see this, observe that if $S$ has measure $0$, we have by Proposition \ref{ch-pv} that 
$$\int_0^{\infty} k_V(\tau,x) \phi_V(x) dx = \lambda = \int_0^{\infty} k_V(b_V,x) \phi_V(x) dx \mbox{ a.s. on the support of $\phi$}, $$
and so 
$$ \int_0^{\infty} k_V(b_V,x) \phi_V(x) dx = \lambda = \int_{0}^\infty \lambda \phi_V(y) dy = \int_{0}^\infty \int_0^{\infty} k_V(y,x)  \phi_V(x) \phi_V(y)  dx dy = F_V.$$
In this case $J(t) = \tilde{J}(t)$ for $t \in [b_V, \infty)$. 

However, if  the measure of the set $S$ is positive and $\int_0^{\infty} k_V(\tau,x) \phi_V(x) dx < \lambda$ for $\tau \in S$ (this is for example the case for $\phi^{\operatorname{Jack}}$ from Lemma \ref{JackMin} when $t < 1$ and $\theta = 1$), then we would have by Proposition \ref{ch-pv} 
$$ \int_0^{\infty} k_V(b_V,x) \phi_V(x) dx = \lambda = \int_{0}^\infty \lambda \phi_V(y) dy > \int_{0}^\infty \int_0^{\infty} k_V(y,x)  \phi_V(x) \phi_V(y)  dx dy = F_V.$$
In this case $J(t) >  \tilde{J}(t)$ for $t \in (b_V, \infty)$. 

The above work shows that if $S$ has measure $0$, we indeed have $J(t) = \tilde{J}(t)$, but in general $J(t) \geq  \tilde{J}(t)$ for $t \in [b_V, \infty)$ and the inequality can be strict. In the case of continuous log-gases one does not have the restriction on the density of $\phi_V$ in the minimization problem of Lemma \ref{iv}, which comes from the discreteness of the model in (\ref{Jo1}). Consequently, for continuous log-gases one always has $J(t) =  \tilde{J}(t)$ and we believe that this mistaken analogy is the source of the error in both \cite{fe} and \cite{jo}.  \\

In the remainder of this section we explain how to modify the arguments in \cite{jo} to obtain (\ref{Jo4}) with $\tilde{J}$ in place of $J$. Let us denote for $M \in \mathbb{N}$
$$Z_{M,N} = \sum_{x \in \mathbb{A}_N^M} |\Delta_M(x)|^{\beta} \cdot \prod_{i = 1}^M \exp \left( - \frac{\beta N}{2} V_N(x_i) \right).$$
Then \cite[Lemma 4.5]{jo} reads (see \cite[Equation (4.22)]{jo})
\begin{equation}\label{Jo5}
\limsup_{N \rightarrow \infty} \frac{1}{N} \log \frac{Z_{N-1,N}}{Z_{N,N}} \mbox{ ``$\leq$''} \beta \cdot \left( F_V - \frac{1}{2} \int_0^\infty V(s) \phi_V(s) ds \right).
\end{equation}
Equation (\ref{Jo5}) is false, and instead should be replaced with 
\begin{equation}\label{Jo6}
\limsup_{N \rightarrow \infty} \frac{1}{N} \log \frac{Z_{N-1,N}}{Z_{N,N}}\leq \beta \cdot \left( \frac{1}{2} V(b_V) + \int_0^{\infty} \log |b_V -  t|^{-1} \phi_V(t) dt \right).
\end{equation}
Equation (\ref{Jo6}) is in fact the one proved in \cite[Lemma 4.5]{jo} (see the top of page 465) and only in the end is the right side of (\ref{Jo6}) (mistakenly) replaced with the right side of (\ref{Jo5}) by invoking the (incorrect) \cite[Equation (4.21)]{jo}. If one uses (\ref{Jo6}) instead of (\ref{Jo5}) in the arguments on pages 465-466 in \cite{jo}, one would obtain in place of \cite[Equation (4.42)]{jo} the following
\begin{equation}\label{Jo7}
\limsup_{N \rightarrow \infty} \frac{1}{N} \log \mathbb{P}_N( \max_{k} x_k > t) \leq - \beta \cdot \tilde{J}(t) \mbox{ for  $t > b_V$}.
\end{equation}
 This proves one half of (\ref{Jo4}) with $\tilde{J}$ in place of $J$. (We mention that \cite[Equation (4.42)]{jo} should have $F_V$ in place of $\Phi_V$ -- this is a small typo in the paper.)

For the other half, \cite[Lemma 4.6]{jo} is used, which reads (see \cite[Equation (4.46)]{jo})
\begin{equation}\label{Jo8}
\liminf_{N \rightarrow \infty} \frac{1}{N} \log \frac{Z_{N-1,N}}{Z_{N,N}} \geq \beta \cdot \left( F_V - \frac{1}{2} \int_0^\infty V(s) \phi_V(s) ds \right).
\end{equation}
While the last equation is valid it is not the correct matching lower bound we need, and instead this equation should be replaced with 
\begin{equation}\label{Jo9}
\liminf_{N \rightarrow \infty} \frac{1}{N} \log \frac{Z_{N-1,N}}{Z_{N,N}} \geq \beta \cdot \left( \frac{1}{2} V(b_V) + \int_0^{\infty} \log |b_V -  t|^{-1} \phi_V(t) dt \right).
\end{equation}
Equation (\ref{Jo9}) is in fact established on page 468, where it is shown that 
\begin{equation*}
\liminf_{N \rightarrow \infty} \frac{1}{N} \log \frac{Z_{N-1,N}}{Z_{N,N}} \geq - \frac{\beta}{2} \int_0^{\infty} V(s) \phi_V(s) ds + \beta \sup_{\tau \geq b_V} \int_0^{\infty} k_V(\tau,x) \phi_V(x) dx .
\end{equation*}
Since by Proposition \ref{ch-pv} we have $\sup_{\tau \geq b_V} \int_0^{\infty} k_V(\tau,x) \phi_V(x) dx = \lambda = \int_0^{\infty} k_V(b_V,x) \phi_V(x) dx$, we see that the last equation implies (\ref{Jo9}). If one uses (\ref{Jo9}) in place of (\ref{Jo8}) on page 468 in \cite{jo} one obtains 
\begin{equation}\label{Jo10}
\liminf_{N \rightarrow \infty} \frac{1}{N} \log \mathbb{P}_N( \max_{k} x_k > t) \geq - \beta \cdot \tilde{J}(t).
\end{equation}
Equations (\ref{Jo7}) and (\ref{Jo10}) together imply (\ref{Jo4}) with $\tilde{J}$ in place of $J$.

To summarize, (\ref{Jo4}) holds with $\tilde{J}$ in place of $J$ and all the work is already present in \cite{jo} -- one needs to simply replace the statements of Lemmas 4.5 and 4.6 in that paper with equations (\ref{Jo6}) and (\ref{Jo9}) respectively and use these two equations when these lemmas are invoked to prove (\ref{Jo7}) and (\ref{Jo10}) in place of (4.42) and the unnumbered equation at the end of Section 4 in \cite{jo}.

\bibliographystyle{alphaabbr}		
\bibliography{discrete}
\end{document}